\documentclass[reqno]{amsart}
\usepackage{comment,amssymb,latexsym,upref,enumerate,fouridx}
\usepackage{mathrsfs,color}
\usepackage{mathtools}
\usepackage{centernot}
\usepackage{longtable}
\usepackage[colorlinks,linkcolor=blue,citecolor=blue]{hyperref}

\numberwithin{equation}{section}

\theoremstyle{plain}
\newtheorem{thm}{Theorem}[section]
\newtheorem{lem}[thm]{Lemma}
\newtheorem{prop}[thm]{Proposition}

\theoremstyle{definition}
\newtheorem{Def}[thm]{Definition}

\theoremstyle{remark}
\newtheorem{rem}[thm]{Remark}

\newtheorem{Not}[thm]{Notation}

\newcommand{\ep}{\epsilon}
\newcommand{\epu}{\underline{\epsilon}{}}

\newcommand{\Piu}{\underline{\Pi}{}}
\newcommand{\Pibru}{\underline{\Pibr}{}}

\newcommand{\pu}{\underline{p}{}}

\newcommand{\mul}{\underline{m}{}}
\newcommand{\vu}{\underline{v}{}}
\newcommand{\Nu}{\underline{N}{}}

\newcommand{\Pbbbr}{\breve{\mathbb{P}}}
\newcommand{\zetau}{\underline{\zeta}{}}
\newcommand{\gu}{\underline{g}{}}
\newcommand{\Ecr}{\mathring{\Ec}{}}
\newcommand{\Bcr}{\mathring{\Bc}{}}
\newcommand{\Lct}{\tilde{\Lc}{}}
\newcommand{\gammah}{\hat{\gamma}{}}
\newcommand{\thetau}{\underline{\theta}{}}
\newcommand{\thetahu}{\underline{\hat{\theta}}{}}
\newcommand{\Ecu}{\underline{\Ec}{}}

\newcommand{\Gammau}{\underline{\Gamma}{}}
\newcommand{\Wcu}{\underline{\Wc}{}}
\newcommand{\ru}{\underline{r}{}}
\newcommand{\au}{\underline{a}{}}
\newcommand{\sigmau}{\underline{\sigma}{}}

\newcommand{\Upsilonch}{\check{\Upsilon}{}}
\newcommand{\Upsilonchu}{\underline{\check{\Upsilon}}{}}
\newcommand{\ellu}{\underline{\ell}{}}

\newcommand{\hu}{\underline{h}{}}
\newcommand{\su}{\underline{s}{}}

\newcommand{\psig}{\grave{\psi}{}}
\newcommand{\phig}{\grave{\phi}{}}
\newcommand{\thetahug}{\grave{\thetahu}{}^0{}}
\newcommand{\zetaug}{\grave{\zetau}{}}
\newcommand{\Uttg}{\grave{\Utt}{}}
\newcommand{\Vttg}{\grave{\Vtt}{}}
\newcommand{\Zttg}{\grave{\Ztt}{}}
\newcommand{\varthetag}{\grave{\vartheta}{}}
\newcommand{\Utta}{\acute{\Utt}{}}
\newcommand{\Vtta}{\acute{\Vtt}{}}
\newcommand{\Ztta}{\acute{\Ztt}{}}
\newcommand{\varthetaa}{\acute{\vartheta}{}}

\newcommand{\phia}{\acute{\phi}{}}
\newcommand{\thetahua}{\acute{\thetahu}{}^0{}}
\newcommand{\zetaua}{\acute{\zetau}{}}
\newcommand{\rhou}{\underline{\rho}{}}

\newcommand{\Ntt}{\mathtt{N}{}}
\newcommand{\Gtt}{\mathtt{G}{}}

\newcommand{\hcbr}{\breve{\hc}{}}
\newcommand{\Hct}{\widetilde{\Hc}{}}
\newcommand{\Hctu}{\underline{\Hct}{}}
\newcommand{\psiu}{\underline{\psi}{}}
\newcommand{\psihu}{\underline{\hat{\psi}}{}}
\newcommand{\Upsilonu}{\underline{\Upsilon}{}}
\newcommand{\Icu}{\underline{\Ic}{}}
\newcommand{\gammahu}{\underline{\hat{\gamma}}{}}
\newcommand{\Lcu}{\underline{\Lc}{}}
\newcommand{\elltu}{\underline{\tilde{\ell}}{}}
\newcommand{\Xcr}{\mathring{\mathcal{X}}{}}
\newcommand{\Eg}{\grave{E}{}}
\newcommand{\Bttg}{\grave{\Btt}{}}

\newcommand{\Xttg}{\grave{\Xtt}{}}
\newcommand{\Fttg}{\grave{\Ftt}{}}
\newcommand{\betag}{\grave{\beta}{}}

\newcommand{\lambdag}{\grave{\lambda}{}}
\newcommand{\Qttg}{\grave{\Qtt}{}}
\newcommand{\Pttg}{\grave{\Ptt}{}}
\newcommand{\Gttg}{\grave{\Gtt}{}}
\newcommand{\Bscg}{\grave{\Bsc}{}}
\newcommand{\Rca}{\acute{\Rc}{}}

\newcommand{\Ncg}{\grave{\Nc}{}}
\newcommand{\Bttt}{\tilde{\Btt}{}}
\newcommand{\Xttt}{\tilde{\Xtt}{}}
\newcommand{\Gttt}{\tilde{\Gtt}{}}
\newcommand{\Fttt}{\tilde{\Ftt}{}}
\newcommand{\Pttt}{\tilde{\Ptt}{}}
\newcommand{\Lttt}{\tilde{\Ltt}{}}

\newcommand{\Xttbr}{\breve{\Xtt}{}}
\newcommand{\Gttbr}{\breve{\Gtt}{}}
\newcommand{\Fttbr}{\breve{\Ftt}{}}

\input Tdef.def

\begin{document}

\title[Dynamical relativistic liquid bodies]{Dynamical relativistic liquid bodies}

\author[T.A. Oliynyk]{Todd A. Oliynyk}
\address{School of Mathematical Sciences\\
9 Rainforest Walk\\
Monash University, VIC 3800\\
Australia}
\email{todd.oliynyk@monash.edu}

\begin{abstract}
\noindent We establish the local-in-time
existence of solutions to the relativistic Euler equations representing
dynamical liquid bodies in vacuum.
\end{abstract}

\maketitle


\section{Introduction} \label{intro}

Over the past two decades, a number of results that guarantee the local-in-time existence and uniqueness of solutions to the (non-relativistic) Euler equations that represent dynamical fluid bodies in vacuum have been established \cite{Coutand_et_al:2013,CoutandShkoller:2007,CoutandShkoller:2012,DisconziEbin:2016,GuLei:2016,JangMasmoudi:2010,Lindblad:2005b,Lindblad:2005a,ShatahZeng:2011,Wu:1997,Wu:1999,ZhangZhang:2006}.  Recently, the first step towards extending these existence results to the relativistic setting have been taken. Specifically,
a priori estimates for solutions to the relativistic Euler equations that satisfy the vacuum boundary conditions have
been established for liquids \cite{Ginsberg:2018c,Oliynyk:Bull_2017} on prescribed spacetimes and for gases \cite{Hadzic_et_al:2015,Jang_et_al:2016}
on Minkowski space.
Unlike most initial boundary value problems where well-known approximation schemes can be used to obtain local-in-time existence and uniqueness results from a priori estimates,
it is highly non-trivial to obtain existence from a priori estimates for dynamical fluid bodies in vacuum, whether relativistic or not.
The main reason for the difficulty is the presence of the free fluid-matter vacuum boundary, which
make it necessary to exploit much of the structure of the Euler equations in order to derive a priori estimates. This makes
the use of approximation methods problematic since any approximate equation would have to possess all of
the essential structure of the Euler equations used to derive estimates, and to find such approximations have proved to be very difficult.

The a priori estimates  from \cite{Oliynyk:Bull_2017} were derived  using a wave formulation of the Euler equations consisting
of a fully non-linear system of wave equations in divergence form together with non-linear acoustic boundary conditions. This system of wave equations and acoustic boundary conditions were obtained by differentiating the Lagrangian representation of the Euler equations and vacuum boundary conditions in time and adding constraints that vanish identically on solutions. A priori
estimates, without derivative loss, were then established using an existence and uniqueness theory that was
developed for linear systems of wave equations with acoustic boundary conditions  together with Sobolev-Moser type inequalities to handle the non-linear estimates. This approach to deriving a priori estimates suggests a two step strategy to obtain the local-in-time existence,
without derivative loss,
of solutions to the relativistic Euler equations on prescribed spacetimes that satisfy the vacuum boundary conditions.
The first step is to show that the constraints used to derive the wave formulation propagate;
that is, to show that if the constraints, when evaluated on a solution of the wave formulation, vanish on the initial hypersurface, then
they must vanish identically everywhere on the world tube defined by support of the solution.
The second step is to establish the local-in-time existence
and uniqueness of solutions to the wave formulation, which would follow from
a standard iteration argument using the linear theory and Sobolev-Moser inequalities developed
in \cite{Oliynyk:Bull_2017}. On a prescribed spacetime, such solutions would, by step one,  then determine solutions of
the relativistic Euler equation that satisfy the vacuum boundary conditions thereby establishing the local-in-time
existence of solutions representing dynamical relativistic liquid bodies.

The main purpose of this article is to carry out this two-step strategy\footnote{For technical reasons,
we do not use the wave formulation from \cite{Oliynyk:Bull_2017}, but instead, we consider a related version that
differs by a choice of constraints. This new wave formulation involves an additional
scalar field that solves a wave equation with Dirichlet boundary conditions.} and establish the local-in-time existence of solutions to the relativistic
Euler equations that represent dynamical liquid bodies. The precise statement of this result is given in Theorem \ref{locexistB}.

\subsection{Related and prior work}
In the non-relativistic setting, a number of different approaches have been used to establish the local-in-time existence
and uniqueness of solutions to the Euler equations that satisfy vacuum boundary conditions.
Important early work was carried out by S. Wu who, in the articles \cite{Wu:1997,Wu:1999}, solved the water waves
problem by establishing the local-in-time existence of solutions for an irrotational incompressible liquid in
vacuum. This work improved on the earlier results \cite{Craig:1985,Nalimov:1974,Yosihara:1982}, where existence for water
waves was established under restrictions on the
initial data. Wu's results were later generalized, using a Nash-Moser scheme combined with extensions to earlier a priori estimates
derived in \cite{ChristodoulouLindblad:2000}, by H. Lindblad to allow for vorticity in \cite{Lindblad:2005a}. This work
was subsequently extended to compressible liquids in \cite{Lindblad:2005b}.

Due to the reliance on Nash-Moser, Lindblad's existence results involve derivative loss. By using an approximation based
on a parabolic regularization that reduces in the limit of vanishing viscosity to the Euler equations, the authors of
\cite{CoutandShkoller:2007} were
able to establish, without derivative loss, a local-in-time existence result for incompressible fluid bodies, which they later
generalized to compressible
gaseous and liquid bodies in \cite{CoutandShkoller:2012} and \cite{Coutand_et_al:2013}, respectively. Existence for
compressible gaseous bodies
was also established using a different approach in \cite{JangMasmoudi:2010}. 
Recently, a new approach to establishing the local-in-time existence of solutions for compressible, self-gravitating, liquid bodies
has been developed in \cite{Ginsberg_et_al:2019}. 
For other related results in the non-relativistic
setting, which includes other approaches to a priori estimates, existence on small and large time scales,  and coupling to Newtonian gravity, see the works
\cite{Alazard_et_al:2011,Alvarez_SamaniegoLannes:2008,Bieri_et_al:2017,BrauerKarp:2018b,BrauerKarp:2018a,Coutand_et_al:2010,DisconziLuo:2019,Germain_etal:2012,Ginsberg:2018a,Ginsberg:2018b,HadzicJang:2018a,
HadzicJang:2018b,Hunter_et_al:2016,IonescuPusateri:2015, LindbladLuo:2018,LindbladNordgren:2009,
Luo:2017,Luo_et_al:2014,Makino:2017,Schweizer:2005,ShatahZeng:2008,Wu:2011} and references cited therein.

In the relativistic setting, much less is known. For gaseous relativistic bodies, the only existence result in the most
physically interesting case where
the square of the sound speed goes to zero like the distance to the boundary that we are
aware of is \cite{Oliynyk:CQG_2012}, which is applicable to fluids on 2 dimensional Minkowsi space. However, based on earlier work
by Makino \cite{Makino:1986} in the non-relativistic setting, Rendall established the existence of solutions to the Einstein-Euler equations
representing self-gravitating gaseous bodies that are undergoing collapse \cite{Rendall:1992}; see also \cite{BrauerKarp:2011,BrauerKarp:2014} for recent progress in this direction.  For relativistic liquids on Mikowski space, a local-in-time
existence result involving derivative loss has been established in \cite{Trakhinin:2009} using a symmetric hyperbolic formulation in conjunction
with a Nash-Moser scheme.

We note that the use of constraints to establish the existence of solutions
has a long history in Mathematical Relativity with, perhaps, the most well known and important
application being the proof of the existence and geometric uniqueness of solutions to the Einstein equations, which was
first established in \cite{ChoquetBruhat:1952}; see also
\cite{Andersson_et_al:2014,ChoquetFriedrich:2006,FriedrichNagy:1999,Kreiss_etal:2007,Kreiss_etal:2009,Rendall:1992} for related
work when boundaries are present. We would also like to add that the work presented
here and in \cite{Oliynyk:Bull_2017} was inspired by the constraint propagation approach to the relativistic
fluid body problem from \cite{Friedrich:1998}.

\subsection{Initial boundary value problem for relativistic liquid bodies\label{introIBVP}}
In order to define the initial boundary value problem (IBVP) for a relativistic fluid, we first need to introduce some geometric
structure starting with a $4$-dimensional manifold\footnote{In this article, we, for simplicity, restrict our considerations
to the physical spacetime dimension of $d=4$. However, the results presented in this article are valid
for all spacetime dimensions $d\geq 3$.} $M$ equipped with
a prescribed smooth Lorentzian metric
\begin{equation} \label{metdef}
g = g_{\mu\nu} dx^\mu dx^\nu
\end{equation}
of signature $(-,+,+,+)$. In the following, we let $\nabla_\mu$ denote the Levi-Civita connection of $g_{\mu\nu}$ and
$\Omega_0 \subset M$ be a bounded, connected\footnote{There is no need for $\Omega_0$ to be either connected or bounded; non-connected $\Omega_0$
correspond to multiple fluid bodies, while non-bounded components represent unbounded fluid bodies.}
spacelike hypersurface  that is properly 
contained and open in a spacelike hypersurface $\Sigma\subset M$ so that $\overline{\Omega}_0=\Omega_0
\cup \del{}\Omega_0$ and $\Omega_0\cap \del{}\Omega_0 = \emptyset$. We further assume that the boundary
$\del{}\Omega_0$ is smooth.

The manifold $\Omega_0$ defines the initial hypersurface where we specify initial data for the fluid. The proper energy density $\rho$ of the fluid is initially non-zero on $\Omega_0$ and vanishes outside, that is on $\Sigma\setminus \Omega_0$.
The initial hypersurface $\Omega_0$ forms the ``bottom'' of the world tube $\Omega_T\subset M$
defined by the motion of the fluid body through spacetime, which is diffeomorphic to the cylinder $[0,T]\times \Omega_0$.
We let $\Gamma_T$ denote the timelike boundary of $\Omega_T$, which is diffeomorphic to
$[0,T]\times \del{}\Omega_0$. By our conventions, $\Gamma_0 = \del{}\Omega_0$.

The motion of the fluid body on the prescribed space time $(M,g)$ is governed by the relativistic Euler equations given by
\begin{equation} \label{eulint1}
\nabla_\mu T^{\mu\nu} = 0
\end{equation}
where
\begin{equation*}
T^{\mu \nu} = (\rho + p)v^\mu v^\nu + p g^{\mu \nu}
\end{equation*}
is the stress energy tensor, $v^\mu$ is the fluid $4$-velocity normalized
by
\begin{equation*} 
g_{\mu\nu}v^\mu v^\nu = -1,
\end{equation*}
$p$ is the pressure, and $\rho$, as above, is the proper energy density of the fluid.
Projecting \eqref{eulint1} into the subspaces parallel and orthogonal to $v^\mu$ yields
the following well-known form of the relativistic Euler equations
\begin{align}
v^\mu \nabla_\mu \rho + (\rho+p)\nabla_\mu v^\mu & =  0, \label{eulint3.1}\\
(\rho + p)v^\mu \nabla_\mu v^\nu + h^{\mu\nu}\nabla_\mu p & = 0, \label{eulint3.2}
\end{align}
where
\begin{equation} \label{pidef}
h_{\mu\nu} = g_{\mu\nu} + v_\mu v_\nu
\end{equation}
is the induced positive  definite metric on the subspace orthogonal to $v^\mu$.

In this article, we restrict our attention to
fluids with a barotropic equation of state of the form
\begin{equation*}
\rho = \rho(p)
\end{equation*}
where $\rho$ satisfies\footnote{In reality, we only require that the equation of state is defined
on the physical domain $p \geq 0$. We extend the range to include $p<0$ as a matter of convenience only.}
\begin{gather}
\rho \in C^{\infty}(-\infty,\infty), \quad  \rho(0) = \rho_0, \notag 
\\
 \intertext{and}
 \frac{1}{s^2_1} \leq \rho'(p) \leq \frac{1}{s^2_0}, \quad -\infty< p < \infty, \label{eosA.2}
\end{gather}
for some positive constants $0<\rho_0$, and $0<s_0<s_1<1$. Since the square of
the sound speed is given by
\begin{equation} \label{s2}
s^2 = \frac{1}{\rho'(p)},
\end{equation}
the assumption \eqref{eosA.2} implies that
$0<s_0^2 \leq s^2 \leq s_1^2  < 1$,
or in other words, that the sound speed is bounded away from zero and strictly less than the speed of light.
Note also that  \eqref{eosA.2} guarantees that the density $\rho$ satisfies
$\frac{1}{s_1^2}p + \rho_0 \leq \rho(p)$ for  $p \in (-\infty, \infty)$, which implies that
\begin{equation}
\rho(p) \geq \tilde{\rho}_0 > 0, \quad s_1^2(\tilde{\rho}_0-\rho_0)\leq  p < \infty, \label{eosA.3}
\end{equation}
for any constant $\tilde{\rho}_0\in (0,\rho_0)$. In particular, this implies that there exists an $\epsilon>0$
such that the density $\rho$ is bounded away from zero for $p \geq -\epsilon$.

The boundary of the world-tube $\Omega_T$, which separates the fluid body from the vacuum region, is defined
by the vanishing of the pressure, i.e. $p|_{\Gamma_T}=0$. By \eqref{eosA.3},
the proper energy density does not vanish at the boundary, and hence, there is a jump in the proper energy density across
$\Gamma_T$. Fluids of this type are referred to as \textit{liquids}. In addition to the vanishing of the pressure, the condition $v|_{\Gamma_T}\in T\Gamma_T$ must be satisfied
to ensure that no fluid moves across $\Gamma_T$. These two conditions
form the vacuum boundary conditions satisfied by freely evolving fluid bodies. Collecting these boundary conditions
together with the evolution equations \eqref{eulint3.1}-\eqref{eulint3.2}, the complete Initial Boundary Value Problem (IBVP) for a relativistic liquid body on the spacetime $(M,g)$ is:
\begin{align}
v^\mu \nabla_\mu \rho + (\rho+p)\nabla_\mu v^\mu & =  0 && \text{in $\Omega_T$,} \label{ibvp.1}\\
(\rho + p)v^\mu \nabla_\mu v^\nu + h^{\mu\nu}\nabla_\mu p & = 0 &&\text{in $\Omega_T$,} \label{ibvp.2}\\
p &=0 &&\text{in $\Gamma_T$,} \label{ibvp.3} \\
n_\nu v^\nu &= 0 &&\text{in $\Gamma_T$,}  \label{ibvp.4}\\
(\rho,v^\mu) &= (\tilde{\rho},\tilde{v}^\mu) &&\text{in $\Omega_0$,} \label{ibvp.5}
\end{align}
where $(\tilde{\rho},\tilde{v}^\mu)$ is the initial data, and $n_\nu$ is the outward pointing unit conormal to $\Gamma_T$.

\begin{rem}
The boundary condition \eqref{ibvp.4} guarantees that the cylinder $\Omega_T$ is invariant under the flow of $v$. Thus, without loss of generality, 
we will assume for the remainder of the article that $\Omega_T$ is of the form
\begin{equation*}
\Omega_T = \bigcup_{0\leq t \leq T} \Gc_t (\Omega_0)
\end{equation*}
where $\Gc_\tau$ is the flow of $v$, that is,
$\frac{d\;}{d t}\Gc_t = v\circ \Gc_t$.
The timelike boundary of the cylinder is then given by
\begin{equation*}
\Gamma_T =  \bigcup_{0\leq t \leq T} \Gc_t (\del{}\Omega_0),
\end{equation*}
and an outward pointing conormal to $\Gamma_T$ can be determined by using the flow to transport an initial outward conormal defined at $t=0$ to all of $\Gamma_T$. Normalizing this conormal then yields the unique, outward pointing unit conormal to $\Gamma_T$; see \eqref{lemBa.2} and \eqref{lemBa.3} for
details.
\end{rem}

\subsection{Overview}
We fix our notation and conventions used throughout this article in Section \ref{prelim} and in Appendix \ref{DG}, where a
number of definitions and formulas from differential geometry are collected.  In Sections \ref{primaux} and \ref{consEul}, we define
the primary fields and constraints, respectively, that will be used in our wave formulation of the relativistic liquid body
IBVP. The Eulerian representation of our wave formulation, which includes the freedom to add constraints, is
introduced in Section \ref{waveibvpEul}. We then state and prove Theorem \ref{cthm} in Section \ref{constprop}.
Informally, this theorem guarantees the constraints, when evaluated on solutions of our wave formulation,
vanish in $\Omega_T$ provided they vanish initially on $\Omega_0$, i.e. they propagate, and moreover, that solutions
to our wave formulation for which constraints vanish correspond to solutions of the relativistic liquid IBVP.
We then, in Section \ref{choice}, make a particular choice of the constraints that appear in the evolution equations
and boundary conditions that define our wave formulation in order to bring the total system into form that is favorable for
establishing the existence and uniqueness of solutions; see Proposition \ref{cprop} for details. 
In Section \ref{lag},  we introduce Lagrangian coordinates
and express the system of equations and boundary conditions from Section \ref{choice} in these coordinates. We then
establish the local-in-time existence and uniqueness of solutions to this system in this section; see Theorem \ref{locexistA}
for the precise statement. The proof of Theorem \ref{locexistA} is based on an iteration method that 
relies on the linear existence and uniqueness theory for wave equations developed in Appendix \ref{LWE}
and the Sobolev-Moser type estimates from Appendix \ref{calculus}.  In the final section, Section \ref{locexistREVBC}, we 
deduce the
local-in-time existence
of solutions to the relativistic Euler with vacuum boundary conditions from Proposition \ref{cprop} and Theorem \ref{locexistA}. This
existence result is the main result of this article and the precise statement is given in Theorem \ref{locexistB}.

\section{Preliminaries\label{prelim}}
In this section, we fix our notation and conventions that we will employ throughout this article; see also
Appendix \ref{DG} where we collect a number of definitions and formulas from differential geometry. 

\subsection{Index of notation} An index containing frequently used definitions and non-standard notation can be found in Appendix \ref{index}.

\subsection{Indexing conventions\label{indexing}}
We will need to index various objects. The conventions that we will employ are as follows:
\medskip
\begin{center}
\begin{tabular}{|l|c|c|l|} \hline
Alphabet & Examples & Index range & Index quantities  \\ \hline
Lowercase Greek & $\mu,\nu,\gamma$ & $0,1,2,3$ & spacetime coordinate components \\ \hline
Uppercase Greek &$\Lambda,\Sigma,\Omega$ & $1,2,3$, & spatial coordinate components \\ \hline
Lowercase Latin &$i,j,k $ & $0,1,2,3$ & spacetime frame components \\ \hline
Uppercase Latin &$I,J,K$ & $1,2,3$, & spatial frame components \\ \hline
Lowercase Calligraphic &$\ic,\jc,\kc$ & 0,1,2 & spacetime boundary frame components\\ \hline
Uppercase Calligraphic &$\Ic,\Jc,\Kc$ & 1,2 & spatial boundary frame components \\ \hline
\end{tabular}
\end{center}

\medskip

\noindent \textit{Note:} In Appendices \ref{calculus} to \ref{LWE}, we work in general dimensions, and so, there lowercase Greek letters will run from $0$ to $n$, while uppercase Greek
indices will run from $1$ to $n$.

\subsection{Partial derivatives\label{pderivatives}}
We use
\begin{equation*}
\del{\mu} = \frac{\del{}\;}{\del{} x^\mu}
\end{equation*}
to denote partial derivatives with respect to the coordinates $(x^\mu)$, and $\del{}$ and $D$ to denote spacetime and spatial gradients, respectively, so
that $\del{}f = (\del{\mu}f)$ and $D{}f = (\del{\Lambda}f)$ for scalar fields $f$. Additionally, we use $\delsl{}$ to denote the derivatives that are tangent to the
boundary $\del{}\Omega$. More generally, for $k\in \Zbb_{\geq 0}$, we use
$\del{}^k f = (\del{\mu_1}\del{\mu_2}\cdots \del{\mu_k} f)$ to denote the set of all partial
derivatives of order $k$, and  define $\delsl{}^k$ and $D^k$ in an analogous manner. We also let $\del{}^{|k|}f = \{\, \del{}^j f \,  |\, 0\leq j\leq k\,\}$, $\delsl{}^{|k|}f = \{\, \delsl{}^j f \,  |\, 0\leq j\leq k\,\}$, and $D^{|k|}f = \{\, D^j f \,  |\, 0\leq j\leq k\,\}$,

\subsection{Raising and lowering indices\label{raising}}
We lower and raise spacetime coordinate indices without comment using the metric $g_{\mu\nu}$ and its inverse $g^{\mu\nu}$, respectively, while
frame indices will be lowered and raised, again without comment, using the frame metric $\gamma_{ij}$ and its inverse $\gamma^{ij}$, respectively; see
\eqref{DGframemet} for a definition of the frame metric. We will have occasion to raise or lower indices using metrics
other than $g_{\mu\nu}$ or $\gamma_{ij}$. In these situations, we will be explicit about this type of operation. For example, given
a metric $m^{\mu\nu}$ and a $1$-form $\lambda_\mu$, we would define the raised version using $m$ explicitly by setting
$\lambda^\mu = m^{\mu\nu}\lambda_\nu$.

\subsection{Norms\label{norms}} For a spacelike $1$-form $\lambda_\mu$, we
define the spacetime norm $|\lambda|_g$ by
\begin{equation*}
|\lambda|_g := \sqrt{g(\lambda,\lambda)} = \sqrt{ g^{\mu\nu}\lambda_\mu \lambda_\nu},
\end{equation*}
while if $m^{\mu\nu}$ is a positive definite metric, then we define the $m$-norm of any 1-form $\lambda_\mu$ by
\begin{equation*}
|\lambda|_{m} := \sqrt{m(\lambda,\lambda)} = \sqrt{ m^{\mu\nu}\lambda_\mu \lambda_\nu}.
\end{equation*}
Similar notation will also be used for inner products involving other objects carrying indices of some type; for example, we
write
$|T|_m^2 = m_{\alpha\beta}m^{\mu\nu} T^\alpha_\nu T^\beta_\mu$, where $(m_{\alpha\beta}):=(m^{\alpha\beta})^{-1}$.
For the special case of the Euclidean inner-product and norm, we employ the notation 
\begin{equation*}
\ipe{\xi}{\zeta} = \xi^{\tr}\zeta \AND 
|\xi| = \sqrt{\ipe{\xi}{\xi}}, \quad \xi,\zeta \in \Rbb^N.
\end{equation*}

\subsection{Constraint terms\label{constnot}} To help encode the freedom to add constraints to evolution equations and boundary conditions, we reserve upper case Fraktur letters, e.. $\Rf$, $\Sf$, $\Tf$, possibly endowed with spacetime
indices, e.g. $\Rf^\nu$, to denote maps that depend linearly on a set of constraints $\Zc$. More precisely, if $\Zc$ is $\Rbb^N$-valued, then
\begin{equation*}
\Rf(\Zc) = \mf\Zc
\end{equation*}
where\footnote{Here and in the following, $\Mbb{N}$ is used to denote the set of $N\times N$ matrices.} $\mf \in C^0(\overline{\Omega}_T,\Mbb{N})$ if $\Rf(\Zc)$ is added to an evolution equation and $\mf \in C^0(\overline{\Gamma}_T,\Mbb{N})$ if $\Rf(\Zc)$ is added to a boundary condition. Since we will, for the most part, not be interested in the exact form of terms involving the constraints, we will often use the same Fraktur letter to denote different combinations of the constraints.

\subsection{Spatial function spaces\label{spatialFS}}
Given a finite dimensional vector space $V$ and  a
bounded, open set $\Omega$ in  $\Rbb^n$ with $C^\infty$ boundary, we let
$W^{s,p}(\Omega,V)$, $s\in \Rbb$ $(s\in \Zbb_{\geq 0})$, $1< p < \infty$ $(1\leq p\leq \infty)$,
denote the space of $V$-valued maps on $\Omega$ with fractional (integral) Sobolev regularity $W^{s,p}$. Particular
cases of interest for $V$ will be $V=\Rbb^N$ and $V=\Mbb{N}$. In the special case of $V=\Rbb$,
we employ the more compact notation $W^{s,p}(\Omega)=W^{s,p}(\Omega,\Rbb)$. 
We will also use the same
notation for function spaces when $\Omega$ is replaced by a closed manifold, e.g. $\partial\Omega$.

When $p=2$, we use the standard
notation $H^s(\Omega,V)=W^{s,2}(\Omega,V)$, and on $L^2(\Omega,\Rbb^N)$, we denote the inner product by
\begin{equation*}
\ip{u}{v}_{\Omega} = \int_{\Omega} \ipe{u(x)}{v(x)}\, d^n x, \qquad u,v\in L^2(\Omega,\Rbb^N).
\end{equation*}
We also use the standard notion $H^1_0(\Omega,V)$ to denote
the subspace of $H^1(\Omega,V)$ consisting of elements whose trace vanishes on the boundary $\del{}\Omega$.

\subsection{Spacetime function spaces:\label{spacetimeFS}} Given $T>0$, and $s=k/2$ for $k\in \Zbb_{\geq 0}$, we define the spaces
\begin{equation}\label{XTdefA}
X^{s,r}_T(\Omega,V) = \bigcap_{\ell=0}^{r} W^{\ell,\infty}\bigl([0,T],H^{s-\frac{\ell}{2}}(\Omega,V)\bigr),\qquad 0\leq r \leq 2s,
\end{equation}
\begin{equation}\label{XTdefB}
X^s_T(\Omega,V) = X^{s,2s}_T(\Omega,V),
\end{equation}
and
\begin{equation}\label{XcTdefB}
\Xc^{s}_T(\Omega,V) = \bigcap_{\ell=0}^{2s-1} W^{\ell,\infty}\bigl([0,T],H^{s-\frac{m(s,\ell)}{2}}(\Omega,V)\bigr)
\end{equation}
where
\begin{equation}\label{melldef}
m(s,\ell) = \begin{cases} \ell & \text{if $0\leq \ell \leq 2s-2$} \\
2s &\text{if $\ell = 2s-1$}
\end{cases}.
\end{equation}
We further let $\mathring{X}^{s,r}_T(\Omega,V)$ and $\Xcr^s_T(\Omega,V)$ denote the subspaces of $X^{s,r}_T(\Omega,V)$ and  $\Xc^{s}_T(\Omega,V)$, respectively,
that are defined by
\begin{equation}\label{Xrdef}
\mathring{X}^{s,r}_T(\Omega,V) = \{ \, u\in X^{s,r}_T(\Omega,V)\, |\, \del{t}u\in X^{s,r}_T(\Omega,V)\, \}
\end{equation}
and
\begin{equation}\label{Xcrdef}
\Xcr^s_T(\Omega,V) = \{ \, u\in \Xc^{s}_T(\Omega,V)\, |\, \del{t}u\in \Xc^{s}_T(\Omega,V)\, \}.
\end{equation}

Next, we define the \textit{energy norms}:
\begin{equation*}
\norm{u}_{E^{s,r}}^2 = \sum_{\ell=0}^r \norm{\del{t}^\ell u}^2_{H^{s-\frac{\ell}{2}}(\Omega)}, \quad\norm{u}^2_{E^s} =  \norm{u}_{E^{s,2s}}^2, \AND
\norm{u}_{\Ec^{s}}^2  = \sum_{\ell=0}^{2s-1} \norm{\del{t}^\ell u}^2_{H^{s-\frac{m(s,\ell)}{2}}(\Omega)}.
\end{equation*}
In terms of these energy norms, we define the \textit{spacetime norms} on the spaces \eqref{XTdefA}-\eqref{XcTdefB} and \eqref{Xcrdef} by
\begin{gather*}
\norm{u}_{X^{s,r}_T} = \sup_{0\leq t \leq T} \norm{u(t)}_{E^{s,r}}, \quad
\norm{u}_{X^s_T} = \sup_{0\leq t \leq T} \norm{u(t)}_{E^s}, \\
\norm{u}_{\Xc^{s}_T} =  \sup_{0\leq t \leq T} \norm{u(t)}_{\Ec^s} \AND 
\norm{u}_{\Xcr^{s}_T} =  \sup_{0\leq t \leq T}\sum_{\ell=0}^1 \norm{\del{t}^\ell u(t)}_{\Ec^s},
\end{gather*}
respectively. We also define the
subspace
\begin{equation} \label{XVdef}
C\Xc^s_T(\Omega,V) = \bigcap_{\ell=0}^{2s-1} C^{\ell}\bigl([0,T],H^{s-\frac{m(s,\ell)}{2}}(\Omega,V)\bigr).
\end{equation}

We conclude this section with the following elementary, but useful, relations:
\begin{align*}
\norm{u}_{E^{s_1,r_1}} &\leq \norm{u}_{E^{s_2,r_2}}, &&s_1\leq s_2, \; r_1\leq r_2, \\
\norm{u}_{\Ec^{s_1}}&\leq \norm{u}_{\Ec^{s_2}}, && s_1\leq s_2,\\
\norm{u}_{\Ec^{s}}^2 &= \norm{u}_{E^{s,2s-2}}^2 + \norm{\del{t}^{2s-1}u}_{L^2(\Omega)}^2,\\
\norm{u}_{E^{s,r}}^2& = \norm{u}_{H^{s}(\Omega)}^2 + \norm{\del{t}u}^2_{E^{s-\frac{1}{2},r-1}},\\
\norm{u}_{E^{s,r}}^2& = \norm{u}_{E^{s-1}}^2 + \norm{\del{t}^ru}^2_{H^{s-\frac{r}{2}}(\Omega)},\\
\norm{u}_{\Ec^{s}}^2 &= \norm{u}_{H^{s}(\Omega)}^2 + \norm{\del{t}u}_{\Ec^{s-\frac{1}{2}}}^2,\\
\norm{\del{t}u}_{E^{s,r}} &\leq \norm{u}_{E^{s+\frac{1}{2},r+1}}, \\
\norm{D u}_{E^{s,r}}  &\leq \norm{u}_{E^{s+1,r}}, \\
\norm{\del{t}u}_{\Ec^s} &\leq \norm{u}_{\Ec^{s+\frac{1}{2}}},\\
\norm{u}_{E^s} &\leq \norm{u}_{\Ec^{s+\frac{1}{2}}}
\intertext{and}
\norm{Du}_{\Ec^s} &\leq \norm{u}_{E^{s+1,2s-1}} \leq \norm{u}_{\Ec^{s+1}},
\end{align*}
which we will use throughout the article without comment.

\subsection{Estimates and constants\label{const}}

We employ that standard notation
\begin{equation*}
a \lesssim b
\end{equation*}
for inequalities of the form
\begin{equation*}
a \leq C b
\end{equation*}
in situations where the precise value or dependence on
other quantities of the constant $C$ is not required.  On the other hand,  when the dependence of the constant
on other inequalities needs to be specified, for example if the constant depends on the norms $\norm{u}_{L^\infty(\Tbb^n)}$ and $\norm{v}_{L^\infty(\Omega)}$, we use the notation
\begin{equation*}
C = C(\norm{u}_{L^\infty(\Tbb^n)},\norm{v}_{L^\infty(\Omega)}).
\end{equation*}
Constants of this type will always be non-negative, non-decreasing, continuous functions of their arguments.

\section{The Eulerian wave formulation\label{waveEul}}

\subsection{Primary and auxiliary fields\label{primaux}}
The \emph{primary fields} for our wave formulation consist of a scalar field $\zeta$ satisfying
\begin{equation*}
\zeta > 0
\end{equation*}
and
a future pointing, timelike $1$-form  $\thetah^0=\thetah^0_\mu dx^\mu$. We use the primary fields to define a timelike $1$-form,
again future pointing, by
\begin{equation} \label{theta0def}
\theta^0 = \zeta \thetah^0,
\end{equation}
which we complete to a coframe by introducing spacelike $1$-forms
\begin{equation*} 
\theta^I = \theta^I_\mu dx^\mu.
\end{equation*}
Along with these $1$-forms, we introduce a collection of scalar fields $\sigma_i{}^k{}_j$.  The set $\{\theta^I,\sigma_i{}^k{}_j\}$
defines the \emph{auxiliary fields} that will evolve via simple transport equations.
For latter use, we introduce a number of additional geometric fields beginning with the frame
\begin{equation} \label{framedef}
e_j = e_j^\mu \del{\mu} \qquad \bigl((e^\mu_j) := (\theta_\mu^j)^{-1}\bigr)
\end{equation}
dual to $\theta^i$.  In accordance with the notation from Appendix \ref{DG}, we use $\gamma_{ij}$ and  $\omega_i{}^k{}_j$
to denote the associated frame metric and
connection coefficients, respectively; see \eqref{DGframemet} and \eqref{DGconndef}.
Finally, we define a future pointing, timelike vector field by
\begin{equation}
\xi^\mu = \frac{1}{\gamma^{00}}\theta^{0\mu}. \label{xidef}
\end{equation}

\subsection{Recovering $\rho$ and $v^\mu$\label{wave2eul}}
The fluid $4$-velocity $v^\mu$ will be shown to be recoverable from the  primary fields
$\{\zeta,\thetah^0_\mu\}$ by normalizing the vector field
$\xi^\mu$ to get
\begin{equation} \label{xihdef}
v^\mu :=\frac{\xi^\mu}{\sqrt{-g(\xi,\xi)}} =   -\frac{\thetah^{0\nu}}{\sqrt{-g(\thetah^0,\thetah^0)}},
\end{equation}
where in obtaining the second equality we used the fact that $\zeta > 0$ and $\gamma^{00}=g(\theta^0,\theta^0) < 0$.
Recovering the proper energy density is more complicated. The first step is to define the pressure
as a solution $p=p(\lambda)$ of the initial value problem
\begin{align}
p'(\lambda) &= \frac{1}{\lambda}\left(\rho\bigl(p(\lambda)\bigr) + p(\lambda)\right),\quad \lambda > 0, \label{prel.1}\\
p(\lambda_0) &= p_0, \label{prel.2}
\end{align}
where $\lambda_0>0$ and  $p_0\geq 0$. To be definite, we set
\begin{equation} \label{zpfix}
p_0 = 0 \AND \lambda = 1.
\end{equation}
From standard ODE theory, we see that $p=p(\lambda)$ is smooth for $\lambda > 0$,
while from the IVP \eqref{prel.1}-\eqref{prel.2}, it follows that $p$ is strictly
increasing, which in turn, implies the
invertibility of the map
$\Rbb_{\geq 1}\ni \lambda  \mapsto p(\lambda)\in \Rbb_{\geq 0}$.
We will show that we can then use this map to recover the proper energy density from the scalar field $\zeta$ by
setting $\lambda = \zeta$ to give $\rho = \rho\bigl(p(\zeta)\bigr)$.

To summarize, $\{\rho,v^\mu\}$ are determined from the primary fields $\{\zeta,\thetah^0_\mu\}$
via the formulas:
\begin{align}
\rho &= \rho\bigl(p(\zeta)\bigr), \label{theta2rho}  \\
v^\mu &=  -\frac{\thetah^{0\nu}}{\sqrt{-g(\thetah^0,\thetah^0)}}. \label{theta2v}
\end{align}

\subsection{Constraints\label{consEul}}
In this section, we define an number of quantities built out of the primary and auxiliary fields that we will refer to as \textit{constraints}.  As we shall
show later in Theorem \ref{cthm} these constraints are distinguished by the property that if the initial data for the IBVP defined below
in Section \ref{waveibvpEul} by the evolution equations 
\eqref{weqn.1}-\eqref{weqn.3} and the boundary conditions \eqref{wbc.1}-\eqref{wbc.3} is chosen so that the constraints vanish on the initial
hypersurface $\Omega_0$, then they must vanish everywhere in $\Omega_T$. We will refer to this property as \textit{constraint propagation}.

The importance of the constraints is two-fold. First, the vanishing of the constraints in $\Omega_T$ will allow us to extract solutions of the relativistic Euler equations satisfying the vacuum boundary conditions from solutions of the IBVP defined by 
\eqref{weqn.1}-\eqref{weqn.3} and  \eqref{wbc.1}-\eqref{wbc.3}.  
Second, the constraints provide us with an enormous freedom to change to form of the IBVP that will be exploited in Section \ref{choice} to bring the evolution equations and boundary conditions
into a form that is advantageous for establishing the existence of solutions.

We separate the constraints into \emph{bulk} and \emph{boundary} constraints,
which are to be interpreted as being associated to $\Omega_T$ and $\Gamma_T$, respectively.

\bigskip

\noindent \underline{Bulk constraints:}
\begin{align}
\ac & = \xi-e_0, \label{acon} \\
\bc^J &= \gamma^{0J}, \label{bcon} \\
\cc^{k}_{ij} &= \sigma_i{}^k{}_j-\sigma_j{}^k{}_i, \label{ccon}\\
\dc^k_j &= \sigma_0{}^k{}_j, \label{dcon}\\
\ec^K &= \Ed \theta^K + \Half \sigma_{i}{}^K{}_j\theta^i\wedge\theta^j, \label{econ}\\
\Fc &= \Ed\theta^0 + \Half\sigma_I{}^0{}_J\theta^I\wedge\theta^J, \label{fcon}\\
\gc &= \delta_g\biggl(\frac{1}{f(\zeta)}\theta^0\biggr) \label{gcon}, \\
\hc &= -\frac{\sqrt{-\det(\gamma^{ij})}}{\sqrt{-\gamma^{00}}}-\frac{\zeta}{f(\zeta)}, \label{hcon} \\
\jc &= g(\thetah^0,\thetah^0)+1 \label{jcon}
\end{align}
where $f(\lambda)$ is defined by
\begin{equation} \label{fdef}
f(\lambda) = -\lambda \exp\biggl(-\int_{1}^\lambda \frac{1}{\eta s^2(\eta)}\, d\eta \biggr)
\end{equation}
with the
square of the sound speed given by
$s^2(\lambda) = \bigl(\rho'(p(\lambda))\bigr)^{-1}$.

We collect together the following bulk constraints
\begin{equation} \label{chidef}
\chi = \bigl(\ac^\mu,\bc^J,\cc^{k}_{ij},\dc^k_j,\ec^K),
\end{equation}
which, as we shall see in the proof of Theorem \ref{cthm}, satisfy simple transport equations.
For later use, we observe from \eqref{theta0def}, \eqref{xidef}, \eqref{xihdef}, \eqref{acon} and \eqref{DGframemet} that
\begin{equation} \label{vrep1}
v^\mu = \sqrt{-\gamma^{00}}\xi^\mu = \sqrt{-\gamma^{00}} e_0^\mu + \Rf^\mu(\ac) =  \sqrt{-\hat{\gamma}^{00}} e_0^\mu
+ \Rf^{\mu}(\ac,\zeta-1),
\end{equation}
where we have set
\begin{equation} \label{gammah00}
\hat{\gamma}^{00} = g(\thetah^0,\thetah^0)
\end{equation}
and we are using $\Rf^\mu$ to denote terms that are proportional to the constraints in line with our notation set out in Section \ref{constnot}.

\bigskip

\noindent \underline{Boundary constraints:}
\begin{align}
\kc &= \thetah^3-n \label{kcon}
\end{align}
where
\begin{equation} \label{thetah3def}
\thetah^3 = \frac{\theta^3}{|\theta^3|_g}
\end{equation}
and as in above, $n$ is the outward pointing unit conormal to $\Gamma_T$.

\subsection{Eulerian IBVP\label{waveibvpEul}}
The formulation \eqref{ibvp.1}-\eqref{ibvp.5} of the vacuum IBVP for the relativistic Euler equations is commonly referred to as the
\emph{Eulerian representation}. In this representation, the matter-vacuum boundary is free, or in other words, dynamical. This
terminology is useful for distinguishing this form of the IBVP from the \textit{Lagrangian representation} where the boundary is fixed.
We will continue to use the Eulerian
terminology for the wave formulation of the relativistic vacuum IBVP that we introduce in this section since the boundary is also free. Later, in Section \ref{lag}, we will consider the Lagrangian representation where the boundary is fixed.

\subsubsection{Eulerian evolution equations\label{wavesysEul}}
Before stating the evolution equations, we first define a number of tensors that will be used repeatedly throughout this article starting
with
\begin{equation} \label{Wcdef}
 \Wc^{\alpha\mu} = m^{\mu\nu} a^{\alpha\beta}\bigl[\zeta \nabla_{\beta} \thetah^0_\nu + \sigma_I{}^0{}_J
\theta^I_\beta\theta^J_\nu\bigr]
\end{equation}
where
\begin{align}
m^{\alpha\beta} &= g^{\alpha\beta} + 2 v^{\alpha}v^{\beta} \label{mdef}
\end{align}
is a Riemannian metric and
\begin{align}
a^{\alpha\beta} &= -\frac{1}{f(\zeta)}\biggl(g^{\alpha\beta} - \biggl(1-\frac{1}{s^2(\zeta)}\biggr)\frac{\xi^\alpha \xi^\beta}{g(\xi,\xi)}\biggr) \label{adef}
\end{align}
is a Lorentzian metric\footnote{For solutions of
our wave formulation that correspond to solutions of the relativistic Euler equations, the metric $a^{\alpha\beta}$ is
conformal to the standard definition of the acoustic metric given by $g^{\alpha\beta}+
\left(1-\frac{1}{s^2}\right)v^{\alpha}v^{\beta}$.}. We then use $\Wc^{\alpha\mu}$ to define
\begin{equation} \label{Ecaldef}
\Ec^\mu = \nabla_\alpha \Wc^{\alpha\mu} - \Hc^\mu
\end{equation}
where
\begin{align}
\Hc^\mu =& \nabla_\alpha m^{\mu\nu}  a^{\alpha\beta}\bigl[\zeta \nabla_{\beta} \thetah^0_\nu + \sigma_I{}^0{}_J
\theta^I_\beta\theta^J_\nu\bigr] \notag \\
&+ m^{\mu\nu}\biggl[ -\nabla_\alpha \thetah^{0\gamma} a^{\alpha\beta}[\zeta \nabla_\beta \thetah^0_\gamma + \sigma_{I}{}^0{}_J\theta^I_\beta
\theta^J_\gamma]\frac{\thetah^0_\nu}{g(\thetah^0,\thetah^0)} + h_\nu^\gamma\bigl(\breve{\Hc}_\gamma - a^{\alpha\beta}\nabla_\beta \zeta \nabla_\alpha \thetah^0_\gamma \bigr) \biggr], \label{Hbrdef}\\
\breve{\Hc}_\nu =& -\frac{1}{f}R^\mu{}_\nu\theta^0_\mu
+\bigl(\delta^\beta_\nu C_{\omega\lambda}^\omega a^{\lambda\alpha}-C_{\omega\nu}^\beta a^{\omega\alpha}\bigr)
\Fc_{\alpha\beta}, \label{Hdef}
\intertext{and}
C_{\alpha\beta}^\lambda =& \Half a^{\lambda\gamma}
\bigl[\nabla_\alpha a_{\beta\gamma} + \nabla_\beta a_{\alpha\gamma}-\nabla_\gamma a_{\alpha\beta}\bigr], \quad \bigl((a_{\alpha\beta}) =
(a^{\alpha\beta})^{-1}\bigr). \label{Cdef}
\end{align}
Here, $R_{\mu\nu}$ denotes the Ricci tensor of the metric $g_{\mu\nu}$,
\begin{equation}
\Fc_{\alpha\beta} = \nabla_\alpha \theta^0_\beta - \nabla_\beta \theta^0_\alpha + \sigma_{I}{}^0{}_J\theta^I_\alpha\theta^J_\beta \label{fcomp}
\end{equation}
denotes the coordinate components of the $2$-form $\Fc$ defined by \eqref{fcon}, i.e. $\Fc = \Half
\Fc_{\alpha\beta}dx^\alpha \wedge dx^\beta$, and $h_{\alpha\beta}$ is as defined previously \eqref{pidef}.
Using the above definitions, the evolution equations for our wave formulation are given by
\begin{align}
\Ec^\mu  + \Rf^\mu\bigl(\del{}^{|1|}\chi,\del{}^{|1|}\jc\bigr)
&=0 && \text{in $\Omega_T$},
\label{weqn.1}\\
\nabla_\alpha\bigl(a^{\alpha\beta}\nabla_\beta \zeta\bigr) & = \Kc
&&  \text{in $\Omega_T$}, \label{weqn.1b}\\
\Ld_v \theta^I &=0 &&
\text{in $\Omega_T$}, \label{weqn.2}\\
v(\sigma_{i}{}^k{}_j) &= 0 &&
\text{in $\Omega_T$}, \label{weqn.3}
\end{align}
where
\begin{equation} \label{Kcdef}
\Kc = -\thetah^{0\nu}\breve{\Hc}_\nu-\nabla_\alpha\thetah^{0\gamma}
a^{\alpha\beta}[\zeta\nabla_\beta \thetah^0_\gamma + \sigma_I{}^0{}_J\theta^I_\beta\theta^J_\gamma]. 
\end{equation}

\begin{rem} \label{Rfrem}
The term $\Rf^\mu(\del{}^{|1|}\chi,\del{}^{|1|}\jc)$ in \eqref{weqn.1} encodes the available freedom to add multiples of $\chi$ and $\jc$ and their first derivatives\footnote{There is nothing stopping us
from adding higher order derivatives of $\chi$ to the evolution equation. The only thing
that would change in the analysis below is that the class of solutions that we are dealing with would have to have enough regularity for the derivatives of $\chi$ to make sense.}
to the evolution equation $\Ec^\mu=0$ in $\Omega_T$.  
Additionally, it is clear from the equations of motion \eqref{weqn.2}-\eqref{weqn.3}
that we are free to add terms proportional to
$\Ld_v \theta^I_\nu$ and $v(\sigma_{i}{}^k{}_j)$ and their derivatives to the evolution equations \eqref{weqn.1}-\eqref{weqn.1b}.
\end{rem}

\subsubsection{Eulerian boundary conditions\label{wavebcs}}
Before stating the boundary conditions for our wave formulation, we define
\begin{equation} \label{Bcdef}
\Bc^\mu = \theta_\alpha^3 \Wc^{\alpha\mu} - \ell^\mu - \Lc^\mu
\end{equation}
where
\begin{align}
\ell^\mu =& \frac{1}{\sqrt{-\gammah^{00}|\gamma|}} \Upsilon^\mu_\omega
h^{\omega\alpha}s_{\alpha\beta}{}^\gamma h^{\beta\nu}\nabla_\gamma \thetah^0_\nu,
\label{elldef}\\
\Lc^{\mu}=& \frac{1}{\sqrt{-\gammah^{00}|\gamma|}}\bigl(- \epsilon|N|_g h^{\mu\nu}\nabla_v \psih_\nu-\kappa v^\mu v^\nu \nabla_v \thetah^0_\nu \bigr),
\label{Lcdef}
\end{align}
and we have made the following definitions:\footnote{Here, $\epsilon_{\nu\alpha\beta\gamma}$ denotes the completely anti-symmetric symbol and $|g|=\sqrt{-\det(g_{\alpha\beta})}$, while
$\nu_{\nu\alpha\beta\gamma}=\sqrt{|g|}\epsilon_{\nu\alpha\beta\gamma}$ defines the volume form of $g$.}
\begin{align}
\Upsilon^\mu_\omega &=  \delta^\mu_\omega - \ep\Pi^\mu_\omega, \label{Upsilondef}\\
\psi_\nu &= \nabla_v\Bigl( \bigl(-g(\thetah^0,\thetah^0)\bigr)^{-\frac{1}{2}}\thetah^0_\nu\Bigr) \overset{\eqref{xihdef}}{=} -\nabla_v v_\nu, \label{psidef} \\
\psih_\nu & = \frac{\psi_\nu}{|\psi|_m}, \label{psihdef}\\
p^{\mu}_{\nu} &= \delta^{\mu}_\nu - \psih^\mu \psih_\nu, \label{pdef} \\
\Pi^\mu_\nu &= p^\mu_\alpha h^{\alpha}_{\nu}, \label{Pidef}\\
\Pi^{\mu\nu} &= \Pi^\mu_{\lambda}g^{\lambda\nu}, \label{Piupdef}\\
\nu_{\nu\alpha\beta\gamma}& = \sqrt{|g|}\epsilon_{\nu\alpha\beta\gamma}, \label{nudef} \\
N_\nu &= -\nu_{\nu\alpha\beta\gamma}v^\alpha e_1^\beta e_2^\gamma \label{Ndef}
\intertext{and}
s_{\mu\nu}{}^{\gamma} &= \nu_{\mu\nu\lambda\omega}e_1^\lambda e_2^\omega v^\gamma
-\nu_{\mu\nu\lambda\omega}v^\lambda e^\omega_2 e^\gamma_1
+ \nu_{\mu\nu\lambda\omega}v^\lambda e^\omega_1 e^\gamma_2 \label{Sdef}.
\end{align}
Here, $\epsilon$ is a positive function and $\kappa$ is a non-negative constant both of which will be fixed later.
With $\Bc^\mu$ so defined, the boundary conditions are given by
\begin{align}
\Bc^\mu + h^\mu_\nu \Qf_1^\nu(\hc,\del{}^{|1|}\jc) +\Qf^\mu_2(\del{}^{{|1|}} \chi,\delsl{}^{|1|}\kc,\jc) &= 0
&&
\text{in $\Gamma_T$}, \label{wbc.1}\\
\zeta & = 1 &&
\text{in $\Gamma_T$}, \label{wbc.2} \\
n_\nu v^\nu &=0 &&
\text{in $\Gamma_T$}. \label{wbc.3}
\end{align}

For use below, we note, by \eqref{theta0def}, \eqref{hcon}, \eqref{fdef} and \eqref{gammah00}, that $\Bc^\mu$ can be expressed as
\begin{equation}   \label{Bcrep}
\Bc^\mu = \theta_\alpha^3 \Wc^{\alpha\mu} -\Lct^\mu +\Qf^\mu(\hc,\zeta-1)
\end{equation} 
where
\begin{align*}
\Lct^\mu =  \Upsilon^\mu_\omega h^{\omega\alpha}s_{\alpha\beta}{}^{\gamma} h^{\beta\nu}\nabla_{\gamma}\thetah^0_\nu 
 -  \epsilon|N|_g h^{\mu\nu}\nabla_v \psih_\nu- \kappa
v^\mu v^\nu \nabla_v \thetah^0_\nu.  
\end{align*}
Furthermore, we also note, with the help of \eqref{pidef}, \eqref{xihdef}, \eqref{gammah00} and \eqref{psidef},  that
\begin{equation}\label{psi2dttheta}
\psi_\mu = (-\gammah^{00})^{-\frac{1}{2}}h_\mu^\nu \nabla_v \thetah^0_\nu.
\end{equation}

\begin{rem} \label{Qfrem}$\;$
\begin{enumerate}[(i)]
\item For the boundary condition \eqref{wbc.1} to make sense, see \eqref{Bcdef}, \eqref{Lcdef} and \eqref{psihdef},  we require that $|\nabla_v v|_m>0$. Consequently, we will restrict our attention to solutions satisfying this condition. With that said, it should be noted that for solutions of our wave formulation that determine a solution to the relativistic Euler equations, then the condition $|\nabla_v v|_m>0$ near the vacuum boundary is equivalent to the Taylor sign condition being satisfied, see \cite[Remark 4.2]{Oliynyk:Bull_2017}, and
so, this assumption is natural near the vacuum boundary. However, since we already can obtain the local-in-time existence of solutions to the relativistic Euler equations away from the vacuum  boundary using standard hyperbolic techniques, there is no loss of generality in assuming that $|\nabla_v v|_m>0$ on $\overline{\Omega}_T$.
\item The terms
$h^\mu_\nu\Qf_1^\nu(\hc)$ and $\Qf^\mu_2(\del{}^{{|1|}} \chi,\delsl{}^{|1|}\kc,\jc)$
in \eqref{wbc.1} encodes the available freedom to add multiples of $\hc$, $\chi$, $\kc$ and $\jc$
 and their indicated derivatives\footnote{Again, there is nothing stopping us
from adding higher order derivatives of $\chi$ and $\kc$ to the boundary conditions. }
to the boundary conditions $\Bc^\mu=0$ in $\Gamma_T$.  We also note that it is clear from the equations of motion \eqref{weqn.2}-\eqref{weqn.3} and the boundary condition \eqref{wbc.2}
that we are free to add terms proportional to
$\Ld_v \theta^I_\nu$ and $v(\sigma_{i}{}^k{}_j)$ and their derivatives to the boundary conditions \eqref{wbc.1}-\eqref{wbc.2}, and $\zeta-1$ and its derivatives tangential to $\Gamma_T$ to
\eqref{wbc.1}.
\end{enumerate}
\end{rem}

\subsubsection{Projections\label{projrem}}
The tensors $h^\mu_\nu$  and $p^\mu_\nu$ define projections into subspaces $g$-othogonal to $v^\mu$ and $\psi^\mu$,
respectively. Being projections,
$h^\mu_\nu$ and $p^\mu_\nu$ satisfy the relations
\begin{gather}
h^\mu_\lambda h^\lambda_\nu = h^\mu_\nu, \quad h^\mu_\nu v^\nu = 0, \label{projrem1} \\
\intertext{and}
 p^\mu_\lambda p^\lambda_\nu = p^\mu_\nu, \quad p^\mu_\nu \psi^\nu = 0, \label{projrem2}
\end{gather}
respectively. Moreover, since $v_\nu$ is unit length, we have that
\begin{equation} \label{orthogA}
g(v,\psi) \overset{\eqref{psidef}}{=} -v^\nu \nabla_v v_\nu = -\Half\nabla_v( v^\mu v_\nu) = 0
\end{equation}
from which we deduce by \eqref{mdef} that
\begin{equation} \label{orthogB}
m(v,\psi) = 0
\end{equation}
and
\begin{equation} \label{psihup}
\psi^\mu = m^{\mu\nu}\psi_\nu. 
\end{equation}
Using the orthogonality relations \eqref{orthogA}-\eqref{orthogB}, it is then not difficult to verify that
\begin{equation} \label{projrem5}
h^\mu_\nu \psi^\nu = \psi^\nu, \qquad  p^\mu_\nu v^\nu = v^\mu,
\end{equation}
and
\begin{equation} \label{projrem3}
h^\mu_\lambda p^\lambda_\nu = p^\mu_\lambda h^\lambda_\nu,
\end{equation}
and that the indices of $p^\mu_\nu$ and $h^\mu_\nu$ can be raised and lowered using either metric $g_{\mu\nu}$ or $m_{\mu\nu}$; for example,
\begin{equation} \label{projrem4}
h^{\mu\nu} = m^{\mu\lambda}h^\nu_\lambda \AND p^{\mu\nu} = m^{\mu\lambda} p^\nu_\lambda.
\end{equation}
Moreover, a short calculation using \eqref{orthogA}-\eqref{orthogB} and  \eqref{projrem3} shows that
\begin{gather}
\nabla_{v} \psih_\mu = p_\mu^\nu \nabla_{v} \psih_\nu =  p_\mu^\nu \frac{\nabla_v \psi_\nu}{|\psi|_m}  \label{dtpsih}
\intertext{and}
h^{\mu\nu}\nabla_{v} \psih_\nu = \Pi^{\mu\nu}  \frac{\nabla_v \psi_\nu}{|\psi|_m}. \label{dtpsihA}
\end{gather}

From the definition \eqref{Pidef}, it is then clear by \eqref{projrem1}-\eqref{projrem2} and \eqref{projrem5}-\eqref{projrem4}
that $\Pi^\mu_\nu$ defines a projection operator satisfying
\begin{equation} \label{Piproj}
\Pi^\mu_\lambda \Pi^\lambda_\nu = \Pi^\mu_\nu, \quad
\Pi^\mu_\nu \psi^\nu =0 , \quad \Pi^\mu_\nu v^\nu = 0,
\end{equation}
\begin{equation} \label{PiprojA}
h^\mu_\lambda\Pi^{\lambda}_\nu =  \Pi^\mu_\lambda h^\lambda_\nu= \Pi^\mu_\nu =p^\mu_\lambda\Pi^{\lambda}_\nu =  \Pi^\mu_\lambda p^\lambda_\nu,
\end{equation}
and
\begin{equation} \label{Piup}
\Pi^{\mu\nu} =  m^{\mu \lambda}\Pi^\nu_\lambda.
\end{equation}
We denote the complementary
projection operator by
\begin{equation} \label{Pibrdef}
\Pibr^\mu_\nu = \delta^\mu_\nu - \Pi^\mu_\nu
\end{equation}
and set
\begin{equation} \label{Pibrup}
\Pibr^{\mu\nu} = m^{\mu \lambda}\Pibr^\nu_\lambda.
\end{equation}
We then observe, from \eqref{Pidef}, \eqref{Piupdef}, \eqref{projrem3}, the symmetry of  $h_{\alpha\beta}$ and $p^{\alpha\beta}$, 
and the calculation
\begin{equation} \label{Pisym}
\Pi^{\mu\nu}= p^\nu_\sigma h^\sigma_\lambda g^{\mu\lambda} = g^{\mu\lambda}p^\sigma_\lambda h^\nu_\sigma
=p^{\sigma\mu}h^\nu_\sigma = p^\mu_\sigma h^{\sigma \nu} = p^\mu_\sigma h^\sigma_\lambda m^{\lambda \nu} = \Pi^\mu_\lambda g^{\lambda \nu}
=\Pi^{\nu\mu},
\end{equation}
that $\Pi^{\mu\nu}$ is symmetric,
which in turn, implies  the symmetry of the complementary projection operator, that is,
\begin{equation} \label{Pisymbr}
\Pibr^{\mu\nu}= \Pibr^{\nu\mu}.
\end{equation}
Finally, we note that the complementrary projection operator satisfies
\begin{equation} \label{Pibrproj}
\Pibr^\mu_\lambda \Pibr^\lambda_\nu = \Pibr^\mu_\nu, \quad
\Pibr^\mu_\nu \psi^\nu =\psi^\mu,\quad  \Pibr^\mu_\nu v^\nu = v^\mu,
\end{equation}
and
\begin{equation} \label{PibrprojA}
h^\mu_\lambda\Pibr^{\lambda}_\nu =  \Pibr^\mu_\lambda h^\lambda_\nu =p^\mu_\lambda\Pibr^{\lambda}_\nu =  \Pibr^\mu_\lambda p^\lambda_\nu.
\end{equation}

\subsubsection{Initial conditions\label{waveics}}
In general, solutions of the IBVP problem consisting of the evolution equations \eqref{weqn.1}-\eqref{weqn.3}
and boundary conditions \eqref{wbc.1}-\eqref{wbc.3} \textit{will not} correspond to solutions of
the relativistic Euler equations with vacuum boundary conditions given by   \eqref{ibvp.1}-\eqref{ibvp.4}.
As we establish in Theorem \ref{cthm} below, a solution $\{\zeta,\thetah^0_\mu,\theta^I_\mu,\sigma_i{}^k{}_j\}$ of \eqref{weqn.1}-\eqref{weqn.3}
and \eqref{wbc.1}-\eqref{wbc.3} will only, via the formulas \eqref{theta2rho}-\eqref{theta2v}, determine a solution of \eqref{ibvp.1}-\eqref{ibvp.4} if
the following constraints on the initial data are satisfied:
\begin{align}
\ac & = 0 && \text{in $\Omega_0$,} \label{aconiv} \\
\bc^J &= 0 && \text{in $\Omega_0$,} \label{bconiv} \\
\cc^{k}_{ij} &= 0 && \text{in $\Omega_0$,} \label{cconiv}\\
\dc^k_j &= 0 && \text{in $\Omega_0$,}  \label{dconiv}\\
\ec^K &= 0 && \text{in $\Omega_0$,}  \label{econiv}\\
\Fc &= 0 && \text{in $\Omega_0$,}  \label{fconiv}\\
\gc &= 0 && \text{in $\Omega_0$,}   \label{gconiv}\\
\hc &= 0 && \text{in $\Omega_0$,}  \label{hconiv} \\
\jc &= 0 && \text{in $\Omega_0$}  \label{jconiv}\\
\kc &= 0 && \text{in $\Gamma_0$}  \label{kconiv}\\
\Ed\Fc &= 0 && \text{in $\Omega_0$,} \label{dfconiv} \\
\nabla_v \gc &=0  && \text{in $\Omega_0$,} \label{gconivtdiff} \\
\nabla_v \hc &= 0 && \text{in $\Omega_0$}  \label{hconivtdiff}\\
\intertext{and}
\nabla_v \jc &= 0 && \text{in $\Omega_0$}  \label{jconivtdiff}
\end{align}
In Section \ref{choice}, we will need to also assume that initial data satisfies the additional constraints
\begin{align}
\Ec^\mu  &= 0  && \text{in $\Omega_0$} \label{evicons} 
\intertext{and}
\Bc^\mu &= 0  && \text{in $\Gamma_0$.}
\label{bcicons}
\end{align}

\begin{rem} The constraints on the initial data \eqref{aconiv}-\eqref{bcicons} do not involve any constraints
on the choice of initial data for the fluid, that is, for $\rho$ and $v^\mu$, or equivalently, $\thetah^0_\mu$ and $\zeta$, beyond the usual compatibility conditions 
for the relativistic Euler equations with a vacuum boundary. Constraints that are unrelated to compatibility conditions for the physical fields involve the auxiliary fields 
$\theta^I_\mu$ and $\sigma_i{}^k{}_j$, which are not physical and we are free to
choose as we like; see \cite[Section 4.2]{Oliynyk:Bull_2017}
for details on how to select initial data for the auxiliary fields.
\end{rem}

\section{Constraint propagation\label{constprop}}
The relationship between solutions to the evolution equations
and boundary conditions defined in the previous section and solutions to the relativistic Euler equations with
vacuum boundary conditions is made precise in the following theorem. The main content of this theorem is that it guarantees that
the constraints \eqref{acon}-\eqref{jcon} and \eqref{kcon} propagate for solutions of the evolution equations \eqref{weqn.1}-\eqref{weqn.3}
and boundary conditions \eqref{wbc.1}-\eqref{wbc.3} provided that the initial condition from \eqref{aconiv}-\eqref{jconivtdiff} are satisfied.

\begin{thm} \label{cthm}
Suppose $\ep \in C^1(\overline{\Omega}_T)$ with $\ep > 0$ in $\overline{\Omega}_T$, $\kappa\geq 0$, $\zeta \in C^2(\overline{\Omega}_T)$, $\thetah^0_\mu\in C^3(\overline{\Omega}_T)$, $\theta^J_\mu\in C^2(\overline{\Omega}_T)$, $\sigma_{i}{}^j{}_k\in C^1(\overline{\Omega}_T)$,
there exists constants  $c_0,c_1>0$ such that\footnote{Here and in the following, we use the notation $\Ip_{X}\!\alpha$ to denote the interior product of a vector field $X$ with a differential
form $\alpha$.} 
\begin{align}
-g(\theta^0,\theta^0) \geq c_0 &> 0  && \text{in $\Omega_T$} \label{cthm1}
\intertext{and} 
-\Ip_{e_3}\nabla_{e_0}\theta^0 \geq c_1 &> 0 && \text{in $\Gamma_T$,} \label{cthm2}
\end{align}
and  the quadruple $\{\zeta,\thetah^0_\mu,\theta^J_\mu,\sigma_i{}^k{}_j\}$ satisfies the initial conditions
\eqref{aconiv}-\eqref{jconivtdiff},
the evolution equations \eqref{weqn.1}-\eqref{weqn.3} and
the boundary conditions
\eqref{wbc.1}-\eqref{wbc.3}.
Then
the constraints \eqref{acon}-\eqref{jcon} and \eqref{kcon} vanish in $\Omega_T$ and $\Gamma_T$, respectively, the pair $\{\rho,v^\mu\}$
determined from $\{\thetah^0_\mu,\zeta\}$ via the formulas \eqref{theta2rho}-\eqref{theta2v}
satisfy the relativistic Euler equations with vacuum boundary conditions given by \eqref{ibvp.1}-\eqref{ibvp.4}, and there exists a constant $c_p>0$
such that the Taylor sign condition
\begin{equation*}
-\nabla_n p \geq c_p > 0 
\end{equation*}
holds on $\Gamma_T$.
\end{thm}
\begin{proof}
\underline{Propagation of $\bc^I$ in $\Omega_T$:}
From the definition of $\bc^I$, see \eqref{bcon}, we compute
\begin{equation} \label{bIev}
v(\bc^I) \oset{\eqref{xihdef}}{=} \Ld_v\biggl(\frac{1}{\sqrt{-\gamma^{00}}} \Ip_v\theta^I\biggr)\oset{\eqref{DGf.3}} =
\frac{1}{\sqrt{-\gamma^{00}}}\Ip_v \Ld_v \theta^I +  \Ld_v\biggl(\frac{1}{\sqrt{-\gamma^{00}}}\biggr)\sqrt{-\gamma^{00}} \bc^I
\oset{\eqref{weqn.2}}{=}   \Ld_v\biggl(\frac{1}{\sqrt{-\gamma^{00}}}\biggr)\sqrt{-\gamma^{00}} \bc^I,
\end{equation}
which holds in $\Omega_T$. The assumption \eqref{cthm1} and the boundary condition \eqref{wbc.3} imply that $v$ is timelike
and tangent to the boundary $\Gamma_T$, which in turn, implies that the set $\Omega_T$ is invariant under
the flow of $v$. From this fact, the transport equation \eqref{bIev} and the initial condition
\eqref{bconiv}, it follows via the uniqueness of solutions to transport equations\footnote{Of course,
this is equivalent to the
uniqueness of solutions of ODEs since transport equations can be solved
using the method of characteristics.}
that
\begin{equation} \label{bconsat}
\bc^I = \gamma^{0I} = 0 \quad \text{in $\Omega_T$.}
\end{equation}

\bigskip

\noindent \underline{Propagation of $\ac$ in $\Omega_T$:}
Since $\Ip_\xi \theta^I =0$ by  \eqref{bconsat} while $\Ip_\xi \theta^0=1$ is a direct
consequence of
the definition \eqref{xidef}, it follows immediately that
\begin{equation} \label{aconsat}
\ac = \xi - e_0 =0 \quad \text{in $\Omega_T$.}
\end{equation}

\bigskip

\noindent \underline{Propagation of $\cc^k_{ij}$ and $\dc^k_j$ in $\Omega_T$:}
From the definitions \eqref{ccon}-\eqref{dcon} and the evolution equation \eqref{weqn.3}, it
is clear that the constraints $\cc^k_{ij}$ and $d^k_j$ satisfy the transport equations
$v(\cc^k_{ij}) = 0$ and
$v(\dc^k_j)=0$ in $\Omega_T$.
Since $\cc^k_{ij}$ and $\dc^k_j$ vanish on the initial hypersurface, see \eqref{cconiv}-\eqref{dconiv},
we conclude, again by the uniqueness of solutions to transport equations, that
\begin{equation} \label{cdconsat}
\cc^k_{ij} = \sigma_i{}^k{}_j-\sigma_j{}^k{}_i = 0 \AND
\dc^k_j = \sigma_0{}^k{}_j = 0 \quad \text{in $\Omega_T$.}
\end{equation}
\bigskip

\noindent \underline{Propagation of $\ec^K$ in $\Omega_T$:} From the definition \eqref{econ} of $\ec^K$,
we have that
\begin{align*}
\Ld_v\ec^K &= \Ld_v\Ed\theta^K + \Half \Bigl(v(\sigma_i{}^K\!{}_j)\theta^i\wedge\theta^j
+\sigma_i{}^K\!{}_j \Ld_{v}\theta^i\wedge\theta^j + \sigma_i{}^K\!{}_j\theta^i\wedge \Ld_v\theta^j \Bigr)
&& \text{ (by \eqref{DGf.4})} \\
&= \Ld_v\Ed\theta^K + \Half \Bigl(\sigma_I{}^K\!{}_J \Ld_{v}\theta^I\wedge\theta^J + \sigma_I{}^K\!{}_J\theta^I\wedge \Ld_v\theta^J \Bigr) && \text{ (by \eqref{weqn.3} and \eqref{cdconsat})} \\
&= \Ed \Ld_v\theta^K + \Half \Bigl(\sigma_I{}^K\!{}_J \Ld_{v}\theta^I\wedge\theta^J + \sigma_I{}^K\!{}_J\theta^I\wedge \Ld_v\theta^J \Bigr) && \text{ (by \eqref{DGf.2})}\\
&= 0 && \text{ (by \eqref{weqn.2})}
\end{align*}
in $\Omega_T$, and so
\begin{equation} \label{econsat}
\ec^K = \Ed\theta^K + \Half \sigma_i{}^K\!{}_j\theta^i\wedge\theta^j = 0 \quad \text{in $\Omega_T$}
\end{equation}
by the uniqueness of solutions to transport equations since $\ec^K$ vanishes on the initial hypersurface, see \eqref{econiv}.
We note, with the help of the Cartan
structure equations \eqref{CartanA}, that \eqref{econsat} is equivalent to
\begin{equation} \label{econsat1}
\sigma_i{}^K{}_j = \omega_i{}^K\!{}_j-\omega_j{}^K\!{}_i \quad \text{in $\Omega_T$}.
\end{equation}

\bigskip

\noindent \underline{Propagation of $\kc$ in $\Gamma_T$:}
Letting $\Gc_t(x^\lambda) = (\Gc^\mu_t(x^\lambda))$
denote the flow of $v$ so that
\begin{equation*}
\frac{d\;}{dt} \Gc^\mu_t(x^\lambda) = v^\mu(\Gc(x^\lambda)) \AND
\Gc^\mu_0(x^\lambda) = x^\mu,
\end{equation*}
we introduce Lagrangian coordinates $(\xb^\mu)$ via
\begin{equation} \label{lemBa.1a}
x^\mu = \phi^\mu(\xb) := \Gc_{\xb^0}^\mu(0,\xb^\Lambda), \quad \forall\, (\xb^0,\xb^\Lambda)\in [0,T]\times \Omega_0.
\end{equation}
Denoting the pull-back of $v$ via the map $\phi$ by
$\vb = \phi^*v$,
we then have that
\begin{equation} \label{lemBa.1c}
\vb = \delb{0} \quad \Longleftrightarrow \quad \vb^\mu = \delta^\mu_0.
\end{equation}
We also observe that
\begin{equation*}
[0,T]\times \Omega_0 = \phi^{-1}(\Omega_T) \AND [0,T]\times \del{}\Omega_0 = \phi^{-1}(\Gamma_T),
\end{equation*}
where $\xb^0$ defines a coordinate on the interval $[0,T]$ and the $(\xb^\Lambda)$ define ``spatial'' coordinates
on $\Omega_0$.

Without loss of generality, we may assume $\xb^3$ is a defining coordinate for the  boundary
$[0,T]\times \del{}\Omega_0$ so that $d\xb^3$ is an outward pointing conormal on the boundary $[0,T]\times \del{}\Omega_0$.
Letting $\thetab^3 = \phi^*\theta^3$
denote the pull-back of $\theta^3$ by $\phi$, we have from the choice of initial data, see
\eqref{kconiv}, that
\begin{equation} \label{lemBa.2}
\thetab^3_\mu(0,\xb^\Lambda) = r(\xb^\Lambda) \delta^3_\mu \quad \quad \forall \, (\xb^\Lambda)\in \del{}\Omega_0
\end{equation}
for some positive function $r>0$. Pulling back the evolution equation \eqref{weqn.2} for $I=3$, we see from
\eqref{lemBa.1c} and \eqref{DGLie} that
\begin{equation*}
\Ld_{\vb} \thetab^3_\nu = 0 \quad \Longleftrightarrow \quad \delb{0}\thetab^3_\mu = 0,
\end{equation*}
which implies by \eqref{lemBa.2} that
$\thetab^3_\mu(\xb^0,\xb^\Lambda) = r(\xb^\Lambda) \delta^3_\mu$ for all
$(\xb^0,\xb^\Lambda)\in [0,T]\times \del{}\Omega_0$.
But this shows that $\thetab^3$ defines an outward pointing conormal to $[0,T]\times \del{}\Omega_0$, and hence,
that $\theta^3$ is an outward pointing conormal to $\Gamma_T$ and
\begin{equation} \label{lemBa.3}
\kc = \thetah^3 - n = 0 \quad \text{in $\Gamma_T$}
\end{equation}
 by \eqref{thetah3def}.

\bigskip

\noindent \underline{Propagation of $\jc$ in $\Omega_T$:} In order to show that the constraint $\jc$, see \eqref{jcon}, propagates, we need to show that $\jc$
satisfies a suitable evolution equation and boundary condition. To derive the boundary condition, we first
contract \eqref{wbc.1} with $v_\mu$ to, with the help of \eqref{projrem1}, get
\begin{equation*} 
v_\mu \Bc^\mu = \Qf(\del{}^{{|1|}} \chi,\delsl{}^{|1|}\kc,\jc) \quad \text{in $\Gamma_T$.}
\end{equation*}
Using this together with \eqref{xihdef}, \eqref{chidef}, \eqref{Wcdef}, \eqref{Bcdef}-\eqref{Lcdef}, \eqref{wbc.2},  \eqref{projrem1}, \eqref{orthogA}, \eqref{Piproj},
\eqref{bconsat}-\eqref{aconsat},  and \eqref{lemBa.3}, it is then not
difficult to verify that $\jc$ satisfies satisfies the Neumann boundary condition
\begin{equation} \label{jcbc1}
|\theta^3|_g n_\alpha a^{\alpha\beta}\nabla_\beta\, \jc = -\frac{\kappa}{\sqrt{-\gammah^{00}|\gamma|}}\nabla_v \jc +  \Qf(\jc) \quad \text{in $\Gamma_T$.}
\end{equation}

To derive an evolution equation for $\jc$,
we first observe from \eqref{bconsat}-\eqref{econsat} and
\eqref{weqn.1} that
$\Ec^\mu = \Rf^\mu(\del{}^{|1|}\jc)$ in $\Omega_T$, and hence, by \eqref{Ecaldef}, that
\begin{equation*}
\nabla_\alpha\bigl(a^{\alpha\beta}[\zeta \nabla_\beta \thetah^0_\nu + \sigma_{I}{}^0{}_J\theta^I_\beta
\theta^J_\nu]\bigr) = h^\mu_\nu\bigl(\Hbr_\mu -a^{\alpha\beta}\nabla_\alpha \zeta\nabla_\beta \thetah^0_\mu \bigr)
-\nabla_\alpha \thetah^{0\gamma} a^{\alpha\beta}[\zeta \nabla_\beta \thetah^0_\gamma + \sigma_{I}{}^0{}_J\theta^I_\beta
\theta^J_\gamma]\frac{\thetah^0_\nu}{g(\thetah^0,\thetah^0)} + \Rf_\nu(\del{}^{|1|}\jc)
\end{equation*}
in $\Omega_T$. Contracting this equation with $\thetah^{0\nu}$, while using \eqref{bconsat}, shows that $\jc$ satisfies the wave equations
\begin{equation}\label{jcpr3}
\nabla_\alpha \bigl(\zeta a^{\alpha\beta}\nabla_\beta \jc\bigr) = \Rf(\del{}^{|1|}\jc)  \quad \text{in $\Omega_T$.}
\end{equation}
Since solutions to wave equations with Neumann boundary conditions of the form \eqref{jcbc1} with $\kappa\geq 0$ are unique by \cite[Theorem 2.2]{Koch:1993}
and $\jc$ and $\nabla_v \jc$ both vanish on the initial hypersurface, see \eqref{jconiv} and \eqref{jconivtdiff}, we have that
\begin{equation} \label{jconsat}
\jc = g(\thetah^0,\thetah^0) + 1 = 0  \quad \text{in $\Omega_T$}
\end{equation}
from which we obtain
\begin{equation} \label{vtothetah}
v_\mu = - \thetah^0_\mu = -\theta^0_\mu \quad \text{in $\Gamma_T$}
\end{equation}
by \eqref{theta0def}, \eqref{xihdef}, and \eqref{wbc.2}.
We further observe from the \eqref{theta0def}, the boundary condition \eqref{wbc.2}, and \eqref{jconsat} that
\begin{equation} \label{gamma00bc}
\gamma^{00}=-1 \quad \text{in $\Gamma_T$,}
\end{equation}
which in turn, implies, by \eqref{vrep1} and \eqref{aconsat} that
\begin{equation} \label{vtoe0}
v^\mu = e^\mu_0 \quad \text{in $\Gamma_T$}.
\end{equation}

\bigskip

\noindent \underline{Integral estimate for $\Half |Y|_{\Lambda}^2  +|\theta^3|_g\omega_0{}^0{}_3 \hcbr$:}
We now turn to establishing an integral estimate for the quantity
\begin{equation*}
\Half |Y|_{\Lambda}^2  +|\theta^3|_g\omega_0{}^0{}_3 \hcbr
\end{equation*}
where
\begin{align} 
\breve{\hc} &= -\sqrt{|g|}\det(e)|\theta^3|_g \hc, \label{lemC.8b}\\
Y_\Kc &= \Fc(e_0,e_\Kc), \label{Ydef}\\
|Y|_\Lambda^2 &= \Lambda^{\Kc\Lc}Y_\Kc Y_\Lc , \label{YLnorm}
\end{align}
and $\Lambda^{\Kc\Lc}$ is a symmetric, positive definite matrix defined by
\begin{equation}
\Lambda^{\Kc\Lc} = \ep|N|_g\theta^{\Kc}_\mu \Pi^{\mu\nu} \theta_\nu^\Lc. \label{Lambdadef}
\end{equation}
We begin by recalling, see \eqref{framedef}, that $\theta^j_\mu$ is the inverse of $e^\mu_j$.   Consequently,
by cofactor expansion, we have
\begin{equation*}
\theta^3_\mu = \frac{\text{cof}(e)^\mu_3}{\det(e)} = -\frac{1}{\det(e)}\epsilon_{\mu\alpha\beta\gamma}
e_0^\alpha e_1^\beta e_2^\gamma, \qquad (e=(e^\mu_j)),
\end{equation*}
which, we observe, by \eqref{nudef}, is equivalent to
\begin{equation} \label{lemC.2}
\sqrt{|g|}\det(e)\theta^3_\mu = -\nu_{\mu\alpha\beta\gamma}
e_0^\alpha e_1^\beta e_2^\gamma.
\end{equation}
Since $\xi=e_0$ by \eqref{aconsat}, we can, using the boundary condition \eqref{wbc.2}, write \eqref{lemC.2} as\footnote{More
explicitly, it follows from \eqref{lemC.2} and the formulas \eqref{theta0def}, \eqref{xidef}, \eqref{acon}, \eqref{gammah00} and  \eqref{Ndef} that
\begin{equation*}
\sqrt{|g|}\det(e)\theta^3_\mu = N_\mu + \Qf_\mu(\ac,\zeta-1),
\end{equation*}
which, in particular, implies \eqref{lemC.2a} when $\ac=0$ and $\zeta=1$.
}
\begin{equation} \label{lemC.2a}
\sqrt{|g|}\det(e)\theta^3_\mu = N_\mu\quad\text{in $\Gamma_T$},
\end{equation}
where $N_\nu$ is defined by \eqref{Ndef}.
But, we also have that
\begin{align}
\theta^3_\mu &= \frac{1}{\det(e)}\bigl(\det(e)\theta^3_\mu\bigr) \notag \\
&= \sqrt{|g|}\sqrt{-\det(\gamma^{kl})} \det(e)\theta^3_\mu &&
\text{(since $\det(\gamma_{ij})=\det(g_{\mu\nu})\det(e)^2$)}\notag\\
&= \sqrt{|g|}\biggl(-\frac{\zeta \sqrt{-\gamma^{00}}}{f}-\sqrt{-\gamma^{00}}\hc\biggr)\det(e)\theta^3_\mu, \label{lemC.3a}
\end{align}
where in deriving the last equality we used \eqref{hcon}.
Since 
\begin{equation*} 
f = f(\zeta) = -1 \quad \text{in $\Gamma_T$}
\end{equation*}
by \eqref{fdef} and \eqref{wbc.2}, we see from \eqref{gamma00bc} and \eqref{lemC.3a} that
\begin{equation} \label{lemC.3}
\theta^3_\mu= \sqrt{|g|}(1-\hc)\det(e)\theta^3_\mu \quad \text{in $\Gamma_T$.}
\end{equation}
 Combining \eqref{lemC.2} and \eqref{lemC.3} yields
\begin{equation} \label{lemC.4}
\theta^3_\mu = -\nu_{\mu\alpha\beta\gamma}
e_0^\alpha e_1^\beta e_2^\gamma - \sqrt{|g|}\hc \det(e)\theta^3_\mu \quad \text{in $\Gamma_T$.}
\end{equation}

Next,  we find,
after a short calculation, that\footnote{Here, we are using the fact that the volume form $\nu_{\alpha\beta\gamma\mu}$ satisfies $\nabla_\lambda \nu_{\alpha\beta\gamma\mu}=0$.}
\begin{align*}
\nabla_{e_0}\bigl(-\nu_{\mu\alpha\beta\gamma}
e_0^\alpha e_1^\beta e_2^\gamma \bigr) &= -\nu_{\mu\alpha\beta\gamma}
\bigl(e_1^\beta e_2^\gamma\nabla_{e_0} e_0^\alpha+
e_0^\alpha e^\gamma_2\nabla_{e_1}e_0^\beta
+e_0^\alpha e_1^\beta\nabla_{e_2}e_0^\gamma\bigr) \notag \\
&\qquad
-\nu_{\mu\alpha\beta\gamma}
\bigl(e^\alpha_0 e^\gamma_2 [e_0,e_1]^\beta +
e^\alpha_0 e_1^\beta [e_0,e_2]^\gamma\bigr).  
\end{align*}
Using \eqref{Sdef} and \eqref{vtoe0}, we can write this as
\begin{align}
\nabla_{e_0}\bigl(-\nu_{\mu\alpha\beta\gamma}
e_0^\alpha e_1^\beta e_2^\gamma \bigr) &= -s_{\mu \nu}{}^\gamma \nabla_\gamma e_0^\nu
-\nu_{\mu\alpha\beta\gamma}\Bigl(e^\alpha_0 e^\gamma_2 [e_0,e_1]^\beta +
e^\alpha_0 e_1^\beta [e_0,e_2]^\gamma\Bigr) \quad \text{ in $\Gamma_T$} \label{lemC.5a}.
\end{align}
But from \eqref{Sdef}, \eqref{vtoe0} and \eqref{lemBa.3}, it is clear that $s_{\mu\nu}^\gamma$ satisfies
\begin{equation}\label{stang}
s_{\mu\nu}{}^\gamma n_\gamma = 0  \quad \text{ in $\Gamma_T$},
\end{equation}
which in turn implies that the differential operator $s_{\mu\nu}{}^\gamma \nabla_\gamma$ only involves derivatives tangential to $\Gamma_T$.
Using this along with \eqref{vtothetah} and \eqref{vtoe0}, we see that \eqref{lemC.5a} can be expressed
as
\begin{align*}
\nabla_{e_0}\bigl(-\nu_{\mu\alpha\beta\gamma}
e_0^\alpha e_1^\beta e_2^\gamma \bigr) &= s_{\mu}{}^{\nu}{}^\gamma \nabla_\gamma \thetah_\nu^0
 -\nu_{\mu\alpha\beta\gamma}\Bigl(e^\alpha_0 e^\gamma_2 [e_0,e_1]^\beta +
e^\alpha_0 e_1^\beta [e_0,e_2]^\gamma\Bigr) \quad \text{ in $\Gamma_T$}.
\end{align*}
But from \eqref{pidef}, \eqref{orthogA} and \eqref{vtothetah}, it is clear that the above expression is equivalent to
\begin{align}
\nabla_{e_0}\bigl(-\nu_{\mu\alpha\beta\gamma}
e_0^\alpha e_1^\beta e_2^\gamma \bigr) &= s_{\mu \beta}{}^\gamma h^{\beta\nu}\nabla_\gamma \thetah_\nu^0
  -\nu_{\mu\alpha\beta\gamma}\Bigl(e^\alpha_0 e^\gamma_2 [e_0,e_1]^\beta +
e^\alpha_0 e_1^\beta [e_0,e_2]^\gamma\Bigr) \quad \text{ in $\Gamma_T$}. \label{lemC.5}
\end{align}

Now, using the identity $[e_0,e_j]^\beta = -\Ed\theta^k(e_0,e_j)e^\beta_k$, we find, with the help of the definitions
\eqref{dcon}-\eqref{fcon}, that
\begin{equation*} 
[e_0,e_j]^\beta = -\Fc(e_0,e_j)e^\beta_0+\ec^K(e_0,e_j)e^\beta_K + \Half \cc_0{}^K\!{}_j e^\beta_K,
\end{equation*}
and hence, that 
\begin{equation*} 
[e_0,e_j]^\beta = \Fc(e_0,e_j)e^\beta_0
\end{equation*}
by \eqref{cdconsat}-\eqref{econsat}. Substituting this into \eqref{lemC.5} gives
\begin{equation*}
\nabla_{e_0}\bigl(-\nu_{\mu\alpha\beta\gamma}
e_0^\alpha e_1^\beta e_2^\gamma \bigr) =  s_{\mu\beta}{}^\gamma h^{\beta\nu} \nabla_\gamma
\thetah^0_\nu
\quad \text{in $\Gamma_T$}. 
\end{equation*}
Applying $\nabla_{e_0}$ to \eqref{lemC.4}, we then find, with the help of the above expression, that
\begin{equation} \label{lemC.8a}
\nabla_{e_0} \theta^3_\mu =  s_{\mu\beta}{}^\gamma h^{\beta\nu} \nabla_\gamma
\thetah^0_\nu 
+\nabla_{e_0}\breve{\hc}\thetah^3_\mu + \breve{\hc} \nabla_{e_0} \thetah^3_\mu \quad \text{in $\Gamma_T$,}
\end{equation}
where $\hcbr$ is defined above by \eqref{lemC.8b}.

To continue, we compute
\begin{align}
\theta^3_\alpha a^{\alpha\beta}\nabla_\beta \theta^0_\nu
&= \theta^3_\alpha g^{\alpha\beta}\nabla_\beta \theta^0_\nu && \text{(by \eqref{bconsat})}\notag \\
&= \theta^3_\alpha g^{\alpha\beta}\bigl(\nabla_\beta \theta^0_\nu - \nabla_\nu \theta^0_\beta + \nabla_\nu \theta^0_\beta\bigr) \notag \\
&= \theta^3_\alpha g^{\alpha\beta}\bigl(\Fc_{\beta\nu}-\sigma_{I}{}^0{}_{J}\theta^I_\beta \theta^J_\nu\bigr)
-\nabla_\nu\theta^3_\alpha g^{\alpha\beta}\theta_\beta^0 && \text{(by \eqref{fcon} and \eqref{bconsat})} \notag \\
&= \theta^3_\alpha g^{\alpha\beta}\bigl(\Fc_{\beta\nu}-\sigma_{I}{}^0{}_{J}\theta^I_\beta \theta^J_\nu\bigr)
-\bigl(\nabla_\nu \theta^3_\alpha - \nabla_\alpha\theta^3_\nu)g^{\alpha\beta}\theta^0_\beta
-g^{\alpha\beta}\theta^0_\beta \nabla_\alpha\theta^3_\nu \notag \\
&=\theta^3_\alpha g^{\alpha\beta}\bigl(\Fc_{\beta\nu}-\sigma_{I}{}^0{}_{J}\theta^I_\beta \theta^J_\nu\bigr)
-g^{\alpha\beta}\theta^0_\beta \nabla_\alpha\theta^3_\nu -\ec^3_{\nu\alpha}g^{\alpha\beta}
+\sigma_{i}{}^3{}_j\theta^i_\nu\theta^j_\alpha g^{\alpha\beta}\theta^0_\beta
&&\text{(by \eqref{econ})}\notag\\
&= \theta^3_\alpha g^{\alpha\beta}\Fc_{\beta\nu}- \theta^3_\alpha a^{\alpha\beta}\sigma_{I}{}^0{}_{J}\theta^I_\beta \theta^J_\nu
-g^{\alpha\beta}\theta^0_\beta \nabla_\alpha\theta^3_\nu, \label{lemC.9}
\end{align}
where in deriving the last equality we used \eqref{bconsat}, \eqref{cdconsat} and \eqref{econsat}. 
Noting that
\begin{equation*}
\theta^3_\alpha a^{\alpha\beta}\nabla_\beta \theta^0_\nu \overset{\eqref{theta0def}}{=} \theta^3_\alpha a^{\alpha\beta}\bigl( \zeta \nabla_\beta\thetah^0_\nu + \nabla_\beta \zeta \thetah^0_\nu\bigr)
\end{equation*}
and, by  \eqref{vtothetah}, \eqref{vtoe0} and \eqref{lemC.8a}, that
\begin{align*}
-g^{\alpha\beta}\theta^0_\beta \nabla_\alpha \theta^3_\nu
=  s_{\nu\beta}{}^\gamma h^{\beta\omega} \nabla_\gamma
\thetah^0_\omega
+\nabla_{e_0}\breve{\hc}\thetah^3_\nu + \breve{\hc} \nabla_{e_0} \thetah^3_\nu,  
\end{align*}
we see, after substituting these two expressions into \eqref{lemC.9} and rearranging, that
\begin{align*}
\theta^3_\alpha g^{\alpha\beta}\Fc_{\beta\nu} =
\theta^3_\alpha a^{\alpha\beta}&\bigl( \zeta \nabla_\beta\thetah^0_\nu + \nabla_\beta \zeta \thetah^0_\nu + \sigma_{I}{}^0{}_{J}\theta^I_\beta \theta^J_\nu\bigr)
-  s_{\nu\beta}{}^\gamma h^{\beta\omega} \nabla_\gamma
\thetah^0_\omega  
-\bigl(\nabla_{e_0}\breve{\hc}\thetah^3_\nu + \breve{\hc} \nabla_{e_0} \thetah^3_\nu\bigr)
\end{align*}
in $\Gamma_T$. 
From \eqref{wbc.1},  \eqref{Bcrep}, \eqref{dtpsih}, \eqref{orthogA}, \eqref{projrem4}, \eqref{bconsat}-\eqref{econsat1}, \eqref{lemBa.3} and \eqref{jconsat}, it is clear that we can, after contracting with $m^{\mu\nu}$, write the above expression as
\begin{align*}
m^{\mu\nu}\theta^3_\alpha g^{\alpha\beta}\Fc_{\beta\nu} =&
 - \ep\Pi^\mu_\omega h^{\omega\alpha}s_{\alpha\beta}{}^{\gamma} h^{\beta\nu}\nabla_{\gamma}\thetah^0_\nu
-  \epsilon |N|_g \Pi^{\mu\nu} \nabla_v\psih_\nu \notag \\
&-m^{\mu\nu}\Bigl(\theta^3_\alpha a^{\alpha\beta}\nabla_\beta \zeta - v^\alpha s_{\alpha\beta}{}^\gamma h^{\beta\omega}\nabla_\gamma \thetah^0_\omega \Bigr)v_\nu
-m^{\mu\nu}\nabla_{e_0}\breve{\hc}\thetah^3_\nu
+\Qf^\mu(\hcbr),
\end{align*}
or with the help of \eqref{vtoe0} and \eqref{lemC.8a}, as
\begin{align}
m^{\mu\nu}\theta^3_\alpha g^{\alpha\beta}\Fc_{\beta\nu} &=
 - \ep\Pi^{\mu\nu} \bigl(\nabla_{v}\theta^3_\nu+ |N|_g\nabla_v\psih_\nu\bigr)
+(\ep\Pi^{\mu\nu} -m^{\mu\nu})\thetah^3_\nu \nabla_{v}\breve{\hc}\notag \\
&\qquad - m^{\mu\nu}\Bigl(\theta^3_\alpha a^{\alpha\beta}\nabla_\beta \zeta - v^\alpha s_{\alpha\beta}{}^\gamma h^{\beta\omega}\nabla_\gamma \thetah^0_\omega \Bigr)v_\nu
+\Qf^\mu(\hcbr). \label{lemC.10}
\end{align}

Next, contracting \eqref{lemC.10} with $\psi_\nu$, we find, with the help of \eqref{orthogB} and \eqref{Piproj}, that 
\begin{equation} \label{mbcidA}
\theta^3_\alpha g^{\alpha\beta}\Fc_{\beta\nu}\psi^\nu =  -\psi^\nu \thetah^3_\nu\nabla_{v}\hcbr + \Qf(\hcbr) \quad \text{in $\Gamma_T$.} 
\end{equation}
We also observe that $\nabla_{e_0} \theta^0_\nu$ can be expressed as
\begin{align}
\nabla_{e_0}\theta^0_\nu
&= -\omega_{0}{}^0{}_{\kc}\theta^{\kc}_\nu -\omega_{0}{}^0{}_3 \theta^3_\nu && \text{(by \eqref{DGconndefA})} \notag \\
&= \bigl(-\omega_{0}{}^0{}_{\Kc} + \omega_{\Kc}{}^0{}_{0}\bigr)\theta_\nu^\Kc - \gamma^{00}\omega_{000} \theta^0_\nu 
-\omega_{0}{}^0{}_3 \theta^3_\nu  && \text{(by \eqref{bconsat})} \notag \\
&= \bigl(-\omega_{0}{}^0{}_{\Kc} + \omega_{\Kc}{}^0{}_{0}\bigr)\theta_\nu^\Kc +\Half e_{\kc}(\gamma_{00})\theta^\kc_\nu
-\omega_{0}{}^0{}_3 \theta^3_\nu && \text{(by \eqref{gamma00bc} \& \eqref{CartanB})} \notag \\
&=  -\Ed\theta^0(e_{\Kc},e_0)\theta_\nu^\Kc +\Half e_{\kc}(\gamma_{00})\theta^\kc_\nu
-\omega_{0}{}^0{}_3 \theta^3_\nu && \text{(by \eqref{CartanA})} \notag \\
&= Y_\Kc \theta_\nu^\Kc  -\omega_{0}{}^0{}_3 \theta^3_\nu + \Half e_{\kc}(\gamma_{00})\theta^\kc_\nu, \label{lemC.19}
\end{align}
where in deriving the last equality we used \eqref{fcon}, \eqref{bconsat}, \eqref{aconsat} and \eqref{Ydef}.
Since the frame components $e_\kc^\mu$ satisfy
$n_\mu e_\kc^\mu = 0$ in $\Gamma_T$ by \eqref{lemBa.3}, the vector fields $e_{\kc}$ are tangent to $\Gamma_T$,
and consequently,
$e_\kc(\gamma_{00})= 0$ in $\Gamma_T$
by \eqref{aconsat} and \eqref{gamma00bc} .
Substituting this into \eqref{lemC.19}, we see using \eqref{psidef}, \eqref{vtothetah} and \eqref{vtoe0} that
\begin{equation} \label{thetadiffxibc}
\psi_\nu = Y_\Kc \theta_\nu^\Kc  -\omega_{0}{}^0{}_3 \theta^3_\nu \quad \text{in $\Gamma_T$,}
\end{equation}
which, in turn, implies that
\begin{equation} \label{lemC.20}
-\frac{\psi_\nu  - Y_\Kc\theta_\nu^\Kc}{\bigl|\psi  - Y_\Kc \theta^\Kc\bigr|_m} =  \thetah^3_\nu \quad \text{in $\Gamma_T$}
\end{equation}
since $\omega_{0}{}^0{}_3 =-\Ip_{e_3}\! \nabla_{e_0}\theta^0 >0$ in $\overline{\Gamma}_T$ by assumption.
Adding $\psih_\nu$ to both sides of \eqref{lemC.20} gives
\begin{equation}\label{psihtothetah3}
\psih_\nu +  \thetah^3_\nu  = \frac{\psi_\nu}{|\psi|_m} -\frac{\psi_\nu  - Y_\Kc\theta_\nu^\Kc}{\bigl|\psi  - Y_\Kc\theta^\Kc\bigr|_m} \quad \text{in $\Gamma_T$.}
\end{equation}
Applying $\nabla_v$ to this expression, we find after a short calculation that
\begin{equation} \label{dtpsihton}
\nabla_v (\psih_\nu +  \thetah^3_\nu) =
\af v(Y_\Kc)\theta_\nu^\Kc
+ \cf_\nu^\Kc Y_\Kc \quad \text{in $\Gamma_T$}
\end{equation}
for some functions $\cf_\mu^\Kc \in C^0(\overline{\Gamma}_T)$ and $\af\in C^0(\overline{\Gamma}_T)$ with
$\af  >  0$  in $\overline{\Gamma}_T$.

Substituting \eqref{thetadiffxibc} into \eqref{mbcidA}, we deduce via the anti-symmetry of $\Fc_{\beta\nu}$ that
\begin{equation} \label{mbcidB}
\theta^{3\beta}\Fc_{\beta\nu}\theta^{\Kc\nu}Y_\Kc =  \frac{\bigl(\gamma^{33}\omega_0{}^0{}_3 -\gamma^{3\Kc}Y_\Kc\bigr)}{|\theta^3|_g}
v(\hcbr) + \Qf(\hcbr) \quad \text{in $\Gamma_T$}
\end{equation}
We also observe from \eqref{Ndef}, \eqref{lemC.4} and \eqref{lemC.8b} that
\begin{equation} \label{theta3Ndiff}
\theta^3_\mu = N_\mu - \Qf_\mu(\hcbr).
\end{equation}
Noticing that
\begin{equation*}
p^{\mu\nu}\nabla_v \thetah^3_\nu = p^{\mu\nu}(\delta_\nu^\omega - \thetah^{3\omega}\thetah^3_\nu)
\frac{\nabla_v\theta^3_\omega}{|\theta^3|_g},
\end{equation*}
a short calculation using \eqref{pdef}, \eqref{psihtothetah3} and  \eqref{theta3Ndiff} shows that
\begin{equation*}
p^{\mu\nu}\nabla_v \thetah^3_\nu = p^{\mu\nu}\frac{\nabla_v\theta^3_\nu}{|N|_g} +  \Qf^\mu(\hcbr,|Y|_\Lambda).
\end{equation*}
Recalling that $\Pi^{\mu\nu} = h^{\mu}_\omega p^{\omega\nu}$, we can use the above expression along
with \eqref{Piproj}, \eqref{thetadiffxibc} and \eqref{dtpsihton}  to write \eqref{lemC.10} as
\begin{align*}
m^{\mu\nu}\theta^3_\alpha g^{\alpha\beta}\Fc_{\beta\nu} &=
 - \ep|N|_g\Pi^{\mu\nu} \af v(Y_\Kc)\theta_\nu^\Kc -m^{\mu\nu}\thetah^3_\nu v(\breve{\hc}) +\Qf^\mu(\hcbr,|Y|_\Lambda)\\
 &\qquad   - m^{\mu\nu}\Bigl(\theta^3_\alpha a^{\alpha\beta}\nabla_\beta \zeta - v^\alpha s_{\alpha\beta}{}^\gamma h^{\beta\omega}\nabla_\gamma \thetah^0_\omega \Bigr)v_\nu.
\end{align*}
Contracting this expression with $\theta_\mu^{\Lc}Y_\Lc$, we deduce, with the help of \eqref{mdef} and \eqref{vtoe0},
that
\begin{align*}
\theta^{3\beta}\Fc_{\beta\nu}\theta^{\Lc\nu}Y_\Lc =
 - \Lambda^{\Kc\Lc} \af v(Y_\Kc)Y_{\Lc}
-\frac{\gamma^{3\Kc}Y_\Kc}{|\theta^3|_g}v(\breve{\hc})
+\Qf^\mu(\hcbr,|Y|^2_\Lambda) \quad \text{in $\Gamma_T$,}
\end{align*}
which we can also write as
\begin{align}
\theta^{3\beta}\Fc_{\beta\nu}\theta^{\Lc\nu}Y_\Lc =
 - \frac{1}{2} v(\af|Y|_\Lambda^2)
-\frac{\gamma^{3\Kc}Y_\Kc}{|\theta^3|_g}v(\breve{\hc})
+\Qf^\mu(\hcbr,|Y|_\Lambda^2) \quad \text{in $\Gamma_T$.}\label{mbcidD}
\end{align}
Subtracting \eqref{mbcidD} from \eqref{mbcidB} then gives
\begin{equation}
v\biggl(\frac{\af}{2}|Y|_{\Lambda}^2  +|\theta^3|_g\omega_0{}^0{}_3 \hcbr\biggr) = \Qf^\mu(\hcbr,|Y|_\Lambda^2) 
 \quad \text{in $\Gamma_T$.}
\label{mbcidE}
\end{equation}

To continue, we introduce a time function $\tau$ that foliates $\Omega_T$ and choose a normalized, future pointing timelike vector field  $\boldsymbol{\tau}=\tau^\nu\del{\nu}$ such that $\tau^\nu\del{\nu}\tau = 1$ in $\overline{\Omega}_T$,
$\tau^\nu n_\nu = 0$ in $\Gamma_T$, and
$\boldsymbol{\tau}$ is normal to the level sets
\begin{equation*}
\Omega(t) = \tau^{-1}(t)\cap \Omega_T \cong \{t\}\times \Omega_0.
\end{equation*}
We additionally define the sets
\begin{align*}
\Gamma(t) &= \tau^{-1}(t)\cap \Gamma_T \cong \{t\}\times \del{}\Omega_0,\\
\Omega_t &= \bigcup_{0\leq \tilde{t} \leq t} \Omega(t) \cong [0,t]\times \Omega_0\\
\intertext{and}
\Gamma_t &= \bigcup_{0\leq \tilde{t} \leq t} \Gamma(t) \cong [0,t]\times \del{}\Omega_0 .
\end{align*}
Then, since $Y_\Kc$ and $\hcbr$ vanish on the initial hypersurface, integrating \eqref{mbcidE}
over $\Gamma_t$ yields the  integral estimate
\begin{equation} \label{conenergyAa}
\int_{\Gamma(t)} \Yc \lesssim
\int_{\Gamma_t} \bigl[ |Y|_{\Lambda}^2 + |\hcbr|+\Yc\bigr],
\end{equation}
where
\begin{equation} \label{Ycaldef}
\Yc = \biggl|\frac{\af}{2} |Y|_{\Lambda}^2  +|\theta^3|_g\omega_0{}^0{}_3\hcbr\biggr|.
\end{equation}
Furthermore, since $\omega_{0}{}^0{}_3 =-\Ip_{e_3}\! \nabla_{e_0}\theta^0 >0$ in $\overline{\Gamma}_T$, it is clear that the estimate
\begin{equation} \label{conenergyAb}
\int_{\Gamma(t)} \Yc \lesssim
\int_{\Gamma_t} \bigl[|Y|_{\Lambda}^2 + \Yc \bigr]
\end{equation}
follows directly from \eqref{conenergyAa}. Furthermore, integrating the above inequality in time gives
\begin{equation} \label{conenergyAc}
\int_{\Gamma_t} \Yc \lesssim
\int_0^t \int_{\Gamma_\tau} \bigl[|Y|_{\Lambda}^2 + \Yc \bigr].
\end{equation}

\bigskip

\noindent \underline{Energy estimate for $\gc$:} The next step is to establish an energy estimate for $\gc$, which we do by showing that $\gc$
satisfies a suitable evolution equation and boundary condition. 
We begin the derivation of the evolution equation by noting from \eqref{jconsat} that
\begin{equation} \label{hprojrem}
h^{\mu}_\nu \nabla_\beta \thetah^0_\mu = \nabla_\beta\thetah^0_\nu.
\end{equation}
We then proceed by computing
\begin{align*}
\nabla_\alpha\bigl(a^{\alpha\beta}[\nabla_\beta \theta^0_\nu + \sigma_{I}{}^0{}_J\theta^I_\beta
\theta^J_\nu]\bigr) &= \nabla_\alpha\bigl(a^{\alpha\beta}[\zeta\nabla_\beta \thetah^0_\nu + \sigma_{I}{}^0{}_J\theta^I_\beta
\theta^J_\nu]\bigr) +\nabla_\alpha(a^{\alpha\beta}\nabla_\beta \zeta)\thetah^0_\nu + a^{\alpha\beta}
\nabla_\beta\zeta\nabla_\alpha\thetah^0_\nu && \text{(by \eqref{theta0def})}\\
& = h^\mu_\nu(\breve{\Hc}_\mu-a^{\alpha\beta}\nabla_\alpha \zeta\nabla_\beta\thetah^0_\mu)
-\nabla_\alpha\thetah^{0\gamma}a^{\alpha\beta}[\zeta\nabla_\beta \thetah^0_\gamma + \sigma_{I}{}^0{}_J\theta^I_\beta
\theta^J_\gamma]\frac{\thetah^0_\nu}{\hat{\gamma}{}^{00}}\\
&-\bigl(\thetah^{0\gamma}\breve{\Hc}_\gamma +
\nabla_\alpha\thetah^{0\gamma}a^{\alpha\beta}[\zeta\nabla_\beta \thetah^0_\gamma + \sigma_{I}{}^0{}_J\theta^I_\beta
\theta^J_\gamma]\bigl)\thetah^0_\nu  + a^{\alpha\beta}\nabla_\beta\zeta\nabla_\alpha\thetah^0_\nu ,
\end{align*}
where in deriving the last equality we used \eqref{chidef}, \eqref{Wcdef}, \eqref{Ecaldef}-\eqref{Hbrdef}, \eqref{weqn.1}-\eqref{weqn.1b}, \eqref{Kcdef},
\eqref{bconsat}-\eqref{econsat} and \eqref{jconsat}. Using \eqref{pidef}, \eqref{xihdef}, \eqref{jconsat} and \eqref{hprojrem}, we can write the above equation as
\begin{equation} \label{Econsat2}
\nabla_\alpha\bigl(a^{\alpha\beta}[\nabla_\beta \theta^0_\nu + \sigma_{I}{}^0{}_J\theta^I_\beta
\theta^J_\nu]\bigr) =  \breve{\Hc}_\nu \quad \text{in $\Omega_T$.}
\end{equation}

Next, observing from \eqref{theta0def} and \eqref{DGframemet} that \eqref{jconsat} implies
\begin{equation} \label{zetaconsat}
\zeta = \sqrt{-\gamma^{00}}  \quad \text{in $\Omega_T$,}
\end{equation}
we calculate
\begin{align}
\nabla_\nu\biggl(\frac{1}{f}\theta^0_\mu\biggr)
&= -\frac{f'}{f^2}\nabla_\nu \zeta \theta^0_\mu + \frac{1}{f}\nabla_\nu \theta^0_\mu \notag \\
&= -\frac{1}{\zeta f}\biggl(1-\frac{1}{s^2}\biggr)\nabla_\nu\zeta \theta^0_\mu
+\frac{1}{f}\delta^\alpha_\mu \nabla_\nu \theta^0_\alpha && \text{(by \eqref{fdef})}\notag \\
&= -\frac{1}{f}\biggl(1-\frac{1}{s^2}\biggr)\frac{\theta^0_\mu \theta^0_\beta}{\gamma^{00}}
\nabla_\nu \theta^0_\alpha +\frac{1}{f}\delta^\alpha_\mu \nabla_\nu \theta^0_\alpha
&& \text{ (by \eqref{zetaconsat})} \notag \\
&= -a_\mu{}^\alpha \nabla_\nu \theta^0_\alpha, \label{lemA.2}
\end{align}
where $a^{\mu\nu}$ is the acoustic metric defined by \eqref{adef}. From this,
we see that
\begin{align}
\nabla_\nu \delta_g \biggl(\frac{1}{f}\theta^0 \biggr)
&= -\nabla_\nu \nabla^\mu \biggl(\frac{1}{f}\theta^0_\mu\biggr) \notag \\
&= -\nabla^\mu \nabla_\nu\biggl(\frac{1}{f}\theta^0_\mu\biggr)
+\frac{1}{f} R_{\nu\mu}{}^{\lambda\mu}\theta^0_\lambda && \text{(by \eqref{curvature})} \notag \\
&= \nabla_\mu\bigl(a^{\mu\alpha}\nabla_\nu\theta^0_\alpha\bigr)
+\frac{1}{f}R_\nu{}^\lambda \theta_\lambda^0 && \text{(by \eqref{lemA.2})} \notag \\
&= \nabla_\mu\bigl(a^{\mu\alpha}\nabla_\alpha \theta^0_\nu
+a^{\mu\alpha}\bigl[\nabla_\nu\theta^0_\alpha-\nabla_\alpha \theta^0_\nu \bigr]
\bigr)+ \frac{1}{f}R_\nu{}^\lambda \theta^0_\lambda, \notag
\end{align}
which, we observe, using \eqref{fcomp} and \eqref{cdconsat}, can be written as
\begin{equation} \label{lemA.3}
\nabla_\nu \delta_g \biggl(\frac{1}{f}\theta^0 \biggr) =
\nabla_\alpha\bigl(a^{\alpha\beta}\bigl[\nabla_\beta \theta^0_\nu
+\sigma_{I}{}^0{}_J\theta^I_\beta \theta^J_\nu \bigr]
\bigr) + \frac{1}{f}R_\nu{}^\lambda \theta^0_\lambda
+\nabla_\alpha\bigl(a^{\alpha\beta}\Fc_{\nu\beta}\bigr) \quad \text{in $\Omega_T$.}
\end{equation}
Using \eqref{DGconnectA} to express the last term in \eqref{lemA.3}
as
\begin{equation} \label{lemA.4}
\nabla_\alpha\bigl(a^{\alpha\beta}\Fc_{\nu\beta}\bigl)
= -\hat{\nabla}_\alpha\bigl(a^{\alpha\beta}\Fc_{\beta\nu}\bigr)
+ \bigl[C^\beta_{\omega\nu} a^{\omega\alpha}-\delta_\nu^\beta C^\omega_{\omega\lambda}
a^{\lambda\alpha}\bigr]\Fc_{\alpha\beta},
\end{equation}
where $C^{\beta}_{\omega\nu}$ is defined by \eqref{Cdef}
and $\hat{\nabla}$ is the Levi-Civita connection of the acoustic metric $a_{\alpha\beta}$,
we find from \eqref{Hdef}, \eqref{Econsat2}, \eqref{lemA.3} and \eqref{lemA.4} that
\begin{equation} \label{lemA.1}
\Ed \gc - \delta_a \Fc = 0  \quad \text{in $\Omega_T$.}
\end{equation}
Applying the codifferential $\delta_{a}$ to to this expression, shows, with the help of \eqref{DGcodiffB},
that $\gc$ satisfies the wave equation
\begin{equation}
\delta_{a}\! \Ed\gc = 0 \quad \text{in $\Omega_T$.} \label{gconsat1}
\end{equation}

Turing our attention to the boundary condition for $\gc$, we observe, see \eqref{hodgeperpB},
that  $*_g(\theta^1\wedge\theta^2\wedge\theta^3)$
is $g$-orthogonal to the co-frame fields $\theta^I$.  Since $\theta^0$
is orthogonal to the $\theta^I$ by \eqref{bconsat}, $\theta^0$ must
be proportional to $*_g(\theta^1\wedge\theta^2\wedge\theta^3)$, and so
\begin{equation} \label{perpA}
q *_g\theta^0 = \theta^1\wedge\theta^2\wedge\theta^3  \quad \text{in $\Omega_T$}
\end{equation}
for some function $q$. To determine $q$, we wedge the above expression with $\theta^0$ to get
$q \theta^0\wedge *_g\theta^0 = \theta^0\wedge \theta^1\wedge\theta^2\wedge\theta^3$
and then use the formulas
\eqref{hodge} and \eqref{volumeA} to obtain
$q =  \frac{\sqrt{-\det(\gamma^{kl})}}{\gamma^{00}}$.
Substituting this into \eqref{perpA} yields
$*_g  \biggl(\frac{\sqrt{-\det(\gamma^{kl})}}{\gamma^{00}} \theta^0\biggr) = \theta^1\wedge\theta^2\wedge\theta^3$
which, after taking the exterior derivative, gives
\begin{align*}
\Ed *_g  \biggl(\frac{\sqrt{-\det(\gamma^{kl})}}{\gamma^{00}} \theta^0\biggr)
&= \Ed\theta^1\wedge \theta^2 \wedge \theta^3 -\theta^1 \wedge \Ed \theta^2 \wedge \theta^3
+\theta^1\wedge \theta^2 \wedge \Ed\theta^3 \\
&= -\Half \sigma_L{}^1\!{}_M\theta^L\wedge \theta^M\wedge \theta^2 \wedge \theta^3
+\Half \sigma_L{}^2\!{}_M\theta^1\wedge \theta^L\wedge \theta^M \wedge \theta^3 - \Half \sigma_L{}^3\!{}_M\theta^1\wedge \theta^2\wedge \theta^L \wedge \theta^M\\
& = 0,
\end{align*}
where in deriving the last two equalities we used
the relations \eqref{cdconsat} and \eqref{econsat} along with the fact
that $\theta^I\wedge \theta^J\wedge \theta^K \wedge \theta^L = 0$ for any choice of
$I,J,K,L \in \{1,2,3\}$. But this implies
$\delta_g\biggl(\frac{\sqrt{-\det(\gamma^{ij})}}{\gamma^{00}}\theta^0\biggr)=0$ in $\Omega_T$,
which, in turn, implies that
\begin{equation}  \label{lemB.1}
\gc = - v(\hc) -\nabla_\nu v^\nu \hc
\quad \text{in $\Omega_T$}
\end{equation}
by \eqref{xidef}-\eqref{xihdef}, \eqref{gcon}-\eqref{hcon}, \eqref{jconsat} and \eqref{zetaconsat}.  Evaluating \eqref{lemB.1} on the boundary $\Gamma_T$ yields
the Dirichlet  boundary condition
\begin{equation} \label{gconsat2}
\gc = -v(\hc) -\nabla_\nu v^\nu \hc \quad \text{in $\Gamma_T$}.
\end{equation}

We can write the wave equation \eqref{gconsat1} in first order form as
\begin{equation*}
A^{\mu\nu\gamma}\hat{\nabla}_\gamma \Gc_\nu = 0,
\end{equation*}
where
\begin{equation*}
A^{\mu\nu\gamma} = -a^{\mu\nu}v^\gamma + v^\mu a^{\nu\gamma} + v^\nu a^{\mu\gamma}
\AND
\Gc_\nu = \nabla_\nu \gc.
\end{equation*}
Then, integrating the identity
\begin{equation*}
\hat{\nabla}_\gamma\bigl( \Gc_\mu A^{\mu\nu\gamma} \Gc_\nu) = \Gc_\mu \hat{\nabla}_\gamma  A^{\mu\nu\gamma} \Gc_\nu
\end{equation*}
over $\Omega_t$, we find, with the help of the Divergence Theorem and the fact that $\Gc_\mu$ vanishes on the initial hypersurface, that
\begin{equation*}
\int_{\Omega(t)}\Gc_\mu A^{\mu\nu\gamma}\tau_\gamma  \Gc_\nu - \int_{\Gamma_t}\Gc_\mu A^{\mu\nu\gamma}n_\gamma  \Gc_\nu
= - \int_{\Omega_t}\Gc_\mu \hat{\nabla}_{\gamma} A^{\mu\nu\gamma}  \Gc_\nu.
\end{equation*}
But since $A^{\mu\nu\gamma}\tau_\gamma$ is positive definite and $A^{\mu\nu\gamma}n_\gamma = v^{\mu}n^{\nu}+v^{\nu}n^{\mu}$
on $\Gamma_T$ by \eqref{wbc.3}, it follows from the above integral
relation that
\begin{equation} \label{conenergyBa}
\int_{\Omega(t)}|\nabla \gc|_m^2 - 2 \int_{\Gamma_t} n(\gc)v(\gc)
\lesssim   \int_{\Omega_t} |\nabla \gc|_m^2 
\end{equation}

Before continuing, we first establish an integral identity. Given two $C^1$ function $f_1$, $f_2$ on $\overline{\Omega}_T$ where $f_1$ vanishes on
$\Omega(0)$, and letting $D_\nu$ denote the induced connection on the boundary $\Gamma_T$, the integral identity
\begin{equation} \label{intid}
\int_{\Gamma_t} f_1 v(f_2) = -\int_{\Gamma(t)}\tau_\mu v^\mu f_1 f_2  - \int_{\Gamma_t}(D_\nu v^\nu f_1 - v(f_1))f_2
\end{equation}
follows from integrating  
 $f_1 v(f_2) = D_\nu( v^\nu f_1 f_2)-(D_\nu v^\nu f_1 - v(f_1))f_2$ on $\Gamma_t$ and using the Divergence Theorem.
Using this identity and \eqref{gconsat2}, we deduce the estimate
\begin{equation*}
\int_{\Omega(t)}|\nabla \gc|_m^2 +2 \int_{\Gamma(t)}\tau_\nu v^\nu n(\gc)v(\hc)
+2\int_{\Gamma_t}(D_\nu v^\nu n(\gc) - v(n(\gc)))v(\hc)
\lesssim  \int_{\Omega_t} |\nabla \gc|_m^2 
\end{equation*}
from \eqref{conenergyBa}. Another application of \eqref{intid} to the above inequality yields
\begin{equation*}
\int_{\Omega(t)}|\nabla \gc|_m^2 +2 \int_{\Gamma(t)}\tau_\nu v^\nu n(\gc)v(\hc)
\lesssim  \int_{\Gamma(t)}|\hc| +   \int_{\Gamma_t}|\hc| +  \int_{\Omega_t} |\nabla \gc|_m^2.
\end{equation*}
Integrating this expression in time, we have
\begin{equation*}
\int_{\Omega_t}|\nabla \gc|_m^2 + 2\int_{\Gamma_t}\tau_\nu v^\nu n(\gc)v(\hc) \lesssim 
\int_{\Gamma_t}|\hc| + \int_0^t  \int_{\Gamma_\tau}|\hc|\, d\tau + \int_{0}^t \int_{\Omega_\tau} |\nabla \gc|_m^2, d\tau.
\end{equation*}
By one final application of the identity \eqref{intid}, we arrive, with the help of \eqref{lemC.8b}, \eqref{Ycaldef} and the
positivity of $\omega_{0}{}^0{}_3 =-\Ip_{e_3}\! \nabla_{e_0}\theta^0$ on $\overline{\Gamma}_T$, at the energy estimate
\begin{equation} \label{conenergyBb}
\int_{\Omega_t}|\nabla \gc|_m^2  \lesssim \int_{\Gamma(t)}\bigl[\Yc + |Y|_{\Lambda}^2\bigr]
+  \int_{\Gamma_t}\bigl[\Yc + |Y|_{\Lambda}^2\bigr] + \int_0^t  \int_{\Gamma_\tau}\bigl[\Yc + |Y|_{\Lambda}^2\bigr]\, d\tau +
 \int_{0}^t \int_{\Omega_\tau} |\nabla \gc|_m^2\, d\tau.
\end{equation}

\bigskip

\noindent \underline{Energy estimate for $\Fc_{\mu\nu}$ and $Y_\Kc$:} We now look for a boundary condition and an evolution equation
for $\Fc_{\mu\nu}$ that will yield a useful energy estimate for $\Fc_{\mu\nu}$ and $Y_\Kc$.
We start by noticing that 
\begin{equation} \label{fconsat2}
\Ld_v \Ed\Fc =\Half \Ed \Bigl(\Ld_v(\sigma_I{}^0{}_J)\theta^I\wedge \theta^J
+ \sigma_I{}^0{}_J\Ld_v\theta^I\wedge \theta^J+ \sigma_I{}^0{}_J\theta^I\wedge \Ld_v\theta^J  \Bigr)
= 0 \quad \text{in $\Omega_T$}
\end{equation}
is a direct consequence of  \eqref{fcon}, \eqref{weqn.2}, \eqref{weqn.3}, \eqref{DGf.2} and \eqref{DGf.4}.
Since $\Ed\Fc$ vanishes initially, see \eqref{dfconiv},
it follows immediately from \eqref{fconsat2} that
\begin{equation} \label{fconsat3}
\Ed\Fc = 0 \quad \text{in $\Omega_T$.}
\end{equation}
Together \eqref{lemA.1} and \eqref{fconsat3} show that $\Fc$ satisfies Maxwell's equations with source $\Ed \gc$ in $\Omega_T$.

Next, we turn to the boundary conditions. We begin by computing
\begin{align}
n^\mu\Fc_{\mu\alpha}a^{\alpha\beta}F_{\nu\beta}v^{\nu} =& n^\mu\Fc_{\mu\alpha}a^{\alpha\beta}\bigl(\theta^\Kc_\beta e^\lambda_\Kc
+\theta^3_\beta e^\lambda_3+\theta^0_\beta e^\lambda_0 \bigr)F_{\nu\lambda}e^{\nu}_0&& \text{(by \eqref{vtoe0})} \notag\\
=& n^\mu\Fc_{\mu\alpha}a^{\alpha\beta}\bigl(\theta^\Kc_\beta e^\lambda_\Kc
+\theta^3_\beta e^\lambda_3\bigr)F_{\nu\lambda}e^{\nu}_0 && \text{(since $\Fc_{(\nu\lambda)}=0$)}\notag \\
=& |\theta^3|_g^{-1} \theta^{3\mu}\Fc_{\mu\alpha}\bigl(\theta^{\Kc\alpha}e^\lambda_\Kc
+\theta^{3\alpha} e^\lambda_3\bigr)F_{\nu\lambda}e^{\nu}_0 && \text{(by \eqref{adef}, \eqref{wbc.2} \& \eqref{lemBa.3})}\notag \\
=&|\theta^3|_g^{-1} \theta^{3\mu}\Fc_{\mu\alpha}\theta^{\Kc\alpha}e^\lambda_\Kc F_{\nu\lambda}e^{\nu}_0
&& \text{(since $\Fc_{(\mu\alpha)}=0$)} \notag\\
=&|\theta^3|_g^{-1} \theta^{3\mu}\Fc_{\mu\alpha}\theta^{\Kc\alpha}Y_\Kc
&& \text{(by \eqref{Ydef})}\notag\\
 =&- \frac{1}{2|\theta^3|_g} v(\af |Y|_\Lambda^2)
-\frac{\gamma^{3\Kc}Y_\Kc}{|\theta^3|^2_g}v(\breve{\hc})
+\Qf^\mu(\hcbr,|Y|_\Lambda^2) && \text{(by \eqref{mbcidD})} \notag
\end{align}  
from which we obtain
\begin{equation} \label{conMaxbc}
n^\mu\Fc_{\mu\alpha}a^{\alpha\beta}F_{\nu\beta}v^{\nu} =   - v\biggl(\frac{\af|Y|_\Lambda^2}{2|\theta^3|_g}+
\frac{\gamma^{3\Kc}Y_\Kc\breve{\hc}}{|\theta^3|^2_g}\biggr)
+\Qf^\mu(\hcbr,|Y|_\Lambda^2) \quad \text{in $\Gamma_T$.}
\end{equation}
Noting that $n^\mu$ satisfies $n^\mu = |n|_a^{-1} n^\mu$ on $\Gamma_T$ by \eqref{adef} and  \eqref{wbc.2}, and recalling that $\Fc$ vanishes on the initial 
hypersurface, we then deduce from Lemma \ref{Maxlem} 
and \eqref{conMaxbc} that inequality
\begin{align}
\int_{\Omega(t)} \tauh^{\nu}T_{\nu\mu}v^{\mu} + 2\int_{\Gamma_t}  v\biggl(&\frac{\af|Y|_\Lambda^2}{2|\theta^3|_g}+
\frac{\gamma^{3\Kc}Y_\Kc\breve{\hc}}{|\theta^3|^2_g}\biggr)  \lesssim \int_{\Gamma_t}\bigl[|\hcbr| +|Y|_\Lambda^2\bigr]\notag\\
& -\int_{\Omega_t} \bigl[{\Half} T_{\alpha\beta}a^{\alpha\mu}a^{\beta\mu}\Ld_v a_{\mu\nu} +2 a^{\alpha\mu}\nabla_\alpha \gc \Fc_{\mu\nu}v^\nu\bigr]
\label{conenergyCa}
\end{align}
holds, where
where $\tauh^\mu = (-a(\tau,\tau))^{-1/2}\tau^\mu$ and
\begin{equation*}
T_{\mu\nu} =a^{\alpha\beta}\Fc_{\mu\alpha}\Fc_{\nu\beta} - \Half a_{\mu\nu} a^{\alpha\omega}a^{\beta\delta}\Fc_{\alpha\beta}\Fc_{\omega\delta}.
\end{equation*}
Employing the identity \eqref{intid}, we can write \eqref{conenergyCa} as
\begin{align}
\int_{\Omega(t)} \tauh^{\nu}T_{\nu\mu}v^{\mu} + \int_{\Gamma(t)}(-\tauh_\mu v^\mu)\biggl(&\frac{\af|Y|_\Lambda^2}{2|\theta^3|_g}+
\frac{\gamma^{3\Kc}Y_\Kc\breve{\hc}}{|\theta^3|^2_g}\biggr)  \lesssim \int_{\Gamma_t}\bigl[|\hcbr| +|Y|_\Lambda^2\bigr]\notag\\
& -\int_{\Omega_t} \bigl[{\Half} T_{\alpha\beta}a^{\alpha\mu}a^{\beta\mu}\Ld_v a_{\mu\nu} +2 a^{\alpha\mu}\nabla_\alpha \gc \Fc_{\mu\nu}v^\nu\bigr]
\label{conenergyCb}
\end{align}
Noting that $-\tauh_\nu \xi^\nu > 0$, by virtue of $\tauh^\nu$ and $v^\nu$ both being
future pointing timelike vector fields, and $|\Fc|^2_m \lesssim \tauh^\nu \Fc_{\nu\mu}v^\nu$
by Lemma \ref{doclem}, we conclude from \eqref{conenergyCb}, Young's inequality, and the positivity of $\af$ on $\overline{\Gamma}_T$, that 
\begin{equation} \label{conenergyCba}
\int_{\Omega(t)}|\Fc|_m^2 + \int_{\Gamma(t)} |Y|_\Lambda^2 \lesssim \int_{\Gamma(t)} |Y|_\Lambda |\hbr| + \int_{\Gamma_t}\bigl[|\hcbr| +|Y|_\Lambda^2\bigr]+
\int_{\Omega_t} \biggl[\delta |\nabla\gc|^2_m + \biggl(\frac{1}{\delta}+1\biggr) |\Fc|_m^2\biggr]
\end{equation}
for any $\tilde{\delta} > 0$, where all implicit constants are independent of $\delta$. 

Since $Y_\Kc = 0$ on the initial hypersurface, for any given $\tilde{\delta} >0$, there exists a time $T_{\tilde{\delta}}\in (0,T]$ such that
$|Y|_{\Lambda} \leq \tilde{\delta}$ in $\Gamma_{T_{\tilde{\delta}}}$.
Using this, it follows from \eqref{Ycaldef}, the the positivity of $\omega_0{}^0_3=-\Ip_{e_3}\nabla_{e_0}\theta^0$ on $\overline{\Gamma}_T$ and \eqref{conenergyCba}
that the inequality
\begin{equation*} 
\int_{\Omega(t)}|\Fc|_m^2 + \int_{\Gamma(t)} |Y|_\Lambda^2 \lesssim \int_{\Gamma(t)}\tilde{\delta}\bigl[\Yc +|Y|_\Lambda^2\bigr] + \int_{\Gamma_t}\bigl[\Yc +|Y|_\Lambda^2\bigr]+
\int_{\Omega_t} \biggl[\delta |\nabla\gc|^2_m + \biggl(\frac{1}{\delta}+1\biggr) |\Fc|_m^2\biggr]
\end{equation*}
holds for $t\in [0,T_{\tilde{\delta}}]$, where now all the implied constants are independent of $\delta$ and $\tilde{\delta}$. Choosing $\tilde{\delta}>0$ small enough then shows, with the help of \eqref{conenergyAb}, that
\begin{equation} \label{conenergyCc}
\int_{\Omega(t)}|\Fc|_m^2 + \int_{\Gamma(t)} |Y|_\Lambda^2 \lesssim  \int_{\Gamma_t}\bigl[\Yc +|Y|_\Lambda^2\bigr]+
\int_{\Omega_t} \biggl[\delta |\nabla\gc|^2_m + \biggl(\frac{1}{\delta}+1\biggr) |\Fc|_m^2\biggr], \quad 0\leq  t \leq T_{\tilde{\delta}}.
\end{equation}
Integrating the above inequality in time, we conclude that
\begin{equation}
\int_{\Omega_t}|\Fc|_m^2 + \int_{\Gamma_t} |Y|_\Lambda^2 \lesssim   \int_0^t\int_{\Gamma_\tau}\bigl[\Yc +|Y|_\Lambda^2\bigr]\, d\tau
+\int_0^t \int_{\Omega_\tau} \biggl[\delta |\nabla\gc|_m + \biggl(\frac{1}{\delta}+1\biggr) |\Fc|_m^2\biggr]\,d\tau \label{conenergyCd}
\end{equation}
for  all $t\in [0,T_{\tilde{\delta}}]$ and $\delta>0$.

\bigskip

\noindent \underline{Propagation of $F$ and $\gc$ in $\Omega_T$:}  Assuming for the moment that $t\in [0,T_{\tilde{\delta}}]$, we see, by choosing $\delta>0$ small enough, that the estimate
\begin{equation} \label{conenergyDa}
\int_{\Omega_t}|\nabla \gc|_m^2  \lesssim \int_{\Omega_t}|\Fc|^2_m
+  \int_{\Gamma_t}\bigl[\Yc + |Y|_{\Lambda}^2\bigr] + \int_0^t  \int_{\Gamma_\tau}\bigl[\Yc + |Y|_{\Lambda}^2\bigr]\, d\tau +
 \int_{0}^t \int_{\Omega_\tau} |\nabla \gc|_m^2\, d\tau
\end{equation}
follows from \eqref{conenergyAb}, \eqref{conenergyBb} and \eqref{conenergyCc}. Then adding \eqref{conenergyCc} to \eqref{conenergyDa}
yields
\begin{equation*}
\int_{\Omega(t)}|\Fc|_m^2+\int_{\Gamma(t)} |Y|_\Lambda^2 + \int_{\Omega_t}|\nabla \gc|_m^2  \lesssim \int_{\Omega_t}|\Fc|^2_m
+  \int_{\Gamma_t}\bigl[\Yc + |Y|_{\Lambda}^2\bigr] + \int_0^t  \int_{\Gamma_\tau}\bigl[\Yc + |Y|_{\Lambda}^2\bigr]\, d\tau +
 \int_{0}^t \int_{\Omega_\tau} |\nabla \gc|_m^2\, d\tau
\end{equation*}
provided that $\delta$ is chosen small enough.
Futhermore, by adding \eqref{conenergyAb}, \eqref{conenergyAc} and \eqref{conenergyCd} to the above
inequality, we have 
\begin{align*}
\int_{\Omega(t)}|\Fc|_m^2 + &\int_{\Omega_t}\bigl[|\Fc|_m^2+|\nabla\gc|_m^2\bigr] + \int_{\Gamma(t)}\bigl[\Yc + |Y|_{\Lambda}^2\bigr]+\int_{\Gamma_t}\bigl[\Yc + |Y|_{\Lambda}^2\bigr]  \\
& \lesssim\int_{\Omega_t}|\Fc|_m^2 + \int_0^t\int_{\Omega_\tau}\bigl[|\Fc|_m^2+|\nabla\gc|_m^2\bigr] + \int_{\Gamma_t}\bigl[\Yc + |Y|_{\Lambda}^2\bigr]+
\int_0^t\int_{\Gamma_\tau}\bigl[\Yc + |Y|_{\Lambda}^2\bigr] d \tau.
\end{align*}
By Gronwall's inequality, we obtain\footnote{Here, we are using $\begin{displaystyle}\int_{\Gamma_t} \approx \int_0^t \int_{\Gamma(t)}\end{displaystyle}$
and $\begin{displaystyle}\int_{\Omega_t} \approx \int_0^t \int_{\Omega(t)}\end{displaystyle}$. } 
\begin{equation*}
\int_{\Omega(t)}|\Fc|_m^2 + \int_{\Omega_t}\bigl[|\Fc|_m^2+|\nabla\gc|_m^2\bigr] + \int_{\Gamma(t)}\bigl[\Yc + |Y|_{\Lambda}^2\bigr]+\int_{\Gamma_t}\bigl[\Yc + |Y|_{\Lambda}^2\bigr] =0, \quad 0\leq t < T_{\tilde{\delta}},
\end{equation*}
which in turn, implies 
\begin{equation*} 
Y_\Kc=0, \quad \Fc =  \Ed\theta^0 +{\Half} \sigma_I{}^0{}_J\theta^I\wedge \theta^J=0 \AND \gc =  \delta_g\biggl(\frac{1}{f}\theta^0\biggr) =0 \quad 
\text{in $\Omega_{T_{\tilde{\delta}}}$.}
\end{equation*}
But $Y_\Kc=0$ implies via the definition of $T_{\tilde{\delta}}$ that $T_{\tilde{\delta}}=T$, and therefore, we conclude that
\begin{equation} \label{Fcgcconsat}
\Fc =  \Ed\theta^0 +{\Half} \sigma_I{}^0{}_J\theta^I\wedge \theta^J=0 \AND \gc =  \delta_g\biggl(\frac{1}{f}\theta^0\biggr) =0 \quad \text{in $\Omega_{T}$.}
\end{equation}
For use below, we note, with the help of the Cartan structure equations \eqref{CartanA} and \eqref{cdconsat}, that
$\Fc=0$ is equivalent to
\begin{equation} \label{fconsat8}
\sigma_i{}^0{}_j = \omega_i{}^0{}_j - \omega_j{}^0{}_i \quad \text{in $\Omega_T$.}
\end{equation}
\bigskip

\noindent \underline{Propagation of $\hc$ in $\Omega_T$:}  Since
 $v(\hc) +\nabla_\nu v^\nu \hc=0$
in $\Omega_T$ by \eqref{lemB.1} and \eqref{Fcgcconsat},
and $\hc$ vanishes initially, see \eqref{hconiv}, it follows immediately from the uniqueness of solutions to transport equations that
\begin{equation} \label{hconsat}
\hc = -\frac{\sqrt{-\det(\gamma^{ij})}}{\sqrt{-\gamma^{00}}}-\frac{\zeta}{f(\zeta)} = 0 \quad \text{in $\Omega_T$.}
\end{equation}

\noindent \underline{Solution of the relativistic Euler equations:} With the proof of the propagation of constraints
complete, we now turn to showing that the frame $\theta^j_\nu$ determines a solution of the relativistic
Euler equations. To start, we use \eqref{bconsat} and \eqref{zetaconsat}  to express \eqref{hconsat} as
$\biggl(\frac{f\bigl((-\gamma_{00})^{-\frac{1}{2}}\bigr)}{(-\gamma_{00})^{-\frac{1}{2}}}\biggr)^2 = \det(\gamma_{IJ})$ in $\Omega_T$.
Applying $e_0$ to this expression, we obtain, after a short calculation using \eqref{fdef}, \eqref{bconsat}, \eqref{zetaconsat} and \eqref{CartanB}, that\footnote{In line with our raising and lowering conventions from Section \ref{raising},  $\omega_{ikj}$
is defined by $\omega_{ikj} = \gamma_{kl}\omega_i{}^l{}_j$.}
\begin{equation} \label{wav2eul1}
\frac{1}{s^2 \gamma_{00}}\omega_{000} - \gamma^{IJ}\omega_{0IJ} = 0 \quad \text{in $\Omega_T$},
\end{equation}
where $s^2 = s^2\bigl((-\gamma_{00})^{-\frac{1}{2}}\bigr)$,
while
\begin{equation} \label{wav2eul2}
\omega_{00J} + \omega_{0J0} = 0 \quad \text{in $\Omega_T$}
\end{equation}
follows from applying $e_0$ to \eqref{bconsat}.
We also note from \eqref{cdconsat}, \eqref{econsat1} and \eqref{fconsat8} that
\begin{equation} \label{wav2eul3}
\omega_{0jI}-\omega_{Ij0} = 0 \quad \text{in $\Omega_T$}.
\end{equation}
From   \eqref{bconsat} and \eqref{wav2eul1}-\eqref{wav2eul3}, it is then clear that the equations
\begin{align*}
\biggl( \biggl(3+\frac{1}{s^2}\biggr)\frac{1}{\gamma_{00}}-3\gamma^{00}\biggr)\omega_{000}
-\gamma^{IJ}\omega_{IJ0}-2\gamma^{0J}(\omega_{J00}+\omega_{0J0}) &= 0, \\
2\gamma^{I0}\omega_{000} + \gamma^{IJ}(\omega_{J00}+\omega_{0J0}) & =0,
\end{align*}
hold in $\Omega_T$. Setting
\begin{align*}
A^{ijk} = \biggl(3+\frac{1}{s^2}\biggr)\frac{\delta^i_0\delta^j_0\delta^k_0}{-\gamma_{00}} +
\delta^i_0\gamma^{jk} + \gamma^{ik}\delta^j_0+\gamma^{ij}\delta^k_0,
\end{align*}
a short calculation shows that the above equations can be written as
\begin{equation*}
A^{ijk}\omega_{kj0}= 0 \quad \text{in $\Omega_T$},
\end{equation*}
which, in turn, are easily seen to be equivalent to
\begin{equation}
A_{\mu\nu}{}^\gamma\nabla_{\gamma}w^\mu =0\quad \text{in $\Omega_T$}, \label{wav2eul4}
\end{equation}
where
\begin{equation*}
w^\mu := e^\mu_0 = \frac{1}{\gamma^{00}}g^{\mu\nu}\theta_\nu^0
\end{equation*}
and
\begin{equation*}
A_{\mu\nu}{}^\gamma = \biggl(3+\frac{1}{s^2}\biggr)\frac{w_\mu w_\nu}{-g(w,w)}w^\gamma
+\delta^\gamma_\nu w_\mu + \delta^\gamma_\mu w_\nu + w^\gamma g_{\mu\nu}.
\end{equation*}
From the discussion in Section II of \cite{Oliynyk:PRD_2012} and the definition $s^2=s^2(\zeta)$ and $p=p(\zeta)$ via
\eqref{s2} and \eqref{prel.1}-\eqref{zpfix}, respectively, we recognize \eqref{wav2eul4} as
the Frauendiener-Walton formulation of the relativistic Euler equations, see \cite{Frauendiener:2003,Walton:2005}.
With the help of the boundary conditions \eqref{wbc.2}-\eqref{wbc.3}, we deduce that the pair $\{\rho,v^\mu\}$ computed from $\theta^0_\mu$ using \eqref{theta0def},  \eqref{theta2rho}-\eqref{theta2v} and $\zeta=\sqrt{-g(\theta^0,\theta^0)}$
satisfy the IBVP for the relativistic liquid body given by \eqref{ibvp.1}-\eqref{ibvp.5}.

\bigskip

\noindent \underline{Taylor sign condition:} Since $\gamma^{0J}=0$ and
$\omega_{i}{}^k{}_0 = \omega_0{}^k{}_i$ in $\Omega_T$, and $\gamma^{00} = -1$ on $\Gamma_T$,
we observe that
\begin{equation*}
-\Ip_{e_3}\! \nabla_{e_0}\theta^0 \oset{\eqref{DGconndefA}}{=} \omega_0{}^0{}_3 = -\omega_{300} \oset{\eqref{CartanB}}{=} - \Half e_3(\gamma_{00}) = -\Half g(e_3,n) \nabla_n \gamma_{00} \quad \text{in $\Gamma_T$},
\end{equation*}
where we note that $g(e_3,n)=|\theta^3|_g^{-1}$ since $n=\thetah^3$ in $\Gamma_T$. From this result and assumption \eqref{cthm2}, it follows that
the inequality
\begin{equation*}
 -\frac{1}{2|\theta^3|_g}\nabla_n \gamma_{00} \geq  c_1 > 0  
\end{equation*}
holds on $\Gamma_T$. By Remark 2.1.(iii) of \cite{Oliynyk:Bull_2017}, this inequality is equivalent to
the Taylor sign condition, which is the condition that the pressure $p$ of the fluid
solution $\{\rho,v^\mu\}$ satisfies
\begin{equation*}
-\nabla_n p \geq c_p > 0  \quad \text{in $\Gamma_T$}
\end{equation*}
for some positive constant $c_p$. This completes the proof of the theorem.

\end{proof}

\begin{rem} \label{thm2rem}
From the proof of Theorem \ref{cthm}, it is clear the propagation of the constraints
$\{\ac,\bc^J,\cc^{k}_{ij},\dc^k_j,\ec^K,\kc\}$
depends only on the evolution equations \eqref{weqn.2} and \eqref{weqn.3} for $\theta^I$ and $\sigma_i{}^k{}_j$, respectively, and the boundary condition \eqref{wbc.3}.
\end{rem}

\section{Choice of constraints\label{choice}}

With the goal of establishing the local-in-time existence of solutions to guide us, we will, in this section, make particular
choices for the constraints that appear in the evolution equations \eqref{weqn.1} and boundary
conditions \eqref{wbc.1}. The reason for these particular choices will be discussed in more detail in
the following sections. We begin by setting
\begin{align}
\Ecr^{\mu}&= v(r\Ec^\mu)- 2|\psi|_m^{-2} r \Upsilon^\mu_\rho h_\sigma^\rho  \Ec^{[\sigma}\psi^{\nu]}v(\psi_\nu) \label{Ecrdef}
\intertext{and}
\Bcr^{\mu}&=  v(\theta^3_\alpha r\Wc^{\alpha\mu}- r\ell^\mu)+\theta^3_\alpha 2|\psi|_m^{-2} \Upsilon^\mu_\rho h_\sigma^\rho r\Wc^{\alpha[\nu}\psi^{\sigma]}
v(\psi_\nu)\notag \\
& \qquad - \Bigl[p^\mu_\sigma v(r\Lc^\sigma) + \Upsilon^\mu_\rho h^\rho_\sigma |\psi|_m^{-2} r\bigl(-\psi^\nu \Lc^\sigma +2\ell^{[\nu}\psi^{\sigma]}\bigr)v(\psi_\nu)
\Bigr] \label{Bcrdef}
\end{align}
where  $r$ is a positive fuction, which will be fixed later. We then consider the following system of evolution equations and boundary
conditions:
\begin{align}
\Ecr^\mu & = 0 && \text{in $\Omega_T$,} \label{eibvp.1} \\
\Bcr^\mu & = 0 && \text{in $\Gamma_T$,}\label{eibvp.2}\\
n_\mu v^\mu & = 0 && \text{in $\Gamma_T$,}\label{eibvp.3}\\
\nabla_\alpha(a^{\alpha\beta}\nabla_\beta \zeta) & = \Kc &&\text{in $\Omega_T$,}\label{eibvp.4} \\
\zeta & = 1 && \text{in $\Gamma_T$,} \label{eibvp.5} \\
\Ld_v\theta^I & = 0 && \text{in $\Omega_T$,} \label{eibvp.6}\\
v(\sigma_i{}^k{}_j) & = 0 && \text{in $\Omega_T$.} \label{eibvp.7}
\end{align}

The following proposition guarantees that solutions of \eqref{eibvp.1}-\eqref{eibvp.7}
will yield solutions to the relativistic Euler equations with vacuum boundary conditions provided that the initial data is chosen so that the constraints
\eqref{aconiv}-\eqref{bcicons} are satisfied initially.
 
\begin{prop} \label{cprop}
Suppose $\ep \in C^1(\overline{\Omega}_T)$ with $\ep > 0$ in $\overline{\Omega}_T$, $\kappa\geq 0$, $r\in C^1(\overline{\Omega}_T)$ and $r>0$, $\zeta \in C^2(\overline{\Omega}_T)$, $\thetah^0_\mu\in C^3(\overline{\Omega}_T)$, $\theta^J_\mu\in C^2(\overline{\Omega}_T)$, $\sigma_{i}{}^j{}_k\in C^1(\overline{\Omega}_T)$,
there exists constants  $c_0,c_1>0$ such that \eqref{cthm1} and \eqref{cthm2} are satisfied,
and  the quadruple $\{\zeta,\thetah^0_\mu,\theta^J_\mu,\sigma_i{}^k{}_j\}$ satisfies the initial conditions
\eqref{aconiv}-\eqref{bcicons}
and the equations \eqref{eibvp.1}-\eqref{eibvp.7}. Then
the constraints \eqref{acon}-\eqref{jcon} and \eqref{kcon} vanish in $\Omega_T$ and $\Gamma_T$, respectively, the pair $\{\rho,v^\mu\}$
determined from $\{\thetah^0_\mu,\zeta\}$ via the formulas \eqref{theta2rho}-\eqref{theta2v}
satisfy the relativistic Euler equations with vacuum boundary conditions given by \eqref{ibvp.1}-\eqref{ibvp.4}, and there exists a constant $c_p>0$
such that the Taylor sign condition $-\nabla_n p \geq c_p > 0$ 
holds on $\Gamma_T$.
\end{prop}
\begin{proof} The proof will follow directly from Theorem \ref{cthm} provided that we can show that
\begin{align}
\Ecr^\mu &= r \nabla_v \Ec^\mu + \Rf^{\mu}(\Ec)&&\text{in $\Omega_T$} \label{cprop1}
\intertext{and}
\Bcr^\mu &= r\nabla_v \Bc^\mu + \Qf^{\mu}(\Bc) && \text{in $\Gamma_T$,} \label{cprop2}
\end{align}
because if these hold, then by the \eqref{eibvp.1} and \eqref{eibvp.2}, we would have that
$r \nabla_v \Ec^\mu + \Rf^{\mu}(\Ec)=0$ in $\Omega_T$ and
$r\nabla_v \Bc^\mu + \Qf^{\mu}(\Bc)=0$ in $\Gamma_T$ . Since $\Ec^\mu$ and $\Bc^\mu$ vanish initially by assumption, see \eqref{evicons}-\eqref{bcicons}, it
would then follows from the equations and the uniqueness of solutions to transport equations that
$\Ec^\mu = 0$ in $\Omega_T$ and $\Bc^\mu=0$ in $\Gamma_T$. 

First, we observe that the validity of \eqref{cprop1} is a direct consequence of the definition \eqref{Ecrdef} of $\Ecr^\mu$.
Next, to establish \eqref{cprop2}, we note from the definition \eqref{Bcdef} of $\Bc^\mu$ and a straightforward calculation involving  
the relations $\psih^\mu= |\psi|_m^{-1}\psi^\mu$, $\Upsilon^\mu_\rho h_\sigma^\rho \psi^\sigma=\psi^\mu$, $p^\mu_\sigma + \psih^\mu\psih_\sigma = \delta^\mu_\sigma$,
and $\psi_\mu \Lc^\mu=0$ (see \eqref{Lcdef}-\eqref{Upsilondef}, \eqref{pdef}, \eqref{orthogA}, \eqref{projrem5}, \eqref{dtpsih} and \eqref{Piproj}) that 
\begin{align*}
v(r\Bc^\mu) &+ |\psi|_m^{-2}  \bigl( \psi^\mu v(\psi_\nu)  -  \psi^\sigma v(\psi_\sigma) \Upsilon_\rho^\mu h_\nu^\rho\bigr)r\Bc^\nu\\
&\;
= p_\sigma^\mu v(r\Bc^\sigma)  + |\psi|_m^{-2} \Upsilon^\mu_\rho h_\sigma^\rho \psi^\sigma v(r\psi_\nu\Bc^\nu) -  |\psi|_m^{-2} \psi^\nu v(\psi_\nu) 
\Upsilon_\rho^\mu h^\rho_\sigma r \Bc^\sigma\\
&\; = p^\mu_\sigma v(\theta^3_\alpha r \Wc^{\alpha\sigma}) + |\psi|_m^{-2} \Upsilon_\rho^\mu h^\rho_\sigma\psi^\sigma v(\theta^3_\alpha r\Wc^{\alpha\nu}\psi_\nu)
- |\psi|_m^{-2}\psi^\nu \Upsilon^\mu_\rho h^\rho_\sigma \theta^3_\alpha r\Wc^{\alpha\sigma}v(\psi_\nu)-p^\mu_\sigma v(r \Lc^\sigma)\\
&\qquad -p^\mu_\sigma v(r\ell^\sigma) - |\psi|^{-2}_m \Upsilon^\mu_\rho h^\rho_\sigma \psi^\sigma v(r\ell^\nu\psi_\nu)
+|\psi|_m^{-2}\psi^\nu \Upsilon^\mu_\rho h^\rho_\sigma (r\Lc^\sigma + r\ell^\sigma)v(\psi_\nu) \\
&\overset{\eqref{Bcrdef}}{=}\Bcr^{\mu}. 
\end{align*}
Thus \eqref{cprop2} holds and the proof is complete.
\end{proof}

We now use the freedom to choose $\ep$ by taking it to be a function of the form
\begin{equation} \label{epfix}
\ep = \ep(|\psi|_m).
\end{equation}
Next, we set
\begin{equation} \label{elltdef}
\ellt^\mu = \frac{1}{\sqrt{-\gammah^{00}|\gamma|}}
h^{\mu\alpha}s_{\alpha\beta}{}^\gamma h^{\beta\nu}\nabla_\gamma \thetah^0_\nu,
\end{equation}
and note by \eqref{elldef}, \eqref{Upsilondef}, \eqref{Pidef} and \eqref{projrem1} that
\begin{equation*}
r\ell^\mu = r\ellt^\mu - \ep p^\mu_\omega r\ellt^\omega.
\end{equation*}
Applying $\nabla_v$ to this expression, we find, using  that \eqref{pdef}, \eqref{orthogA}, \eqref{psihup}, \eqref{dtpsih} and \eqref{epfix}, that
\begin{align}
\nabla_v(r\ell^\mu) &= \nabla_v(r\ellt^\mu)-\nabla_v(\ep p^\mu_\omega r \ellt^\omega) \notag \\
&=(\delta^\mu_\omega - \ep p^\mu_\omega\bigr)\nabla_v(r\ellt^\omega) -\nabla_v(\ep p^\mu_\omega) r \ellt^\omega\notag \\
&=(\delta^\mu_\omega - \ep p^\mu_\omega\bigr)\nabla_v(r\ellt^\omega) -\ep' \psih^\nu \nabla_v \psi_\nu p^\mu_\omega r \ellt^\omega 
+\ep\nabla_v\psih^\mu \psih_\omega r \ellt^\omega+\ep \psih^\mu \nabla_v\psih_\omega  r \ellt^\omega  \notag \\
&= (\delta^\mu_\omega - \ep p^\mu_\omega\bigr)\nabla_v(r\ellt^\omega)+\frac{\ep}{|\psi_m|} \psih_\omega r \ellt^\omega p^{\mu\nu}\nabla_v\psi_\nu
-\ep' \psih^\nu \nabla_v \psi_\nu p^\mu_\omega r \ellt^\omega + \frac{\ep}{|\psi_m|} \psih^\mu r \ellt^\omega p_\omega^\nu \nabla_v\psi_\nu.
 \label{dtellA}
\end{align}
But since $\ellt^\omega$ satisfies $\ellt^\omega = \ellt^\lambda h_\lambda^\omega$ by \eqref{projrem1} and \eqref{elltdef}, 
it is clear from \eqref{Pidef} and \eqref{projrem3} that $\ellt^\omega p_\omega^\nu= \ellt^\omega \Pi_\omega^\nu$. Substituting this into
\eqref{dtellA} shows that
\begin{equation} \label{dtellB}
\nabla_v(r\ell^\mu) =  (\delta^\mu_\omega - \ep p^\mu_\omega\bigr)\nabla_v(r\ellt^\omega)+\frac{\ep}{|\psi_m|} \psih_\omega r \ellt^\omega p^{\mu\nu}\nabla_v\psi_\nu
-\ep' \psih^\nu  \Pi^\mu_\omega r \ellt^\omega \nabla_v \psi_\nu + \frac{\ep}{|\psi_m|} \psih^\mu \Pi_\omega^\nu  r \ellt^\omega \nabla_v\psi_\nu.
\end{equation}

Next, we deduce from \eqref{Upsilondef},  \eqref{Piproj}, \eqref{Pibrdef} and \eqref{Pibrproj} that the inverse of $\Upsilon^\mu_\nu$, denoted $\Upsilonch^\mu_\nu$, is given by
\begin{equation}\label{Upsilonchdef}
\Upsilonch^\mu_\nu = \frac{1}{1-\ep}\Pi^\mu_\nu + \Pibr^\mu_\nu.
\end{equation}
Applying $\Upsilonch^\sigma_\mu$ to \eqref{dtellB}, we find, with the help of \eqref{Piproj} and \eqref{Pibrproj}, that 
\begin{align}
\Upsilonch^\sigma_\mu \nabla_v(r\ell^\mu) &=  \Upsilonch^\sigma_\mu(\delta^\mu_\omega - \ep p^\mu_\omega\bigr)\nabla_v(r\ellt^\omega)+
\frac{\ep}{|\psi_m|} \psih_\omega r \ellt^\omega \Upsilonch^\sigma_\mu p^{\mu\nu} \nabla_v\psi_\nu \notag \\
&\qquad 
-\frac{\ep'}{1-\ep} \psih^\nu  \Pi^\sigma_\omega r \ellt^\omega \nabla_v \psi_\nu + \frac{\ep}{|\psi_m|} \psih^\sigma \Pi_\omega^\nu  r \ellt^\omega \nabla_v\psi_\nu.
 \label{dtellC}
\end{align}
We then fix the functional form of $\ep(y)$ by demanding that it satisfy the differential equation 
\begin{equation*} 
\frac{\ep'}{1-\ep} = \frac{\ep}{y},
\end{equation*}
which after integrating yields the explicit formula
\begin{equation} \label{epformula}
\ep(|\psi|_m) = \frac{|\psi|_m}{\ep_0+|\psi|_m},
\end{equation}
where $\ep_0$ is a arbitrary constant that we take to be positive.
This result and \eqref{dtpsih} allows us to write \eqref{dtellC} as
\begin{align}
\Upsilonch^\sigma_\mu \nabla_v(r\ell^\mu) &=  \Upsilonch^\sigma_\mu(\delta^\mu_\omega - \ep p^\mu_\omega\bigr)\nabla_v(r\ellt^\omega)+
\frac{\ep}{|\psi|_m} \psih_\omega r \ellt^\omega \Upsilonch^\sigma_\mu p^{\mu\nu} \nabla_v\psi_\nu
+ \frac{2\ep}{|\psi|_m} \psih^{[\sigma} \Pi_\omega^{\nu]}  r \ellt^\omega \nabla_v\psi_\nu.
 \label{dtellD}
\end{align}
This identity will be used in the following sections and is crucial to establishing the existence of solutions.

\section{Lagrangian wave formulation}

\subsection{Lagrangian coordinates\label{lag}}

In order for the wave formulation of the relativistic Euler equations given by \eqref{eibvp.1}-\eqref{eibvp.7} to be useful
for establishing the local-in-time existence of solutions or for other purposes such as constructing numerical solutions, the dynamical matter-vacuum
boundary must be fixed. We achieve this through the use of Lagrangian coordinates
\begin{equation*}
\phi\; : \: [0,T]\times \Omega_0 \longrightarrow \Omega_T \: :\: (\xb^\lambda) \longmapsto (\phi^{\mu}(\xb^\lambda))
\end{equation*}
that were defined previously by \eqref{lemBa.1a}. In the following, we use
\begin{equation*}
\delb{\mu} = \frac{\del{}\;}{\del{}\xb^\mu}
\end{equation*}
to denote partial derivatives with respect to the Lagrangian coordinates $(\xb^\mu)$, and we define the Jacobian and
its inverse by
\begin{equation} \label{Jdef}
J^\mu_\nu = \delb{\nu}\phi^\mu
\end{equation}
and
\begin{equation*} 
(\Jch^\mu_\nu) = (J^\mu_\nu)^{-1},
\end{equation*}
respectively.

\begin{Not} \label{Lagnot}
For scalars fields $f$ defined on $\Omega_T$, we employ the notation
\begin{equation} \label{ful}
\underline{f} = f\circ \phi
\end{equation}
to denote the pullback of $f$ by $\phi$ to $[0,T]\times \Omega_0$. More generally, we use this notation
also to denote the pullback of tensor field components $Q^{\mu_1\ldots\mu_r}_{\nu_1\ldots\nu_s}$ that are treated as scalar fields defined on
$\Omega_T$, that is
\begin{equation*} \label{Qul}
\underline{Q}^{\mu_1\ldots\mu_r}_{\nu_1\ldots\nu_s} = Q^{\mu_1\ldots\mu_r}_{\nu_1\ldots\nu_s}\circ \phi.
\end{equation*}
Using this notation, we can then write the geometric pullback of a tensor field $Q^{\mu_1\ldots\mu_r}_{\nu_1\ldots\nu_s}$
by $\phi$ as
\begin{equation*} \label{bardef}
\Qb^{\mu_1\ldots\mu_r}_{\nu_1\ldots\nu_s} := (\phi^*Q)^{\mu_1\ldots\mu_r}_{\nu_1\ldots\nu_s} =
\Jch^{\mu_1}_{\alpha_1}\cdots \Jch^{\mu_r}_{\alpha_r}J^{\beta_1}_{\nu_1}\cdots J^{\beta_s}_{\nu_s}
\underline{Q}^{\alpha_1\ldots\alpha_r}_{\beta_1\ldots\beta_s}.
\end{equation*}
\end{Not}

\medskip

Since the Lagrangian coordinates are defined via the flow of the fluid velocity $v$, it follows that the components of the pullback $\vb = \phi^*v$ are given by
\begin{equation} \label{xibcomp}
\vb^\mu = \delta^\mu_0.
\end{equation}
Substituting this into the transformation law $J^\mu_\nu \vb^\nu = \underline{v}^\mu$ shows that $\phi^\mu$ satisfies
\begin{equation} \label{phiev}
\delb{0}\phi^\mu = \underline{v}^\mu \oset{\eqref{theta2v}}{=} \frac{1}{\sqrt{-\underline{\hat{\gamma}}^{00}}}\underline{g}^{\mu\nu}
\underline{\thetah}{}^0_\nu,
\end{equation}
where we note, see \eqref{gammah00}, that
\begin{equation*}
\underline{\hat{\gamma}}^{00} = \underline{g}^{\mu\nu}\underline{\thetah}{}^0_\mu\underline{\thetah}{}^0_\nu.
\end{equation*}
In the Lagrangian representation, the map $\phi=(\phi^\mu)$ is treated as an unknown and \eqref{phiev} is viewed as an evolution equation for $\phi$.

Pulling back the evolution equations \eqref{eibvp.6} and \eqref{eibvp.7} using the map $\phi$, we see, with the help of the naturalness property
$\phi^* L_v = L_{\vb} \phi^*$ of Lie derivatives and
 formulas  \eqref{xibcomp} and \eqref{DGLie}, that
\begin{equation} \label{thetabevA}
\delb{0}\thetab^I_\mu = 0 \AND \delb{0}\underline{\sigma}_{i}{}^k{}_j = 0 \hspace{0.4cm} \text{in $[0,T]\times\Omega_0$}.
\end{equation}
By \eqref{lemBa.1a}, it is clear that $\phi$ satisfies $\phi(\Omega_0)=\Omega_0$ from which it follows that
\begin{equation*}
J^\mu_\nu(0,\xb^\Sigma) =  v^\mu(0,\xb^\Sigma)\delta^0_\nu + \delta^\mu_\Lambda \delta^\Lambda_\nu, \qquad \forall \, (\xb^\Sigma)\in \Omega_0 ,
\end{equation*}
by \eqref{phiev}. We will always choose initial data for $\theta_\mu^I$ so that \eqref{bconiv}, see also \eqref{bcon}, is satisfied, which is equivalent to
\begin{equation*}
v^\mu(0,\xb^\Sigma)\theta_\mu^I(0,\xb^\Sigma)=0, \quad \forall \,(\xb^\Sigma)\in \Omega_0.
\end{equation*}
The above two results together with \eqref{phiev}, \eqref{thetabevA} and the transformation law $\thetab^I_\mu = J^\nu_\mu \thetat^I_\nu$ yield the
explicit representations
\begin{align}
\thetab^I_\mu(\xb^0,\xb^\Sigma) = \delta^\Lambda_\mu\theta^I_\Lambda(0,\xb^\Sigma)  \AND \underline{\sigma}_i{}^k{\!}_j(\xb^0,\xb^\Sigma) = \sigma_i{}^k{\!}_j(0,\xb^\Sigma),
\quad \forall \; (\xb^0,\xb^\Sigma)\in [0,T]\times \Omega_0,  \label{sigmatev}
\end{align}
for the unique solution to the evolution equations \eqref{eibvp.6} and \eqref{eibvp.7}. We will futher assume that the initial data for
$\theta^I_\mu$ is chosen, see \eqref{lemBa.2}, so that
\begin{equation} \label{thetab3}
\thetab^3 = \thetab^3_\mu d\xb^\mu =\theta^3_\Lambda(0,\xb^\Sigma) d\xb^\Lambda
\end{equation}
defines an outward pointing conormal to the boundary $[0,T]\times \del{}\Omega_0$.

\begin{rem} \label{thetatIrem}
$\;$

\begin{enumerate}[(i)]
\item Since \eqref{sigmatev} represents the unique solution of the evolution equations \eqref{eibvp.6} and \eqref{eibvp.7} given the
choice of initial data,
we consider them as solved and remove these equations from the systems of equations under consideration. Furthermore, from the transformation law
$\thetab^I_\mu = J^\nu_\mu \underline{\theta}{}^I_\nu$ and  \eqref{sigmatev}, it is clear that we can express
$\underline{\theta}{}^I_\nu$ as
\begin{equation}\label{thetatI}
\underline{\theta}{}^I_\nu(\xb^0,\xb^\Sigma) = \Jch_\nu^\Lambda(\xb^0,\xb^\Sigma) \theta^I_\Lambda(0,\xb^\Sigma),
\quad \forall \: (\xb^0,\xb^\Sigma) \in [0,T]\times\Omega_0,
\end{equation}
 which, in particular, shows that the components $\underline{\theta}{}^I_\nu(\xb^0,\xb^\Sigma)$ are determined completely in terms of the initial data $\theta^I_\Lambda(0,\xb^\Sigma)$
and the derivatives of $\phi^\mu$, since by definition $(\Jch^\mu_\nu) = (\delb{\mu}\phi^\nu)^{-1}$.
\item The boundary condition \eqref{eibvp.3} is automatically satisfied since
$\thetab^3$ defines an outward pointing conormal to the boundary $[0,T]\times \partial \Omega_0$ and
$\vb^\mu \thetab^3_\mu =  0$ follows immediately from \eqref{xibcomp} and \eqref{thetab3}. We, therefore, consider the boundary condition \eqref{eibvp.3} as
satisfied and do not consider it further.
\end{enumerate}
\end{rem}

By definition, the frame field components $\eb^{\lambda}_j$ are given by
\begin{equation*}
(\eb^\lambda_j) = (\thetab^j_\lambda)^{-1} = \begin{pmatrix}\begin{displaystyle} \frac{1}{\thetab^0_0 - \thetab^0_\Omega \check{\thetab}^\Omega_K \thetab^K_0}\end{displaystyle} &
\begin{displaystyle} -\frac{\thetab^0_\Sigma \check{\thetab}^\Sigma_J}{\thetab^0_0 - \thetab^0_\Omega \check{\thetab}^\Omega_K \thetab^K_0}\end{displaystyle} \\
\begin{displaystyle} -\frac{\thetab^K_0 \check{\thetab}^\Lambda_K}{\thetab^0_0 - \thetab^0_\Omega \check{\thetab}^\Omega_K \thetab^K_0}\end{displaystyle}
& \begin{displaystyle} \check{\thetab}^\Lambda_J + \frac{\check{\thetab}^\Lambda_K \thetab^K_0 \thetab^0_\Sigma \check{\thetab}^\Sigma_J }{\thetab^0_0 - \thetab^0_\Omega \check{\thetab}^\Omega_K \thetab^K_0}
\end{displaystyle}
\end{pmatrix},
\end{equation*}
where $(\check{\thetab}^\Lambda_J) = (\thetab^J_\Lambda)^{-1}$,
which by \eqref{sigmatev}, reduces to
\begin{equation} \label{ebformula}
(\eb^\lambda_j) = (\thetab^j_\lambda)^{-1} = \begin{pmatrix}\begin{displaystyle} \frac{1}{\thetab^0_0}\end{displaystyle} &
\begin{displaystyle} -\frac{\thetab^0_\Omega \check{\theta}^\Omega_J(0,\xb^\Sigma)}{\thetab^0_0}\end{displaystyle} \\
0
& \check{\theta}^\Lambda_J(0,\xb^\Sigma)
\end{pmatrix}, \qquad  \bigl(\check{\theta}^\Lambda_J(0,\xb^\Sigma)\bigr) = \bigl(\theta^J_\Lambda(0,\xb^\Sigma)\bigr)^{-1}.
\end{equation}
Since $\eb^\lambda_\kc \thetab^3_\lambda = 0$ by duality, and $\thetab^3$ is conormal to $[0,T]\times \del{}\Omega_0$, we have that
$T\bigl([0,T]\times\del{}\Omega_0\bigr) = \text{Span}\{\, \eb_\kc=  \eb^\lambda_\kc \delb{\lambda} \, \}$,
which, by \eqref{ebformula}, implies that $\delb{0}$ and the vector fields $\Zb_\Kc$ defined by
\begin{equation} \label{Zdef}
\Zb_\Kc(\xb^0,\xb^\Sigma) =  \check{\theta}^\Lambda_\Kc(0,\xb^\Lambda)\delb{\Lambda}
\end{equation}
span the tangent space to the boundary $[0,T]\times \del{}\Omega_0$, that is
\begin{equation} \label{Zspan}
T\bigl([0,T]\times\del{}\Omega_0\bigr) = \text{Span}\{\,\delb{0}, Z_\Kc \, \}.
\end{equation}
Appealing to transformation law $\thetab^0_\nu = J^\mu_\nu \underline{\theta}{}^0_\mu$,
we see from \eqref{ebformula} and \eqref{Zdef} that
\begin{equation*} \label{ebKcrepA}
(\eb^\lambda_\Kc) =  \begin{pmatrix}
\begin{displaystyle} -\frac{\underline{\theta}{}_\omega^0 \Zb_\Kc(\phi^\omega)}{\underline{\theta}{}_\gamma^0 \delb{0}\phi^\gamma}\end{displaystyle} \\
 \Zb^\Lambda_\Kc
\end{pmatrix},
\end{equation*}
while
\begin{equation}
\thetab^0_0 = \underline{\theta}^0_\gamma \delb{0}\phi^\gamma = \frac{\underline{g}(\underline{\theta}{}^0,
\underline{\thetah}{}^0)  }{\sqrt{-\underline{\hat{\gamma}}^{00}}}  \label{thetab00}
\end{equation}
follows from \eqref{phiev}.
Combining these two results yields
\begin{equation} \label{ebKcrepB}
\eb_\Kc = -\frac{\sqrt{-\underline{\hat{\gamma}}^{00}}\underline{\theta}{}_\omega^0 \Zb_\Kc(\phi^\omega)}{\underline{g}(\underline{\theta}{}^0,
\underline{\thetah}{}^0) }\delb{0} +  \Zb_\Kc,
\end{equation}
which, when used in conjunction with the transformation law $\underline{e}{}^\mu_\Kc = J^\mu_\nu \eb^\nu_\Kc$, gives
\begin{equation} \label{etKrep}
\underline{e}{}^\mu_\Kc = \eb_\Kc(\phi^\mu) =-\frac{\sqrt{-\underline{\hat{\gamma}}^{00}}\underline{\theta}{}_\omega^0 \Zb_\Kc(\phi^\omega)}{\underline{g}(\underline{\theta}{}^0,
\underline{\thetah}{}^0) }\delb{0}\phi^\mu +  \Zb_\Kc(\phi^\mu).
\end{equation}
Using \eqref{theta0def}, we can express \eqref{ebKcrepB} and \eqref{etKrep} as
\begin{align}
\eb_\Kc &= \frac{\underline{\thetah}{}_\omega^0 \Zb_\Kc(\phi^\omega)}{\sqrt{-\underline{\hat{\gamma}}^{00}}}\delb{0} +  \Zb_\Kc  \label{ebKbndry}
\intertext{and}
\underline{e}{}^\mu_\Kc &= \frac{\underline{\thetah}{}_\omega^0 \Zb_\Kc(\phi^\omega)}{\sqrt{-\underline{\hat{\gamma}}^{00}}}\delb{0}\phi^\mu +  \Zb_\Kc(\phi^\mu),
\label{etbndry}
\end{align}
respectively.

Next, we consider the determinant of the frame metric evaluated in the Lagrangian coordinates. By
definition $|\gamma|=-\det(g(e_i,g_j))$, and so treating this as a function and pulling back by $\phi$ gives
\begin{equation} \label{detgammaLagA}
\underline{|\gamma|} = -\det\bigl(\underline{e}^\mu_i\underline{g}_{\mu\nu} \underline{e}^\nu_j\bigr) =\underline{|g|}\det(J)^2\det(\eb)^2,
\end{equation}
where in deriving the second equality we have used the transformation law $J^\mu_\nu \eb^\nu_j = \underline{e}^\mu_j$. We also
observe from \eqref{theta0def} and the boundary condition \eqref{eibvp.4} that
\begin{equation} \label{thetabndry}
\thetau^0_\nu =\thetahu^0_\nu \hspace{0.4cm} \text{in $[0,T]\times \del{}\Omega_0$}.
\end{equation}
Then   by \eqref{ebformula}, \eqref{thetab00} and \eqref{detgammaLagA}-\eqref{thetabndry}, we deduce that
\begin{equation} \label{detgammaLag}
\frac{1}{\sqrt{-\underline{\hat{\gamma}}{}^{00}\underline{|\gamma|}}} =
\frac{\det(\thetab^J_\Lambda)}{\det(J)\sqrt{\underline{|g|}}}
 \hspace{0.3cm}\text{in $[0,T]\times \del{}\Omega_0$},
\end{equation}
where the coframe components $\thetab_\Lambda^J$ are given by \eqref{sigmatev}.

\subsection{Wave formulation in Lagrangian coordinates\label{wformlag}}

We now turn to transforming the remaining evolution equations and boundary conditions, i.e. \eqref{eibvp.1}-\eqref{eibvp.2} and \eqref{eibvp.4}-\eqref{eibvp.5}, 
into Lagrangian coordinates. Writing \eqref{Ecaldef} as
\begin{equation} \label{EcaldefA}
\Ec^\mu = \del{\alpha}\Wc^{\alpha\mu}-\Hct^\mu,
\end{equation}
where
\begin{equation} \label{Hctdef}
\Hct^{\mu} = -\Gamma^\alpha_{\alpha\lambda}\Wc^{\lambda\mu}-
\Gamma^{\mu}_{\alpha\lambda}\Wc^{\alpha\lambda}+\Hc^\mu,
\end{equation}
we observe that the pull-back of the components $\Ec^{\mu}$, treated as scalars, by $\phi$ is given by
\begin{equation}\label{EcurepA}
\Ecu^\mu = \underline{\del{\alpha}\Wc^{\alpha\mu}} -\Hctu^\mu,
\end{equation}
where here and below we freely employ the notion \eqref{ful}.
Letting $\eta=\eta_{\mu\nu}dx^\mu dx^\nu$, where $(\eta_{\mu\nu})=\text{diag}(-1,1,1,1)$, denote the Minkowsi metric, we
recall the following transformation formula for the pull-back of the divergence of a vector field $Y=Y^\mu \del{\mu}$ by $\phi$:
\begin{equation*}
|\bar{\eta}|^{-\frac{1}{2}}\delb{\mu}\bigl(|\bar{\eta}|^{\frac{1}{2}} \bar{Y}^\mu\bigr) =  \underline{|\eta|^{-\frac{1}{2}}\del{\mu}\bigl(|\eta|^{\frac{1}{2}} Y^\mu\bigr)}.
\end{equation*}
Since $|\eta|=-\det(\eta_{\mu\nu}) = 1$,  $|\bar{\eta}|=-\det(J_\mu^\alpha \eta_{\alpha\beta} J^\beta_\nu) = \det(J)^2$, and
$\bar{Y}^\mu = \Jch^{\mu}_\nu \underline{Y}^\nu$, the above formula can be written as
\begin{equation}\label{divtrans}
\det(J)^{-1}\delb{\mu}\bigl(\det(J) \Jch^\mu_\nu \underline{Y}^\nu\bigr) =  \underline{\del{\mu}Y^\mu}.
\end{equation}
Using this,
we can write \eqref{EcurepA} as
\begin{equation} \label{EcurepB} 
 \Ecu^\mu = \frac{1}{\det(J)}\delb{\alpha}(\det(J)\Jch^\alpha_\lambda \Wcu^{\lambda \mu}) -\Hctu^\mu
\end{equation}
We now fix the free function $r$ in the definition of $\mathring{\Ecu}{}^{\mu}$, see \eqref{Ecrdef}, by setting 
\begin{equation}\label{rfix}
r = \det(J)\circ\phi^{-1} \quad \Longleftrightarrow \quad \ru = \det(J).
\end{equation}
Multiplying \eqref{EcurepB} by $\underline{r}$ gives
\begin{equation}\label{EcurepC}
\underline{r\Ec^\mu} = \delb{\alpha}(\det(J)\Jch^\alpha_\lambda \Wcu^{\lambda \mu}) - \det(J)\Hctu^\mu.
\end{equation}
Applying $\delb{0}$ to this expression, we get that
\begin{equation} \label{EcurepD}
\underline{v(r\Ec^{\mu})} = \delb{0}\delb{\alpha}(\det(J)\Jch^\alpha_\lambda \Wcu^{\lambda \mu}) - \delb{0}\bigl( \det(J)\Hctu^\mu\bigr).
\end{equation}
Using the the transformation law 
\begin{equation} \label{pdlaw}
\del{\mu} = \Jch^\nu_\mu \delb{\nu}
\end{equation}
for partial derivatives, we see from \eqref{Wcdef} that
\begin{equation}\label{Wcurep}
\det(J)\Jch^\alpha_\lambda \Wcu^{\lambda \mu} = \det(J)\zetau \mul^{\mu\nu}\ab^{\alpha\beta}\delb{\beta}\thetahu^0_\nu
+\det(J)\mul^{\mu\nu}\Jch^\alpha_\lambda \au^{\lambda\beta}\bigl[-\zetau \Gammau^\omega_{\beta\nu}\thetahu^0_\omega
+\sigmau_I{}^0{}_J\thetau^I_\beta\thetau^J_\nu\bigr].
\end{equation}
Substituting this into \eqref{EcurepD} yields
\begin{align}
\underline{v(r\Ec^{\mu})} &= 
\delb{\alpha}\Bigl( \det(J)\zetau \mul^{\mu\nu}\ab^{\alpha\beta}\delb{\beta}\delb{0}\thetahu^0_\nu
+\delb{0}\bigl( \det(J)\zetau \mul^{\mu\nu}\ab^{\alpha\beta}\bigr)\delb{\beta}\thetahu^0_\nu
\notag \\
&\qquad +\delb{0}\bigl(\det(J)\mul^{\mu\nu}\Jch^\alpha_\lambda \au^{\lambda\beta}\bigl[-\zetau \Gammau^\omega_{\beta\nu}\thetahu^0_\omega
+\sigmau_I{}^0{}_J\thetau^I_\beta\thetau^J_\nu\bigr]\bigr)\Bigr) - \delb{0}\bigl( \det(J)\Hctu^\mu\bigr). \label{EcurepE}
\end{align}
Introducing a new variable $\vartheta_\nu$ by
\begin{equation}\label{varthetadef}
\vartheta_\nu := \underline{\nabla_v \thetah^0_\nu}= \delb{0}\thetahu^0_\nu - \vu^\lambda \Gammau^\omega_{\lambda \nu}\thetahu^0_\omega,
\end{equation}
we can use $\vartheta_\nu$ to replace $\delb{0}\thetahu^0_\nu$ in \eqref{EcurepE}. Doing so gives
\begin{align}
\underline{v(r\Ec^{\mu})} &= 
\delb{\alpha}\Bigl( \det(J)\zetau \mul^{\mu\nu}\ab^{\alpha\beta}\delb{\beta}\vartheta_\nu
+\det(J)\zetau \mul^{\mu\nu}\ab^{\alpha\beta}\delb{\beta}( \vu^\lambda \Gammau^\omega_{\lambda \nu}\thetahu^0_\omega)
+\delb{0}\bigl( \det(J)\zetau \mul^{\mu\nu}\ab^{\alpha\beta}\bigr)\delb{\beta}\thetahu^0_\nu
\notag \\
&\qquad +\delb{0}\bigl(\det(J)\mul^{\mu\nu}\Jch^\alpha_\lambda \au^{\lambda\beta}\bigl[-\zetau \Gammau^\omega_{\beta\nu}\thetahu^0_\omega
+\sigmau_I{}^0{}_J\thetau^I_\beta\thetau^J_\nu\bigr]\bigr)\Bigr) - \delb{0}\bigl( \det(J)\Hctu^\mu\bigr). \label{EcurepEb}
\end{align}
Multiplying \eqref{EcurepEb} by $\Upsilonchu^\delta_\mu$, we obtain
\begin{align}
 \underline{\Upsilonch^\delta_\mu v(r\Ec)^{\mu}} &= 
\delb{\alpha}\Bigl( \det(J)\zetau \Upsilonchu^{\delta\nu}\ab^{\alpha\beta}\delb{\beta}\vartheta_\nu
+\Upsilonchu^{\delta}_\mu\Bigl[\det(J)\zetau \mul^{\mu\nu}\ab^{\alpha\beta}\delb{\beta}( \vu^\lambda \Gammau^\omega_{\lambda \nu}\thetahu^0_\omega)
+\delb{0}\bigl( \det(J)\zetau \mul^{\mu\nu}\ab^{\alpha\beta}\bigr)\delb{\beta}\thetahu^0_\nu
\notag \\
&\qquad +\delb{0}\bigl(\det(J)\mul^{\mu\nu}\Jch^\alpha_\lambda \au^{\lambda\beta}\bigl[-\zetau \Gammau^\omega_{\beta\nu}\thetahu^0_\omega
+\sigmau_I{}^0{}_J\thetau^I_\beta\thetau^J_\nu\bigr]\bigr)\Bigr]\Bigr) - \Upsilonchu^\delta_\mu\delb{0}\bigl( \det(J)\Hctu^\mu\bigr) \notag \\
& \quad - \delb{\alpha}\Upsilonchu^\delta_\mu \Bigl( \det(J)\zetau \mul^{\mu\nu}\ab^{\alpha\beta}\delb{\beta}\vartheta_\nu
+\det(J)\zetau \mul^{\mu\nu}\ab^{\alpha\beta}\delb{\beta}( \vu^\lambda \Gammau^\omega_{\lambda \nu}\thetahu^0_\omega)
+\delb{0}\bigl( \det(J)\zetau \mul^{\mu\nu}\ab^{\alpha\beta}\bigr)\delb{\beta}\thetahu^0_\nu
\notag \\
&\qquad +\delb{0}\bigl(\det(J)\mul^{\mu\nu}\Jch^\alpha_\lambda \au^{\lambda\beta}\bigl[-\zetau \Gammau^\omega_{\beta\nu}\thetahu^0_\omega
+\sigmau_I{}^0{}_J\thetau^I_\beta\thetau^J_\nu\bigr]\bigr)\Bigr),
  \label{EcurepEc}
\end{align}
where we have set
\begin{equation}\label{Upsilonchuup}
\Upsilonchu^{\delta\nu}= \Upsilonchu^\delta_\mu \mul^{\mu\nu}.
\end{equation}

Next, we observe that
\begin{align*}
r\del{\alpha}\bigl(v^{[\beta}\Wc^{\alpha][\mu}\psi^{\nu]} \del{\beta}\psi_\nu\bigr)
&= \frac{r}{2}\del{\alpha}\bigl(v^{\beta}\Wc^{\alpha[\mu}\psi^{\nu]}-v^{\alpha}\Wc^{\beta[\mu}\psi^{\nu]}  \bigr) \del{\beta}\psi_\nu \\
&= \frac{r}{2}\Bigl(\del{\alpha}v^{\beta} \Wc^{\alpha[\mu}\psi^{\nu]} + v^{\beta} \del{\alpha}\Wc^{\alpha[\mu}\psi^{\nu]}
+ v^\beta \Wc^{\alpha[\mu}\del{\alpha}\psi^{\nu]}\\
&\qquad -\del{\alpha}v^\alpha \Wc^{\beta [\mu}\psi^{\nu]}
-v\bigl(\Wc^{\beta[\mu}\psi^{\nu]}\bigr) \Bigr)\del{\beta}\psi_\nu.
\end{align*}
Using \eqref{EcaldefA} to replace $\del{\alpha}\Wc^{\alpha\mu}$ with $\Ec^\mu+\Hct^\mu$ in the above expression, we get 
\begin{equation} \label{EcurepEa}
4r\del{\alpha}\bigl(v^{[\beta}\Wc^{\alpha][\mu}\psi^{\nu]} \del{\beta}\psi_\nu\bigr)
= 2r \Ec^{[\mu}\psi^{\nu]}v(\psi_\nu) + \Ic^\mu
\end{equation}
where
\begin{equation} \label{Icdef}
\Ic^\mu = 2r\Bigl(\del{\alpha}v^{\beta} \Wc^{\alpha[\mu}\psi^{\nu]} + v^{\beta}\Hct^{[\mu}\psi^{\nu]}
+ v^\beta \Wc^{\alpha[\mu}\del{\alpha}\psi^{\nu]} -\del{\alpha}v^\alpha \Wc^{\beta [\mu}\psi^{\nu]}
-v\bigl(\Wc^{\beta[\mu}\psi^{\nu]}\bigr) \Bigr)\del{\beta}\psi_\nu.
\end{equation}
With the help of  \eqref{xibcomp}, \eqref{divtrans} and \eqref{rfix}, we can can express the contraction of \eqref{EcurepEa} with 
$\Upsilon^\mu_\rho h^\rho_\sigma$ in Lagrangian coordinates as
\begin{align}
\underline{-2r |\psi|_m^{-2}\Upsilon^\mu_\rho h^\rho_\sigma \Ec^{[\sigma}\psi^{\nu]}v(\psi_\nu)}
&= |\psiu|_{\mul}^{-2}\Upsilonu^{\mu}_\rho \hu^\rho_\sigma \delb{\alpha}\bigl( 4\det(J) \delta_0^{[\beta}\Jch^{\alpha]}_\lambda\Wcu^{\lambda[\nu}\psiu^{\sigma]} \delb{\beta}\psiu_\nu\bigr)
+ |\psiu|_{\mul}^{-2}\Upsilonu^{\mu}_\delta \hu^\delta_\sigma \Icu^{\sigma} \notag \\
&= \Upsilonu^{\mu}_\rho \delb{\alpha}\bigl( 4\det(J) |\psiu|_{\mul}^{-2} \hu^\rho_\sigma \delta_0^{[\beta}\Jch^{\alpha]}_\lambda\Wcu^{\lambda[\nu}\psiu^{\sigma]} \delb{\beta}\psiu_\nu\bigr)
\notag \\
&\quad -  \Upsilonu^{\mu}_\rho \delb{\alpha}(|\psiu|_{\mul}^{-2} \hu^\rho_\sigma) 4\det(J) \delta_0^{[\beta}\Jch^{\alpha]}_\lambda\Wcu^{\lambda[\nu}\psiu^{\sigma]} \delb{\beta}\psiu_\nu
+ |\psiu|_{\mul}^{-2}\Upsilonu^{\mu}_\rho \hu^\rho_\sigma \Icu^{\sigma}. \label{EcurepEd}
\end{align}
From \eqref{psi2dttheta} and \eqref{varthetadef}, we note that $\psiu_\nu$ can be expressed in terms of $\vartheta_\mu$ by
\begin{equation}\label{psi2dtthetaA}
\psiu_\mu = (-\gammahu^{00})^{-\frac{1}{2}}\hu^\nu_\mu \vartheta_\nu.
\end{equation}
Using this and multiplying \eqref{EcurepEd} by $\Upsilonchu^\delta_\mu$, we arrive at the expression
\begin{align}
 \underline{-\Upsilonch^\delta_\mu 2r |\psi|_m^{-2}\Upsilon^\mu_\rho h^\rho_\sigma \Ec^{[\sigma}\psi^{\nu]}v(\psi_\nu)}
&= \delb{\alpha}\Bigl( 4\det(J) |\psiu|_{\mul}^{-2} (-\gammahu^{00})^{-\frac{1}{2}} \hu^\delta_\sigma
 \delta_0^{[\beta}\Jch^{\alpha]}_\lambda\Wcu^{\lambda[\chi}\psiu^{\sigma]}\hu_\chi^\nu \delb{\beta}\vartheta_\nu
  \notag \\
&\qquad + 4\det(J) |\psiu|_{\mul}^{-2} \hu^\delta_\sigma \delta_0^{[\beta}\Jch^{\alpha]}_\lambda\Wcu^{\lambda[\chi}\psiu^{\sigma]} \vartheta_\nu
  \delb{\beta}\bigl((-\gammahu^{00})^{-\frac{1}{2}} \hu_\chi^\nu\bigr)
 \Bigr)
\notag \\
&\qquad -  \delb{\alpha}(|\psiu|_{\mul}^{-2} \hu^\delta_\sigma) 4\det(J) \delta_0^{[\beta}\Jch^{\alpha]}_\lambda\Wcu^{\lambda[\nu}\psiu^{\sigma]} \delb{\beta}\psiu_\nu
+ |\psiu|_{\mul}^{-2}\hu^\delta_\sigma \Icu^{\sigma}. \label{EcurepEe}
\end{align}
Adding \eqref{EcurepEc} and \eqref{EcurepEe} shows by the invertibility of $\Upsilonch^\delta_\mu$ that the evolution equation \eqref{eibvp.1}
can be expressed in Lagrangian coordinates as
\begin{equation}
\delb{\alpha}\bigl(\Asc^{\alpha\beta\mu\nu}\delb{\beta}\vartheta_\nu + \Xsc^{\alpha\mu}\bigr)=\Hsc^\mu \quad
\text{in $[0,T]\times \Omega_0$}, \label{EcrurepH}
\end{equation}
where
\begin{align}
\Asc^{\alpha\beta\mu\nu} &=   \det(J)\zetau \Upsilonchu^{\mu\nu}\ab^{\alpha\beta}+ 4\det(J) |\psiu|_{\mul}^{-2} (-\gammahu^{00})^{-\frac{1}{2}} 
 \delta_0^{[\beta}\Jch^{\alpha]}_\lambda\Wcu^{\lambda\chi}\hu_\chi^{[\nu}\hu_\sigma^{\mu]}\psiu^{\sigma}, \label{Ascdef}\\
\Xsc^{\alpha\mu} &= \Upsilonchu^{\mu}_\delta\Bigl[\det(J)\zetau \mul^{\delta\nu}\ab^{\alpha\beta}\delb{\beta}( \vu^\lambda \Gammau^\omega_{\lambda \nu}\thetahu^0_\omega)
+\delb{0}\bigl( \det(J)\zetau \mul^{\delta\nu}\ab^{\alpha\beta}\bigr)\delb{\beta}\thetahu^0_\nu
\notag \\
&\qquad +\delb{0}\bigl(\det(J)\mul^{\delta\nu}\Jch^\alpha_\lambda \au^{\lambda\beta}\bigl[-\zetau \Gammau^\omega_{\beta\nu}\thetahu^0_\omega
+\sigmau_I{}^0{}_J\thetau^I_\beta\thetau^J_\nu\bigr]\bigr)\Bigr]
\notag \\
&\qquad + 4\det(J) |\psiu|_{\mul}^{-2} \hu^\mu_\sigma \delta_0^{[\beta}\Jch^{\alpha]}_\lambda\Wcu^{\lambda[\chi}\psiu^{\sigma]} \vartheta_\nu
  \delb{\beta}\bigl((-\gammahu^{00})^{-\frac{1}{2}} \hu_\chi^\nu\bigr)\label{Xscdef}
\intertext{and}
\Hsc^\mu &= \delb{\alpha}\Upsilonchu^\mu_\delta \Bigl( \det(J)\zetau \mul^{\delta\nu}\ab^{\alpha\beta}\delb{\beta}\vartheta_\nu
+\det(J)\zetau \mul^{\delta\nu}\ab^{\alpha\beta}\delb{\beta}( \vu^\lambda \Gammau^\omega_{\lambda \nu}\thetahu^0_\omega)
+\delb{0}\bigl( \det(J)\zetau \mul^{\delta\nu}\ab^{\alpha\beta}\bigr)\delb{\beta}\thetahu^0_\nu
\notag \\
&\quad +\delb{0}\bigl(\det(J)\mul^{\delta\nu}\Jch^\alpha_\lambda \au^{\lambda\beta}\bigl[-\zetau \Gammau^\omega_{\beta\nu}\thetahu^0_\omega
+\sigmau_I{}^0{}_J\thetau^I_\beta\thetau^J_\nu\bigr]\bigr)\Bigr) +\Upsilonchu^\mu_\delta\delb{0}\bigl( \det(J)\Hctu^\delta\bigr) \notag \\
&\qquad +  \delb{\alpha}(|\psiu|_{\mul}^{-2} \hu^\mu_\sigma) 4\det(J) \delta_0^{[\beta}\Jch^{\alpha]}_\lambda\Wcu^{\lambda[\nu}\psiu^{\sigma]} \delb{\beta}\psiu_\nu
- |\psiu|_{\mul}^{-2}\hu^\mu_\sigma \Icu^{\sigma}.\label{Hscdef}
\end{align}
This complete the transformation of the evolution equation \eqref{eibvp.1} into Lagrangian coordinates.

Turing the the boundary conditions, we observe from \eqref{Bcrdef} and \eqref{xibcomp} that $\Bcr^\mu$, when expressed in Lagrangian coordinates, becomes
\begin{align*}
\underline{\Bcr}{}^{\mu} &= \delb{0}(\thetab^3_\alpha \det(J)\Jch^\alpha_\lambda \Wcu^{\lambda \mu}) +
\thetab^3_\alpha  2 \det(J) |\psiu|_{\mul}^{-2}\Upsilonu^\mu_\rho\hu^\rho_\sigma 
\Jch^{\alpha}_\lambda \Wcu^{\lambda[\nu}\psiu^{\sigma]}\delb{0}\psiu_\nu
-\delb{0}(\det(J)\ellu^\mu) \\
&\qquad -\Bigl[\pu^\mu_\sigma\delb{0}(\det(J)\Lcu^\sigma)+\Upsilonu^\mu_\rho \hu^\rho_\sigma |\psiu|_{\mul}^{-2}\det(J)
\bigl(-\psiu^\nu \Lcu^\sigma + 2 \ellu^{[\nu}\psiu^{\sigma]}\bigr)\delb{0}\psiu_\nu\Bigr].
\end{align*}
Using the fact that $\delb{0}\thetab^3_\alpha = 0$ and $\thetab^3_\alpha \delta^\alpha_0 =0$, we can write the above expression as
\begin{align*}
\underline{\Bcr}{}^{\mu} &= \thetab^3_\alpha\Bigl(\delb{0}(\det(J)\Jch^\alpha_\lambda \Wcu^{\lambda \mu}) +
 4 \det(J) |\psiu|_{\mul}^{-2}\Upsilonu^\mu_\rho\hu^\rho_\sigma 
\delta_0^{[\beta}\Jch^{\alpha]}_\lambda \Wcu^{\lambda[\nu}\psiu^{\sigma]}\delb{\beta}\psiu_\nu\Bigr)
-\delb{0}(\det(J)\ellu^\mu) \notag \\
&\qquad -\Bigl[\pu^\mu_\sigma\delb{0}(\det(J)\Lcu^\sigma)+\Upsilonu^\mu_\rho \hu^\rho_\sigma |\psiu|_{\mul}^{-2}\det(J)
\bigl(-\psiu^\nu \Lcu^\sigma + 2 \ellu^{[\nu}\psiu^{\sigma]}\bigr)\delb{0}\psiu_\nu\Bigr]. 
\end{align*}
Contracting this with $\Upsilonchu^\delta_\mu$, we get, with the help \eqref{Wcurep}, \eqref{varthetadef}, \eqref{psi2dtthetaA},
\eqref{Ascdef} and \eqref{Xscdef}, that
\begin{align}
\Upsilonchu^\delta_\mu\underline{\Bcr}{}^{\mu} &= \thetab^3_\alpha\bigl(\Asc^{\alpha\beta\delta\nu}\delb{\nu}\vartheta_\nu+\Xsc^{\alpha\delta}\bigr) 
-\Upsilonchu^\delta_\mu\delb{0}(\det(J)\ellu^\mu)  -\Bigl[\Upsilonchu^\delta_\mu \pu^\mu_\sigma\delb{0}(\det(J)\Lcu^\sigma) \notag \\
&\qquad +\hu^\delta_\sigma |\psiu|_{\mul}^{-2}\det(J)
\bigl(-\psiu^\nu \Lcu^\sigma + 2 \ellu^{[\nu}\psiu^{\sigma]}\bigr)\delb{0}\psiu_v\Bigr]. \label{BcrrepA}
\end{align}

Next, we observe
\begin{equation} \label{dtpsinorm}
\delb{0}(|\psiu|_{\mul}^2)\overset{\eqref{orthogA}}{=}\delb{0}(\gu^{\tau \nu}\psiu_\rho \psiu_\nu) = 2\psiu^\nu \delb{0}\psiu_\nu + \delb{0}\gu^{\rho \nu}\psiu_\rho \psiu_\nu,
\end{equation}
and with the help of 
\eqref{elldef}-\eqref{Lcdef}, 
\eqref{detgammaLag}, 
and \eqref{varthetadef}, that we can write $\det(J)\ellu^\mu$ and $\det(J)\Lcu^\mu$ as
\begin{align}
\det(J)\ellu^\mu =&  |\gu|^{-\frac{1}{2}}\det(\thetab^J_\Lambda) \Upsilonu^\mu_\omega
\hu^{\omega\alpha}\su_{\alpha\beta}{}^\xi \hu^{\beta\nu}\bigl(\Jch_\xi^\gamma \delb{\gamma} \thetahu^0_\nu
-\Gammau_{\xi \nu}^\tau \thetahu^0_\tau\bigr)
\label{det(J)ellu}
\intertext{and}
\det(J)\Lcu^{\mu}=& |\gu|^{-\frac{1}{2}}\det(\thetab^J_\Lambda) \bigl(- \epu|\Nu|_{\gu} \hu^{\mu\nu}\underline{\nabla_v\psih_\nu}-\kappa \vu^\mu \vu^\nu
\vartheta_\nu  \bigr),
\label{det(J)Lcu}
\end{align}
respectively. Moreover,   we
see that relations
\begin{equation} \label{elluLcu}
\hu^\delta_\sigma \Lcu^\sigma =  - \epu(1-\epu) \Upsilonchu^\delta_\mu \pu^\mu_\sigma |\gu|^{-\frac{1}{2}}\det(\thetab^J_\Lambda)|\Nu|_{\gu} \hu^{\sigma \nu}\underline{\nabla_v\psih_\nu}\AND \hu^\nu_\rho \ellu^\rho = \ellu^\nu
\end{equation}
hold by \eqref{Upsilondef}, \eqref{projrem1}, \eqref{dtpsihA}, \eqref{PiprojA}, \eqref{PibrprojA} and \eqref{Upsilonchdef}.
We further observe, for any vector field $Z^\sigma$, that
\begin{align*}
\epu \delb{0}\Bigl(\frac{1}{\epu}Z^\sigma\Bigr)=\delb{0}Z^\sigma - \delb{0}(\ln(\epu))Z^\sigma 
\overset{\eqref{epformula}}{=} \delb{0}Z^\sigma - \frac{(1-\epu)}{|\psiu|_{\mul}}\delb{0}|\psiu|_{\mul} Z^\sigma.
\end{align*}

Setting $Z^\sigma =-\epu  |\gu|^{-\frac{1}{2}}\det(\thetab^J_\Lambda)|\Nu|_{\gu} \hu^{\sigma\nu} \underline{\nabla_v\psih_\nu}$
in this formula, we then get from \eqref{dtpsinorm}, \eqref{det(J)Lcu} and \eqref{elluLcu} that
\begin{align}
\Upsilonchu^\delta_\mu \pu^\mu_\sigma\delb{0}(&\det(J)\Lcu^\sigma) -\hu^\delta_\sigma |\psiu|_{\mul}^{-2}\det(J)
\psiu^\nu \Lcu^\sigma \delb{0}\psiu_v = \Upsilonchu^\delta_\mu \pu^\mu_\sigma\Biggl[
-\epu \delb{0}\Bigl( |\gu|^{-\frac{1}{2}}\det(\thetab^J_\Lambda)|\Nu|_{\gu} \hu^{\sigma\nu}\underline{\nabla_v\psih_\nu}\Bigr)
\notag \\
& -\frac{\epu(1-\epu)}{2} \delb{0}\gb^{\rho\tau}\psiu_\rho\psiu_\tau |\gu|^{-\frac{1}{2}}\det(\thetab^J_\Lambda) |\Nu|_{\gu} \hu^{\sigma\nu} |\psiu|^{-2}_{\mul}\underline{\nabla_v\psih_\nu}
 -\kappa\delb{0}\bigl( |\gu|^{-\frac{1}{2}}\det(\thetab^J_\Lambda)\vu^\sigma \vu^\nu
\vartheta_\nu \bigr) \Biggr]. \label{dtdet(J)Lcu}
\end{align}
Defining
\begin{equation} \label{Psidef}
\Psi_\nu =  
 \frac{1}{|\psiu|_{\mul}}\Bigl( \bigl(-\gammahu^{00}\bigr)^{-\frac{1}{2}} \delb{0}\vartheta_\nu +\hu_\nu^\omega \delb{0}\bigl(\bigl(-\gammahu^{00}\bigr)^{-\frac{1}{2}}\hu_\omega^\tau\bigr)\vartheta_\tau
- \hu_\nu^\omega \vu^\lambda \Gammau_{\lambda\omega}^\tau\psiu_\tau \Bigr),
\end{equation}
we have from \eqref{projrem1} and \eqref{psi2dtthetaA} that
\begin{equation}\label{Psi2dtthetaA}
\hu^\nu_\mu \underline{\nabla_v\psi_\nu}=|\psiu|_{\mul}\hu^\nu_\mu \Psi_\nu,
\end{equation}
which in turn, implies by \eqref{dtpsihA} and \eqref{PiprojA} that
\begin{equation} \label{Psi2dttheta}
  \hu^{\mu\nu}\underline{\nabla_v\psih_\nu}=  \Piu^{\mu\nu}\Psi_\nu.
\end{equation}
Recalling that $v^\nu\psi_\nu=0$ by \eqref{orthogA}, we find, after differentiating, that
$\nabla_v v^\nu \psi_\nu + v^\nu \nabla_v \psi_\nu=0$, which, we observe, by \eqref{psidef} and \eqref{psihup}, is equivalent to
\begin{equation*}
 v^\nu \nabla_v \psi_\nu = |\psi|^2_{m}.
\end{equation*} 
We can use this result together with \eqref{pidef} and \eqref{Psi2dtthetaA} to express $\underline{\nabla_v\psi_\mu}$
as
\begin{equation} \label{Psi2dtthetaB}
 \underline{\nabla_v\psi_\mu}=  |\psiu|_{\mul}\hu_{\mu}^{\nu}\Psi_\nu -|\psiu|_{\mul}^2\vu_{\mu}.
\end{equation}
Moreover, by differentiating $v^\nu\psih_\nu=0$, it is not difficult to verify using a similar calculation and the identity \eqref{dtpsihA} that
\begin{equation} \label{Psi2dtthetaC}
\underline{\nabla_v\psih_\mu}=  \Piu_{\mu}^{\nu}\Psi_\nu -|\psiu|_{\mul}\vu_{\mu}.
\end{equation}

With the help of \eqref{projrem1}, \eqref{projrem5}, \eqref{Piproj}, \eqref{PiprojA}, \eqref{Pibrdef}, \eqref{Pibrproj}, \eqref{Upsilonchdef}, 
\eqref{Psidef} and \eqref{Psi2dttheta}, we  then observe that \eqref{dtdet(J)Lcu} can be expressed as
\begin{align}
\Upsilonchu^\delta_\mu \pu^\mu_\sigma\delb{0}(&\det(J)\Lcu^\sigma) -\hu^\delta_\sigma |\psiu|_{\mul}^{-2}\det(J)
\psiu^\nu \Lcu^\sigma \delb{0}\psiu_v = -\epu\Upsilonchu^\delta_\mu \pu^\mu_\sigma
\delb{0}\bigl( |\gu|^{-\frac{1}{2}}\det(\thetab^J_\Lambda)|\Nu|_{\gu} \Piu^{\sigma\nu}\Psi_\nu\bigr)
\notag \\
& -\frac{\epu}{2} \delb{0}\gb^{\rho\tau}\psiu_\rho\psiu_\tau |\gu|^{-\frac{1}{2}}\det(\thetab^J_\Lambda) |\Nu|_{\gu} |\psiu|^{-2}_{\mul}
\Piu^{\delta\nu}\Psi_\nu
 -\kappa|\gu|^{-\frac{1}{2}}\det(\thetab^J_\Lambda)|\psiu|_{\mul}\bigl(-\gammahu^{00}\bigr)^{\frac{1}{2}} \vu^\delta \vu^\nu
\Psi_\nu \notag \\
&\qquad -\Upsilonchu^\delta_\mu \pu^\mu_\sigma\kappa\delb{0}\bigl( |\gu|^{-\frac{1}{2}}\det(\thetab^J_\Lambda)\vu^\sigma \vu^\nu\bigr)
\vartheta_\nu. \label{dtdet(J)LcuA}
\end{align}
We further observe by \eqref{elltdef}, \eqref{dtellD}, \eqref{detgammaLag} and \eqref{rfix} that
\begin{align}
\Upsilonchu^\delta_\mu\delb{0}\bigl(\det(J)\ellu^\mu\bigr) &=  \Upsilonchu^\delta_\mu(\delta^\mu_\omega - \epu \pu^\mu_\omega\bigr)\underline{\nabla_v(r\ellt^\omega)} +
\epu|\psiu|_{\mul}^{-1}  \psihu_\omega \det(J) \elltu^\omega \Upsilonchu^\delta_\mu \pu^{\mu\nu} \underline{\nabla_v\psi_\nu} \notag \\
&\qquad 
+ 2\epu|\psiu|_{\mul}^{-1} \psihu^{[\delta} \Piu_\omega^{\nu]}  \det(J) \elltu^\omega \underline{\nabla_v\psi_\nu}- \Upsilonchu^\delta_\mu \vu^\lambda\Gammau_{\lambda\tau}^\mu \det(J)\ellu^\tau
 \label{dtellE}
\end{align}
where
\begin{equation}  \label{det(J)elltu}
\det(J)\elltu^{\mu}=\underline{r}\elltu^\mu = \Ssc^{\mu\nu\gamma}\bigl(\delb{\gamma}\thetahu^0_\nu - J_\gamma^\chi \Gammau_{\chi\nu}^\tau \thetahu^0_\tau\bigr)
\end{equation}
and
\begin{equation} \label{Sscdef}
\Ssc^{\mu\nu\gamma}= |\gu|^{-\frac{1}{2}}\det(\thetab^{J}_{\Lambda})\hu^{\mu\alpha}\su_{\alpha\beta}{}^\xi \hu^{\beta\nu}\Jch_\xi^\gamma.
\end{equation}
From \eqref{varthetadef} and \eqref{det(J)elltu}, we then have
\begin{align}
\underline{\nabla_v(r\ellt^\omega)} &= \delb{0}(\underline{r}\elltu^\omega) + \vu^\lambda\Gammau_{\lambda\tau}^\omega \underline{r}\elltu^\tau \notag \\
& = \Ssc^{\omega\nu\gamma}\delb{\gamma}\bigl(\vartheta_\nu + \vu^\lambda \Gamma_{\lambda\nu}^\tau \thetahu^0_\tau\bigr)
-\delb{0}\Ssc^{\omega\nu\gamma}\delb{\gamma}\thetahu^0_\nu 
 -\delb{0}\bigl(\Ssc^{\omega\nu\gamma}J_\gamma^\chi\Gammau_{\chi\nu}^\tau \thetahu^0_\tau\bigr)+ \vu^\lambda\Gammau_{\lambda\tau}^\omega \det(J)\elltu^\tau,
\label{dtellF}
\end{align}
and we note that
\begin{align}
\Upsilonchu^\delta_\mu\bigl(\delta^\mu_\omega -\epu \pu^\mu_\omega\bigr)\Ssc^{\omega\nu\gamma} &=
\Upsilonchu^\delta_\mu\Ssc^{\mu\nu\gamma} -\epu \Upsilonchu^\delta_\mu \pu^\mu_\omega\hu^\omega_\sigma \Ssc^{\sigma\nu\gamma} && \text{(by \eqref{projrem1} \& \eqref{Sscdef})} \notag \\
&= \bigl(\Upsilonchu^\delta_\mu-\epu \Upsilonchu^\delta_\sigma \Piu^\sigma_\mu\bigr) \Ssc^{\mu\nu\gamma} && \text{(by \eqref{Pidef})} \notag \\
&= \biggl(\frac{1}{1-\epu}\Piu^\delta_\mu + \Pibru^\delta_\mu - \frac{\epu}{1-\epu}\Piu^\delta_\mu\biggr)\Ssc^{\mu\nu\gamma} && \text{(by \eqref{Upsilonchdef}, \eqref{Piproj} \& \eqref{Pibrproj})} \notag \\
&= \Ssc^{\delta\nu\gamma},  \label{dtdellG}
\end{align}
\begin{align}
 \Upsilonchu^\delta_\mu \pu^{\mu\nu} |\psiu|_{\mul}^{-1} \underline{\nabla_v\psi_\nu} &=  |\psiu|_{\mul} \Upsilonchu^\delta_\mu \Piu^{\mu\nu}\Psi_\nu - |\psiu|_{\mul}  \Upsilonchu^\delta_\mu 
 \vu^\mu && \text{(by \eqref{Pidef}, \eqref{projrem5} \& \eqref{Psi2dtthetaB})} \notag \\
 &= \frac{1}{1-\epu} \Piu^{\delta\nu}\Psi_\nu -|\psiu|_{\mul}\vu^\delta  && \text{(by \eqref{Piproj}, \eqref{PiprojA}, \eqref{Pibrdef}, \eqref{Pibrproj}  \& \eqref{Upsilonchdef})}
\label{dtdellH}
\end{align}
and
\begin{align}
2|\psiu|_{\mul}^{-1}\psihu^{[\delta}\Piu^{\nu]}_\omega \underline{\nabla_v \psi_\nu} &= \psihu^\delta \Piu^\nu_\omega \hu^\rho_\nu \Psi_\rho - |\psiu|_{\mul} \psihu^\delta  \Piu^\nu_\omega\vu_\nu
-\Piu^\delta_\omega\psihu^\nu \hu_\nu^\rho \Psi_\rho + |\psiu|_{\mul}\Piu^\delta_\omega\psihu^\nu \vu_\nu && \text{(by \eqref{Psi2dtthetaB})}\notag \\
& = 2 \psihu^{[\delta}\Piu^{\nu]}_\omega\Psi_\nu, \label{dtdellI}
\end{align}
where in deriving the last equality we used \eqref{psihdef}, \eqref{orthogA}, \eqref{Piproj} and \eqref{PiprojA}.
Substituting \eqref{dtellF}-\eqref{dtdellI} into \eqref{dtellE} then gives
\begin{align}
\Upsilonchu^\delta_\mu\delb{0}\bigl(\det(J)\ellu^\mu\bigr) &=   \Ssc^{\delta\nu\gamma}\delb{\gamma}\vartheta_\nu  +
\biggl(\frac{\epu}{1-\epu}\det(J) \psihu_\omega  \elltu^\omega \Piu^{\delta\nu}+ 2\epu \det(J) \elltu^\omega\psihu^{[\delta} \Piu_\omega^{\nu]} \biggr)\Psi_\nu  \notag \\
&\qquad + \Upsilonchu^\delta_\mu\bigl(\delta^\mu_\omega -\epu \pu^\mu_\omega\bigr)\Bigl(
-\delb{0}\Ssc^{\omega\nu\gamma}\delb{\gamma}\thetahu^0_\nu 
 -\delb{0}\bigl(\Ssc^{\omega\nu\gamma}J_\gamma^\chi\Gammau_{\chi\nu}^\tau \thetahu^0_\tau\bigr)+ \vu^\lambda\Gammau_{\lambda\tau}^\omega \det(J)\elltu^\tau\Bigr) \notag\\
&\qquad  +  \Ssc^{\delta\nu\gamma}\delb{\gamma}\bigl(\vu^\lambda \Gammau_{\lambda\nu}^\tau \thetahu^0_\tau \bigr)
  -\epu\psihu_\omega \det(J) \elltu^\omega |\psiu|_{\mul}\vu^\delta- \Upsilonchu^\delta_\mu \vu^\lambda\Gammau_{\lambda\tau}^\mu \det(J)\ellu^\tau.
 \label{dtdellJ}
 \end{align}
We also observe from \eqref{elluLcu} and \eqref{Psi2dtthetaA}  that
\begin{align} 
2|\psiu|_{\mul}^{-2}\det(J) \ellu^{[\nu}\psiu^{\sigma]}\hu^\delta_\sigma  \delb{0}\psiu_v &= 2|\psiu|_{\mul}^{-2}\det(J) \ellu^{[\nu}\psiu^{\sigma]}\hu^\delta_\sigma
\bigl(\underline{\nabla_v \psi_\nu}+\vu^\lambda \Gammau_{\lambda \nu }^\tau \psi_\tau \bigr)\notag\\
&= 2|\psiu|_{\mul}^{-1}\det(J) \ellu^\omega\psiu^{\sigma}\hu_\omega^{[\nu} \hu^{\delta]}_\sigma \Psi_\nu +
2|\psiu|_{\mul}^{-2}\det(J) \ellu^{[\nu}\psiu^{\sigma]}\hu^\delta_\sigma \vu^\lambda \Gammau_{\lambda \nu }^\tau \psi_\tau .
\label{dtdellK}
\end{align}

From \eqref{BcrrepA}, \eqref{dtdet(J)LcuA}, \eqref{dtdellJ}, \eqref{dtdellK} and the invertibility of $\Upsilonchu^\delta_\mu$, we conclude that the Lagrangian formulation
of the boundary conditions \eqref{eibvp.2} is given by
\begin{align}
\thetab^3_\alpha\bigl(\Asc^{\alpha\beta\mu\nu}\delb{\nu}\vartheta_\nu+\Xsc^{\alpha\mu}\bigr)  &=  \Ssc^{\mu\nu\gamma}\delb{\gamma}\vartheta_\nu  
-\epu\Upsilonchu^\mu_\kappa \pu^\kappa_\sigma
\delb{0}\bigl( |\gu|^{-\frac{1}{2}}\det(\thetab^J_\Lambda)|\Nu|_{\gu} \Piu^{\sigma\nu}\Psi_\nu\bigr)\notag \\
&\qquad 
+\Psc^{\mu\nu}\Psi_\nu +\Gsc^\mu &&
\text{in $[0,T]\times \del{}\Omega_0$,} \label{BcrrepB}
\end{align}
where
\begin{align}
\Psc^{\mu\nu} &= \biggl(\frac{\epu}{1-\epu}\det(J) \psihu_\omega  \elltu^\omega  -\frac{\epu}{2} \delb{0}\gu^{\rho\tau}\psiu_\rho\psiu_\tau |\gu|^{-\frac{1}{2}}\det(\thetab^J_\Lambda) |\Nu|_{\gu} |\psiu|^{-2}_{\mul}\biggr)\Piu^{\mu\nu} \notag \\
&\qquad  -\kappa|\gu|^{-\frac{1}{2}}\det(\thetab^J_\Lambda)|\psiu|_{\mul}\bigl(-\gammahu^{00}\bigr)^{\frac{1}{2}} \vu^\mu \vu^\nu
+2\epu \det(J) \elltu^\omega\psihu^{[\mu} \Piu_\omega^{\nu]} +2|\psiu|_{\mul}^{-1}\det(J) \ellu^\omega\psiu^{\sigma}\hu_\omega^{[\nu} \hu^{\mu]}_\sigma
 \label{Pscdef}
\intertext{and}
\Gsc^{\mu} &= -\Upsilonchu^\mu_\kappa \pu^\kappa_\sigma\kappa\delb{0}\bigl( |\gu|^{-\frac{1}{2}}\det(\thetab^J_\Lambda)\vu^\sigma \vu^\nu\bigr)
\vartheta_\nu  + \Upsilonchu^\mu_\kappa\bigl(\delta^\kappa_\omega -\epu \pu^\kappa_\omega\bigr)\Bigl(
-\delb{0}\Ssc^{\omega\nu\gamma}\delb{\gamma}\thetahu^0_\nu \notag\\
 &\quad-\delb{0}\bigl(\Ssc^{\omega\nu\gamma}J_\gamma^\chi\Gammau_{\chi\nu}^\tau \thetahu^0_\tau\bigr) 
  + \vu^\lambda\Gammau_{\lambda\tau}^\omega \det(J)\elltu^\tau\Bigr)
 +  \Ssc^{\mu\nu\gamma}\delb{\gamma}\bigl(\vu^\lambda \Gammau_{\lambda\nu}^\tau \thetahu^0_\tau \bigr)
 \notag\\
 &\quad -\epu\psihu_\omega \det(J) \elltu^\omega |\psiu|_{\mul}\vu^\mu 
- \Upsilonchu^\mu_\kappa \vu^\lambda\Gammau_{\lambda\tau}^\kappa \det(J)\ellu^\tau  +
2|\psiu|_{\mul}^{-2}\det(J) \ellu^{[\nu}\psiu^{\sigma]}\hu^\mu_\sigma\vu^\lambda \Gammau_{\lambda \nu }^\tau \psi_\tau. \label{Gscdef}
\end{align}
This completes the transformation of the boundary condition \eqref{eibvp.2} into Lagrangian coordinates.

Finally, it is not difficult from \eqref{divtrans} and \eqref{pdlaw} to verify that the evolution equation \eqref{eibvp.4} and
boundary condition \eqref{eibvp.5} can be expressed in Lagrangian coordinates as
\begin{align}
\delb{\alpha}(\Msc^{\alpha\beta}\delb{\beta}\zetau) &= \Ksc && \text{in $[0,T]\times \Omega_0$,} \label{zetaueqn}\\
\zetau &= 1 && \text{in $[0,T]\times \del{}\Omega_0$,} \label{zetaubc}
\end{align}
where
\begin{align}
\Msc^{\alpha\beta} &= \det(J)|\gu|^{\frac{1}{2}}\ab^{\alpha\beta} \label{Mscdef}
\intertext{and}
\Ksc &= \det(J)|\gu|^{\frac{1}{2}} \underline{\Kc}. \label{Kscdef}
\end{align}

Summarizing, \eqref{phiev}, \eqref{varthetadef}, \eqref{EcrurepH}, \eqref{Psidef}, \eqref{BcrrepB}, \eqref{zetaueqn} and \eqref{zetaubc} show that the Lagrangian representation of the IBVP for the system \eqref{eibvp.1}-\eqref{eibvp.5} is given by
\begin{align}
\delb{\alpha}\bigl(\Asc^{\alpha\beta}\delb{\beta}\vartheta + \Xsc^{\alpha}\bigr)&=\Hsc &&
\text{in $[0,T]\times \Omega_0$}, \label{libvp.1} \\
\thetab^3_\alpha \bigl(\Asc^{\alpha\beta}\delb{\beta}\vartheta
+ \Xsc^{\alpha}\bigr) &=  \Ssc^{\gamma}\delb{\gamma}\vartheta
- \Rsc^{\tr}
\delb{0}(\Qsc \Psi) 
+ \Psc\Psi_\nu +\Gsc && \text{in $[0,T]\times \del{}\Omega_0$,} \label{libvp.2} \\
\delb{0}\phi &= \alpha \thetahu^0  && \text{in $[0,T]\times \Omega_0$}, \label{libvp.3} \\
\delb{0}\thetahu^0 & = \vartheta+\betat  \thetahu^0 &&
\text{in $[0,T]\times \Omega_0$}, \label{libvp.4} \\
\delb{\alpha}(\Msc^{\alpha\beta}\delb{\beta}\zetau) &= \Ksc && \text{in $[0,T]\times \Omega_0$,} \label{libvp.5}\\
\zetau &= 1 && \text{in $[0,T]\times \del{}\Omega_0$,} \label{libvp.6}\\
(\vartheta,\delb{0}\vartheta) &= (\vartheta_0,\vartheta_1)  && \text{in $\{0\}\times \Omega_0$,} \label{libvp.7}\\
\phi &= \phi_0  && \text{in $\{0\}\times \Omega_0$,} \label{libvp.8}\\
\thetahu^0 &= \thetahu^0_0  && \text{in $\{0\}\times \Omega_0$,} \label{libvp.9}\\
(\zetau,\delb{0}\zetau) &= (\zetau_{0},\zetau_{1})  && \text{in $\{0\}\times \Omega_0$,} \label{libvp.10}
\end{align} 
where we are now employing matrix notation,\footnote{Here, we are using the following matrix conventions: if $L=(L^{\mu\nu})$, $M=(M_\mu^\nu)$ and $N=(N_\mu^\nu)$
are matrices
and $Y=(Y_\mu)$ and $X=(X^\mu)$ are vectors, then the various matrix products are defined as follows:
\begin{equation*}
LY=(L^{\mu\nu}Y_\nu), \quad MY=(M_\mu^\nu Y_\nu), \quad LM=(L^{\mu\nu}M_{\nu}^\lambda)
\AND MN = (M_\mu^\nu N_\nu^\lambda).
\end{equation*}
With these conventions, we have that
\begin{equation*}
M^{\tr} L = (M^\lambda_\mu L^{\mu\nu}), \quad M^{\tr} X = (M^\nu_\mu X^\mu) \AND
M^{\tr}N^{\tr} = (M^\nu_\mu N^\mu_\lambda).
\end{equation*}} and we have made the definitions
\begin{align}
\Qsc^{\mu\nu}
&= |\gu|^{-\frac{1}{2}}\det(\thetab^J_\Lambda)|\Nu|_{\gu}\Piu^{\mu\nu}, \label{Qscdef}\\
\Rsc^{\mu}_\nu &= \epu\Upsilonchu^\mu_\sigma\pu^\sigma_\nu, \label{Rscdef} \\
\alpha^{\mu\nu} &= (-\underline{\hat{\gamma}}{}^{00})^{-\frac{1}{2}}\gu^{\mu\nu}, \label{alphadef}\\ 
\betat^\mu_\nu &= \vu^\lambda\Gammau_{\lambda \nu}^\mu, \label{betatdef}
\end{align}
and set
\begin{gather*}
\Asc^{\alpha \beta} = (\Asc^{\alpha\beta\mu\nu}), \quad \Ssc^\gamma  = (\Ssc^{\mu\nu\gamma}), \quad
  \Psc = (\Psc^{\mu\nu}), \quad \Qsc  = (\Qsc^{\mu\nu}), \quad \Rsc = (\Rsc^{\mu}_\nu) , \\ 
\alpha = (\alpha^{\mu\nu}), \quad \betat=(\betat^\mu_\nu), \quad
\Xsc^\alpha  = (\Xsc^{\alpha\mu}), \quad
\Hsc  = (\Hsc^\mu), \quad \Gsc = (\Gsc^\mu), \\
 \phi=(\phi^\mu), \quad \vartheta  = (\vartheta_\mu),  \quad
\underline{\thetah}{}^0 = (\underline{\thetah}{}^0_\mu), \AND \Psi  = (\Psi_\nu).
\end{gather*}
In this formulation, the unknowns to solve for are $\{\vartheta,\phi,\underline{\thetah}{}^0,\zetau\}$, while $\Psi$ is determined, see
\eqref{psi2dtthetaA} and \eqref{Psidef}, by
\begin{equation} \label{Psitopsi}
\delb{0}\vartheta = \lambda \Psi + \beta \vartheta
\end{equation}
where
\begin{align}
\lambda &=  (-\gammahu^{00})^{\frac{1}{2}}|\psiu|_{\mul}, \label{lambdadef} \\
\beta &= (\beta^\mu_\nu) := \Bigl(-\bigl(-\gammahu^{00}\bigr)^{\frac{1}{2}} \hu^\sigma_\nu \delb{0}\bigl(\bigl(-\gammahu^{00}\bigr)^{-\frac{1}{2}}
\hu_\sigma^\mu\bigr)
+\hu_\nu^\sigma \vu^\lambda \Gammau_{\lambda\sigma}^\tau \hu_\tau^\mu \Bigr) \label{betadef}
\intertext{and}
\psiu &= (\psiu_\nu) = \Bigl(\bigl(-\gammahu^{00}\bigr)^{-\frac{1}{2}}\hu_\nu^\mu\vartheta_\mu\Bigr).\label{psivecdef}
\end{align}

By \eqref{Zspan} and  \eqref{ebKbndry}, we know that the vector fields $\{\vb =\delb{0},\eb_{\Kc} = \eb^\mu_{\Kc}\delb{\mu}\}$ are
tangent to the boundary $[0,T]\times \del{}\Omega_0$ and  $\thetab^3 = \thetab^3_\mu d\xb^\mu$ is conormal to $[0,T]\times \del{}\Omega_0$.
From these facts and formulas \eqref{Sdef} and \eqref{Sscdef}, it it then clear that $\Ssc^{\mu\nu\gamma}$ satisfies 
\begin{align}
&\Ssc^{\mu\nu\gamma}\thetab^3_\gamma = 0 \hspace{0.4cm} \text{in  $[0,T]\times \del{}\Omega_0$} \label{Sscprop.1}
\intertext{and}
&\Ssc^{\mu\nu\gamma}= -\Ssc^{\nu\mu\gamma} \hspace{0.4cm} \text{in  $[0,T]\times \Omega_0$}. \label{Sscprop.2}
\end{align}
Note that \eqref{Sscprop.1} is equivalent to the statement that the operator $\Ssc^{\mu\nu\gamma}\delb{\gamma}$ involves
only derivatives that are tangential to the boundary $[0,T]\times \del{}\Omega_0$. The importance of \eqref{Sscprop.1}
and \eqref{Sscprop.2} is that they allow us to write the evolution equation \eqref{libvp.1} and boundary condition \eqref{libvp.2}
as
\begin{align}
\delb{\alpha}\bigl(\Bsc^{\alpha\beta}\delb{\beta}\vartheta+ \Xsc^{\alpha}\bigr) &= \Fsc  &&
\text{in $[0,T]\times \Omega_0$,} \label{libvp.1a} \\
\thetab^3_\alpha \bigl(\Bsc^{\alpha\beta}\delb{\beta}\vartheta+ \Xsc^{\alpha}\bigr)  &=-\Rsc^{\tr}\delb{0}(\Qsc \Psi) + \Psc\Psi + \Gsc
&& \text{in $[0,T]\times \del{}\Omega_0$,} \label{libvp.2a}
\end{align}
where
\begin{align}
\Bsc^{\alpha\beta} &= (\Bsc^{\alpha\beta\mu\nu}) := \biggl(\Asc^{\alpha\beta\mu\nu}+\frac{2}{|\thetab^3|_{\mb}^2}
\Ssc^{\mu\nu [\alpha}\mb^{\beta]\lambda}\thetab^3_\lambda \biggr) \label{Bscdef}
\intertext{and}
\Fsc &= (\Fsc^\mu) := \biggl(\Hsc^\mu + 
\delb{\alpha}\biggl(\frac{2}{|\thetab^3|_{\mb}^2}
\Ssc^{\mu\nu [\alpha}\mb^{\beta]\lambda}\thetab^3_\lambda \biggr)\delb{\beta}\vartheta_\nu\biggr). \label{Fscdef}
\end{align}
As we shall see in the following sections, this formulation of the evolution equation  \eqref{libvp.1} and boundary condition \eqref{libvp.2}  is crucial for establishing the existence of solutions. 
We further observe from \eqref{Ascdef} and \eqref{Bscdef} that  the spatial and $00$ components of $\Bsc^{\alpha\beta}$ are given by
\begin{align} 
\Bsc^{\Lambda\Sigma} &= \biggl(\det(J)\zetau\Upsilonchu^{\mu\nu}\ab^{\Lambda\Sigma}+ \frac{2}{|\thetab^3|_{\mb}^2}
\Ssc^{\mu\nu [\Lambda}\mb^{\Sigma]\lambda}\thetab^3_\lambda \biggr) \label{Bscspat}
\intertext{and}
\Bsc^{00} &= \bigl(\det(J)\zetau\Upsilonchu^{\mu\nu}\ab^{00}\bigr), \label{Bsctemp}
\end{align}
respectively.

\subsection{The time differentiated IBVP\label{tdiffIBVP}}
The  IBVP  \eqref{libvp.1}-\eqref{libvp.10}, see also \eqref{libvp.1a}-\eqref{libvp.2a}, is not yet in a form that will allow us apply the linear existence theory for wave equations from Appendix \ref{LWE}.
To remedy this, we will modify the IBVP by introducing a non-linear change of variables and differentiating the evolution equations and boundary conditions in time.
In order to define the change of variables, we first set
\begin{equation} \label{E0E3def}
E_0^\mu=\vu^\mu \AND E_3^\mu=\psihu^\mu.
\end{equation}
Then, by \eqref{theta2v}, \eqref{mdef}, \eqref{psihdef} and \eqref{orthogB}, we have that
\begin{equation}\label{E0E3orth}
\mul(E_0,E_0)=\mul(E_3,E_3) = 1 \AND \mul(E_0,E_3)=0,
\end{equation}
where $\mul(X,Y)$ is the positive definite inner-product defined by $\mul(X,Y) = \mul_{\alpha\beta}X^\alpha X^\beta$.
We can use the Gram-Schmidt algorithm to complete $\{E_0^\mu,E^\mu_3\}$ to an orthonormal
basis $\{E^\mu_0,E^\mu_1,E^\mu_2,E^\mu_3\bigr\}$
for $\mul$,  that is,
\begin{equation} \label{Eorth}
\mul(E_i,E_j) = \mul_{\mu\nu}E^\mu_i E^\nu_j = \delta_{ij}.
\end{equation}
We then let $\{\Theta^0_\mu,\Theta^1_\mu,\Theta^2_\mu,\Theta^3_\mu\}$ denote the dual basis, which is defined by
\begin{equation}
\Theta^i_\mu E^\mu_j = \delta^i_j, \label{Thetadef1}
\end{equation}
or equivalently, using the matrix notation
\begin{equation*}
E = (E^\mu_i) \AND \Theta = (\Theta^i_\mu),
\end{equation*}
by
\begin{equation} \label{Thetadef2}
\Theta E = \id\!\!,
\end{equation}
where $\id\!\!$ is the $4\times 4$ identity matrix. The dual basis is also orthonormal, that is,
\begin{equation}\label{Thetaorth}
\mul(\Theta^i,\Theta^j) := \mul^{\mu\nu}\Theta_\mu^i \Theta_\nu^j = \delta^{ij},
\end{equation}
as can be easily verified  from \eqref{Eorth} and \eqref{Thetadef1}.

Noting from \eqref{pidef}, \eqref{mdef}, \eqref{pdef} and \eqref{psihup} that we can write $\hu^\mu_\nu$ and $\pu^\mu_\nu$ as
\begin{equation*}
\hu^\mu_\nu = \delta^\mu_\nu - \vu^\mu \mul_{\nu\lambda}\vu^\lambda \AND \pu^\mu_\nu =\delta^\mu_\nu -\psih^\mu \mul_{\nu\lambda}\psih^\lambda,
\end{equation*} 
respectively, a short calculation using \eqref{E0E3def} and \eqref{Eorth}-\eqref{Thetadef1} shows that
\begin{equation} \label{hupubasis}
\Theta_\mu^i\hu^\sigma_\lambda E^\lambda_j = \delta^i_j - \delta^i_0\delta_j^0 \AND \Theta_\mu^i\hu^\sigma_\lambda E^\lambda_j = \delta^i_j - \delta^i_3\delta_j^3.
\end{equation}
Defining 
\begin{equation} \label{Piudef}
(\pu)=(\pu^\mu_\nu), \quad \Piu = (\Piu^{\mu\nu}), \quad \Upsilonchu = (\Upsilonchu^{\mu\nu}) \AND \vu\otimes \vu = (\vu^\mu\vu^\nu),
\end{equation}
it is then not difficult to see from \eqref{Pidef}, \eqref{Piup}-\eqref{Pibrup}, \eqref{Upsilonchdef}, \eqref{Upsilonchuup}, \eqref{Eorth}-\eqref{Thetadef1}, \eqref{Thetaorth} and \eqref{hupubasis} that
\begin{align}
E\pu \Theta &=\Pbb_3, \label{pumat} \\
\Theta^{\tr} \Piu \Theta &= \Pbb, \label{Piumat} \\
\Theta^{\tr}\Upsilonchu\Theta &= \frac{1}{1-\epu}\Pbb+ \Pbbbr, \label{Upsilonchumat}
\intertext{and}
\Theta^{\tr}\vu\otimes \vu \Theta &= \Pbb_0 \label{vuvumat}
\end{align}
where
\begin{align}
\Pbb_3 &= \begin{pmatrix} 1 & 0 & 0 & 0\\
0 & 1 & 0 & 0\\
0&  0 & 1 & 0 \\
0 & 0 & 0 & 0  \end{pmatrix},  \label{Pbb3def} \\
\Pbb &= \begin{pmatrix} 0 & 0 & 0 & 0\\
0 & 1 & 0 & 0\\
0&  0 & 1 & 0 \\
0 & 0 & 0 & 0  \end{pmatrix},  \label{Pbbdef} \\
\Pbbbr &= \begin{pmatrix} 1 & 0 & 0 & 0\\
0 & 0 & 0 & 0\\
0&  0 & 0 & 0 \\
0 & 0 & 0 & 1  \end{pmatrix}  \label{Pbbbrdef}
\intertext{and}
\Pbb_0 &= \begin{pmatrix} 1 & 0 & 0 & 0\\
0 & 0 & 0 & 0\\
0&  0 & 0 & 0 \\
0 & 0 & 0 & 0  \end{pmatrix}.  \label{Pbb0def}
\end{align}
With the help of \eqref{pumat}-\eqref{Upsilonchumat}, we then see from \eqref{Qscdef}-\eqref{Rscdef} that
\begin{align}
\Theta^{\tr}\Qsc\Theta &= q \Pbb \label{Qscmat}
\intertext{and}
E \Rsc \Theta &=\frac{\epu}{1-\epu}\Pbb + \epu\Pbb_0  \label{Rscmat}
\end{align}
where
\begin{equation}\label{qdef}
q= |\gu|^{-\frac{1}{2}}\det(\thetab^J_\Lambda)|\Nu|_{\gu}.
\end{equation}

With the preliminaries out of the way, we are now ready to define our non-linear change of variables by replacing $\Psi$ in favour of the variable $\Utt=(\Utt_j)$ defined by
\begin{equation} \label{Psi2U}
\Utt = E\Psi.
\end{equation}
Then multiplying \eqref{libvp.1a} and \eqref{libvp.2a} on the left by $\Theta^{\tr}$, a straightforward calculation shows, with the help of
\eqref{Thetadef2}, \eqref{Pbbdef}-\eqref{Rscmat}, and \eqref{Psi2U}, that
\begin{align}
\delb{\alpha}\bigl(\Theta^{\tr}\Bsc^{\alpha\beta}\delb{\beta}\vartheta+\Theta^{\tr}\Xsc^{\alpha}\bigr) &= 
\delb{\alpha}\Theta^{\tr}\bigl(\Bsc^{\alpha\beta}\delb{\beta}\vartheta+\Xsc^{\alpha}\bigr) +
\Theta^{\tr}\Fsc  &&
\text{in $[0,T]\times \Omega_0$,} \label{libvp.1b} \\
\thetab^3_\alpha \bigl(\Theta^{\tr}\Bsc^{\alpha\beta}\delb{\beta}\vartheta+\Theta^{\tr}\Xsc^{\alpha}\bigr)  &=-\frac{\epu q}{1-\epu}\Pbb\delb{0} \Utt
- q\biggl(\frac{\epu}{1-\epu}\Pbb + \epu \Pbb_0\biggr)\Theta^{\tr}\delb{0}E^{\tr} \Pbb \Utt
\notag \\
&\qquad +\biggl( -\frac{\epu}{1-\epu}\delb{0}q \Pbb+ \Theta^{\tr}\Psc\Theta\biggr) \Utt + \Theta^{\tr}\Gsc
&& \text{in $[0,T]\times \del{}\Omega_0$.} \label{libvp.2b}
\end{align}
Differentiating \eqref{libvp.1b}-\eqref{libvp.2b} with respect to $\xb^0$, we find, after  replacing $\delb{0}\vartheta$
with $\Utt$ using \eqref{Psitopsi} and \eqref{Psi2U}, that
\begin{align}
\delb{\alpha}\bigl(\Btt^{\alpha\beta}\delb{\beta}\Utt+\Xtt^{\alpha}\bigr) &= \Ftt
  &&
\text{in $[0,T]\times \Omega_0$,} \label{libvp.1c} \\
\thetab^3_\alpha \bigl(\Btt^{\alpha\beta}\delb{\beta}\Utt+\Xtt^{\alpha}\bigr)  &=\Qtt\delb{0}^2 \Utt
+\biggl[ -\biggl(\delb{0}\biggl(\frac{\epu}{1-\epu}\biggr)q  +\frac{2\epu}{1-\epu}\delb{0}q\biggr)\Pbb + \Theta^{\tr}\Psc\Theta  \notag \\
&\quad - q\biggl(\frac{\epu}{1-\epu}\Pbb + \epu \Pbb_0\biggr)\Theta^{\tr}\delb{0}E^{\tr} \Pbb\biggr]\delb{0}\Utt
- q\biggl(\frac{\epu}{1-\epu}\Pbb  \notag \\
&\quad + \epu \Pbb_0\biggr)\Theta^{\tr}\delb{0}^2 E^{\tr} \Pbb \Utt
- \delb{0}\biggl(q\biggl(\frac{\epu}{1-\epu}\Pbb + \epu \Pbb_0\biggr)\Theta^{\tr}\biggr)\delb{0}E^{\tr} \Pbb \Utt  \notag\\
&\quad+
\delb{0}\biggl( -\frac{\epu}{1-\epu}\delb{0}q \Pbb+ \Theta^{\tr}\Psc\Theta\biggr)\Utt+ \delb{0}(\Theta^{\tr}\Gsc)
&& \text{in $[0,T]\times \del{}\Omega_0$,} \label{libvp.2c}
\end{align}
where
\begin{align}
\Qtt &= -\frac{\epu q}{1-\epu}\Pbb, \label{Qttdef}\\
\Btt^{\alpha\beta} &= \lambda\Theta^{\tr}\Bsc^{\alpha\beta}\Theta, \label{Bttdef} \\
\Xtt^\alpha &= \Theta^{\tr}\Bsc^{\alpha\beta}\delb{\beta}(\lambda\Theta)\Utt+\Theta^{\tr}\Bsc^{\alpha\beta}\delb{\beta}(\beta\vartheta)
+\delb{0}(\Theta^{\tr}\Bsc^{\alpha\beta})\delb{\beta}\vartheta + \delb{0}(\Theta^{\tr}\Xsc^{\alpha}), \label{Xttdef} 
\intertext{and}
\Ftt &= \delb{0}\Bigl(\delb{\alpha}\Theta^{\tr}\bigl(\Bsc^{\alpha\beta}\delb{\beta}\vartheta+\Xsc^{\alpha}\bigr) +
\Theta^{\tr}\Fsc\Bigr). \label{Fttdef}
\end{align}

To proceed, we need to analyze the term $\delb{0}^2 E^{\tr}$ that appears in the boundary condition \eqref{libvp.2c} in more detail. To do so,
we, for fixed constants $\uc_0,\uc_1>0$, define an open set $\Uc$ $\subset$ $\Rbb\times \Rbb^4\times \Mbb{4}\times \Rbb^{4}\times \Rbb^{4}$ by
\begin{equation*}
(\zetau,\phi,J,\thetahu^0,\vartheta) \in \Uc 
\end{equation*}
if and only if
\begin{gather}
\uc_0 < \det(\gu^{\alpha\beta}), \quad |\gu^{\alpha\beta}|< \uc_1, \label{gubnd}\\
\uc_0 < \det(J), \quad|J| < \uc_1 , \label{Jbnd} \\
 \uc_0 < -\gu^{\mu\nu}\thetahu_\mu^0\thetahu_\nu^0, \quad
|\thetahu^0| < \uc_1,  \label{thetahubnd}\\
 \uc_0 < \mul^{\mu\nu}\psiu_\mu\psiu_\nu <\uc_1, \label{psibnd} \\
 \uc_0< \zetau < \uc_1,  \label{zubnd} 
\end{gather}
where we recall that $\psiu_\nu$ is defined by \eqref{psivecdef}. We further assume that time-independent spatial coframe $\thetab^I_\mu(\xb^\Sigma)$ satisfy 
\begin{equation} \label{thetabIbnd}
\vc_0 < \det(\thetab^I_\Lambda) \AND |\thetab^I| < \vc_1 \quad  \text{in $\Omega_0$}
\end{equation}
for some constants $\vc_0,\vc_1 > 0$,
where we have set
\begin{equation*}
\thetab^I = (\thetab^I_\mu).
\end{equation*}
Now, it follows from the Gram-Schmidt process and the analyticity of the inversion map that
$E$ and $\Theta$ can be viewed as matrix valued maps that depend analytically on the vectors
$\vu$ and $\psihu$.
Consequently, the  functional dependence of $E$ and $\Theta$ on
the  variables $(\phi,\thetahu^0,\psihu)$ is given by
\begin{equation}\label{EThetasmooth}
E = E(\phi,\thetahu^0,\psihu) \AND \Theta = \Theta(\phi,\thetahu^0,\psihu),
\end{equation}
where $E$ and $\Theta$ depend smoothly on their arguments.
Differentiating $E$ with respect to $\xb^0$ then gives
\begin{equation} \label{dtE}
\delb{0}E^{\tr} = D_{\phi}E^{\tr}\cdot \delb{0}\phi+ D_{\thetahu^0}E^{\tr}\cdot \delb{0}\thetahu^0 + D_{\psihu}E^{\tr}\cdot \delb{0}\psihu.
\end{equation}
But, from \eqref{Psi2dtthetaC} and \eqref{Psi2U}, we observe, with the help of  \eqref{Piup}, \eqref{Eorth} and \eqref{Piumat}, that we
can write $\delb{0}\psihu$ as
\begin{equation} \label{dtpsihu}
\delb{0}\psihu = \Theta \Pbb \Utt +\bigl(-|\psiu|_{\mul}\vu_\nu + \vu^\lambda \Gammau_{\lambda\nu}^\tau \psihu_\tau \bigr).
\end{equation}
Substituting this into \eqref{dtE}  and recalling that $\psihu_\nu$ is determined in terms of $\phi^\mu$, $\thetah^0_\mu$ and
$\vartheta_\mu$ via \eqref{gammah00}, \eqref{psihdef} and \eqref{psi2dtthetaA}, we see that $\delb{0}E^{\tr}$ can
be expressed as
\begin{equation*}
\delb{0}E^{\tr} = \tilde{\Ett}_0(\phi,\thetahu^0,\vartheta) \Pbb \Utt + \tilde{\Ett}_1(\phi,\thetahu^0,\vartheta)
\end{equation*}
where  $\tilde{\Ett}_0$ is a linear map from $\Rbb^4$ to $\Mbb{4}$,  $\tilde{\Ett}_1$ is $\Mbb{4}$-valued, and both maps are smooth in their arguments
for $(\phi,\thetahu^0,\vartheta)$
satisfying \eqref{gubnd}, \eqref{thetahubnd} and \eqref{psibnd}. Differentiating this expression once more with respect to $\xb^0$ yields
\begin{equation*}
\delb{0}^2E^{\tr} = \tilde{\Ett}_0\Pbb \delb{0}\Utt + \delb{0}\tilde{\Ett}_0 \Pbb \Utt+ \delb{0}\tilde{\Ett}_1.
\end{equation*}
Using this we can, with the help of \eqref{libvp.3}-\eqref{libvp.4}, \eqref{Psitopsi}, \eqref{Psi2U} and \eqref{EThetasmooth}-\eqref{dtpsihu},  write  $\delb{0}^2E^{\tr} \Pbb \Utt$ as 
\begin{equation} \label{dttE}
\delb{0}^2E^{\tr} \Pbb \Utt = \Ett_0(\phi,\thetahu^0,\vartheta,\Utt)\Pbb \delb{0}\Utt+ 
\Ett_1(\phi,\thetahu^0,\vartheta,\Utt)
\end{equation}
where $\Ett_0$ is $\Mbb{4}$-valued, $\Ett_1$ is $\Rbb^4$-valued, and both maps are smooth in their arguments provided that
$(\phi,\thetahu^0,\vartheta)$
satisfies \eqref{gubnd}, \eqref{thetahubnd} and \eqref{psibnd}.
Substituting this expression into \eqref{libvp.2b} shows that
\begin{equation} \label{libvp.3c}
\thetab^3_\alpha \bigl(\Btt^{\alpha\beta}\delb{\beta}\Utt+\Xtt^{\alpha}\bigr)  =\Qtt\delb{0}^2 \Utt + \Ptt \delb{0}\Utt + \Gtt
\quad \text{in $[0,T]\times \del{}\Omega_0$,}
\end{equation}
where
\begin{align}
\Ptt &=  -\biggl(\delb{0}\biggl(\frac{\epu}{1-\epu}\biggr)q  +\frac{2\epu}{1-\epu}\delb{0}q\biggr)\Pbb + \Theta^{\tr}\Psc\Theta  - q\biggl(\frac{\epu}{1-\epu}\Pbb + \epu \Pbb_0\biggr)\Theta^{\tr}\bigl(\delb{0}E^{\tr}+\Ett_0\bigr) \Pbb, \label{Pttdef}\\
\Gtt &= - q\biggl(\frac{\epu}{1-\epu}\Pbb + \epu \Pbb_0\biggr)\Theta^{\tr}\Ett_1 - \delb{0}\biggl(q\biggl(\frac{\epu}{1-\epu}\Pbb + \epu \Pbb_0\biggr)\Theta^{\tr}\biggr)\delb{0}E^{\tr} \Pbb \Utt  \notag\\
&\qquad +
\delb{0}\biggl( -\frac{\epu}{1-\epu}\delb{0}q \Pbb+ \Theta^{\tr}\Psc\Theta\biggr)\Utt+ \delb{0}(\Theta^{\tr}\Gsc), \label{Gttdef}
\end{align}
and we note, by \eqref{Pscdef}, \eqref{Piumat} and \eqref{vuvumat},  that
\begin{align}
 \Theta^{\tr}&\Psc\Theta =  \biggl(\frac{\epu}{1-\epu}\det(J) \psihu_\omega  \elltu^\omega  -\frac{\epu}{2} \delb{0}\gu^{\rho\tau}\psiu_\rho\psiu_\tau |\gu|^{-\frac{1}{2}}\det(\thetab^J_\Lambda) |\Nu|_{\gu} |\psiu|^{-2}_{\mul}\biggr)\Pbb  \notag \\
 &-\kappa|\gu|^{-\frac{1}{2}}\det(\thetab^J_\Lambda)|\psiu|_{\mul}\bigl(-\gammahu^{00}\bigr)^{\frac{1}{2}} \Pbb_0
 +\Bigl(\Bigl(2\epu \det(J) \elltu^\omega\psihu^{\alpha} \Piu_\omega^{\beta} -2|\psiu|_{\mul}^{-1}\det(J) \ellu^\omega\psiu^{\sigma}\hu_\omega^{\alpha} \hu^{\beta}_\sigma\Bigr) \Theta_\alpha^{[i}\Theta_\beta^{j]} 
 \Bigr). \label{Pscmat}
\end{align}
Finally, differentiating \eqref{libvp.5} and \eqref{libvp.6} with respect to $\xb^0$ yields
\begin{align}
\delb{\alpha}(\Msc^{\alpha\beta}\delb{\beta}\Ztt + \Ytt^\alpha) &=\Ktt && \text{in $[0,T]\times \Omega_0$,} \label{libvp.4c}\\
\Ztt &= 0 && \text{in $[0,T]\times \del{}\Omega_0$,} \label{libvp.5c}
\end{align}
where  we have set
\begin{align}
\Ktt &= \delb{0}\Ksc, \label{Kttdef}\\
\Ytt^\alpha &= \delb{0}\Msc^{\alpha\beta}\delb{\beta}\zetau \label{Yttdef}
\intertext{and}
\Ztt&=\delb{0}\zeta. \label{Zttdef}
\end{align}

\subsubsection{The complete time-differentiated system}
Collecting \eqref{libvp.3}-\eqref{libvp.4}, \eqref{Psitopsi} where we use \eqref{Psi2U} to replace $\Psi$ with $\Utt$, \eqref{libvp.1c}, \eqref{libvp.3c}, 
\eqref{libvp.4c}-\eqref{libvp.5c} and  \eqref{Zttdef}, we arrive at the following IBVP:
\begin{align}
\delb{\alpha}\bigl(\Btt^{\alpha\beta}\delb{\beta}\Utt+\Xtt^{\alpha}\bigr) &= \Ftt &&
\text{in $[0,T]\times \Omega_0$,} \label{libvp.1d} \\
\thetab^3_\alpha \bigl(\Btt^{\alpha\beta}\delb{\beta}\Utt+\Xtt^{\alpha}\bigr)  &=\Qtt\delb{0}^2 \Utt + \Ptt \delb{0}\Utt + \Gtt
&&\text{in $[0,T]\times \del{}\Omega_0$,} \label{libvp.2d}\\
\delb{0}\phi &= \alpha \thetahu^0  && \text{in $[0,T]\times \Omega_0$}, \label{libvp.3d} \\
\delb{0}\thetahu^0 & = \vartheta+\betat  \thetahu^0 &&
\text{in $[0,T]\times \Omega_0$}, \label{libvp.4d} \\
\delb{0}\vartheta &= \lambda \Theta \Utt + \beta \vartheta  &&
\text{in $[0,T]\times \Omega_0$}, \label{libvp.5d} \\
\delb{\alpha}(\Msc^{\alpha\beta}\delb{\beta}\Ztt+\Ytt^\alpha) &= \Ktt && \text{in $[0,T]\times \Omega_0$,} \label{libvp.6d}\\
\Ztt &= 0 && \text{in $[0,T]\times \del{}\Omega_0$,} \label{libvp.7d}\\
\delb{0}\zetau & = \Ztt &&
\text{in $[0,T]\times \Omega_0$}, \label{libvp.8d} \\
(\Utt,\delb{0}\Utt) &= (\Utt_0,\Utt_1)  && \text{in $\{0\}\times \Omega_0$,} \label{libvp.9d}\\
\phi &= \phi_0  && \text{in $\{0\}\times \Omega_0$,} \label{libvp.10d}\\
\thetahu^0 &= \thetahu^0_0  && \text{in $\{0\}\times \Omega_0$,} \label{libvp.11d}\\
\vartheta &= \vartheta_0  && \text{in $\{0\}\times \Omega_0$,} \label{libvp.12d}\\
(\Ztt,\delb{0}\Ztt) &= (\Ztt_{0},\Ztt_{1})  && \text{in $\{0\}\times \Omega_0$,} \label{libvp.13d}\\
\zetau &= \zetau_0  && \text{in $\{0\}\times \Omega_0$.} \label{libvp.14d}
\end{align} 

\begin{rem} \label{tdiffIBVPrem}
From the derivation of the IBVP \eqref{libvp.1d}-\eqref{libvp.14d}, it is clear that if the initial data
is chosen so that constraints 
\begin{align}
\delb{\alpha}\bigl(\Bsc^{\alpha\beta}\delb{\beta}\vartheta+ \Xsc^{\alpha}\bigr) - \Fsc& =0  &&
\text{in $\{0\}\times \Omega_0$,} \label{libvp.1idata} \\
\thetab^3_\alpha \bigl(\Bsc^{\alpha\beta}\delb{\beta}\vartheta+ \Xsc^{\alpha}\bigr) +\Rsc^{\tr}\delb{0}(\Qsc \Psi) - \Psc\Psi - \Gsc &=0
&& \text{in $\{0\}\times \del{}\Omega_0$,} \label{libvp.2idata}\\
\delb{\alpha}(\Msc^{\alpha\beta}\delb{\beta}\zetau) - \Ksc&=0 && \text{in $\{0\}\times \Omega_0$,} \label{libvp.3idata}\\
\zetau - 1&=0 && \text{in $\{0\}\times \del{}\Omega_0$,} \label{libvp.4idata}
\end{align}
all hold on the initial hypersurface, then they will propagate. Consequently, any solution 
$(\Utt,\phi,\thetahu^0,\vartheta,\Ztt,\zetau)$ of the IBVP  \eqref{libvp.1d}-\eqref{libvp.14d} that is generated from
initial data satisfying the constraints \eqref{libvp.1idata}-\eqref{libvp.4idata} will determine a solution of the IBVP \eqref{libvp.1}-\eqref{libvp.10} given by $(\phi,\thetahu^0,\vartheta,\zetau)$.
\end{rem}


\subsubsection{Coefficient smoothness and estimates\label{coefsmooth}}

It is not difficult to verify that functional dependence of the coefficients that appear in the system
\eqref{libvp.1d}-\eqref{libvp.6d}
on the variables $\{\sigmau,\thetab^I,\zetau,\phi,\thetahu^0,\vartheta,\Utt\}$, where
\begin{equation*}
\sigmau=(\sigmau_i{}^k{}_j),
\end{equation*}
is as follows:\footnote{Here, we are freely employing the evolution equations \eqref{libvp.3d}, \eqref{libvp.4d}, \eqref{libvp.5d} and
\eqref{libvp.8d} to replace any appearance of $\delb{0}\phi$, $\delb{0}\thetahu$, $\delb{0} \vartheta$, and $\delb{0}\zeta$ with
the corresponding right hand side of those equations, and equation \eqref{Psi2U} to replace any appearance of $\Psi$ with $\Theta \Utt$.} 
\begin{align} 
\alpha &= \alpha(\phi,\thetahu^0), \label{fdepA.1} \\
\betat &= \betat(\phi,\thetahu^0),   \label{fdepA.2} \\
\beta &= \beta(\phi,\thetahu^0,\vartheta),   \label{fdepA.3} \\
\Theta &= \Theta(\phi,\thetahu^0,\vartheta), \label{fdepA.4} \\
\lambda &= \lambda(\phi,\thetahu^0,\vartheta),  \label{fdepA.5} \\
\Msc^{\alpha\beta} &= \Msc^{\alpha\beta}(\zeta,\Db^{|1|}\phi,\thetahu^0),  \label{fdepA.6} \\
\Ktt &= \Ktt(\sigma,\thetab^I,\Db^{|1|}\zetau,\Db^{|1|}\Ztt,\delb{0}\Ztt,\Db^{|1|}\phi,\Db^{|1|}\thetahu^0,\Db^{|1|}\vartheta,\Utt),  \label{fdepA.7} \\
\Ytt^\alpha &= \Ytt^\alpha(\Db^{|1|}\zetau,\Ztt,\Db^{|1|}\phi,\Db^{|1|}\thetahu^0,\vartheta)  \label{fdepA.8} \\
\Btt^{\alpha\beta} &=  \Btt^{\alpha\beta}(\sigmau,\thetab^I,\zetau,\Db^{|1|}\phi,\Db^{|1|}\thetahu^0,\vartheta),  \label{fdepA.9}  \\
\Bsc^{\Lambda\Sigma} &=  \Bsc^{\Lambda\Sigma}(\thetab^I,\zetau,\Db^{|1|}\phi,\thetahu^0,\vartheta),  \label{fdepA.9a}  \\
\Ftt &= \Ftt(\sigmau,\Db^{|1|}\thetab^I,\Db^{|1|}\zetau,\delb{0}^{|2|}\Ztt,\Db\delb{0}^{|1|}\Ztt,\Db^{|2|}\phi,\Db^{|2|}\thetahu^0,
\Db^{|1|}\vartheta,\Db^{|1|}\Utt,\delb{0}\Utt), \label{fdepA.10}\\
\Xtt^{\alpha} & = \Xtt^{\alpha}(\sigmau,\thetab^I,\zetau,\delb{0}^{|1|}\Ztt,\Db^{|1|}\phi,\Db^{|1|}\thetahu^0,\Db^{|1|}\vartheta,\Utt),  \label{fdepA.11} \\
\Qtt & = \Qtt(\thetab^I,\Dbsl^{|1|}\!\phi,\thetahu^0,\vartheta), \label{fdepA.12}\\
\Ptt &= \Ptt(\thetab^I,\Dbsl^{|1|}\!\phi,\Dbsl^{|1|}\!\thetahu^0,\vartheta,\Utt)  \label{fdepA.13} 
\intertext{and}
\Gtt & = \Gtt(\thetab^I,\Dbsl^{|1|}\!\phi,\Dbsl^{|1|}\!\thetahu^0,\Dbsl^{|1|}\!\vartheta,\Utt),  \label{fdepA.14}
\end{align}
where $\Dbsl{}$ denotes the derivatives tangent to the spatial boundary $\del{}\Omega_0$.
Moreover, it is also not difficult to verify that the coefficients depend smoothly on their arguments provided that $(\zetau,\phi,\delb{}\phi,\thetahu^0,\vartheta)\in \Uc$, see \eqref{gubnd}-\eqref{zubnd} above, and the spatial coframe
$\thetab^I$ satisfies \eqref{thetabIbnd}.


\begin{lem} \label{coeflemA}
Suppose $s>n/2+1/2$, $\sigmau_i{}^k{}_j \in H^{s+\frac{1}{2}}(\Omega_0)$, $\thetab^I\in H^{s+\frac{1}{2}}(\Omega_0,\Rbb^4)$,
$\, \zetau$,$\,\delb{0}\zetau$,$\,\Ztt$ $\in \Xc_T^{s+\frac{3}{2}}(\Omega_0)$, $\, \phi$,$\,\delb{0}\phi$,$\,\thetahu^0$, $\,\delb{0}\thetahu^0$,$\,\vartheta$ $\in  \Xc_T^{s+\frac{3}{2}}(\Omega_0,\Rbb^4)$,
$\Utt \in \Xc_T^{s+1}(\Omega_0,\Rbb^4)$,
and $(\zetau,\phi,\thetahu^0,\vartheta,\thetab^I)$ satisfy \eqref{gubnd}-\eqref{zubnd} and \eqref{thetabIbnd},
where $J=\delb{}\phi$, and let
\begin{equation*} 
\Nc = \sum_{\ell=0}^1 \Bigl(\norm{\delb{0}^\ell\zetau}_{\Ec^{s+\frac{3}{2}}} +\norm{\delb{0}^\ell\phi}_{\Ec^{s+\frac{3}{2}}} +  
\norm{\delb{0}^\ell \thetahu^0}_{\Ec^{s+\frac{3}{2}}}\Bigr)+\norm{\Ztt}_{\Ec^{s+\frac{3}{2}}}+\norm{\vartheta}_{\Ec^{s+\frac{3}{2}}} +  \norm{\Utt}_{\Ec^{s+1}}. 
\end{equation*}
Then
\begin{equation*}
\sum_{\ell=0}^1\Bigl(\norm{\delb{0}^\ell\alpha}_{\Ec^{s+\frac{3}{2}}}+\norm{\delb{0}^\ell\betat}_{\Ec^{s+\frac{3}{2}}}\Bigr) +\norm{\beta}_{\Ec^{s+\frac{3}{2}}}+ \norm{\Theta}_{\Ec^{s+\frac{3}{2}}}+\norm{E}_{\Ec^{s+\frac{3}{2}}}+
\norm{\lambda\Theta}_{\Ec^{s+\frac{3}{2}}}\leq C(\Nc),
\end{equation*}
\begin{align*}
\sum_{\ell=0}^1\Bigl(\norm{\delb{0}^\ell\Btt}_{E^{s+\frac{1}{2}}}+\norm{\delb{0}^\ell\Qtt}_{E^{s+\frac{1}{2}}}+\norm{\delb{0}^\ell\Ytt}_{E^{s+\frac{1}{2}}}
+\norm{\delb{0}^\ell\Msc}_{E^{s+\frac{1}{2}}} \Bigr)
 &
\\
+\norm{\delb{0}\vec{\Bsc}}_{E^{s+\frac{1}{2}}}+
\norm{\Gtt}_{E^{s+\frac{1}{2}}}+ \norm{\Ktt}_{E^{s+\frac{1}{2}}}
+ \norm{\Ptt}_{E^{s+\frac{1}{2}}} + &
\norm{\Xtt}_{E^{s+\frac{1}{2}}} \\
+ 
\norm{\Ftt}_{E^{s-\frac{1}{2}}}
+ \norm{\delb{0}^{2s}\Ftt}_{L^2(\Omega_0)}  \leq 
C\Bigl(\norm{\thetab^I}_{H^{s+\frac{1}{2}}(\Omega_0)}&,\norm{\sigmau}_{H^{s+\frac{1}{2}}(\Omega_0)},\Nc\Bigr)\Nc
\end{align*}
and
\begin{align*}
\norm{\alpha(\xb^0)}_{\Ec^{s+\frac{3}{2}}}+\norm{\betat(\xb^0)}_{\Ec^{s+\frac{3}{2}}}+\norm{\Ptt(\xb^0)}_{H^{s}(\Omega_0)}
+\norm{\vec{\Bsc}(\xb^0)}_{H^{s+\frac{1}{2}}(\Omega_0)}\qquad & \\ 
+\norm{\Btt(\xb^0)}_{H^{s+\frac{1}{2}}(\Omega_0)}+
\norm{\Msc(\xb^0)}_{H^{s+\frac{1}{2}}(\Omega_0)} +\sum_{\ell=0}^1\norm{\delb{0}^\ell \Qtt(\xb^0)}_{H^{s+\frac{1-\ell}{2}}(\Omega_0)}&\leq \Xi(\xb^0),
\end{align*}
where $\vec{\Bsc}=(\Bsc^{\Lambda\Sigma})$ and
\begin{equation*}
\Xi(\xb^0) =C\Bigl(\norm{\thetab^I}_{H^{s+\frac{1}{2}}(\Omega_0)},\norm{\sigmau}_{H^{s+\frac{1}{2}}(\Omega_0)},\Nc(0)\Bigr)+  \int_0^{\xb^0} 
C\Bigl(\norm{\thetab^I}_{H^{s+\frac{1}{2}}(\Omega_0)},\norm{\sigmau}_{H^{s+\frac{1}{2}}(\Omega_0)},\Nc(\tau)\Bigr)\Nc(\tau) \, d\tau.
\end{equation*}
\end{lem}
\begin{proof}
We will only establish the estimates for $\norm{\Btt}_{H^{s+\frac{1}{2}}(\Omega_0)}$,  
$\norm{\delb{0}\Btt}_{\Ec^{s+\frac{1}{2}}}$  and  $\norm{\delb{0}^{2s}\Ftt}_{L^2(\Omega_0)}$ with the rest established in a similar fashion. 
We begin by setting
\begin{equation*}
\ytt =(\zetau,\Db^{|1|}\phi,\Db^{|1|}\thetahu^0,\vartheta)
\end{equation*}
and observing, by \eqref{fdepA.9}, that
\begin{equation*}
\delb{0}\Btt^{\alpha\beta} =  D_{\ytt}\Btt^{\alpha\beta}(\sigmau,\thetab^I,\ytt)\delb{0}\ytt.
\end{equation*}
We then deduce from Proposition \ref{stpropB} that
\begin{align}
\norm{\Btt}_{E^{s+\frac{1}{2}}} &\leq C\Bigl((\norm{\thetab^I}_{H^{s+\frac{1}{2}}(\Omega_0)},\norm{\sigmau}_{H^{s+\frac{1}{2}}(\Omega_0)} ,
\norm{\ytt}_{E^{s+\frac{1}{2}}}\Bigr) \label{coeflemA0} 
\intertext{and}
\norm{\delb{0}\Btt}_{E^{s+\frac{1}{2}}} &\leq C\biggl((\norm{\thetab^I}_{H^{s+\frac{1}{2}}(\Omega_0)},\norm{\sigmau}_{H^{s+\frac{1}{2}}(\Omega_0)} , \sum_{\ell=0}^1\norm{\delb{0}^\ell \ytt}_{E^{s+\frac{1}{2}}}\biggr) \sum_{\ell=0}^1\norm{\delb{0}^\ell \ytt}_{E^{s+\frac{1}{2}}}.  \label{coeflemA1}
\end{align}
But, noting that 
\begin{align*}
\norm{(\Db^{|1|}v,w)}_{E^{s+\frac{1}{2}}} = \norm{(\Db^{|1|}v,w)}_{E^{s+\frac{1}{2},2s+1}} 
\lesssim \norm{v}_{E^{s+\frac{3}{2},2s+1}}+ \norm{w}_{E^{s+\frac{1}{2},2s+1}}
\lesssim  \norm{v}_{\Ec^{s+\frac{3}{2}}}+ \norm{w}_{\Ec^{s+1}},
\end{align*}
the estimates
\begin{align}
\norm{\Btt}_{E^{s+\frac{1}{2}}} &\leq C\Bigl(\norm{\thetab^I}_{H^{s+\frac{1}{2}}(\Omega_0)},\norm{\sigmau}_{H^{s+\frac{1}{2}}(\Omega_0)},\Nc\Bigr)
\label{coeflemA2}
\intertext{and}
\norm{\delb{0}\Btt}_{E^{s+\frac{1}{2}}} &\leq C\Bigl(\norm{\thetab^I}_{H^{s+\frac{1}{2}}(\Omega_0)},\norm{\sigmau}_{H^{s+\frac{1}{2}}(\Omega_0)},\Nc\Bigr)\Nc
\label{coeflemA3}
\end{align}
then follow directly from \eqref{coeflemA0} and \eqref{coeflemA1}, respectively.
Writing $\Btt^{\alpha\beta}$ as $\Btt^{\alpha\beta}(\xb^0) = \Btt^{\alpha\beta}(0) + \int_0^{\xb^0} \delb{0}\Btt^{\alpha\beta}(\tau)\, d\tau$, we
find, after applying $H^{s+\frac{1}{2}}(\Omega_0)$ norm and using the triangle inequality and the estimates \eqref{coeflemA2}-\eqref{coeflemA3}, that
\begin{equation*}
\norm{\Btt(\xb^0)}_{H^{s+\frac{1}{2}}(\Omega_0)} \leq 
C\Bigl(\norm{\thetab^I}_{H^{s+\frac{1}{2}}(\Omega_0)},\norm{\sigmau}_{H^{s+\frac{1}{2}}(\Omega_0)},\Nc(0)\Bigr)
+\int_0^{\xb^0}C\Bigl(\norm{\thetab^I}_{H^{s+\frac{1}{2}}(\Omega_0)},\norm{\sigmau}_{H^{s+\frac{1}{2}}(\Omega_0)},\Nc(\tau)\Bigr)\Nc(\tau) \, d\tau.
\end{equation*}

Turning to the estimate involving $\Ftt$, we have from \eqref{fdepA.10} that
\begin{equation} \label{coeflemA4}
\Ftt = \Fsc(\sigmau,\Db^{|1|}\thetab^I,\ztt)
\end{equation}
where
\begin{equation*}
\ztt= \bigl(\Db^{|1|}\zetau,\delb{0}^{|2|}\Ztt,\Db\delb{0}^{|1|}\Ztt,\Db^{|2|}\phi,\Db^{|2|}\thetahu^0,
\Db^{|1|}\vartheta,\Db^{|1|}\Utt,\delb{0}\Utt\bigr).
\end{equation*}
Before proceeding, we observe that
\begin{align*}
\norm{\Db^{|2|}v}_{E^{s-\frac{1}{2},2s-2}}+\norm{\delb{0}\Db^{|2|}v}_{E^{s-\frac{1}{2},2s-2}} &\lesssim
\norm{v}_{E^{s+\frac{3}{2},2s-2}} + \norm{\delb{0}v}_{E^{s+\frac{3}{2},2s-2}} \notag \\
 &\lesssim \norm{v}_{\Ec^{s+\frac{3}{2}}} + \norm{\delb{0} v}_{\Ec^{s+\frac{3}{2}}}, \\
 \norm{\delb{0}^{|3|} v}_{E^{s-\frac{1}{2},2s-2}}+\norm{\delb{0}\Db^{|1|}\delb{0} v}_{E^{s-\frac{1}{2},2s-2}} &\lesssim
 \norm{v}_{E^{s+1,2s+1}} + \norm{v}_{E^{s+\frac{3}{2},2s}} \notag \\
 &\lesssim \norm{v}_{\Ec^{s+\frac{3}{2}}}
 \intertext{and}
\norm{\Db^{|1|}w}_{E^{s-\frac{1}{2},2s-2}}+\norm{\delb{0}\Db^{|1|}w}_{E^{s-\frac{1}{2},2s-2}} + \norm{\delb{0}^{|2|}w}_{E^{s-\frac{1}{2},2s-2}}&\lesssim
\norm{w}_{E^{s+\frac{1}{2},2s-2}} + \norm{w}_{E^{s+1,2s-1}}+\norm{w}_{E^{s+\frac{1}{2},2s}}\notag \\
&\lesssim  \norm{w}_{\Ec^{s+1}}.
\end{align*}
From these inequalities, we then get 
\begin{equation} \label{coeflemA5}
\norm{\ztt}_{E^{s-\frac{1}{2},2s-2}}+\norm{\delb{0}\ztt}_{E^{s-\frac{1}{2},2s-2}} \lesssim \Nc.
\end{equation}
Moreover, from the inequalities
\begin{align*}
 \norm{\delb{0}^{2s}\Db^{|2|}v}_{L^2(\Omega_0)}&\lesssim  \norm{\delb{0}^{2s-1}\delb{0}v}_{H^2(\Omega_0)} \notag\\
&\lesssim \norm{\delb{0}v}_{\Ec^{s+\frac{3}{2}}}, \notag \\
\norm{\delb{0}^{2s}\delb{0}^{|2|}v}_{L^2(\Omega_0)} + \norm{\delb{0}^{2s}\Db^{|1|}\delb{0}v}_{L^2(\Omega_0)}
&\lesssim \norm{\delb{0}^{2s+2}v}_{L^2(\Omega_0)}
+ \norm{\delb{0}^{2s+1}v}_{H^1(\Omega_0)} \notag\\
&\lesssim  \norm{v}_{\Ec^{s+\frac{3}{2}}}
\intertext{and}
\norm{\delb{0}^{2s}\delb{0}w}_{L^2(\Omega_0)} + \norm{\delb{0}^{2s}\Db^{|1|}w}_{L^2(\Omega_0)}  &\lesssim
\norm{w}_{\Ec^{s+1}}+ \norm{\delb{0}^{2s}w}_{H^1(\Omega_0)} \\
&\lesssim \norm{w}_{\Ec^{s+1}},
\end{align*}
we also have
\begin{equation} \label{coeflemA6}
\norm{\delb{0}^{2s}\ztt}_{L^2(\Omega_0)} \lesssim \Nc.
\end{equation}

Next, differentiating  \eqref{coeflemA4} with respect to $\xb^0$ yields 
\begin{equation*}
\delb{0}\Ftt = D_{\ztt}\Ftt(\sigmau,\Db^{|1|}\thetab^I,\ztt)\delb{0}\ztt.
\end{equation*}
Using this, we can express $\delb{0}^{2s}\Ftt$ as
\begin{equation*}
\delb{0}^{2s}\Ftt = [\delb{0}^{2s-1},D_{\ztt}\Ftt(\sigmau,\Db^{|1|}\thetab^I,\ztt)]\delb{0}\ztt+ D_{\ztt}\Ftt(\sigmau,\Db^{|1|}\thetab^I,\ztt)\delb{0}^{2s} \ztt.
\end{equation*}
Then using Proposition \ref{stcomA}, with $s_3=s-\frac{1}{2}$, $s_1=s$, $s_2=s-\frac{1}{2}$ and $\ell=2s-1$,  to bound the first term on the right hand side of the above expression, and Sobolev and H\"{o}lder's inequalities (Theorems \ref{Holder}
and \ref{FSobolev}) to bound the
second term, we obtain the estimate
\begin{equation}  \label{coeflemA7}
\norm{\delb{0}^{2s}\Ftt}_{L^2(\Omega)} \lesssim \norm{\delb{0} D_{\ztt}\Ftt}_{E^{s-\frac{1}{2},2s-2}}\norm{\delb{0}\ztt}_{E^{s-\frac{1}{2},2s-2}}+
\norm{D_{\ztt}\Ftt}_{H^{s-\frac{1}{2}}(\Omega)}\norm{\delb{0}^{2s}\ztt}_{L^2(\Omega)}.
\end{equation}
Noting that
\begin{equation*}
\delb{0} \bigl(D_{\ztt}\Ftt (\sigmau,\Db^{|1|}\thetab^I,\ztt)\bigr) = D_{\ztt}^2\Ftt (\sigmau,\Db^{|1|}\thetab^I,\ztt) \delb{0}\ztt,
\end{equation*}
the estimate
\begin{equation*}
\norm{\delb{0}^{2s}\Ftt}_{L^2(\Omega_0)} \leq C\bigl(\norm{\thetab^I}_{H^{s+\frac{1}{2}}(\Omega_0)},
\norm{\sigmau}_{H^{s+\frac{1}{2}}(\Omega_0)},\Nc\bigr)\Nc
\end{equation*}
follows directly from \eqref{coeflemA5}, \eqref{coeflemA6}, \eqref{coeflemA7} and an application of Proposition \ref{stpropB}.
\end{proof}

\subsubsection{Coefficient relations, inequalities and coercive estimates\label{coefcoercive}}
From the definitions \eqref{Bscdef}, \eqref{Bsctemp} and \eqref{Bttdef}, we note that
the matrices $\Bsc^{\alpha\beta}$ and $\Btt^{\alpha\beta}$ satisfy
\begin{equation}
(\Bsc^{\alpha\beta})^{\tr} = \Bsc^{\beta\alpha} \AND  (\Btt^{\alpha\beta})^{\tr} = \Btt^{\beta\alpha},  \label{BscBttsym}
\end{equation}
respectively.
Furthermore, recalling that $\ep$ is given by \eqref{epformula} where $\ep_0>0$, we see from \eqref{psibnd} that 
\begin{equation} \label{epbnd2}
0<\frac{\ep_0+\sqrt{\uc_0}}{\ep_0}\leq  \frac{1}{1-\epu} \leq \frac{\ep_0+ \sqrt{\uc_1}}{\ep_0}.
\end{equation}
Using this bound, it is not difficult to verify that there exists constants $\bc_a=\bc_a(\ep_0,\vec{\uc},\vec{\vc})>0$, $a=0,1$, where
\begin{equation*}
\vec{\vc}=(\vc_0,\vc_1) \AND \vec{\uc} = (\uc_1,\uc_2),
\end{equation*}
 such that the inequalities
 \begin{align}
 \bigl(\chi\bigl|\Btt^{00}(\thetab^I,\zetau,\delb{}^{|1|}\phi,\thetahu^0,\vartheta)\chi\bigr) &\leq -\bc_0|\chi|^2 \label{Bttbnd1}\\
 \intertext{and} 
 \bigl(\chi\bigl|\Bsc^{\Sigma \Lambda}(\thetab^I,\zetau,\phi,\delb{}\phi,\thetahu^0)\chi\bigr)\eta_\Sigma \eta_\Lambda\ &\geq \bc_1 |\chi|^2|\eta|^2 \label{Bscbnd2}
 \end{align}
hold  for all 
$(\chi,\eta)\in \Rbb^4 \times \Rbb^3$, and  
 $(\thetab^I ,\zetau,\phi,J=\delb{}\phi,\thetahu^0,\vartheta)$ that satisfy \eqref{gubnd}-\eqref{zubnd} and  \eqref{thetabIbnd}
Similarly, we see from the definition \eqref{Mscdef} that $\Msc^{\alpha\beta}$ satisfies
\begin{equation}
\Msc^{\alpha\beta} = \Msc^{\beta\alpha},  \label{Mscsym}
\end{equation}
and there exists constants $\mc_a=\mc_a(\vec{\uc})>0$, $a=0,1$, 
 such that 
 \begin{align}
 \Msc^{00}(\phi,\delb{}\phi,\thetahu^0) &\leq -\mc_0 \label{Mscbnd1}
  \intertext{and} 
 \Msc^{\Sigma \Lambda}(\phi,\delb{}\phi,\thetahu^0)\eta_\Sigma \eta_\Lambda &\geq \mc_1 |\eta|^2 \label{Mscbnd2}
 \end{align}
 for all $\eta \in \Rbb^3$ and $ (\phi,J=\delb{}\phi,\thetahu^0)$ satisfying \eqref{gubnd}-\eqref{thetahubnd}. 
Moreover, it is clear from the  definition \eqref{lambdadef} that the exists a constant $\lc(\vec{\uc},\vec{\vc})>0$ such that the inequality
\begin{equation}
 0< \uc_0 \leq \lambda(\phi,\thetahu^0,\vartheta) \leq  \lc \label{lambdabnd}
 \end{equation}
holds for all $(\zetau,\phi,\vartheta)$ satisfying \eqref{gubnd},  \eqref{thetahubnd} and \eqref{psibnd}.
 
Next, we establish coercive estimates in the following two lemmas for the bilinear forms associated to $\Msc^{\Lambda\Sigma}$, $\Bsc^{\Lambda\Sigma}$ and $\Btt^{\Lambda\Sigma}$. 
\begin{lem} \label{Msccoercive}
Suppose, $\phi \in C^{1}(\overline{\Omega}_0,\Rbb^4)$, $\thetahu^0\in C^{0}(\overline{\Omega}_0,\Rbb^4)$, and $(\phi,\delb{}\phi,\thetahu^0)$ satisfies
\eqref{gubnd}-\eqref{thetahubnd}  on $\Omega_0$. Then there exists positive constants $\kappa_0 =\kappa_0(\vec{\uc})>0$ and $\mu_0=\mu_0(\vec{\uc})>0$
such that
\begin{equation*}
\ip{\delb{\Lambda}u}{\Msc^{\Lambda\Sigma}\delb{\Sigma}u}_{\Omega_0} \geq \kappa_0\norm{u}_{H^1(\Omega_0)}^2 -\mu_0\norm{u}^2_{L^2(\Omega_0)}
\end{equation*}
for all $u \in H^1_0(\Omega_0)$.
\end{lem}
\begin{proof}
The stated coercive estimate is a direct consequence of the inequality \eqref{Mscbnd2}.
\end{proof}

\begin{lem} \label{Bsccoercive}
Suppose
$\zeta \in C^{0}(\overline{\Omega}_0)$, $\phi \in C^{1}(\overline{\Omega}_0,\Rbb^4)$, $\, \thetab^I,\thetahu^0,\vartheta \in C^{0}(\overline{\Omega}_0,\Rbb^4)$, and $(\thetab^I,\zeta,\phi,\delb{}\phi,\thetahu^0)$ satisfies
\eqref{gubnd}-\eqref{zubnd} and \eqref{thetabIbnd} on $\Omega_0$.
Then there exists constants $\wc>0$, $\kappa_1 =\kappa_1(\ep_0,\vec{\uc},\vec{\vc},\wc)>0$ and $\mu_1=\mu_1(\ep_0,\vec{\uc},\vec{\vc},\wc)>0$
such that
\begin{align*}
\ip{\delb{\Lambda}u}{\Bsc^{\Lambda\Sigma}\delb{\Sigma}u}_{\Omega_0} &\geq \kappa_1\norm{u}_{H^1(\Omega_0)}^2 -\mu_1\norm{u}^2_{L^2(\Omega_0)}
\intertext{and}
\ip{\delb{\Lambda}u}{\Bsc^{\Lambda\Sigma}\delb{\Sigma}u}_{\Omega_0} &\geq \kappa_1\norm{u}_{H^1(\Omega_0)}^2 -\mu_1\norm{u}^2_{L^2(\Omega_0)}
\end{align*}
and
for all $u \in H^1(\Omega_0,\Rbb^4)$ provided that\footnote{See \eqref{thetah3def}, \eqref{psihdef} and \eqref{psivecdef}.}  $\underline{\hat{\psi}}$ satisfies $|\thetahu^3+\underline{\hat{\psi}}|<\wc$ on $\Omega_0$.
\end{lem}
\begin{proof}
From the bound \eqref{epbnd2} and the formulas \eqref{mdef}-\eqref{adef}, \eqref{Upsilonchdef}, \eqref{Sscdef} and \eqref{Bscspat}, it
is clear that the matrices $\Bsc^{\Lambda\Sigma}$ are closely related to the matrices 
$B^{\Lambda\Sigma}$ from Lemma 8.3 from \cite{Oliynyk:Bull_2017} when $\thetahu^3+\underline{\hat{\psi}}=0$, and moreover, that a straightforward adaptation of
the arguments used
to derive the coercive estimates given in Lemma 8.3 from \cite{Oliynyk:Bull_2017} can be used establish the existence
of constants
$\tilde{\kappa}_1 =\tilde{\kappa}_1(\ep_0,\vec{\uc},\vec{\vc})>0$ and $\tilde{\mu}_1=\tilde{\mu}_1(\ep_0,\vec{\uc},\vec{\vc})>0$ such that
the coercive estimate 
\begin{equation}  \label{Bsccoercive0}
\ip{\delb{\Lambda}u}{\Bsc^{\Lambda\Sigma}\delb{\Sigma}u}_{\Omega_0} \geq \tilde{\kappa}_1\norm{u}_{H^1(\Omega_0)}^2 -\tilde{\mu}_1\norm{u}^2_{L^2(\Omega_0)}, \quad \forall \; u \in H^1(\Omega_0,\Rbb^4),
\end{equation}
holds.  With the help of a perturbation and covering argument similar to Steps 2 \& 3 from the proof of Theorem 3.42 in 
\cite{GiaquintaMartinazzi:2012}, it is not difficult to see that the coercive estimate \eqref{Bsccoercive0} implies,
for $\wc>0$ small enough, the existence of constants $\kappa_1 =\kappa_1(\ep_0,\vec{\uc},\vec{\vc},\wc)>0$ and $\mu_1=\mu_1(\ep_0,\vec{\uc},\vec{\vc},\wc)>0$
such that
\begin{equation} \label{Bsccoercive1}
\ip{\delb{\Lambda}u}{\Bsc^{\Lambda\Sigma}\delb{\Sigma}u}_{\Omega_0} \geq \kappa_1\norm{u}_{H^1(\Omega_0)}^2 -\mu_1\norm{u}^2_{L^2(\Omega_0)}, \quad \forall \; u \in H^1(\Omega_0,\Rbb^4),
\end{equation}
provided that $|\thetahu^3+\underline{\hat{\psi}}|<\wc$ on $\Omega_0$. Furthermore, we see from \eqref{Bttdef} that $\Btt^{\Lambda\Sigma}$ and $\Bsc^{\Lambda\Sigma}$ are related the non-singular, see \eqref{Thetaorth}, \eqref{EThetasmooth} and \eqref{lambdabnd},
change of basis $\sqrt{\lambda}\Theta$, that is, $\Btt^{\Lambda\Sigma} = (\sqrt{\lambda}\Theta)^{\tr}\Bsc^{\Lambda\Sigma}\sqrt{\lambda}\Theta$.
Using this fact together with another perturbation and covering argument, we deduce from the coercive estimate \eqref{Bsccoercive1} the existence of constants
 $\kappa_2 =\kappa_2(\ep_0,\vec{\uc},\vec{\vc},\wc)>0$ and $\mu_2=\mu_2(\ep_0,\vec{\uc},\vec{\vc},\wc)>0$
such that
\begin{equation*} 
\ip{\delb{\Lambda}u}{\Btt^{\Lambda\Sigma}\delb{\Sigma}u}_{\Omega_0} \geq \kappa_1\norm{u}_{H^1(\Omega_0)}^2 -\mu_1\norm{u}^2_{L^2(\Omega_0)}, \quad \forall \; u \in H^1(\Omega_0,\Rbb^4).
\end{equation*}
\end{proof}

Next, we observe from \eqref{Pttdef} and \eqref{Pscmat} that we can decompose the matrix $\Ptt$ as
\begin{equation}\label{Pttmat}
\Ptt =  \Pbb_0 \Ptt_0\Pbb -\kappa \ptt \Pbb_0 + \Pbb \Ptt_1\Pbb + \Ptt_2
\end{equation}
where
\begin{align}
\Ptt_0&= - q\epu\Theta^{\tr}\bigl(\delb{0}E^{\tr}+\Ett_0\bigr), \label{Ptt0def} \\
\Ptt_1 &=  \biggl(-\delb{0}\biggl(\frac{\epu}{1-\epu}\biggr)q -\frac{2\epu}{1-\epu}\delb{0}q + \frac{\epu}{1-\epu}\det(J) \psihu_\omega  \elltu^\omega \notag\\
&\qquad   -\frac{\epu}{2} \delb{0}\gu^{\rho\tau}\psiu_\rho\psiu_\tau |\gu|^{-\frac{1}{2}}\det(\thetab^J_\Lambda) |\Nu|_{\gu} |\psiu|^{-2}_{\mul}\biggr)\id 
- \frac{\epu q}{1-\epu}\Theta^{\tr}\bigl(\delb{0}E^{\tr}+\Ett_0\bigr), \label{Ptt1def}\\
\Ptt_2 &= \Bigl(\bigl(2\epu \det(J) \elltu^\omega\psihu^{\alpha} \Piu_\omega^{\beta} -2|\psiu|_{\mul}^{-1}\det(J) \ellu^\omega\psiu^{\sigma}\hu_\omega^{\alpha} \hu^{\beta}_\sigma\bigr) \Theta_\alpha^{[i}\Theta_\beta^{j]} 
 \Bigr) \label{Ptt2def}
\intertext{and}
\ptt &= |\gu|^{-\frac{1}{2}}\det(\thetab^J_\Lambda)|\psiu|_{\mul}\bigl(-\gammahu^{00}\bigr)^{\frac{1}{2}} . \label{pttdef}
\end{align}

\begin{lem} \label{matlemB}
Suppose $s\in \Rbb$ and
\begin{equation*}
\frac{1}{2}|\Ptt_0|^2_{\op}+|\Ptt_1|_{\op}+|s| \biggl| \delb{0}\biggl(\frac{\epu q}{1-\epu}\biggr)\biggr|\leq \rc.
\end{equation*}
Then  there exists positive constants  $\chi=\chi(\rc,\vec{\uc},\vec{\vc})>0$,
$\kappa =\kappa(\vec{\uc},\vec{\vc})>0$, and $\qc_a(\ep_0,\vec{\uc},\vec{\vc})>0$, $a=0,1$, such that
\begin{equation*}
\qc_0\Pbb \leq -\Qtt \leq \qc_1 \Pbb
\end{equation*}
and
\begin{equation*}
\Ptt +s\delb{0}\Qtt + \chi \Qtt \leq 0.
\end{equation*}
\end{lem}
\begin{proof}
From the bounds \eqref{gubnd}-\eqref{psibnd} and \eqref{thetabIbnd},  and the definitions \eqref{epformula},  \eqref{qdef} and \eqref{pttdef}, it 
is clear that there exists positive constants $\qc_a=\qc_a(\ep_0,\vec{\uc},\vec{\vc})>0$ and $\pc=\pc(\vec{\uc},\vec{\vc})>0$
such that
\begin{gather}
0<\qc_0 \leq \frac{\epu q}{1-\epu}=\frac{|\psiu|_{\mul}q}{\ep_0} \leq \qc_1 \label{matlemB1}
\intertext{and}
\ptt \geq \pc \label{matlemB2}.
\end{gather}
Using \eqref{matlemB1}, it then follows from \eqref{Qttdef} that the matrix $\Qtt$ satisfies  
\begin{equation*}
\qc_0\Pbb \leq -\Qtt \leq \qc_1 \Pbb,
\end{equation*}
which establishes the first of the stated inequalities.

Next, we fix a vector $Y=(y_0,y_1,y_2,y_3)\in \Rbb^4$ and decompose it as
\begin{equation*}
Y=y_0 e_0+ \vec{Y}+ y_3 e_3
\end{equation*}
where $e_i$, $i=0,1,2,3$, is the standard basis for $\Rbb^4$ and
\begin{equation*}
\vec{Y} = (0,y_1,y_2,0).
\end{equation*}
Then after a short calculation using \eqref{Pbbdef}, \eqref{Pbb0def}, \eqref{Qttdef} and \eqref{Pttmat},  we find that
\begin{equation*}
\ipe{Y}{\Ptt Y}+s\ipe{Y}{\delb{0}\Qtt Y}+\chi \ipe{Y}{\Qtt Y} 
= y_0\ipe{e_0}{\Ptt_0\vec{Y}}-\kappa \ptt y_0^2 +\ipe{\vec{Y}}{\Ptt_1\vec{Y}} -s \delb{0}\biggl(\frac{\epu q}{1-\epu}\biggr)|\vec{Y}|^2
-\chi q|\vec{Y}|^2,
\end{equation*}
and hence, with the help of the Cauchy-Schwartz inequality, that
\begin{align*}
\ipe{Y}{\Ptt Y}+s\ipe{Y}{\delb{0}\Qtt Y}+\chi \ipe{Y}{\Qtt Y} 
&\leq  |y_0||\Ptt_0\vec{Y}|-\kappa \ptt y_0^2 +|\vec{Y}| |\Ptt_1\vec{Y}| +|s| \biggl|\delb{0}\biggl(\frac{\epu q}{1-\epu}\biggr)
\biggr||\vec{Y}|^2
-\chi q|\vec{Y}|^2\notag \\
&\leq |y_0||\Ptt_0|_{\op}|\vec{Y}|-\kappa \ptt y_0^2 +\biggl(|\Ptt_1|_{\op}+|s|\biggl|\delb{0}\biggl(\frac{\epu q}{1-\epu}\biggr)
\biggr|-\chi \frac{\epu q}{1-\epu}\biggr)|\vec{Y}| ^2.
\end{align*}
From this, we deduce
\begin{align*}
\ipe{Y}{\Ptt Y}+s\ipe{Y}{\delb{0}\Qtt Y}+\chi \ipe{Y}{\Qtt Y}& \leq \biggl(\frac{1}{2}-\kappa \ptt\biggr) y_0^2 +\biggl(\frac{1}{2}|\Ptt_0|_{\op}^2+|\Ptt_1|_{\op}+|s|\biggl|\delb{0}\biggl(\frac{\epu q}{1-\epu}\biggr)
\biggr|-\chi \frac{\epu q}{1-\epu} \biggr)|\vec{Y}| ^2
\end{align*}
via an application of Young's inequality. But
\begin{equation*}
\frac{1}{2}|\Ptt_0|^2_{\op}+|\Ptt_1|_{\op}+|s| \biggl| \delb{0}\biggl(\frac{\epu q}{1-\epu}\biggr)\biggr|\leq \rc
\end{equation*}
by assumption, and so, we have that
\begin{align*}
\ipe{Y}{\Ptt Y}+s\ipe{Y}{\delb{0}\Qtt Y}+\chi \ipe{Y}{\Qtt Y}& \leq \biggl(\frac{1}{2}-\kappa \ptt\biggr) y_0^2 +\biggl(r-\chi \frac{\epu q}{1-\epu} \biggr)|\vec{Y}| ^2.
\end{align*}
Assuming that $\kappa$ and $\chi$ are both positive, we see from \eqref{matlemB1}, \eqref{matlemB2} and the above inequality that
\begin{align*}
\ipe{Y}{\Ptt Y}+s\ipe{Y}{\delb{0}\Qtt Y}+\chi \ipe{Y}{\Qtt Y}& \leq \biggl(\frac{1}{2}-\kappa \pc\biggr) y_0^2 + (\rc-\chi \qc_0)|\vec{Y}| ^2
\end{align*}
from which we conclude that
\begin{align*}
\ipe{Y}{\Ptt Y}+s\ipe{Y}{\delb{0}\Qtt Y}+\chi \ipe{Y}{\Qtt Y} \leq 0
\end{align*}
by setting  $\chi=\qc_0/\rc$ and $\kappa = 1/(2\pc)$.
This establishes the second stated inequality and completes the proof.
\end{proof}

\subsection{Local-in-time existence and uniqueness of solutions\label{Lwfsec}}

\begin{Def} \label{LwfDef}
Given $s=k/2$ with $k\in \mathbb{Z}$, we say the the initial data\footnote{We assume here that the functions $\sigmau_{i}{}^k{}_j \in H^{s+\frac{1}{2}}(\Omega)$ and the spatial coframe $\thetab^I\in H^{s+\frac{1}{2}}(\Omega_0,\Rbb^4)$, $I=1,2,3$,
have been fixed and
$\thetab^I$ satisfies the inequalities \eqref{thetabIbnd} for some positive constants $\vc_0,\vc_1$.}
\begin{align}
(\Utt_0,\Utt_1,\tilde{\Utt}_{2s+1})& \in H^{s+1}(\Omega_0,\Rbb^4)\times
H^{s+\frac{1}{2}}(\Omega_0,\Rbb^4)\times L^2(\del{}\Omega_0,\Ran(\Pbb)), \label{LwfDef1.1} \\
(\phi_0,\thetahu^0_0,\vartheta_0) &\in H^{s+\frac{3}{2}}(\Omega_0,\Rbb^4)\times H^{s+\frac{3}{2}}(\Omega_0,\Rbb^4)\times
H^{s+\frac{3}{2}}(\Omega_0,\Rbb^4), \label{LwfDef1.2}\\
(\zeta_0,\Ztt_0,\Ztt_1)&\in H^{s+\frac{3}{2}}(\Omega_0)\times  H^{s+\frac{3}{2}}(\Omega_0)\times H^{s+1}(\Omega_0),  \label{LwfDef1.3}
\end{align}
for the IBVP \eqref{libvp.1d}-\eqref{libvp.14d} satisfies the \textit{compatibility conditions} if
the higher formal
time derivatives
\begin{align*}
\Utt_\ell &= \delb{0}^\ell \Utt \bigl|_{\xb^0=0}, \quad \ell=2,3\ldots,2s+1,\\
\phi_\ell &= \delb{0}^{\ell}\phi\bigl|_{\xb^0=0}, \quad \ell=1,2,\ldots,2s+2,\\
\thetahu^0_\ell &= \delb{0}^{\ell}\thetahu^0 \bigl|_{\xb^0=0}, \quad \ell=1,2,\ldots,2s+2, \\
\vartheta_\ell &= \delb{0}^{\ell}\vartheta \bigl|_{\xb^0=0}, \quad \ell=1,2,\ldots,2s+2, \\
\zetau_\ell &= \delb{0}^\ell \zetau \bigl|_{\xb^0=0}, \quad \ell=1,2,\ldots,2s+2,\\
\Ztt_\ell &= \delb{0}^\ell \Ztt \bigl|_{\xb^0=0}, \quad \ell=2,3\ldots,2s+2,
\end{align*} 
which are  generated from the initial data by differentiating \eqref{libvp.1d}, \eqref{libvp.3d}-\eqref{libvp.6d} and \eqref{libvp.8d} formally with respect to $\xb^0$ the required number of times and setting $\xb^0=0$, satisfy 
\begin{align*}
\Utt_\ell &\in H^{s+1-\frac{m(s+1,\ell)}{2}}(\Omega_0,\Rbb^4), \quad \ell =2,3,\ldots, 2s+1, \\
\phi_\ell,  \thetahu^0_\ell, \vartheta_\ell &\in H^{s+\frac{3}{2}-\frac{m(s+\frac{3}{2},\ell)}{2}}(\Omega_0,\Rbb^4),
\quad \ell=1,2,\ldots,2s+2,\\
\zetau_\ell &\in H^{s+\frac{3}{2}-\frac{m(s+\frac{3}{2},\ell)}{2}}(\Omega_0),
\quad \ell=1,2,\ldots,2s+2, \\
\Ztt_\ell &\in H^{s+\frac{3}{2}-\frac{m(s+\frac{3}{2},\ell)}{2}}(\Omega_0),
\quad \ell=2,3,\ldots,2s+2, \\
\end{align*}
and
\begin{align*}
\delb{0}^\ell \Bigl(\thetab^3_\alpha \bigl(\Btt^{\alpha\beta}\delb{\beta}\Utt
+ \Xtt^{\alpha}\bigr)-
\Qtt\delb{0}^2\Utt
- \Ptt\delb{0}\Utt -\Gtt\Bigr) &= 0 \quad \text{in $\{0\}\times \del{}\Omega_0$},\quad \ell=0,1,\ldots,2s-1,\\
\delb{0}^\ell \Ztt &= 0\quad \text{in $\{0\}\times \del{}\Omega_0$}, \quad \ell=0,1,\ldots,2s+1,
\end{align*}
where the $(2s+1)$ formal time derivative of $\Utt$ restricted to the boundary $\del{}\Omega_0$ at $\xb^0=0$ is given by $\tilde{\Utt}_{2s+1}$, 
that is\footnote{In general, the this expression has to be taken as a definition and not as the restriction of $\Utt_{2s+1}$ to the boundary
$\del{}\Omega_0$ since the trace of $\Utt_{2s+1}$ on the boundary is not necessarily well defined due to our assumption that
$\Utt_{2s+1}\in L^2(\Omega_0)$.}
\begin{equation*} 
\delb{0}^{2s+1}\Utt|_{\{0\}\times\del{}\Omega_0} = \tilde{\Utt}_{2s+1}. 
\end{equation*}
\end{Def}

For use below, we also define the function space
\begin{align} 
\Yc^s_T= &\Xc^{s+1}_T(\Omega_0,\Rbb^4)\times X^0_T(\del{}\Omega_0,\Ran(\Pbb))
\times \Xcr^{s+\frac{3}{2}}_T(\Omega_0,\Rbb^4)  \notag \\
&\times \Xcr^{s+\frac{3}{2}}_T(\Omega_0,\Rbb^4)
\times \Xc^{s+\frac{3}{2}}_T(\Omega_0,\Rbb^4)
\times \Xcr^{s+\frac{3}{2}}_T(\Omega_0)\times \Xc^{s+\frac{3}{2}}_T(\Omega_0).\label{YcsTdef}
\end{align}

\begin{thm} \label{locexistA}
Suppose  $s> n/2+1/2$ and  $s=k/2$ for $k\in \Zbb_{\geq0}$, $\sigmau_{i}{}^k{}_j \in H^{s+\frac{1}{2}}(\Omega)$, 
the spatial coframe fields $\thetab^I\in H^{s+\frac{1}{2}}(\Omega_0,\Rbb^4)$
satisfy the inequalities \eqref{thetabIbnd} for some positive constants $\vc_0,\vc_1$, the initial data \eqref{LwfDef1.1}-\eqref{LwfDef1.3}
satisfies the compatibility conditions from Definition \ref{LwfDef}, the inequalities \eqref{gubnd}-\eqref{zubnd}
for some positive constants $\uc_0, \uc_1$, and the restriction $|\underline{\psih}+\thetahu^3|\leq \wc$ where $\wc$ is
the positive constant from Lemma \ref{Bsccoercive},   and the constant
$\kappa>0$ is chosen as in the statement of Lemma \ref{matlemB}.
Then there exists a $T>0$ and maps 
$(\Utt,\Vtt,\phi,\thetahu^0,\vartheta,\zetau,\Ztt)\in \Yc^s_T$
that determine a solution $(\Utt,\phi,\thetahu^0,\vartheta,\zetau,\Ztt)$ to the IBVP \eqref{libvp.1d}-\eqref{libvp.14d}
that satisfies the additional property that $\Vtt|_{\xb^0=0}=\tilde{\Utt}_{2s+1}$ and the pair $(\delb{0}^{2s}\Utt,\Vtt)$
defines a weak solution to the linear wave equation
that is obtained from differentiating \eqref{libvp.1d}-\eqref{libvp.2d} $2s$-times with respect to $\xb^0$. 
Moreover, the solution $(\Utt,\phi,\thetahu^0,\vartheta,\zetau,\Ztt)$ is unique in the space
\begin{align*}
X^{s+1,2}_T(\Omega_0,\Rbb^4)
\times \mathring{X}^{s+\frac{3}{2},3}_T(\Omega_0,\Rbb^4) \times \mathring{X}^{s+\frac{3}{2},3}_T(\Omega_0,\Rbb^4)
\times X^{s+\frac{3}{2},3}_T(\Omega_0,\Rbb^4)
\times \mathring{X}^{s+\frac{3}{2},3}_T(\Omega_0)\times X^{s+\frac{3}{2},3}_T(\Omega_0)
\end{align*}
and satisfies the energy estimate
\begin{equation*}
 \Rc(\xb^0) \leq c\biggl( \Rc(0) + \int_0^{\xb^0}C(\Nc(\tau))\bigl(\Rc(\tau)+\Nc(\tau)\bigr)\, d\tau \biggr)
\end{equation*}
where 
\begin{align*}
\Rc &= \Nc+ \norm{\Vtt}_{L^2(\del{}\Omega_0)},\\
\Nc &= \sum_{\ell=0}^1 \Bigl(\norm{\delb{0}^\ell\zetau}_{\Ec^{s+\frac{3}{2}}} +\norm{\delb{0}^\ell\phi}_{\Ec^{s+\frac{3}{2}}} +  
\norm{\delb{0}^\ell \thetahu^0}_{\Ec^{s+\frac{3}{2}}}\Bigr)+\norm{\Ztt}_{\Ec^{s+\frac{3}{2}}}+\norm{\vartheta}_{\Ec^{s+\frac{3}{2}}} +  \norm{\Utt}_{\Ec^{s+1}}
\intertext{and}
c &=  C(\Nc(0)) + T C(\norm{\Nc}_{L^\infty([0,T])}) \norm{\Nc}_{L^\infty([0,T])}.
\end{align*}
\end{thm}

\begin{proof}$\;$

\smallskip

\noindent \underline{Existence:} 
We begin by defining the formal solution generated from the initial data by setting\footnote{Note, by the definition of the compatibility conditions, that $\phi_\ell$ and $\thetah^0_\ell$ are related via formal
differentiation of the evolution equation \eqref{libvp.3d}. This allows us to define $\phi_{2s+3}$ in terms
of  $(\phi_{\ell},\thetah^0_\ell)$, $0\leq \ell\leq 2s+2$. Moreover, with the help of the calculus inequalities from
Appendix \ref{calculus}, it is not difficult to verify that 
\begin{equation*}
\norm{\delb{0}\phi_f}_{\Xc^{s+\frac{3}{2}}_1} \leq
 C\biggl(\norm{\phi_f}_{\Xc^{s+\frac{3}{2}}_1}, \norm{\thetahu^0_f}_{\Xc^{s+\frac{3}{2}}_1}\biggr).
\end{equation*} 
By using similar arguments involving the evolution equations \eqref{libvp.4d} and \eqref{libvp.8d}, we can also define
 $\thetahu^0_{2s+3}$ and $\zetau_{2s+3}$ in terms of the initial data $(\phi_{\ell},\thetah^0_\ell,\vartheta_\ell,\Ztt_\ell)$, $0\leq \ell\leq 2s+2$,
 and bound $\delb{0}\thetahu^0_f$ and $\delb{0}\zetau_f$  by
 \begin{equation*}
 \norm{\delb{0}\thetahu^0_f}_{\Xc^{s+\frac{3}{2}}_1} + \norm{\delb{0}\zetau_f}_{\Xc^{s+\frac{3}{2}}_1}\leq
 C\biggl(\norm{\phi_f}_{\Xc^{s+\frac{3}{2}}_1}, 
 \norm{\thetahu^0_f}_{\Xc^{s+\frac{3}{2}}_1},\norm{\vartheta_f}_{\Xc^{s+\frac{3}{2}}_1},\norm{\Ztt_f}_{\Xc^{s+\frac{3}{2}}_1}\biggr).
 \end{equation*}  
 } 
\begin{gather*}
\Utt_f(\xb^0,\xb^\Lambda) = \sum_{\ell=0}^{2s+1} \frac{(\xb^0)^\ell}{\ell !}\Utt_\ell(\xb^\Lambda),\quad
\Vtt_f(\xb^0,\xb^\Lambda) = \tilde{\Utt}_{2s+1}(\xb^\Lambda),\quad
\phi_f(\xb^0,\xb^\Lambda) = \sum_{\ell=0}^{2s+3} \frac{(\xb^0)^\ell}{\ell !}\phi_\ell(\xb^\Lambda),\\
\thetahu^0_f(\xb^0,\xb^\Lambda) =  \sum_{\ell=0}^{2s+3} \frac{(\xb^0)^\ell}{\ell !}\thetahu^0_\ell(\xb^\Lambda),
\quad  
\vartheta_f(\xb^0,\xb^\Lambda) =  \sum_{\ell=0}^{2s+2} \frac{(\xb^0)^\ell}{\ell !}\vartheta_\ell(\xb^\Lambda), \\
\zetau_f(\xb^0,\xb^\Lambda) = \sum_{\ell=0}^{2s+3} \frac{(\xb^0)^\ell}{\ell !}\zetau_\ell(\xb^\Lambda) \AND
\Ztt_f(\xb^0,\xb^\Lambda) = \sum_{\ell=0}^{2s+2} \frac{(\xb^0)^\ell}{\ell !}\Ztt_\ell(\xb^\Lambda),
\end{gather*}
and we let
\begin{equation*}
R_f = \norm{\Utt_f}_{\Xc^{s+1}_1}+ \norm{\Vtt_f}_{X^0_1}+\norm{\phi_f}_{\Xcr^{s+\frac{3}{2}}_1} +
\norm{\thetahu^0_f}_{\Xcr^{s+\frac{3}{2}}_1} 
+  \norm{\vartheta_f}_{\Xc^{s+\frac{3}{2}}_1}+   \norm{\zetau_f}_{\Xcr^{s+\frac{3}{2}}_1}+
 \norm{\Ztt_f}_{\Xc^{s+\frac{3}{2}}_1} .
\end{equation*}
Assuming that $T\in (0,1]$ and $\delta>0$, we fix $R>R_f$ and we define a non-empty, closed set 
\begin{equation*}
B_{R,\delta} \subset \Yc^s_T
\end{equation*}
by
\begin{equation*}
(\Uttg,\Vttg,\phig,\thetahug,\varthetag,\zetaug,\Zttg) \in B_{R,\delta} 
\end{equation*}
if and only if 
\begin{equation*}
  \norm{\Uttg}_{\Xc^{s+1}_T}+ \norm{\Vttg}_{X^0_T}+\norm{\phig}_{\Xcr^{s+\frac{3}{2}}_T} +
\norm{\thetahug}_{\Xcr^{s+\frac{3}{2}}_T} 
+  \norm{\varthetag}_{\Xc^{s+\frac{3}{2}}_T}+   \norm{\zetaug}_{\Xcr^{s+\frac{3}{2}}_T}+
 \norm{\Zttg}_{\Xc^{s+\frac{3}{2}}_T} \leq R,
\end{equation*}
\begin{equation*}
\sup_{0\leq \xb^0\leq T}\Bigl( \norm{\zetaua(\xb^0)-\zetau_f(0)}_{E^{s,0}}+\norm{\phia(\xb^0)-\phi_f(0)}_{E^{s+1,1}}+ 
\norm{\thetahua(\xb^0)-\thetahu^0_f(0)}_{E^{s,0}}
+ \norm{\varthetaa(\xb^0)-\vartheta_f(0)}_{E^{s,0}}\Bigr)\leq \delta,
\end{equation*}
\begin{align*}
\delb{0}^\ell \Uttg \bigl|_{\xb^0=0} &= \Utt_\ell, \quad \ell=0,1,\ldots,2s+1,\\
\Vttg\bigl|_{\xb^0=0} &= \tilde{\Utt}_{2s+1}, \\
\delb{0}^{\ell}\phig\bigl|_{\xb^0=0} &= \phi_\ell, \quad \ell=0,1,\ldots,2s+3,\\
\delb{0}^{\ell}\thetahug \bigl|_{\xb^0=0}&= \thetahu^0_\ell  , \quad \ell=0,1,\ldots,2s+3,\\
\delb{0}^\ell \varthetag \bigl|_{\xb^0=0}&= \vartheta_\ell , \quad \ell=0,1,\ldots,2s+2,\\
\delb{0}^\ell \zetaug \bigl|_{\xb^0=0}&= \zetau_\ell , \quad \ell=0,1,\ldots,2s+3
\intertext{and}
\delb{0}^\ell \Zttg \bigl|_{\xb^0=0}&= \zetau_\ell , \quad \ell=0,1,\ldots,2s+2.
\end{align*} 

By the choice of initial data, the formal solution satisfies 
\begin{equation*}
(\zetau_f(\xb),\phi_f(\xb),\delb{}\phi_f(\xb),\thetahu^0_f(\xb),\vartheta(\xb))\in \Uc, \qquad \forall\; \xb\in \{0\}\times \overline{\Omega}_0,
\end{equation*} 
and
\begin{equation*}
|\psihu_f(\xb)-\hat{\thetau}{}^3(\xb)| <\wc, \qquad \forall\; \xb\in \{0\}\times \del{}\Omega_0.
\end{equation*}
Choosing $\delta >0$ small enough, we see from the definition of the set $B_{R,\delta}$ and Sobolev's inequality (Theorem \ref{ISobolev}) that
any element $(\Uttg,\Vttg,\phig,\thetahug,\varthetag,\zetaug,\Zttg)\in B_{R,\delta}$ will satisfy
\begin{equation} \label{Ucsat}
(\zetaug(\xb),\phig(\xb),\delb{}\phig(\xb),\thetahug(\xb),\varthetag(\xb))\in \Uc, \qquad \forall\; \xb\in [0,T]\times \overline{\Omega}_0,
\end{equation} 
and
\begin{equation} \label{coercsat}
|\grave{\psihu}(\xb)-\hat{\thetau}{}^3(\xb)| <\wc, \qquad \forall\; \xb\in [0,T]\times \del{}\Omega_0.
\end{equation}
We then define a map
\begin{align} \label{Jcmap}
\Jc_{\!T} &\: :\: B_{R,\delta} \longrightarrow \Yc^s_T
\end{align}
by letting
\begin{equation*} 
(\Utta,\Vtta,\phia,\thetahua,\varthetaa,\zetaua,\Ztta):= \Jc_{\!T}(\Uttg,\Vttg,\phig,\thetahug,\varthetag,\zetaug,\Zttg)
\end{equation*}
denote the unique solution to the linear IBVP
\begin{align}
\delb{\alpha}\bigl(\grave{\Btt}{}^{\alpha\beta}\delb{\beta}\Utta + \grave{\Xtt}{}^{\alpha}\bigr)&=\grave{\Ftt} &&
\text{in $[0,T]\times \Omega_0$}, \label{llibvp.1} \\
\thetab^3_\alpha \bigl(\grave{\Btt}{}^{\alpha\beta}\delb{\beta}\Utta
+ \grave{\Xtt}{}^{\alpha}\bigr) &=  
\grave{\Qtt}{}
\delb{0}^2\Utta 
+ \grave{\Ptt}\delb{0}\Utta +\grave{\Gtt} && \text{in $[0,T]\times \del{}\Omega_0$,} \label{llibvp.2} \\
\delb{0}\phia &= \grave{\alpha} \thetahua && \text{in $[0,T]\times \Omega_0$}, \label{llibvp.3} \\
\delb{0}\thetahua & = \varthetaa+\grave{\betat} \thetahua &&
\text{in $[0,T]\times \Omega_0$}, \label{llibvp.4} \\
\delb{0}\varthetaa & = \grave{\lambda}\grave{\Theta}\Utta+\grave{\beta}\varthetag &&
\text{in $[0,T]\times \Omega_0$}, \label{llibvp.5} \\
\delb{\alpha}(\grave{\Msc}{}^{\alpha\beta}\delb{\beta}\Ztta+\grave{\Ytt}{}^\alpha) &= \grave{\Ktt} && \text{in $[0,T]\times \Omega_0$,} \label{llibvp.6}\\
\Ztta &= 0 && \text{in $[0,T]\times \del{}\Omega_0$,} \label{llibvp.7}\\
\delb{0}\zetaua &= \Ztta && \text{in $[0,T]\times \Omega_0$}, \label{llibvp.8} \\
(\Utta,\delb{0}\Utta) &= (\Utt_0,\Utt_1)  && \text{in $\{0\}\times \Omega_0$,} \label{llibvp.9}\\
\phia &= \phi_0  && \text{in $\{0\}\times \Omega_0$,} \label{llibvp.10}\\
\thetahua &= \thetahu^0_0  && \text{in $\{0\}\times \Omega_0$,} \label{llibvp.11}\\
\varthetaa &= \vartheta_0  && \text{in $\{0\}\times \Omega_0$,} \label{llibvp.12}\\
(\Ztta,\delb{0}\Ztta) &= (\Ztt_0,\Ztt_1)  && \text{in $\{0\}\times \Omega_0$,} \label{llibvp.13}\\
\zetaua &= \zetau_0  && \text{in $\{0\}\times \Omega_0$,} \label{llibvp.14}\\
\Vtta &= \tilde{\Utt}_{2s+1}  && \text{in $\{0\}\times\del{}\Omega_0$,}\label{llibvp.15}
\end{align} 
where the pair $(\delb{0}^{2s}\Utta,\Vtta)$
defines a weak solution to the linear wave equation
that is obtained from differentiating \eqref{llibvp.1}-\eqref{llibvp.2} $2s$-times with respect to $\xb^0$,
and we are using the grave accent over the coefficients that appear in the IBVP  to indicate that those coefficients are being
evaluated using the fields
$(\Uttg,\phig,\thetahug,\varthetag,\zetaug,\Zttg)$, e.g. $\grave{\lambda} = \lambda(\phig,\thetahug,\psig)$.

For any $T\in (0,1]$, the existence of a solution 
\begin{equation} \label{locexistA9}
(\Utta,\Vtta,\phia,\thetahua,\varthetaa,\zetaua,\Ztta) \in \tilde{\Yc}{}^s_T, 
\end{equation}
where
\begin{align*}
\tilde{\Yc}^s_T = &\Xc^{s+1}_T(\Omega_0,\Rbb^4)\times X^0_T(\del{}\Omega_0,\Ran(\Pbb))
\times \Xcr^{s+\frac{1}{2}}_T(\Omega_0,\Rbb^4) \\
&\times \Xcr^{s+\frac{1}{2}}_T(\Omega_0,\Rbb^4) 
\times \Xc^{s+1}_T(\Omega_0,\Rbb^4)
\times \Xcr^{s+\frac{3}{2}}_T(\Omega_0)\times \Xc^{s+\frac{3}{2}}_T(\Omega_0),
\end{align*}
is then a direct consequence of the Theorems \ref{DCthm} and \ref{LAWthm}, Proposition \ref{stpropE}, and the estimates from
Lemmas \ref{coeflemA}-\ref{matlemB} and Propositions  \ref{stpropC}, \ref{stpropE} and \ref{stpropF} (see also
\eqref{fdepA.13} and \eqref{fdepA.14}). Furthermore, from these results, we deduce that the energy estimate
\begin{equation}  \label{locexistA11a}
\Rca_0(\xb^0) \leq c\biggl( \Rca_0(0) + \int_0^{\xb^0}C(\Ncg(\tau))\bigl(\Rca_0(\tau)+\Ncg(\tau)\bigr)\, d\tau 
\biggr)
\end{equation}
holds for $\xb^0\in [0,T]$, where
\begin{equation*}
\Rca_0 = \sum_{\ell=0}^1 \Bigl(\norm{\delb{0}^\ell\zetaua}_{\Ec^{s+\frac{3}{2}}} +\norm{\delb{0}^\ell\phia}_{\Ec^{s+1}} +  
\norm{\delb{0}^\ell \thetahua}_{\Ec^{s+1}}\Bigr)+\norm{\Ztta}_{\Ec^{s+\frac{3}{2}}}+\norm{\varthetaa}_{\Ec^{s+1}} +  \norm{\Utta}_{\Ec^{s+1}}+ \norm{\Vtta}_{L^2(\del{}\Omega_0)},
\end{equation*}
\begin{equation*}
\Ncg = \sum_{\ell=0}^1 \Bigl(\norm{\delb{0}^\ell\zetaug}_{\Ec^{s+\frac{3}{2}}} +\norm{\delb{0}^\ell\phig}_{\Ec^{s+\frac{3}{2}}} +  
\norm{\delb{0}^\ell \thetahug}_{\Ec^{s+\frac{3}{2}}}\Bigr)+\norm{\Zttg}_{\Ec^{s+\frac{3}{2}}}+\norm{\varthetag}_{\Ec^{s+\frac{3}{2}}} +  \norm{\Uttg}_{\Ec^{s+1}}
\end{equation*}
and
\begin{equation} \label{locexistA10}
c=  C(\grave{\Nc}(0)) + T C(\norm{\grave{\Nc}}_{L^\infty([0,T])}) \norm{\grave{\Nc}}_{L^\infty([0,T])}.
\end{equation}

At the moment, all we know is that \eqref{locexistA9} holds.  In order to for the map \eqref{Jcmap} to be well defined, we must show that the solution
$(\Utta,\Vtta,\phia,\thetahua,\varthetaa,\zetaua,\Ztta)$ lies in $\Yc^s_T$ instead of just in $\tilde{\Yc}{}^s_T$. To show this, we use \eqref{llibvp.5} to replace  $\Utta$
in \eqref{llibvp.1}-\eqref{llibvp.2} with $\delb{0}\varthetaa$ and mulitply both equations on the left by $\Eg^{\tr}$. Carrying this out, we obtain, with the help \eqref{sigmatev},
\eqref{Thetadef2} and \eqref{Bttdef}, the equations
\begin{align}
\delb{\Lambda}\bigl(\delb{0}(\Bscg^{\Lambda\Sigma}\delb{\Sigma}\varthetaa) + \Xf^{\Lambda}\bigr)&=-\delb{0}\bigl(\Eg^{\tr}\Bttg^{0\beta}\delb{\beta}\Utta
+ \Eg^{\tr}\Xttg^{0}\bigr)+\Ff &&
\text{in $[0,T]\times \Omega_0$}, \label{llibvp.1a} \\
\delb{\Lambda}\bigl(\delb{0}(\Bscg^{\Lambda\Sigma}\delb{\Sigma}\varthetaa)  + \Xf^{\Lambda}\bigr)&=  
\delb{0}(\Eg^{\tr}\Qttg
\delb{0}\Utta)  +\Gf && \text{in $[0,T]\times \del{}\Omega_0$,} \label{llibvp.2a}
\end{align}
where
\begin{align*}
\Ff &= \Eg^{\tr}\Fttg +\delb{\alpha}(\Eg^{\tr})(\Bttg^{\alpha\beta}\delb{\beta}\Utta+\Xttg^{\alpha}) ,\\
\Xf^\Lambda&= -\delb{0}\Bscg^{\Lambda\Sigma}\delb{\Sigma}\varthetaa-\Eg^{\tr}\Bttg^{\Lambda\Sigma}\delb{\Sigma}\bigl(\lambdag^{-1}\Eg\betag\varthetag\bigr)+\Eg^{\tr}\Bttg^{\Lambda 0}\delb{0}\Utta+\Eg^{\tr}\Xttg^\Lambda 
\intertext{and}
\Gf &= (-\delb{0}(\Eg^{\tr}\Qttg)+\Eg^{\tr}\Pttg)\delb{0}\Utta+ \Eg^{\tr}\Gttg.
\end{align*}
Integrating \eqref{llibvp.1a} and \eqref{llibvp.2a} in time then yields
\begin{align}
\delb{\Lambda}\biggl(\Bscg^{\Lambda\Sigma}(\xb^0)\delb{\Sigma}\varthetaa(\xb^0) - \Bscg^{\Lambda\Sigma}(0)\delb{\Sigma}\varthetaa(0) + 
\int_{0}^{\xb^0}\Xf^{\Lambda}(\tau)\,d\tau\biggr)&=-\Eg^{\tr}(\xb^0)\Bttg^{0\beta}(\xb^0)\delb{\beta}\Utta(\xb^0) \notag\\
- \Eg^{\tr}(\xb^0)\Xttg^{0}(\xb^0) + \Eg^{\tr}(0)\Bttg^{0\beta}(0)\delb{\beta}\Utta(0) &
+ \Eg^{\tr}(0)\Xttg^{0}(0)+\int_0^{\xb^0}\Ff(\tau)\, d\tau &&
\text{in $[0,T]\times \Omega_0$}, \label{llibvp.1b}\\
\delb{\Lambda}\biggl(\Bscg^{\Lambda\Sigma}(\xb^0)\delb{\Sigma}\varthetaa(\xb^0) - \Bscg^{\Lambda\Sigma}(0)\delb{\Sigma}\varthetaa(0) + 
\int_{0}^{\xb^0}\Xf^{\Lambda}(\tau)\,d\tau\biggr)&=  
\Eg^{\tr}(\xb^0)\Qttg(\xb^0)
\delb{0}\Utta(\xb^0) \notag  \\
 - \Eg^{\tr}(0)&\Qttg(0)\delb{0}\Utta(0)  +\int_0^{\xb^0}\Gf(\tau)\, d\tau  && \text{in $[0,T]\times \del{}\Omega_0$.} \label{llibvp.2b}
\end{align}

Using the estimate from Lemma \ref{coeflemA}, the multiplication estimates from Proposition \ref{calcpropB}, \ref{stpropA} and \ref{stpropC},
the integral estimates from Proposition \ref{stpropE}, and the Trace Theorem (Theorem \ref{trace}), we can estimate the coefficients from \eqref{llibvp.1b} and \eqref{llibvp.2b} as follows:
\begin{align*}
\norm{\Bscg^{\Lambda\Sigma}(0)\delb{\Sigma}\varthetaa(0)}_{H^{s+\frac{1}{2}}(\Omega_0)} &\leq
C(\Ncg(0))\Rca(0), \\
\norm{\Eg^{\tr}(\xb^0)\Bttg^{0\beta}(\xb^0)\delb{\beta}\Utta(\xb^0)}_{H^{s-\frac{1}{2}}(\Omega_0)} &\leq
C(\Ncg(0))\Rca(0) +
\int_0^{\xb^0} C(\Ncg(\tau))\Rca(\tau)\, d\tau, \\ 
\norm{\Eg^{\tr}(\xb^0)\Xtt^{0}(\xb^0)}_{H^{s-\frac{1}{2}}(\Omega_0)} &\leq
C(\Ncg(0))\Ncg(0)  +
\int_0^{\xb^0} C(\Ncg(\tau))\Ncg(\tau)\, d\tau, \\ 
\norm{\Eg^{\tr}(\xb^0)\Qttg(\xb^0)
\delb{0}\Utta(\xb^0)}_{H^{s}(\del{}\Omega_0)} &\lesssim \norm{\Eg^{\tr}(\xb^0)\Qttg(\xb^0)
\delb{0}\Utta(\xb^0)}_{H^{s+\frac{1}{2}}(\Omega_0)}\\
&\leq c\Rca_0(\xb^0),\\
\biggl\|\int_0^{\xb^0}\Xf^{\lambda}(\tau)\, d\tau\biggr\|_{H^{s+\frac{1}{2}}(\Omega_0)} &\lesssim \int_0^{\xb^0} \norm{\Xf^{\Lambda}(\tau)}_{H^{s+\frac{1}{2}}(\Omega_0)}\, d\tau \\
&\leq \int_0^{\xb^0} C(\Ncg(\tau))(\Ncg(\tau)+\Rca(\tau))\, d\tau,\\
\biggl\|\int_0^{\xb^0}\Ff(\tau)\, d\tau\biggr\|_{H^{s-\frac{1}{2}}(\Omega_0)} &\lesssim \int_0^{\xb^0} \norm{\Ff(\tau)}_{H^{s-\frac{1}{2}}(\Omega_0)}\, d\tau \\
&\leq \int_0^{\xb^0} C(\Ncg(\tau))(\Ncg(\tau)+\Rca(\tau))\, d\tau
\intertext{and}
\biggl\|\int_0^{\xb^0}\Gf(\tau)\, d\tau\biggr\|_{H^{s}(\del{}\Omega_0)} &\lesssim \int_0^{\xb^0} \norm{\Gf(\tau)}_{H^{s+\frac{1}{2}}(\Omega_0)}\, d\tau \\
&\leq \int_0^{\xb^0} C(\Ncg(\tau))(\Ncg(\tau)+\Rca(\tau))\, d\tau,
\end{align*}
where the constant $c$ is of the form \eqref{locexistA10}, $\Ncg$ is defined above and
\begin{equation*}
\Rca = \sum_{\ell=0}^1 \Bigl(\norm{\delb{0}^\ell\zetaua}_{\Ec^{s+\frac{3}{2}}} +\norm{\delb{0}^\ell\phia}_{\Ec^{s+\frac{3}{2}}} +  
\norm{\delb{0}^\ell \thetahua}_{\Ec^{s+\frac{3}{2}}}\Bigr)+\norm{\Ztta}_{\Ec^{s+\frac{3}{2}}}+\norm{\varthetaa}_{\Ec^{s+\frac{3}{2}}} +  \norm{\Utta}_{\Ec^{s+1}}+ \norm{\Vtta}_{L^2(\del{}\Omega_0)}.
\end{equation*}
From the above coefficient estimates, it then follows from Neumann BVP \eqref{llibvp.1a}-\eqref{llibvp.2a} and elliptic regularity, see Theorem \ref{ellipticregN},
that 
\begin{align}
\norm{\varthetaa(\xb^0)}_{H^{s+\frac{3}{2}}(\Omega_0)}\leq c \biggl(C(\Ncg(0))\Rca(0)+ \Rca_0(\xb^0) 
  +
 \int_0^{\xb^0} C(\Ncg(\tau))(\Ncg(\tau)+\Rca(\tau))\, d\tau
 \biggr),\quad \xb^0\in [0,T], \label{locexistA11b}
\end{align} 
where the constant $c$ is of the form \eqref{locexistA10} and we have used the fact that
\begin{equation} \label{locexistA11c}
\Rca(0) = \Ncg(0)+ \norm{\Vtta(0)}_{L^2(\del{}\Omega_0)}.
\end{equation}
Moreover, from the evolution equation \eqref{llibvp.4}, it follows, with the help of Lemma \ref{coeflemA}, Theorem \ref{calcpropB} and
Theorem \ref{stpropE} with $r=0$, that
\begin{equation} \label{locexistA11da}
\norm{\thetahua(\xb^0)}_{H^{s+\frac{3}{2}}(\Omega_0)} \leq c\biggl(\norm{\thetahua(0)}_{H^{s+\frac{3}{2}}(\Omega_0)}
+ \int_0^{\xb^0} \norm{\thetahua(\tau)}_{H^{s+\frac{3}{2}}(\Omega_0)}+\norm{\varthetaa(\tau)}_{H^{s+\frac{3}{2}}(\Omega_0)}\, d\tau \biggr), 
\end{equation}
for $\xb^0\in [0,T]$, where again the constant $c$ is of the from \eqref{locexistA10}, while  the estimates
\begin{gather}
\norm{\phia(\xb^0)}_{H^{s+\frac{3}{2}}(\Omega_0)}\leq 
 c\biggl(\norm{\phia(0)}_{H^{s+\frac{3}{2}}(\Omega_0)}
+ \int_0^{\xb^0} \norm{\thetahua(\tau)}_{H^{s+\frac{3}{2}}(\Omega_0)}\, d\tau \biggr) \label{locexistA11db}
\intertext{and}
\norm{\delb{0}\phia(\xb^0)}_{H^{s+\frac{3}{2}}(\Omega_0)}+\norm{\delb{0}\thetahua(\xb^0)}_{H^{s+\frac{3}{2}}(\Omega_0)} \leq 
  c\Bigl(\norm{\thetahua(\xb^0)}_{H^{s+\frac{3}{2}}(\Omega_0)}+
\norm{\varthetaa(\xb^0)}_{H^{s+\frac{3}{2}}(\Omega_0)} \Bigr), \label{locexistA11dc}
\end{gather}
which hold for $\xb^0\in [0,T]$, are a direct consequence of the evolution equations \eqref{llibvp.3}-\eqref{llibvp.4},
Lemma \ref{coeflemA}, Theorem \ref{calcpropB} and Theorem \ref{stpropE}.

By combining the estimates \eqref{locexistA11a}, \eqref{locexistA11b} and \eqref{locexistA11da}-\eqref{locexistA11dc}, we arrive, with the help of \eqref{locexistA11c}, at the energy estimate
\begin{equation} \label{locexistA11e}
\Rca(\xb^0) \leq c\biggl(\Rca(0) + \int_0^{\xb^0}C(\Ncg(\tau))\bigl(\Rca(\tau)+\Ncg(\tau)\bigr)\, d\tau \biggr),\quad \xb^0\in [0,T], 
\end{equation}
where, as above, the constant $c$ is of the form \eqref{locexistA10}. We then observe from this estimate and Proposition \ref{stpropE}
that
\begin{align} 
&\norm{\zetaua(\xb^0)-\zetau_f(0)}_{E^{s,0}}+\norm{\phia(\xb^0)-\phi_f(0)}_{E^{s+1,1}} \notag \\
&\qquad+ 
\norm{\thetahua(\xb^0)-\thetahu^0_f(0)}_{E^{s,0}}
+ \norm{\varthetaa(\xb^0)-\vartheta_f(0)}_{E^{s,0}} \leq \int_0^{\xb^0}\Rca(\tau)\,d \tau.\label{locexistA11f}
\end{align}
We also conclude from \eqref{locexistA11e} that the map $\Jc_{\!T}$ is well defined since it implies
that  $(\Utta,\Vtta,\phia,\thetahua,\varthetaa,\zetaua,\Ztta)$ lies in $\Yc^s_T$.

Recalling that $T\in (0,1]$ and noting that
\begin{equation*}
\Ncg(0)\leq \grave{\Rc}(0)=\Rca(0)\leq R_f \AND \norm{\Ncg}_{L^\infty([0,T])} \leq R,
\end{equation*}
we see from \eqref{locexistA10}, \eqref{locexistA11e} and Gronwall's inequality that 
\begin{equation*}
\norm{\Rca}_{L^\infty([0,T])} \leq C_0 := (C(R_f)+TC(R))e^{C(R)T},
\end{equation*}
and hence, by \eqref{locexistA11f}, that
\begin{equation*}
\sup_{0\leq \xb^0 \leq T}\Bigl(\norm{\zetaua(\xb^0)-\zetau_f(0)}_{E^{s,0}}+\norm{\phia(\xb^0)-\phi_f(0)}_{E^{s+1,1}} + 
\norm{\thetahua(\xb^0)-\thetahu^0_f(0)}_{E^{s,0}}
+ \norm{\varthetaa(\xb^0)-\vartheta_f(0)}_{E^{s,0}}\Bigr)<TC_0.
\end{equation*}
By choosing $R$ large enough so that $R>C(R_f) +1$ and setting $T=\min\bigl\{\frac{1}{C(R)}, \frac{\delta}{C_0},1\bigr\}$, we can guarantee that
\begin{gather*}
\norm{\acute{\Rc}}_{L^\infty([0,T])}\leq R
\intertext{and}
\sup_{0\leq \xb^0 \leq T}\Bigl(\norm{\zetaua(\xb^0)-\zetau_f(0)}_{E^{s,0}}+\norm{\phia(\xb^0)-\phi_f(0)}_{E^{s+1,1}} + 
\norm{\thetahua(\xb^0)-\thetahu^0_f(0)}_{E^{s,0}}
+ \norm{\varthetaa(\xb^0)-\vartheta_f(0)}_{E^{s,0}}\Bigr)\leq\delta.
\end{gather*}
We therefore conclude from the definition of the map $\Jc_{\! T}$ and the set $B_{R,\delta}$ that
\begin{equation}\label{locexistA13c}
\Jc_{\! T}(B_{R,\delta})\subset B_{R,\delta}.
\end{equation}
This allows us to use $\Jc_{\! T}$ to define a sequence in $B_{R,\delta}$ by setting
\begin{equation} \label{locexistA13a}
(\Utt_{(j)},\Vtt_{(j)},\phi_{(j)},\thetahu^0_{(j)},\vartheta_{(j)},\zetau_{(j)},\Ztt_{(j)}):= 
\underset{\text{j times}}{\underbrace{\Jc_{\!T}\circ \Jc_{\!T} \circ \cdots \circ \Jc_{\!T}}}(\Utt_{f},\Vtt_{f},\phi_{f},\thetahu^0_{f},\vartheta_{f},\zetau_{f},\Ztt_{f}), \quad j\in \Zbb_{\geq 0}.
\end{equation}

Since 
\begin{equation*}
  \norm{\phi_{(j)}}_{\Xcr^{s+\frac{3}{2}}_T} +
\norm{\thetahu^0_{(j)}}_{\Xcr^{s+\frac{3}{2}}_T} 
+   \norm{\zetau_{(j)}}_{\Xcr^{s+\frac{3}{2}}_T}+
 \norm{\Ztt_{(j)}}_{\Xc^{s+\frac{3}{2}}_T}+  \norm{\vartheta_{(j)}}_{\Xc^{s+\frac{3}{2}}_T} +\norm{\Utt_{(j)}}_{\Xc^{s+1}_T}+ \norm{\Vtt_{(j)}}_{X^0_T}\leq R  
\end{equation*}
for $j\in \Zbb_{\geq 0}$, by \eqref{locexistA13c},
it follows from the sequential Banach-Alaoglu Theorem that we can extract a subsequence from \eqref{locexistA13a}
that, for $q\in (1,\infty)$, converges weakly to a limit $(\Utt,\Vtt,\phi,\thetahu^0,\vartheta,\zetau,\Ztt)$ where
\begin{gather}
\Utt \in \bigcap_{\ell=0}^{2s+1} W^{\ell,q}\bigl([0,T],H^{s+1-\frac{m(s+1,\ell)}{2}}(\Omega_0,\Rbb^4)\bigr), \quad 
\Vtt \in L^q\bigl([0,T],L^2(\del{}\Omega_0,\Ran(\Pbb))\bigr), \label{locexistA13b.1} \\
\phi,\delb{0}\phi,\thetahu^0, \delb{0}\thetahu^0, \vartheta \in \bigcap_{\ell=0}^{2s+2} W^{\ell,q}\bigl([0,T],H^{s+\frac{3}{2}-\frac{m(s+\frac{3}{2},\ell)}{2}}(\Omega_0,\Rbb^4)\bigr),
 \label{locexistA13b.2}\\
\zetau,\delb{0}\zetau,\Ztt \in \bigcap_{\ell=0}^{2s+2} W^{\ell,q}\bigl([0,T],H^{s+\frac{3}{2}-\frac{m(s+\frac{3}{2},\ell)}{2}}(\Omega_0)\bigr),
 \label{locexistA13b.3}
\end{gather}
and the norm bounds
\begin{align}
\norm{\Utt}_{W^{\ell,q}\bigl([0,T],H^{s+1-\frac{m(s+1,\ell)}{2}}(\Omega_0,\Rbb^4)\bigr)} &\lesssim 1, \quad \ell=0,1,\ldots,2s+1,
\label{locexistA14.1} \\
\norm{\Vtt}_{L^q([0,T],L^2(\del{}\Omega_0,\Ran(\Pbb))} &\lesssim 1,\label{locexistA14.2}\\
\sum_{\ell=0}^1\norm{\delb{0}^\ell \phi}_{W^{\ell,q}\bigl([0,T],H^{s+\frac{3}{2}-\frac{m(s+\frac{3}{2},\ell)}{2}}(\Omega_0,\Rbb^4)\bigr)} &\lesssim 1, \quad \ell=0,1,\ldots,2s+2,
\label{locexistA14.3}\\
\sum_{\ell=0}^1\norm{\delb{0}^\ell\thetahu^0}_{W^{\ell,q}\bigl([0,T],H^{s+\frac{3}{2}-\frac{m(s+\frac{3}{2},\ell)}{2}}(\Omega_0,\Rbb^4)\bigr)} &\lesssim 1, \quad \ell=0,1,\ldots,2s+2, \label{locexistA14.4}\\
\norm{\vartheta}_{W^{\ell,q}\bigl([0,T],H^{s+\frac{3}{2}-\frac{m(s+\frac{3}{2},\ell)}{2}}(\Omega_0,\Rbb^4)\bigr)} &\lesssim 1, \quad \ell=0,1,\ldots,2s+2, \label{locexistA14.4a}\\
\sum_{\ell=0}^1\norm{\delb{0}^\ell\zetau}_{W^{\ell,q}\bigl([0,T],H^{s+\frac{3}{2}-\frac{m(s+\frac{3}{2},\ell)}{2}}(\Omega_0)\bigr)}& \lesssim 1, \quad \ell=0,1,\ldots,2s+2,
\label{locexistA14.5a}\\
\norm{\Ztt}_{W^{\ell,q}\bigl([0,T],H^{s+\frac{3}{2}-\frac{m(s+\frac{3}{2},\ell)}{2}}(\Omega_0)\bigr)}& \lesssim 1, \quad \ell=0,1,\ldots,2s+2,
\label{locexistA14.5}
\end{align}
hold uniformly for $q\in (1,\infty)$. 
Using the fact that $\sup_{0\leq t\leq T}|f(t)|= \lim_{q\rightarrow \infty}\bigl( \frac{1}{T}\int_0^T |f(\tau)|^q \, d\tau\bigr)^{\frac{1}{q}}$
holds for measurable functions $f(t)$ on $[0,T]$ satisfying $|f|^{q_*}\in L^1([0,T])$ for some $q_*\in \Rbb$, we deduce from
the norm bounds \eqref{locexistA14.1}-\eqref{locexistA14.5} that the limit satisfies
\begin{equation*} 
(\Utt,\Vtt,\phi,\thetahu^0,\vartheta,\zetau,\Ztt) \in \Yc^s_T,
\end{equation*}
where $\Yc^s_T$ is defined above by \eqref{YcsTdef}. Since the limit $(\Utt,\Vtt,\phi,\thetahu^0,\vartheta,\zetau,\Ztt)$ satisfies
\begin{equation*}
(\phi,\thetahu^0,\vartheta,\zetau)|_{\xb=0} = (\phi_f,\thetahu^0_f,\vartheta_f,\zetau_f),
\end{equation*}
we then have by Proposition \ref{stpropE} that
\begin{equation*} 
\norm{\zetau(\xb^0)-\zetau_f(0)}_{E^{s,0}}+\norm{\phi(\xb^0)-\phi_f(0)}_{E^{s+1,1}} + 
\norm{\thetahu(\xb^0)-\thetahu^0_f(0)}_{E^{s,0}}
+ \norm{\vartheta(\xb^0)-\vartheta_f(0)}_{E^{s,0}}\lesssim T<\delta.
\end{equation*}
for all $\xb^0\in [0,T]$ provided we choose $T>0$ small enough, which, as discussed above, implies that
\begin{equation} \label{locexistA14b}
(\zetau(\xb),\phi(\xb),\delb{}\phi(\xb),\thetahu^0(\xb),\vartheta(\xb))\in \Uc, \qquad \forall\; \xb\in [0,T]\times \overline{\Omega}_0,
\end{equation} 
and
\begin{equation}\label{locexistA14c}
|\psih(\xb)-\hat{\thetau}{}^3(\xb)| <\wc, \qquad \forall\; \xb\in [0,T]\times \del{}\Omega_0.
\end{equation}
Moreover, from the Rellich-Kondrachov Compactness Theorem, we can, for $\mu>0$, extract a subsequence
from \eqref{locexistA13a} that converges strongly in the following
spaces:
\begin{align}
\Utt_{(j)}\longrightarrow \Utt \quad &\text{in}\quad  \bigcap_{\ell=1}^{2s+1} W^{\ell-1,q}\bigl([0,T],H^{s+1-\frac{m(s+1,\ell)}{2}-\mu}(\Omega_0,\Rbb^4)\bigr),
\label{locexistA15.1} \\
\delb{0}^a \phi_{(j)}\longrightarrow \delb{0}^a\phi \quad &\text{in}\quad \bigcap_{\ell=1}^{2s+2} W^{\ell-1,q}\bigl([0,T],H^{s+\frac{3}{2}-\frac{m(s+\frac{3}{2},\ell)}{2}-\mu}(\Omega_0,\Rbb^4)\bigr), \qquad a=0,1,
\label{locexistA15.2}\\
\delb{0}^a \thetahu^0_{(j)}\longrightarrow \delb{0}^a \thetahu^0 \quad &\text{in}\quad \bigcap_{\ell=1}^{2s+2} 
W^{\ell-1,q}\bigl([0,T],H^{s+\frac{3}{2}-\frac{m(s+\frac{3}{2},\ell)}{2}-\mu}(\Omega_0,\Rbb^4)\bigr),\qquad a=0,1, \label{locexistA15.3}\\
\vartheta_{(j)}\longrightarrow \vartheta \quad &\text{in}\quad \bigcap_{\ell=1}^{2s+2} 
W^{\ell-1,q}\bigl([0,T],H^{s+\frac{3}{2}-\frac{m(s+\frac{3}{2},\ell)}{2}-\mu}(\Omega_0,\Rbb^4)\bigr)\label{locexistA15.3a}\\
\delb{0}^a\zetau_{(j)}\longrightarrow \delb{0}^a\zetau \quad &\text{in}\quad \bigcap_{\ell=1}^{2s+2} W^{\ell-1,q}\bigl([0,T],H^{s+\frac{3}{2}-\frac{m(s+\frac{3}{2},\ell)}{2}-\mu}(\Omega_0)\bigr), \qquad a=0,1,\label{locexistA15.4a}
\intertext{and}
\Ztt_{(j)}\longrightarrow \Ztt \quad &\text{in}\quad \bigcap_{\ell=1}^{2s+2} W^{\ell-1,q}\bigl([0,T],H^{s+\frac{3}{2}-\frac{m(s+\frac{3}{2},\ell)}{2}-\mu}(\Omega_0)\bigr), \label{locexistA15.4}
\end{align}
for any $q\in (1,\infty)$.

Next, we observe from the definition of the sequence \eqref{locexistA13a} and the map $\Jc_{\! T}$ that if we set
\begin{align}
(\Utta,\Vtta,\phia,\thetahua,\varthetaa,\zetaua,\Ztta) &= (\Utt_{(j)},\Vtt_{(j)},\phi_{(j)},\thetahu^0_{(j)},\vartheta_{(j)},\zetau_{(j)},\Ztt_{(j)}) \label{locexistA16.1}
\intertext{and}
(\Uttg,\Vttg,\phig,\thetahug,\varthetag,\zetaug,\Zttg) &=(\Utt_{(j-1)},\Vtt_{(j)},\phi_{(j-1)},\thetahu^0_{(j-1)},\vartheta_{(j-1)},\zetau_{(j-1)},\Ztt_{(j-1)}),
 \label{locexistA16.2}
\end{align}
then the pair \eqref{locexistA16.1}-\eqref{locexistA16.2} will solve the linear IBVP \eqref{llibvp.1}-\eqref{llibvp.15} and
the pair $(\delb{0}^{2s}\Utta,\Vtta)$ will
define a weak solution to the linear wave equation
that is obtained from differentiating \eqref{llibvp.1}-\eqref{llibvp.2} $2s$-times with respect to $\xb^0$. Using these facts, it is then
not difficult to verify from the weak convergence
of the sequence \eqref{locexistA13a} in the spaces \eqref{locexistA13b.1}-\eqref{locexistA13b.3},
the strong convergence in the spaces
\eqref{locexistA15.1}-\eqref{locexistA15.4}, and the calculus inequalities from Appendix \ref{calculus} that the limit $(\Utt,\Vtt,\phi,\thetahu^0,\vartheta,\zetau,\Ztt)$ determines a solution
of the IBVP \eqref{libvp.1d}-\eqref{libvp.14d} such that $\Utt|_{\xb^0=0}=\Utt_{2s+1}$ and the pair $(\delb{0}^{2s}\Utt,\Vtt)$
defines a weak solution to the linear wave equation
that is obtained from differentiating \eqref{libvp.1}-\eqref{libvp.2} $2s$-times with respect to $\xb^0$. 
Furthermore, the stated energy estimates for the solution  $(\Utt,\Vtt,\phi,\thetahu^0,\vartheta,\zetau,\Ztt)$ follow from
the same arguments that were used to derive the estimates  \eqref{locexistA11a}, \eqref{locexistA11b}, and \eqref{locexistA11da}-\eqref{locexistA11dc}, which were combined to yield the energy estimate \eqref{locexistA11e}.

\bigskip

\noindent\underline{Uniqueness:} 
Given two solutions 
\begin{align}
(\Utt_{(a)},\phi_{(a)},\thetahu^0_{(a)},\vartheta_{(a)},\zetau_{(a)},\Ztt_{(a)})\in &X^{s+1,2}_T(\Omega_0,\Rbb^4)
\times \mathring{X}^{s+\frac{3}{2},3}_T(\Omega_0,\Rbb^4) \times \mathring{X}^{s+\frac{3}{2},3}_T(\Omega_0,\Rbb^4)\notag \\
&
\times X^{s+\frac{3}{2},3}_T(\Omega_0,\Rbb^4)
\times \mathring{X}^{s+\frac{3}{2},3}_T(\Omega_0)\times X^{s+\frac{3}{2},3}_T(\Omega_0),\qquad a=1,2, \label{unique0}
\end{align}
to the IBVP \eqref{libvp.1d}-\eqref{libvp.14d} that are generated from the same initial data and satisfy the bounds \eqref{locexistA14b}-\eqref{locexistA14c}, we set
\begin{gather*}
\delta \Utt = \Utt_{(2)}-\Utt_{(1)}, \quad
\delta \phi = \phi_{(2)}- \phi_{(1)}, \quad
\delta \thetahu^0 = \thetahu^0_{(2)}-\thetahu^0_{(1)}, \\
 \delta \vartheta = \vartheta_{(2)}-\vartheta_{(1)}, \quad
\delta\zetau = \zetau_{(1)}-\zetau_{(1)} \AND \delta \Ztt = \Ztt_{(2)}-\Ztt_{(1)}.
\end{gather*}
More generally, for maps $f$ that depend on one of the solutions $(\Utt_{(a)},\phi_{(a)},\thetahu^0_{(a)},\vartheta_{(a)},\zetau_{(a)},\Ztt_{(a)})$,
we will employ the notation
\begin{equation*}
f_{(a)} = f\bigl(\Utt_{(a)},\phi_{(a)},\thetahu^0_{(a)},\vartheta_{(a)},\zetau_{(a)},\Ztt_{(a)}\bigr),
\end{equation*}
which, for example, gives 
\begin{equation*}
\Btt^{\alpha\beta}_{(a)} = \Btt^{\alpha\beta}\bigl(\sigmau,\thetab^I,\zetau_{(a)},\delb{}^{|1|}\phi_{(a)},\delb{}^{|1|}\thetahu^0_{(a)}\bigr).
\end{equation*}
Using this notation, we similarly define
\begin{equation*}
\delta f = f_{(2)}-f_{(1)}.
\end{equation*}

Now, multiplying \eqref{libvp.1d}-\eqref{libvp.2d} by $\frac{1-\epu}{\epu q}$
shows, with the help of \eqref{Qttdef}, that
\begin{align}
\delb{\alpha}\bigl(\Bttt^{\alpha\beta}\delb{\beta}\Utt+\Xttt^{\alpha}\bigr) &= \Fttt &&
\text{in $[0,T]\times \Omega_0$,} \label{libvp.1dd} \\
\thetab^3_\alpha \bigl(\Bttt^{\alpha\beta}\delb{\beta}\Utt+\Xttt^{\alpha}\bigr)  &=-\Pbb\delb{0}^2 \Utt + \Pttt \delb{0}\Utt + \Gttt
&&\text{in $[0,T]\times \del{}\Omega_0$,} \label{libvp.2dd}
\end{align}
where
\begin{align}
\Bttt^{\alpha\beta} &=  \frac{1-\epu}{\epu q}\Btt^{\alpha\beta}, \label{Btttdef}\\
\Xttt^{\alpha} &=  \frac{1-\epu}{\epu q}\Xtt^{\alpha}, \label{Xtttdef}\\
\Pttt &=  \frac{1-\epu}{\epu q}\Ptt, \label{Ptttdef}\\
\Gttt &=  \frac{1-\epu}{\epu q}\Gtt \label{Gtttdef}
\intertext{and}
\Fttt &=  \frac{1-\epu}{\epu q}\Ftt+\delb{\alpha}\biggl(\frac{1-\epu}{\epu q}\biggr)\bigl(\Btt^{\alpha\beta}\delb{\beta}\Utt+\Xtt^\alpha\bigr). \label{Ftttdef}
\end{align}
Since $\Utt_{(a)}$, $a=1,2$, both satisfy \eqref{libvp.1dd}-\eqref{libvp.2dd}, a short calulation shows that the difference $\delta \Utt$ 
 solves 
\begin{align}
\delb{\alpha}\bigl(\Bttt^{\alpha\beta}_{(2)}\delb{\beta}\delta\Utt+\delta\Bttt^{\alpha\beta}\delb{\beta}\Utt_{(1)}+\delta\Xttt^{\alpha}\bigr) &=\delta\Fttt  &&
\text{in $[0,T]\times \Omega_0$,} \label{unique6.1} \\
\thetab^3_\alpha \bigl(\Bttt^{\alpha\beta}_{(2)}\delb{\beta}\delta\Utt+\delta\Bttt^{\alpha\beta}\delb{\beta}\Utt_{(1)}+\delta\Xttt^{\alpha}\bigr)  
&=-\Pbb\delb{0}^2 \delta\Utt + \Pttt_{(2)}\delb{0}\delta\Utt + \delta\Pttt\delb{0}\Utt_{(1)}+\delta\Gttt
&& \text{in $[0,T]\times \del{}\Omega_0$,} \label{unique6.2}
\end{align}
where 
\begin{equation} \label{unique6a}
\delb{0}^{|2|}\delta\Utt = 0 \quad \text{in $\Omega_0$}.
\end{equation}
Noting that $\delta\Xttt^{\alpha}$
depends smoothly on the variables 
\begin{equation*}
\bigl(\sigmau,\thetab^I,\zetau_{(a)},\delb{0}^{|1|}\Ztt_{(a)},\Db^{|1|}\phi_{(a)},\Db^{|1|}\thetahu^0_{(a)},\Db^{|1|}\vartheta_{(a)},\Utt_{(a)} ,\delta\zetau,\delb{0}^{|1|}\delta\Ztt,\Db^{|1|}\delta\phi,\Db^{|1|}\delta\thetahu^0,\Db^{|1|}\delta\vartheta,\delta\Utt \bigr)
\end{equation*}
and is linear in
\begin{equation*}
\bigl(\delta\zetau,\delb{0}^{|1|}\delta\Ztt,\Db^{|1|}\delta\phi,\Db^{|1|}\delta\thetahu^0,\Db^{|1|}\delta\vartheta,\delta\Utt \bigr),
\end{equation*}
we  deduce from the calculus inequalities from Appendix \ref{calculus}  and
\eqref{unique0} that
\begin{equation} \label{unique7}
\norm{\delta \Xttt}_{H^{1}(\Omega_0)}+\norm{\delb{0}\delta \Xttt}_{E^{\frac{1}{2}}(\Omega_0)} \lesssim \Ntt
\end{equation}
where
\begin{equation*}
\Ntt = \sum_{\ell=0}^1\Bigl(\norm{\delb{0}^\ell \delta\zetau}_{\Ec^{2}} +  \norm{\delb{0}^\ell \delta\phi}_{\Ec^{2}} +  \norm{\delb{0}^\ell\delta\thetahu^0}_{\Ec^{2}}\Bigr)
+  \norm{\delta\Ztt}_{\Ec^{2}}+ \norm{\delta\vartheta}_{\Ec^{2}}
 +  \norm{\delta\Utt}_{\Ec^{\frac{3}{2}}}.
\end{equation*}
By similar arguments, it is also not difficult to verify that
\begin{equation} \label{unique8}
\norm{\delta\Bttt^{\alpha\beta} \delb{\beta} \Utt_{(1)}}_{H^{\frac{1}{2}}(\Omega_0)}+\norm{\delb{0}(\delta\Bttt^{\alpha\beta} \delb{\beta} \Utt_{(1)})}_{E^{\frac{1}{2}}} +
 \norm{\delta \Fttt}_{L^2(\Omega_0)}+\norm{\delb{0}\delta \Fttt}_{L^2(\Omega_0)} \lesssim \Ntt.
\end{equation}

Noting that  $\delta\Pttt\delb{0}\Utt_{(1)}+ \delta\Gttt$ depends smoothly on
\begin{equation*}
\bigl(\thetab^I,\Dbsl^{|1|}\!\phi_{(a)},\Dbsl^{|1|}\!\thetahu^0_{(a)},\Dbsl^{|1|}\!\vartheta_{(a)},\delb{0}^{|1|}\Utt_{(a)},
\Dbsl^{|1|}\!\delta\phi,\Dbsl^{|1|}\!\delta\thetahu^0,\Dbsl^{|1|}\!\delta\vartheta,
\delta\Utt\bigr)
\end{equation*}
and is linear in  
\begin{equation*}
\bigl(\Dbsl^{|1|}\!\delta\phi,\Dbsl^{|1|}\!\delta\thetahu^0,\Dbsl^{|1|}\!\delta\vartheta,\delta\Utt\bigr),
\end{equation*}
it is also not difficult to see that
\begin{equation} \label{unique9}
\delta\Pttt\delb{0}\Utt_{(1)}+ \delta\Gttt = K^\Sigma \del{\Sigma} \Xi + G,
\end{equation}
where
\begin{equation*}
\Xi =\bigl(\delta\phi,\delta\thetahu^0,\delta\vartheta\bigr),
\end{equation*}
 $K^\Sigma$ depends smoothly
\begin{equation*}
\bigl(\thetab^I,\Dbsl^{|1|}\!\phi_{(a)},\Dbsl^{|1|}\!\thetahu^0_{(a)},\Dbsl^{|1|}\!\vartheta_{(a)},\delb{0}^{|1|}\Utt_{(a)}\bigr)
\end{equation*}
and satisfies  $\thetab^3_\Sigma K^\Sigma =0$, and $\Gtt$ depends smoothly on
\begin{equation*}
\bigl(\thetab^I,\Dbsl^{|1|}\!\phi_{(a)},\Dbsl^{|1|}\!\thetahu^0_{(a)},\Dbsl^{|1|}\!\vartheta_{(a)},\delb{0}^{|1|}\Utt_{(a)},
\delta\phi,\delta\thetahu^0,\delta\vartheta,
\delta\Utt\bigr)
\end{equation*}
and is linear in  
\begin{equation*}
\bigl(\delta\phi,\delta\thetahu^0,\delta\vartheta,\delta\Utt\bigr).
\end{equation*}
From this, \eqref{unique0} and the calculus inequalities from Appendix \ref{calculus}, the estimates
\begin{gather} \label{unique10}
\norm{K}_{H^{s+\frac{1}{2}}(\Omega_0)}+\norm{\delb{0}K}_{H^{s}(\Omega_0)} \lesssim
1
\intertext{and}
 \norm{\delb{0}\Xi}_{H^{\frac{3}{2}}(\Omega_0)}+\norm{\delb{0}^2\Xi}_{H^1(\Omega_0)}+ \norm{ K^\Sigma \del{\Sigma} \Xi + G}_{H^{1}(\Omega_0)}+\norm{\delb{0}(K^\Sigma \del{\Sigma} \Xi + G)}_{H^{\frac{1}{2}}(\Omega_0)} \lesssim
\Ntt \label{unique10a}
\end{gather}
then follow.

By \eqref{Bttbnd1}, \eqref{BscBttsym},  \eqref{matlemB1} and \eqref{Btttdef}, it is clear that the matrices $\Bttt^{\alpha\beta}_{(2)}$ satisfies
\begin{equation}\label{unique10b}
(\Bttt^{\alpha\beta}_{(2)})^{\tr}=\Bttt^{\beta\alpha}_{(2)}
\end{equation}
and 
\begin{equation} \label{unique10c}
\Bttt^{00}_{(2)}\geq \tilde{\bc}_0 \id
\end{equation}
for some constant $\tilde{\bc}_0>0$. Furthermore, from \eqref{matlemB1} and  \eqref{unique0}, 
it is also not difficult, with the help of the calculus inequalities, to verify that $\Bttt_{(2)}^{\alpha\beta}$ and $\Pttt_{(2)}$ satisfy the estimates\footnote{Note that
$\Bttt^{\alpha\beta}$ and $\Pttt$ have the same functional dependence as $\Btt^{\alpha\beta}$ and $\Ptt$ given by \eqref{fdepA.9} and \eqref{fdepA.13}, respectively.}
\begin{equation} \label{unique10d}
\norm{\Bttt_{(2)}}_{E^{s+\frac{1}{2},2}}+\norm{\Pttt_{(2)}}_{H^{s+\frac{1}{2}}(\Omega_0)}+\norm{\delb{0}\Pttt_{(2)}}_{E^{s+\frac{1}{2},1}} \lesssim 1.
\end{equation}
Moreover, we observe that $\Btt_{(2)}^{\Lambda\Sigma}$ satisfies a coercive inequality by  Lemma \ref{Bsccoercive}. But then, by using similar arguments as in the proof
of Lemma \ref{Bsccoercive}, we can deduce from this coercive inequality and \eqref{matlemB1} that  $\Bttt^{\Lambda\Sigma}$ also satisfies a coercive inequality for an appropriate choice
of constants. Additionally, it is not difficult to see from the definition of $\Pttt_{(2)}$, see \eqref{Ptttdef} above, the inequality \eqref{matlemB1}, and Lemma \ref{matlemB} that  
$\Pttt - \tilde{\chi}\Pbb \leq 0$ for some constant $\tilde{\chi}$. From these two observations, the symmetry property \eqref{unique10b}, the decomposition  \eqref{unique9},
and the inequalities \eqref{unique7}-\eqref{unique8},  \eqref{unique10}-\eqref{unique10a}, \eqref{unique10c} and \eqref{unique10d}, we see that the system
\eqref{unique6.1}-\eqref{unique6a} satisfies all the necessary assumptions to apply Theorem \ref{LAWthm} for $\st=\frac{1}{2}$. Doing so, we conclude
from the initial condition \eqref{unique6a} that the solution $\delta U$ satisfies the energy estimate 
\begin{equation}\label{unique11}
\norm{\delta \Utt(\xb^0)}_{\Ec^{\frac{3}{2}}}+\norm{\Pbb\delb{0}^2\delta \Utt(\xb^0)}_{L^2(\del{}\Omega_0)} \lesssim \int_0^{\xb^0}
\norm{\Pbb\delb{0}^2\delta\Utt(\tau)}_{L^2(\del{}\Omega_0)}+ \Ntt(\tau) \, d\tau, \qquad 0\leq \xb^0 \leq T.
\end{equation}

Since $(\Utt_{(a)},\phi_{(a)},\thetahu^0_{(a)},\vartheta_{(a)},\zetau_{(a)},\Ztt_{(a)})$, $a=1,2$,  solves the IBVP \eqref{libvp.1d}-\eqref{libvp.14d}, we see, in particular, that
$\delta \phi$, $\delta\thetah^0$, $\delta\vartheta$ and $\delta\zetau$ satisfy 
\begin{align}
\delb{0} \delta\phi &= \Ltt_1 && \text{in $[0,T]\times \Omega_0$,} \label{unique11a.1}\\
\delb{0} \delta \thetahu^0 &= \Ltt_2 && \text{in $[0,T]\times \Omega_0$,}  \label{unique11a.2}\\
\delb{0}\delta \vartheta &= \Ltt_3 && \text{in $[0,T]\times \Omega_0$,}  \label{unique11a.3}
\intertext{and}
\delb{0}\delta \zetau &= \delta\Ztt && \text{in $[0,T]\times \Omega_0$,}  \label{unique11a.4}
\end{align}
respectively,
where $\Ltt_1$, $\Ltt_2$ and $\Ltt_3$  depend smoothly on 
\begin{equation*}
(\phi_{(a)},\thetahu^0_{(a)},\delta\phi,\delta \thetahu^0), \quad (\phi_{(a)},\thetahu^0_{(a)},\vartheta_{(a)},\delta\phi,\delta \thetahu^0,\delta\vartheta) \AND 
(\phi_{(a)},\thetahu^0_{(a)},\vartheta_{(a)},\Utt_{(a)},\delta\phi,\delta \thetahu^0,\delta\psi,\delta \Utt), 
\end{equation*}
and  are linear in
\begin{equation*}
(\delta\phi,\delta \thetahu^0), \quad (\delta\phi,\delta \thetahu^0,\delta\vartheta) \AND (\delta\phi,\delta \thetahu^0,\delta\psi,\delta \Utt), 
\end{equation*}
respectively. Using this we can, with the help of the calculus inequalities from Appendix \ref{calculus} and  \eqref{unique0}, bound the derivatives $\delb{0} \delta\phi$, $\delb{0} \delta \thetah^0$,
$\delb{0}\delta \vartheta$ and $\delb{0}\delta \zetau$   by
\begin{align}
\norm{\delb{0}\delta\phi}_{\Ec^2} &\lesssim \norm{\delta\phi}_{\Ec^2}+\norm{\delta\thetahu^0}_{\Ec^2}, \label{unique11b.1} \\
\norm{\delb{0}\delta\thetahu^0}_{\Ec^{2}} &\lesssim \norm{\delta\phi}_{\Ec^{2}}+\norm{\delta\thetahu^0}_{\Ec^{2}}+ \norm{\delta\vartheta}_{\Ec^{2}},  \label{unique11b.2} \\
\norm{\delb{0}\delta\vartheta}_{\Ec^{\frac{3}{2}}} &\lesssim \norm{\delta\phi}_{\Ec^{\frac{3}{2}}}+\norm{\delta\thetahu^0}_{\Ec^{\frac{3}{2}}}+ \norm{\delta\vartheta}_{\Ec^{\frac{3}{2}}}+
\norm{\delta\Utt}_{\Ec^{\frac{3}{2}}} \label{unique11b.3}
\intertext{and}
\norm{\delb{0}\delta\zetau}_{\Ec^2} &\lesssim \norm{\delta\Ztt}_{\Ec^2}, \label{unique11b.4}
\end{align}
respectively. Moreover, integrating \eqref{unique11a.1}-\eqref{unique11a.4} in time, we obtain from an application of Proposition \ref{stpropE} and
the vanishing of $\delta \phi$, $\delta\thetah^0$, $\delta\vartheta$ and $\delta\zetau$ and their time derivatives at $\xb^0=0$ that
\begin{equation}\label{unique11c}
 \norm{\delta\zetau(\xb^0)}_{\Ec^2}+ \norm{\delta\phi(\xb^0)}_{\Ec^2}+\norm{\delta\thetahu^0(\xb^0)}_{\Ec^{2}}+ \norm{\delta\vartheta(\xb^0)}_{\Ec^{\frac{3}{2}}} \lesssim \int_0^{\xb^0} \Ntt(\tau)\, d\tau.
\end{equation}


Next, we can write \eqref{unique11a.3} more explicitly as
\begin{equation*}
\delb{0}\delta\vartheta =\lambda_{(2)}\Theta_{(2)} \delta\Utt -\Lttt_3
\end{equation*}
where $\Lttt_3$ depends smoothly on
\begin{equation*}
(\phi_{(a)},\thetahu^0_{(a)},\vartheta_{(a)},\Utt_{(a)},\delta\phi,\delta\thetahu^0,\delta\vartheta)
\end{equation*}
and is linear in
\begin{equation*}
(\delta\phi,\delta\thetahu^0,\delta\vartheta).
\end{equation*}
Using this to replace $\delta \Utt$ with $\lambda_{(2)}^{-1}E_{(2)}(\delb{0}\delta\vartheta+\Lttt_3)$ in \eqref{unique6.1}-\eqref{unique6.2}, we find after a short calculation that
\begin{align}
\delb{\Lambda}\bigl(\delb{0}(\lambda_{(2)}^{-1}\Bttt^{\Lambda\Sigma}_{(2)}E_{(2)}\delb{\Sigma}\delta\vartheta)+ \Xttbr^{\Lambda}\bigr) 
&= \Fttbr && 
\text{in $[0,T]\times \Omega_0$,} \label{unique17.1} \\
\thetab^3_\Lambda\bigl(\delb{0}(\lambda_{(2)}^{-1}\Bttt^{\Lambda\Sigma}_{(2)}E_{(2)}\delb{\Sigma}\delta\vartheta)+ \Xttbr^{\Lambda}\bigr)  &=-\delb{0}(\Pbb\delb{0}\delta\Utt)
 + \Gttbr
&& \text{in $[0,T]\times \del{}\Omega_0$,} \label{unique17.2}
\end{align}
where 
\begin{align*}
\Xttbr^{\Lambda} &= -\delb{0}\bigl(\lambda_{(2)}^{-1}\Bttt^{\Lambda\Sigma}_{(2)}E_{(2)}\bigr)\delb{\Sigma}\delta \vartheta +
\lambda_{(2)}^{-1}\Bttt^{\Lambda\Sigma}_{(2)}E_{(2)}\delb{\Sigma}\Lttt_3 \\
&\quad +\Bttt^{\Lambda\Sigma}_{(2)}\delb{\Sigma}\bigl(\lambda_{(2)}^{-1}E_{(2)}\bigr)
\lambda_{(2)}\Theta_{(2)}\delta\Utt+ \Bttt^{\Lambda 0}_{(2)}\delb{0}\delta \Utt + \delta \Btt^{\Lambda\beta}\delb{\beta}\Utt_{(1)}
+\delta\Xttt^{\Lambda}, \\
\Fttbr&=-\delb{0}\bigl(\Bttt^{0\beta}_{(2)}\delb{\beta}\delta\Utt + \delta \Bttt^{0\beta}\delb{\beta}U_{(1)}+\delta\Xttt^0\bigr)+ \delta\Fttt 
\intertext{and}
\Gttbr &=\Pttt_{(2)}\delb{0}\delta\Utt + \delta\Pttt\delb{0}\Utt_{(1)}+\delta\Gttt.
\end{align*}
By integrating \eqref{unique17.1} and \eqref{unique17.2} with respect to $\xb^0$ and then multiplying the results on the left by $E^{\tr}$, we get, with the help of \eqref{Thetadef2}, \eqref{Bttdef} and \eqref{Btttdef}, that
\begin{align}
\delb{\Lambda}\biggl(\frac{1-\epu_{(2)}(\xb^0)}{\epu_{(2)}(\xb^0)q_{(2)}(\xb^0)}\Bsc^{\Lambda\Sigma}_{(2)}(\xb^0)\delb{\Sigma}\delta\vartheta(\xb^0)+ &E^{\tr}_{(2)}(\xb^0)
\int_0^{\xb^0} \Xttbr^{\Lambda}(\tau)\, d\tau \biggr) =  E^{\tr}(\xb^0)\int_0^{\xb^0} \Fttbr(\tau)\, d\tau\notag \\
&+
\lambda_{(2)}^{-1}(\xb^0)\delb{\Lambda}E^{\tr}_{(2)}(\xb^0)\Bttt^{\Lambda\Sigma}_{(2)}(\xb^0) E_{(2)}(\xb^0)\delb{\Sigma}\delta\vartheta (\xb^0) 
 &&
\text{in $[0,T]\times \Omega_0$,} \label{unique19.1} \\
\thetab^3_\Lambda\biggl(\frac{1-\epu_{(2)}(\xb^0)}{\epu_{(2)}(\xb^0)q_{(2)}(\xb^0)}\Bsc^{\Lambda\Sigma}_{(2)}(\xb^0)\delb{\Sigma}\delta\vartheta(\xb^0)+ &E^{\tr}_{(2)}(\xb^0)\int_0^{\xb^0} \Xttbr^{\Lambda}(\tau)\, d\tau \biggr) \notag\\
&=-E^{\tr}_{(2)}(\xb^0)\Pbb\delb{0}\delta \Utt(\xb^0) + 
E^{\tr}_{(2)}(\xb^0)\int_0^{\xb^0}\Gttbr(\tau)\, d\tau
&& \text{in $[0,T]\times \del{}\Omega_0$.} \label{unique19.2}
\end{align}
Using similar arguments as above, we find that the coefficients of the Neumann BVP \eqref{unique19.1}-\eqref{unique19.1} can be estimated as follows:
\begin{align}
\biggl\|\frac{1-\epu_{(2)}(\xb^0)}{\epu_{(2)}(\xb^0)q_{(2)}(\xb^0)}\Bsc^{\Lambda\Sigma}_{(2)}(\xb^0)\bigg\|_{H^{s}(\Omega_0)} &\lesssim 1, \label{unique20.1}\\
\biggl\|E^{\tr}_{(2)}(\xb^0)\int_0^{\xb^0} \Xttbr^{\Lambda}(\tau)\, d\tau \biggr\|_{H^1(\Omega_0)}&\lesssim \int_0^{\xb^0} \Ntt(\tau)\, d\tau, \label{unique20.2}\\ 
\biggl\|E^{\tr}_{(2)}(\xb^0)\int_0^{\xb^0} \Gttbr(\tau)\, d\tau \biggr\|_{H^1(\Omega_0)}&\lesssim \int_0^{\xb^0} \Ntt(\tau)\, d\tau, \label{unique20.3}\\ 
\biggl\|E^{\tr}_{(2)}(\xb^0)\int_0^{\xb^0} \Fttbr(\tau)\, d\tau \biggr\|_{L^2(\Omega_0)}&\lesssim \int_0^{\xb^0} \Ntt(\tau)\, d\tau, \label{unique20.4}\\
\bigl\|\lambda_{(2)}^{-1}(\xb^0)\delb{\Lambda}E^{\tr}_{(2)}(\xb^0)\Bttt^{\Lambda\Sigma}_{(2)}(\xb^0) E_{(2)}(\xb^0)\delb{\Sigma}\delta\vartheta (\xb^0) 
\bigr\|_{L^2(\Omega_0)} &\lesssim \norm{\delta\vartheta(\xb^0)}_{H^1(\Omega_0)} \label{unique20.5}
\intertext{and}
\norm{E^{\tr}_{(2)}(\xb^0)\Pbb\delb{0}\delta \Utt(\xb^0)}_{H^1(\Omega_0)}&\lesssim \norm{\delta\Utt(\xb^0)}_{\Ec^{\frac{3}{2}}}.
 \label{unique20.6}
\end{align}
We further recall that $\Bsc_{(2)}^{\Lambda\Sigma}$ satisfies a coercive inequality by Lemma \ref{Bsccoercive}, and from this it follows by \eqref{matlemB1} and similar arguments as in the proof
of Lemma \ref{Bsccoercive} that  $\frac{1-\epu_{(2)}}{\epu_{(2)}q_{(2)}}\Bsc^{\Lambda\Sigma}_{(2)}$ also satisfies a coercive inequality for an appropriate choice
of constants. Using this fact and the inequalities \eqref{unique20.1}-\eqref{unique20.6}, we can, since $\delta\vartheta$ satisfies the Neumann BVP  \eqref{unique19.1}-\eqref{unique19.1},
appeal to elliptic regularity (Theorem \ref{ellipticregN}) to conclude that
\begin{equation} \label{unique21} 
\norm{\delta\vartheta(\xb^0)}_{H^2(\Omega_0)} \lesssim \norm{\delta\vartheta(\xb^0)}_{H^1(\Omega_0)} + \norm{\delta\Utt(\xb^0)}_{\Ec^{\frac{3}{2}}} 
+ \int_0^{\xb^0} \Ntt(\tau)\, d\tau .
\end{equation}

Using that fact that $\Ztt_{(a)}$, $a=1,2$, both solve \eqref{libvp.6d}-\eqref{libvp.7d}, a short calculation shows that $\delta\Ztt$ satisfies
the scalar wave equation with Dirichlet boundary condition given by
\begin{align*}
\delb{\alpha}\bigl(\Msc^{\alpha\beta}_{(2)}\delb{\beta}\delta\Ztt + \delta\Msc^{\alpha\beta}\delb{\beta}\Ztt_{(1)}+\delta\Ytt\bigr) &=\delta \Ktt && \text{in $[0,T]\times \Omega_0$,}\\
\delta\Ztt &=0,  && \text{in $[0,T]\times \del{}\Omega_0$,}
\end{align*}
where
\begin{equation*}
\delb{0}^{|3|}\delta\Ztt= 0 \quad \text{in $\{0\}\times \Omega_0$.}
\end{equation*}
Furthermore, from \eqref{unique0}, Lemma \ref{coeflemA},  and the calculus inequalities, we have
\begin{gather*}
\norm{\Msc^{\alpha\beta}_{(2)}}_{E^{s+\frac{1}{2},3}}\lesssim 1
\intertext{and}
\norm{\delta\Msc^{\alpha\beta}\delb{\beta}\Ztt_{(1)}}_{H^1(\Omega_0)}+\norm{\delb{0}(\delta\Msc^{\alpha\beta}\delb{\beta}\Ztt_{(1)})}_{E^1(\Omega_0)}
+\sum_{\ell=0}^2\norm{\delb{0}^\ell \delta\Ktt}_{L^2(\Omega_0)} \lesssim \Ntt.
\end{gather*}
Since $\Msc^{\alpha\beta}_{2}$ satisfies the coercive inequality from Lemma \ref{Bsccoercive} along with the bound \eqref{Mscbnd1} and the
symmetry property \eqref{Mscsym}, we conclude from Theorem \ref{DCthm} and the above inequalities that $\delta\Ztt$ satisfies the energy estimate
\begin{equation}  \label{unique22}
\norm{\delta \Ztt(\xb^0)}_{\Ec^{2}} \lesssim \int_0^{\xb^0}\Ntt(\tau)\, d\tau.
\end{equation}

Taken together,  the estimates \eqref{unique11b.1}-\eqref{unique11c} imply that 
\begin{equation*}
\Ntt(\xb^0) \lesssim
\norm{\delta\Ztt(\xb^0)}_{\Ec^2}+\norm{\delta\vartheta(\xb^0)}_{H^2(\Omega_0)}+\norm{\delta\Utt(\xb^0)}_{\Ec^{\frac{3}{2}}}
+\int_0^{\xb^0}\Ntt(\tau)\, d\tau.
\end{equation*}
Adding $\norm{\Pbb\delb{0}^2\delta\Utt(\xb^0)}_{L^2(\del{}\Omega_0)}$ to both sides of this inequality, we then have
\begin{equation*}
\Ntt(\xb^0)+\norm{\Pbb\delb{0}^2\delta\Utt(\xb^0)}_{L^2(\del{}\Omega_0)} \lesssim
\int_0^{\xb^0}\Ntt(\tau)+\norm{\Pbb\delb{0}^2\delta\Utt(\tau)}_{L^2(\del{}\Omega_0)}\, d\tau,
\end{equation*}
by \eqref{unique11} and \eqref{unique21}-\eqref{unique22}. Thus
$\Ntt(\xb^0)+ \norm{\Pbb\delb{0}^2\delta\Utt(\xb^0)}_{L^2(\del{}\Omega_0)}= 0$, for $0\leq \xb^0 \leq T$,
by Gronwall's inequality, and consequently
\begin{equation*}
 (\Utt_{(1)},\phi_{(1)},\thetahu^0_{(1)},\vartheta_{(1)},\zetau_{(1)},\Ztt_{(1)})=(\Utt_{(2)},\phi_{(2)},\thetahu^0_{(2)},\vartheta_{(2)},\zetau_{(2)},\Ztt_{(2)})
 \quad \text{in $[0,T]\times \Omega_0$,}
\end{equation*}
which establishes uniqueness.
\end{proof}

\begin{rem} \label{spacetimerem} Although the proof of Theorem \ref{locexistA} was carried out under the assumption that the spacetime dimension is the physical one, i.e. $d=4$, it not difficult to verify that
the proof of this theorem continues to hold in all spacetime dimensions $d\geq 3$. 
\end{rem}

\section{Local-in-time existence for the relativistic Euler equations\label{locexistREVBC}}

In Lagrangian coordinates, the relativistic Euler equations with vacuum boundary conditions, see \eqref{ibvp.1}-\eqref{ibvp.4} and \eqref{ful}, 
are given by 
\begin{align}
\vu^\mu\delb{\mu}\rhou + (\rhou+\pu)\bigl(\delb{\mu}\vu^\mu+\Gammau_{\mu\nu}^\mu \vu^\mu\bigr)&=0
&&\text{in $[0,T]\times \Omega_0$,} \label{REVBC.1}\\
(\rhou+\pu)\bigl(\vu^\mu\delb{\mu}\vu^\nu+\vu^\mu\Gammau_{\mu\lambda}^\nu \vu^\lambda\bigr)
+\underline{h}^{\mu\nu}\delb{\mu}\pu &= 0 &&\text{in $[0,T]\times \Omega_0$,} \label{REVBC.2}\\
\delb{0}\phi^\mu &= \vu^\mu &&\text{in $[0,T]\times \Omega_0$,} \label{REVBC.3}\\
\pu &= 0 &&\text{in $[0,T]\times \del{}\Omega_0$.} \label{REVBC.4}
\end{align}
The local-in-time existence of solutions to this system is then an immediate consequence of Proposition \ref{cprop}, Remark
\ref{tdiffIBVPrem}, and
Theorem \ref{locexistA}. We formalize this statement in the following theorem.

\begin{thm} \label{locexistB}
In addition to the assumptions of Theorem \ref{locexistA}, assume that the initial data chosen so that
the constraints \eqref{aconiv}-\eqref{bcicons} and  \eqref{libvp.1idata}-\eqref{libvp.4idata} are satisfied initially and let 
$(\Utt,\Vtt,\phi,\thetahu^0,\vartheta,\zetau,\Ztt)\in \Yc^s_T$
be the map from Theorem \ref{locexistA}. Then the triple 
\begin{equation*}
(\phi,\underline{v}^\mu,\underline{\rho}) \in \Xcr^{s+\frac{3}{2}}_T(\Omega_0,\Rbb^4) \times \Xcr^{s+\frac{3}{2}}_T(\Omega_0,\Rbb^4)
\times \Xcr^{s+\frac{3}{2}}_T(\Omega_0) 
\end{equation*}
determined from the map $(\Utt,\Vtt,\phi,\thetahu^0,\vartheta,\zetau,\Ztt)$ via the formulas \eqref{theta2rho}-\eqref{theta2v} and \eqref{ful}
satisfies the Lagrangian representation of relativistic Euler equations with vacuum boundary conditions given by
\eqref{REVBC.1}-\eqref{REVBC.4}. Moreover,
there exists a constant $c_p>0$
such that the Taylor sign condition $-\bar{n}{}^\mu\delb{\mu}\underline{p} \geq c_p > 0$ 
holds in $[0,T]\times \del{}\Omega_0$.
\end{thm}

\begin{rem} \label{REVBCrem}
$\;$
\begin{enumerate}[(i)]
\item
From \cite{Oliynyk:Bull_2017}, we know, for $s>n/2+1/2$, that every solution
\begin{equation*} 
(\phi,\underline{v}^\mu,\underline{\rho})  \in \Xcr^{s+\frac{3}{2}}_T(\Omega_0,\Rbb^4) \times \Xcr^{s+\frac{3}{2}}_T(\Omega_0,\Rbb^4)
\times \Xcr^{s+\frac{3}{2}}_T(\Omega_0) 
\end{equation*}
of \eqref{REVBC.1}-\eqref{REVBC.4} determines, via the formulas  $\thetahu^0_\mu =-\vu_\mu$, 
\eqref{theta2rho}, \eqref{ful}, \eqref{sigmatev} and \eqref{varthetadef},
a solution  
\begin{equation} \label{REVBCrem1}
(\phi^\mu,\thetahu^0_\mu,\vartheta_\mu,\zetau)\in \Xcr^{s+\frac{3}{2}}_T(\Omega_0,\Rbb^4)
\times \Xcr^{s+\frac{3}{2}}_T(\Omega_0,\Rbb^4) \times \Xc^{s+\frac{3}{2}}_T(\Omega_0,\Rbb^4)
\times \Xcr^{s+\frac{3}{2}}_T(\Omega_0) 
\end{equation}
of \eqref{libvp.1}-\eqref{libvp.6}  provided that
the time-independent  spatial coframe fields $(\thetab^I_\mu)\in H^{s+\frac{1}{2}}(\Omega_0,\Rbb^4)$  and the time-independent functions $\sigmau_i{}^k{}_j\in H^{s+\frac{1}{2}}(\Omega_0)$ are chosen so that
the constraints \eqref{acon}-\eqref{fcon}, \eqref{hcon} and \eqref{kcon} vanish on the initial hypersurface. Moreover, we
know from the computations carried out in Section \ref{tdiffIBVP} that the solution \eqref{REVBCrem1} determines by \eqref{Psidef}, \eqref{Psi2U} and \eqref{Zttdef} a solution
\begin{align*}
(\Utt,\phi,\thetahu^0,\vartheta,\zetau,\Ztt)\in &\Xc^{s+1}_T(\Omega_0,\Rbb^4)
\times \Xcr^{s+\frac{3}{2}}_T(\Omega_0,\Rbb^4) \times \Xcr^{s+\frac{3}{2}}_T(\Omega_0,\Rbb^4) \\
&\times \Xc^{s+\frac{3}{2}}_T(\Omega_0,\Rbb^4)
\times \Xcr^{s+\frac{3}{2}}_T(\Omega_0)\times \Xc^{s+\frac{3}{2}}_T(\Omega_0)
\end{align*}
of \eqref{libvp.1d}-\eqref{libvp.8d}.
Consequently, we can infer the uniqueness of solutions to \eqref{REVBC.1}-\eqref{REVBC.4} for initial data satisfying suitable compatibility conditions from  the uniqueness statement from Theorem \ref{locexistA}.
\item[(ii)] By Remark \ref{spacetimerem} and the proof of Theorem \ref{locexistB}, it is not difficult to verify that Theorem \ref{locexistB} is valid not only for the physical spacetime dimension $n=4$, but for
all spacetime dimensions $n\geq 3$.
\end{enumerate}
\end{rem}

\bigskip

\noindent \textit{Acknowledgements:}
Part of the work for this article was completed at the Institut Henri Poincar\'{e}
during the Mathematical Relativity trimester in 2015 and while visiting the Department of Mathematics at Princeton University during the Fall semester of 2017. I thank both institutions for their support and hospitality.
This work was partially supported by the Australian Research Council grant DP170100630 and a Fulbright Senior Scholarship.

\appendix

\section{Differential geometry formulas\label{DG}}

In this appendix, we collect together some useful formulas from differential geometry that will be used throughout this article. In the following, we let
\begin{equation*} 
g = g_{\mu\nu}dx^\mu dx^\nu
\end{equation*}
denote a smooth Lorentzian metric on a four dimensional manifold $M$, and we use $\nabla$ to denote the Levi-Civita connection of this metric.
We also use the indexing conventions from Section \ref{indexing}. Given a local frame
\begin{equation}\label{DGframe}
e_j = e_j^\mu \del{\mu}
\end{equation}
on $M$, we denote the dual frame by
\begin{equation} \label{DGcoframeA}
\theta^j = \theta^j_\mu dx^{\mu} \qquad \bigl( (\theta^j_\mu):= (e^\mu_j)^{-1}\bigl),
\end{equation}
and use the notation
\begin{equation} \label{DGframemet}
\gamma_{ij} = g(e_i,e_j)=g_{\mu\nu}e^{\mu}_i e^{\nu}_j \AND \gamma^{ij} = g(\theta^i,\theta^j)=g^{\mu\nu}\theta^i_\mu \theta^j_\nu
\end{equation}
for the frame metric and its inverse, respectively.

\subsection{Lie and exterior derivatives}
Given vector fields $X,Y$, a scalar field $f$, a $q$-form $\alpha$, and a $p$-form $\beta$, the following identities hold:
\begin{align}
\Ld_X \alpha &= \Ip_X \Ed\alpha + \Ed\Ip_X \alpha, \label{DGf.1} \\
\Ed\Ld_X \alpha &= \Ld_X \Ed\alpha, \label{DGf.2} \\
\Ip_Y \Ld_X \alpha & = \Ld_X \Ip_Y \alpha + \Ip_{[Y,X]}\alpha, \label{DGf.3} \\
\Ld_X(\alpha\wedge\beta) &= (\Ld_X\alpha)\wedge \beta + \alpha \wedge \Ld_X\beta \label{DGf.4}
\intertext{and}
\Ld_{fX} \alpha &= f \Ld_X \alpha + \Ed f \wedge \Ip_X \alpha, \label{DGf.5}
\end{align}
where here we are using $\Ld_X$, $\Ed$, and $\Ip_X$ to denote the Lie derivative, exterior derivative, and interior product, respectively. 
Expressing $\alpha$ locally as
\begin{equation*} 
\alpha = \frac{1}{q!}\alpha_{\mu_1\mu_2 \ldots \mu_q} dx^{\mu_1}\wedge dx^{\mu_2} \wedge \cdots \wedge dx^{\mu_q},
\end{equation*}
the exterior derivative of $\alpha$ can be computed using the formula
\begin{equation*}
(\Ed\alpha)_{\mu_1\mu_2\ldots\mu_{q+1}} = (q+1)\del{[\mu_1}\alpha_{\mu_2\mu_3\ldots\mu_{q+1}]} = (q+1)\nabla_{[\mu_1}\alpha_{\mu_2\mu_3\ldots\mu_{q+1}]}.
\end{equation*}
Furthermore, given the local expression
\begin{equation*}
X=X^\mu\del{\mu},
\end{equation*}
the Lie derivative $L_X$ of $\alpha$ can be computed using
\begin{equation} \label{DGLie}
(\Ld_X \alpha)_{\mu_1\mu_2 \ldots \mu_q} = X^\nu\del{\nu}\alpha_{\mu_1\mu_2 \ldots \mu_q} + \alpha_{\nu\mu_2 \ldots \mu_q}\del{\mu_1}X^{\nu}
+ \alpha_{\mu_1\nu\mu_3 \ldots \mu_q}\del{\mu_2}X^{\nu} + \cdots + \alpha_{\mu_1 \mu_2 \ldots \mu_{q-1}\nu}\del{\mu_q}X^{\nu}.
\end{equation}
For functions, we employ the alternate notation
\begin{equation*} 
X(f) = \Ld_X(f) = X^{\mu}\del{\mu}f
\end{equation*}
for the Lie derivative, and more generally, we use this notation locally on coordinate components of tensors, e.g. $X(Y^\nu)=X^{\mu}\del{\mu}Y^{\nu}$.

\subsection{Volume form}
We use $\nu$ to denote the volume form of the metric $g$, which is given locally by
\begin{equation*}
\nu = \frac{1}{4!}\nu_{\mu\alpha\beta\gamma} dx^{\mu}\wedge dx^\alpha \wedge dx^\beta \wedge dx^{\gamma},
\end{equation*}
where the components are computed using
\begin{equation*}
\nu_{\mu\alpha\beta\gamma} = \sqrt{|g|}\epsilon_{\mu\alpha\beta\gamma}.
\end{equation*}
Here, $\epsilon_{\nu\alpha\beta\gamma}$ denotes the completely anti-symmetric symbol and we employ the standard notation
\begin{equation*}
|g|=-\det(g_{\mu\nu})
\end{equation*}
for the negative of the determinant of the metric $g_{\mu\nu}$.
The volume form is also given locally in terms of the coframe by
\begin{equation} \label{volumeA}
\nu = \sqrt{|\gamma|}\theta^0\wedge\theta^1\wedge\theta^2\wedge\theta^3,
\end{equation}
where
\begin{equation*} \label{detgammadef}
|\gamma| = -\det(\gamma_{kl}).
\end{equation*}

\subsection{Hodge star operator}
The Hodge star operator $*_g$ associated to $g$ satisfies
\begin{equation} \label{hodge}
\alpha \wedge *_g \beta = g(\alpha,\beta)\nu
\end{equation}
for any $1$-forms $\alpha$, $\beta$, where $\nu$, as above, is the volume form of $g$. From \eqref{volumeA},
\eqref{hodge} and the identity
\begin{equation*} 
*_g *_g\alpha = (-1)^{q(4-q)+1}\alpha
\end{equation*}
for $q$-forms, we obtain
\begin{equation*} 
g(\theta^j,*(\theta^1\wedge\theta^2\wedge\theta^3))\nu_g
= \theta^j\wedge\theta^1\wedge\theta^2\wedge\theta^3,
\end{equation*}
from which it follows that
\begin{equation} \label{hodgeperpB}
g(\theta^I,*(\theta^1\wedge\theta^2\wedge\theta^3)) = 0, \quad I \in \{1,2,3\}.
\end{equation}

\subsection{Codifferential}
The codifferential $\delta_g$ associated to $g$ is given by the formula
\begin{equation} \label{DGcodiffA}
\delta_g \alpha = (-1)^{4(q-1)}*_g \Ed *_g \alpha
\end{equation}
when acting on $q$-forms,
and it can be computed locally via the formula
\begin{equation} \label{DGcodiffC}
(\delta_g \alpha)_{\mu_2 \ldots \mu_q}  = -\nabla^{\mu_1}\alpha_{\mu_1\mu_2 \ldots \mu_q}  = - \frac{1}{\sqrt{|g|}} \del{\nu}\bigl(\sqrt{|g|} g^{\nu \mu_1}
\alpha_{\mu_1\mu_2 \ldots \mu_q}\bigr).
\end{equation}
The codifferential satisfies the identity
\begin{equation} \label{DGcodiffB}
\delta_g^2 = 0.
\end{equation}

\subsection{Cartan structure equations}
The connection coefficients $\omega_i{}^k{}_j$ of the metric  $g$  with respect to the frame \eqref{DGframe} are defined by
\begin{equation*} \label{DGconndef}
\nabla_{e_i} e_j = \omega_i{}^k{}_j e_k,
\end{equation*}
or equivalently
\begin{equation} \label{DGconndefA}
\nabla_{e_i}\theta^k = -\omega_{i}{}^k{}_{j}\theta^j.
\end{equation}
Defining the connection one forms $\omega^k{}_j$ by
\begin{equation*}
\omega^k{}_j = \omega_{i}{}^k{}_j\theta^i,
\end{equation*}
the \textit{Cartan structure equations} are given by
\begin{align}
\Ed\theta^i &= -\omega^i{}_j\wedge \theta^j, \label{CartanA} \\
\Ed\gamma_{ij} &= \omega_{ij}+\omega_{ji}, \label{CartanB}
\end{align}
where
\begin{equation*} 
\omega_{ij} = \gamma_{ik}\omega^k{}_j.
\end{equation*}

\subsection{Curvature}
We use
\begin{equation} \label{curvature}
\nabla_\mu\nabla_\nu \lambda_\gamma - \nabla_\nu\nabla_\mu\lambda_\gamma
= R_{\mu\nu\gamma}{}^\sigma \lambda_\sigma,
\end{equation}
to define the curvature tensor $R_{\mu\nu\gamma}{}^\sigma$ of the metric $g$,
and we define the Ricci tensor $R_{\mu\gamma}$ by
\begin{equation*}
R_{\mu\gamma} = R_{\mu\nu\gamma}{}^\nu.
\end{equation*}

\subsection{Covariant derivatives and changes of metrics}
Letting $\hat{\nabla}$ denote the Levi-Civita connection of another Lorentian metric
\begin{equation*} 
\gh = \gh_{\mu\nu}dx^\mu dx^\nu
\end{equation*}
on M, the covariant derivatives with respect to the metrics $g$ and $\gh$
are related via the formula
\begin{equation} \label{DGconnectA}
\hat{\nabla}_\gamma T^{\mu_1\ldots \mu_r}_{\nu_1\ldots \nu_s}
= \nabla_\gamma T^{\mu_1\ldots \mu_r}_{\nu_1\ldots \nu_s}
+ C_{\gamma\lambda}^{\mu_1}T^{\lambda \mu_2\ldots \mu_r}_{\nu_1\ldots \nu_s}+ \cdots +
C_{\gamma\lambda}^{\mu_r}T^{\mu_1\ldots \mu_{r-1} \lambda}_{\nu_1\ldots \nu_s}
-C_{\gamma\nu_1}^{\lambda}T^{\mu_1\ldots \mu_r}_{\lambda\nu_2\ldots \nu_s}- \cdots -
C_{\gamma\nu_s}^{\lambda}T^{\mu_1\ldots \mu_{r}}_{\nu_1\ldots \nu_{s-1}\lambda}
\end{equation}
where
\begin{equation*}
C_{\alpha\beta}^\lambda = \Half \gh^{\lambda\gamma}
\bigl[\nabla_\alpha \gh_{\beta\gamma} + \nabla_\beta \gh_{\alpha\gamma}-\nabla_\gamma \gh_{\alpha\beta}\bigr].
\end{equation*}

\section{Maxwell's equations\label{Maxwell}}
As in the introduction, let $\Omega_0\subset M$ be a bounded, connected spacelike hypersurface with smooth boundary $\del{}\Omega_0$, and $\Omega_T$ be
a timelike cylinder diffeomorphic to $[0,T]\times \Omega_0$. We use $\Gamma_T$ to denote the timelike boundary of
$\Omega_T$, which is diffeomorphic to $[0,T]\times \del{}\Omega_0$, and we note that
$\Gamma_0 = \del{}\Omega_0.$ We denote the outward unit conormal to $\Gamma_T$ by
$n = n_\nu dx^\nu$, which we arbitrarily extend to all of $M$,
and we let $\Omega(T) \cong \{T\}\times \Omega_0$ and $\Gamma(T) \cong \{T\} \times \del{}{\Omega_0}$
denote the ``top'' of the spacetime cylinder and its boundary, respectively.
We further assume that $\tau = \tau^\mu\del{\mu}$ and  $\xi = \xi^\mu \del{\mu}$
are timelike, future pointing  $C^1$ vector fields on $\overline{\Omega}_T$  that satisfy $\tau(x),\xi(x) \in T_x\Gamma_T$ for all $x\in \Gamma_T$, and that $\tau$ 
has unit length and is normal to $\Omega_0$ and $\Omega(T)$.

Maxwell's equations on the world tube $\Omega_T$  are given by
\begin{align}
\delta_g \Fc &= J\hspace{0.4cm} \text{in $\Omega_T$}, \label{Max.1}\\
\Ed \Fc & = 0 \hspace{0.4cm} \text{in $\Omega_T$}, \label{Max.2}
\end{align}
where $\Fc = \Half \Fc_{\mu\nu} dx^\mu \wedge dx^\nu$
is the electromagnetic tensor and $J=J_\mu dx^\mu$ is the current source, which satisfies $\delta_g J=0$. We recall that the stress-energy tensor $T^{\mu\nu}$ of the electromagnetic field is defined by
\begin{equation} \label{stressdef}
T^{\mu\nu} = 2 \Fc^{\mu\alpha}\Fc^{\nu}{}_\alpha - \Half g^{\mu\nu} \Fc_{\alpha\beta}\Fc^{\alpha\beta}.
\end{equation}
For solutions to Maxwell's equations,  $T^{\mu\nu}$ satisfies
\begin{equation} \label{divT}
\nabla_\mu T^{\mu\nu} = 2 J_\mu F^{\mu\nu} \hspace{0.4cm} \text{in $\Omega_T$.}
\end{equation}
Integrating this expression over $\Omega_T$ leads to the following well-known integral relation, which will be needed in
the proof of Theorem \ref{cthm}.
\begin{lem} \label{Maxlem}
Suppose $\Fc\in C^1(\overline{\Omega}_T)$ solves \eqref{Max.1}-\eqref{Max.2}.
Then
\begin{equation*}
\int_{\Omega(T)} \tau_\mu T^{\mu\nu} \xi_\nu
  = \int_{\Omega_0} \tau_\mu T^{\mu\nu} \xi_\nu + 2 \int_{\Gamma_T} n_\mu \Fc^{\mu\alpha}\Fc^{\nu}{}_\alpha \xi_\nu
-\int_{\Omega_T}\biggl[ 2 J_\mu F^{\mu\nu}\xi_\nu + \frac{1}{2} T^{\mu\nu}\Ld_\xi g_{\mu\nu}\biggr].
\end{equation*}
\end{lem}
\begin{proof}
Since any solution $\Fc \in C^1(\overline{\Omega}_T)$ of  \eqref{Max.1}-\eqref{Max.2} satisfies \eqref{divT} in $\Omega_T$, we have that
\begin{equation*}
\nabla_\mu (T^{\mu\nu}\xi_\nu) =\nabla_\mu T^{\mu\nu}\xi_\nu + T^{\mu\nu}\nabla_{\mu}\xi_\nu = 2 J_\mu F^{\mu\nu}\xi_\nu + \frac{1}{2} T^{\mu\nu}\Ld_\xi g_{\mu\nu} \quad \text{in $\Omega_T$}.
\end{equation*}
Integrating this expression over $\Omega_T$, we find via an application of the Divergence Theorem that
\begin{equation*} 
\int_{\Omega(T)} \tau_\mu T^{\mu\nu}\xi_\nu - \int_{\Omega_0} \tau_\mu T^{\mu\nu} \xi_\nu - 2\int_{\Gamma_T} n_\mu \Fc^{\mu\alpha}\Fc^{\nu}{}_\alpha \xi_\nu
= -\int_{\Omega_T} \biggl[2 J_\mu F^{\mu\nu}\xi_\nu + \frac{1}{2} T^{\mu\nu}\Ld_\xi g_{\mu\nu}\biggr],
\end{equation*}
where in deriving this we have used $n_\mu \xi^\mu = 0$ in $\Gamma_T$.
\end{proof}

In the proof of Theorem \ref{cthm}, we will also need the inequality from the following lemma, which is used in
literature to show that the electromagnetic stress-energy tensor satisfies the Dominant Energy Condition.
Before stating the lemma, we first denote the unit-normalized version of $\xi^\mu$ by
$v^\mu = (-g(\xi,\xi))^{-1/2}\xi^\mu$
and we let $h_{\mu\nu}=g_{\mu\nu}+v_\mu v_\nu$ denote the induced positive definite metric on the subspace $g$-orthogonal
to $v^\mu$. We also define a positive definite
metric by  $m_{\mu\nu}=g_{\mu\nu}+2v_\mu v_\nu$.
\begin{lem} \label{doclem}
There exists a constant $c>0$, independent of $\Fc$,  such that
\begin{equation*}
v_\mu T^{\mu\nu} \tau_\nu \geq c |\Fc|_m^2 \quad \text{in $\overline{\Omega}_T$.}
\end{equation*}
\end{lem}
\begin{proof}
Starting from the standard decomposition, see \cite[Ch. 13]{Geroch:1972},
\begin{equation} \label{doclem1}
F_{\mu\nu} = 2 v_{[\mu}E_{\nu]}-\nu_{\mu\nu\alpha\beta}v^\alpha B^\beta
\end{equation}
of the electromagnetic tensor in terms of the electric and magnetic fields relative to $v^\mu$, which are defined by
\begin{equation} \label{doclem2}
E_\mu = F_{\mu\nu}v^\nu \AND B_\mu = \frac{1}{2}\nu_{\mu\nu\alpha\beta}v^\nu F^{\alpha \beta},
\end{equation}
respectively, a straightforward calculation, see  \cite[Ch. 15]{Geroch:1972}, shows that
$F^{\mu\nu}F_{\mu\nu} = 2(B_\mu B^\mu - E_\mu E^\mu)$.
Using
\begin{equation} \label{doclem3}
E_\mu v^\mu = B_\mu v^\mu=0,
\end{equation}
we can then write $F^{\mu\nu}F_{\mu\nu}$ as
$F^{\mu\nu}F_{\mu\nu} = 2(|B|_m^2 - |E|_m^2)$.
From this and \eqref{doclem3}, it is then not difficult to verify that
\begin{equation} \label{doclem5}
|F|_m^{2} = m^{\alpha\beta}m^{\mu\nu}F_{\alpha\mu} F_{\beta\nu} = 2(|B|_m^2+|E|_m^2).
\end{equation}

Next, we recall that the energy density relative to $v^\mu$ is given by, see \cite[Ch. 15]{Geroch:1972}, the formula
\begin{equation} \label{doclem6}
v_\mu T^{\mu\nu} v_\nu = |E|_m^2 + |B|_m^2.
\end{equation}
Using this, we compute
\begin{align*}
v_\mu T^{\mu\nu}\tau_\nu &= (-\tau_\lambda v^\lambda) v_\mu T^{\mu\nu} \tau_\nu + v_\mu T^{\mu\nu}h_\nu^\lambda \tau_\lambda
\notag \\
& = (-\tau_\lambda v^\lambda)\bigl( |E|_m^2 + |B|_m^2\bigr) + 2 E^\alpha F_{\alpha\beta}h^\beta_\lambda \tau^\lambda && \text{(by  \eqref{stressdef}, \eqref{doclem2} \& \eqref{doclem6})} \notag \\
&= (-\tau_\lambda v^\lambda)\bigl( |E|_m^2 + |B|_m^2\bigr) -2 E^\mu B^\nu \nu_{\mu\nu\alpha}h^\alpha_\lambda \tau^\lambda,
\end{align*}
where $\nu_{\mu\nu\alpha}=\nu_{\mu\nu\alpha\beta}v^\beta$ and in deriving the last equality we used \eqref{doclem1}.
This result together with the inequality
\begin{equation*}
|E^\mu B^\nu \nu_{\mu\nu\alpha}h^\alpha_\lambda \tau^\lambda|
\leq |E|_h |B|_h |\tau|_h =  |E|_m |B|_m |\tau|_h \leq \frac{1}{2}(|E|_m^2+|B|_m^2)|\tau|_h
\end{equation*}
implies
\begin{equation}
v_\mu T^{\mu\nu}\tau_\nu \geq
 ( (-\tau_\lambda v^\lambda)-|\tau|_h)(|E|_m^2+|B|_m^2) \overset{\eqref{doclem5}}{=}
\frac{ ( (-\tau_\lambda v^\lambda)-|\tau|_h)}{2}|F|_m^2. \label{doclem8}
\end{equation}
Since $v^\mu$ and $\tau^\mu$ are both future pointing and timelike by assumption, we
have that
$v_\mu \tau^\mu > 0$ and  $\tau_\mu \tau^\mu=|\tau|^2_h-(-v_\mu \tau^\mu)^2<0$, and consequently,
$(-\tau_\lambda v^\lambda)-|\tau|_h \geq c >0$
in $\overline{\Omega}_T$
for some positive constant $c$. The proof now follows from this inequality and \eqref{doclem8}.
\end{proof}

\section{Calculus inequalities\label{calculus}}

In this appendix, we state a number of calculus inequalities that will be used throughout this article.
In the following, $\Omega$ will denote a bounded, open subset of $\Rbb^n$ with a
smooth boundary.

\subsection{Spatial inequalities\label{Sineq}}

The proofs of the following calculus inequalities are well known and may be found, for example, in
the references \cite{AdamsFournier:2003,Friedman:1976,RunstSickel:1996,TaylorIII:1996}.
In the following, $M$ will denote either $\Omega$, or a closed $n$-manifold.

\begin{thm}{\emph{[H\"{o}lder's inequality]}} \label{Holder}
If $0< p,q,r \leq \infty$ satisfy $1/p+1/q = 1/r$, then
\begin{equation*}\label{HolderA}
\norm{uv}_{L^r(M)} \leq \norm{u}_{L^p(M)}\norm{v}_{L^q(M)}
\end{equation*}
for all $u\in L^p(M)$ and $v\in L^q(M)$.
\end{thm}

\begin{thm}{\emph{[Integral Sobolev inequalities]}} \label{ISobolev}
Suppose $s\in \Zbb_{\geq 1}$ and $1\leq p < \infty$.
\begin{enumerate}[(i)]
\item If $sp<n$, then
\begin{equation*} \label{ISobolev1}
\norm{u}_{L^q(M)} \lesssim \norm{u}_{W^{s,p}(M)}, \qquad p\leq q \leq \frac{np}{n-s p},
\end{equation*}
for all $u\in W^{s,p}(M)$.
\item (Morrey's inequality)
If $sp > n$, then
\begin{equation*}\label{ISobolev2}
\norm{u}_{C^{0,\mu}(M)} \lesssim \norm{u}_{W^{s,p}(M)}, \qquad 0 < \mu \leq \min\{1,s-n/p\},
\end{equation*}
for all $u\in W^{s,p}(M)$.
\end{enumerate}
\end{thm}

\begin{thm}{\emph{[Fractional Sobolev inequalities]}} \label{FSobolev} Suppose $s>0$ and
$1< p < \infty$.
\begin{itemize}
\item[(i)] If $sp<n$, then
\begin{equation*}\label{FSobolev1}
\norm{u}_{L^q(M)} \lesssim \norm{u}_{W^{s,p}(M)}, \qquad p\leq q \leq \frac{np}{n-s p},
\end{equation*}
for all $u\in W^{s,p}(M)$.
\item[(ii)] If $sp > n$, then
\begin{equation*}\label{FSobolev2}
\norm{u}_{L^\infty(M)} \lesssim \norm{u}_{W^{s,p}(M)}
\end{equation*}
for all $u\in W^{s,p}(M)$.
\end{itemize}
\end{thm}

\begin{thm} {\emph{[Trace theorem]}} \label{trace}
If $s>1/2$, then the trace operator
\begin{equation*}\label{trace1}
H^s(\Omega) \ni u \longmapsto u|_{\del{}\Omega} \in H^{s-\frac{1}{2}}(\del{}\Omega)
\end{equation*}
is well-defined, continuous (i.e. bounded), and surjective.
\end{thm}

\begin{lem} {\emph{[Ehrling's lemma]}}  \label{Ehrling}
Suppose $1<p<\infty$, $0\leq s_0 < s < s_1$. Then for any $\delta>0$ there exists a constant $C=C(\delta)$ such
that
\begin{equation*} \label{Ehrling1}
\norm{u}_{W^{s,p}(M)} \leq \delta \norm{u}_{W^{s_1,p}(M)} + C(\delta)\norm{u}_{W^{s_0,p}(M)}
\end{equation*}
for all $u\in W^{s_1,p}(M)$.
\end{lem}

\begin{thm}{\emph{[Integral multiplication inequality]}} \label{Iprod} $\;$
Suppose $s_1,s_2,s_3\in \Zbb_{\geq 0}$, $s_1,s_2\geq s_3$, $1\leq p \leq \infty$, and $s_1+s_2-s_3 > n/p$. Then
\begin{equation*} \label{Iprod.3}
\norm{uv}_{W^{s_3,p}(M)} \lesssim \norm{u}_{W^{s_1,p}(M)}\norm{v}_{W^{s_2,p}(M)}
\end{equation*}
for all $u\in W^{s_1,p}(M)$ and $v\in W^{s_2,p}(M)$.
\end{thm}

\begin{thm}{\emph{[Fractional multiplication inequality]}} \label{calcpropB}
Suppose $1< p <\infty$, $s_1,s_2,s_3\in \Rbb$, $s_1+s_2 > 0$, $s_1,s_2 \geq s_3$, and
$s_1+s_2 - s_3 > n/p$. Then
\begin{equation*} \label{calcpropB.1}
\norm{u v}_{W^{s_{3},p}(M)} \lesssim \norm{u}_{W^{s_1,p}(M)} \norm{v}_{W^{s_2,p}(M)}
\end{equation*}
for all $u \in W^{s_1,p}(M)$ and $v \in W^{s_2,p}(M)$.
\end{thm}

\begin{thm}{\emph{[Moser estimate]}} \label{calcpropC}
Suppose $1<p <\infty$, $s\in \Rbb$, $k\in \Zbb_{\geq 0}$, $s\leq k$, $f\in C^{k}\cap W^{k,\infty}(\Rbb)$, and $f(0)=0$. Then
\begin{equation*}
\norm{u}_{W^{s,p}(M)} \leq C(\norm{u}_{L^\infty(M)}\bigr)\norm{u}_{W^{s,p}(M)}
\end{equation*} 
for all $u\in W^{s,p}(M)\cap L^\infty(M)$.
\end{thm}

\subsection{Spacetime inequalities\label{STineq}}
We recall, in this section, a number of spacetime estimates from \cite{Oliynyk:Bull_2017} that will be used repeatedly throughout this
article. Additionally, we establish two new results in the last two propositions of this section.
\begin{prop} \label{stpropA}
Suppose $s_1=k_1/2$ and  $s_2=k_2/2$ for $k_1,k_2\in \Zbb_{\geq 0}$, $s_3\in \Rbb$, $s_1,s_2 \geq s_3$,
$s_1+s_2 - s_3 > n/2$, $r\in \Zbb$, and $0\leq r \leq 2 s_3$.
Then
\begin{equation*}\label{stpropA1}
\norm{\del{t}^\ell(u(t)v(t))}_{H^{s_3-\frac{\ell}{2}}(M)}\leq
\norm{u(t)v(t)}_{E^{s_3,r}}  \lesssim \norm{u(t)}_{E^{s_1,r}} \norm{v(t)}_{E^{s_2,r}}
\end{equation*}
for $0\leq t \leq T$, $0\leq \ell \leq r$, and all $u\in X^{s_1,r}_T(M)$ and $v\in X^{s_2,r}_T(M)$.
\end{prop}

\begin{prop} \label{stpropC}
Suppose $s_1=k_1/2$ and  $s_2=k_2/2$ for $k_1,k_2\in \Zbb_{\geq 0}$, $s_3\in \Rbb$, $s_1,s_2 \geq s_3$,
and $s_1+s_2 - s_3 > n/2$.
Then
\begin{equation*}\label{stpropC1}
\norm{u(t)v(t)}_{\Ec^{s_3}}  \lesssim \norm{u(t)}_{\Ec^{s_1}} \norm{v(t)}_{\Ec^{s_2}}
\end{equation*}
for $0\leq t \leq T$, and all $u\in \Xc^{s_1}_T(M)$ and $v\in \Xc^{s_2}_T(M)$.
\end{prop}

\begin{prop} \label{stcomA}
Suppose $s_1=k_1/2$ and  $s_2=k_2/2$ for $k_1,k_2\in \Zbb_{\geq 0}$, $\ell\in \Zbb_{\geq 1}$,
$s_1+s_2 -\ell/2>0$, $\ell\leq \min\{2s_1,2s_2+1\}$,   $s_3\in \Rbb$, $s_1\geq s_3$,  $s_2\geq s_3-1/2$ and $s_1+s_2-s_3>n/2$. Then
\begin{equation*}\label{stcomA1}
\norm{[\del{t}^\ell,u(t)]v(t)}_{H^{s_3-\frac{\ell}{2}}(M)} \lesssim \norm{\del{t}u(t)}_{E^{s_1-\frac{1}{2},\ell-1}}
\norm{v(t)}_{E^{s_2,\ell-1}}
\end{equation*}
for $0\leq t \leq T$ and all $u\in X^{s_1,\ell}_T(M)$ and $v\in X^{s_2,\ell-1}_T(M)$.
\end{prop}

\begin{prop} \label{stpropB}
Suppose $s =k/2$, $s>n/2$, $r\in \Zbb$, $0\leq r \leq 2s$, $f\in C^r(\Rbb)$, and $f(0)=0$.
Then
\begin{equation*}
\norm{\del{t}^\ell f(u(t))}_{H^{s-\frac{\ell}{2}}(M)}\leq
\norm{f(u(t))}_{E^{s,r}} \leq C(\norm{u(t)}_{E^{s,r}})\norm{u(t)}_{E^{s,r}}
\end{equation*}
for $0\leq t \leq T$, $0\leq \ell \leq r$, and all $u\in X^{s,r}_T(M)$.
\end{prop}

\begin{prop} \label{stpropD}
Suppose $s =k/2$, $s>n/2$,  $f\in C^{2s-1}(\Rbb)$ and $f(0)=0$.
Then
\begin{equation*}
\norm{f(u(t))}_{\Ec^{s}}\leq  C(\norm{u(t)}_{\Ec^{s}})\norm{u(t)}_{\Ec^{s}}
\end{equation*}
for $0\leq t \leq T$ and all $u\in \Xc^{s}_T(M)$.
\end{prop}

\begin{prop} \label{stpropE}
Suppose $s=k/2$ for $k\in \Zbb_{\geq 0}$, $0\leq r\leq 2s$, $g_1\in \Xc^s_T(M)$, $g_2\in X_T^{s,r}(M)$, $f^0_a \in H^{s}(\Omega_0)$ for $a=1,2$, and the
$f_a$ satisfies the IVP
\begin{align*}
\del{t}f_a &= g_a && \text{in $[0,T]\times M$,} \\ 
f_a &= f^0_a && \text{in $\{0\}\times M$,}
\end{align*}
for $a=1,2$.
Then $f_1 \in C\Xc^s_T(M)$, $f_2\in C\Xc^{s,r}_T(M)$, and the $f_a$ satisfy 
\begin{align*}
\norm{f_1(t)-f_1(0)}_{\Ec^s} \leq \int_0^t \norm{g_1(\tau)}_{\Ec^s} d\tau
\AND
\norm{f_2(t)-f_2(0)}_{E^{s,r}} \leq  \int_0^t \norm{g_2(\tau)}_{E^{s,r}} d\tau
\end{align*}
for all $t\in [0,T]$.
\end{prop}
\begin{proof}
Integrating 
\begin{equation} \label{stpropE1}
\del{t}f_1 = g_1 
\end{equation}
in time shows, with the help of the initial condition $f_1(0)=f^0_1$, that
\begin{equation*}
f_1(t) = f^0_1 + \int_{0}^t g_1(\tau)\, d\tau
\end{equation*}
from which we have $f_1 \in C([0,T],H^{s+1}(M))$. Similarly, differentiating \eqref{stpropE1} with respect to $t$ repeatedly gives
\begin{equation*}
\del{t}^{\ell+1} f_1 = \del{t}^\ell g_1, \qquad 0\leq \ell \leq 2s-1.
\end{equation*}
Integrating this in time, we find that $f_1 \in C\Xc^s_T(M)$ and that
\begin{equation*}
\del{t}^\ell f_1(t) -\del{t}^\ell f_1(0) = \int_{0}^t \del{t}^\ell g_1(\tau)\, d\tau, \qquad 0\leq \ell \leq 2s-1.
\end{equation*}
We then conclude via an application of the triangle inequality that
\begin{equation*}
\norm{\del{t}^\ell f_1(t)-\del{t}^\ell f_1(0)}_{H^{s-\frac{m(s,\ell)}{2}}(M)} \leq  \int_0^t \norm{\del{t}^\ell g_1(\tau)}_{H^{s-\frac{m(s,\ell)}{2}}(M)} d\tau,  \qquad 0\leq \ell \leq 2s-1,
\end{equation*}
which, after summing over $\ell$, yields
\begin{equation*}
\norm{f_1(t)-f_1(0)}_{\Ec^s} \leq \int_0^t \norm{g_1(\tau)}_{\Ec^s} d\tau.
\end{equation*}
This establishes the desired estimates for $f_1$. The estimate for $f_2$ can be established in a similar fashion.
\end{proof}

\begin{prop} \label{stpropF}
Suppose $s =k/2$, $s>n/2$,  $f\in C^{2s}(\Rbb\times \Rbb^{n-1})$, $u \in  X_T^{s+1,2s}(\Omega)\cap  X_T^{s+\frac{1}{2}}(\Omega)$,
$v \in \Xc_T^{s+\frac{3}{2}}(\Omega)$, and let
$\Dsl f$ denote the collection of that are tangent to the boundary\footnote{Locally, 
$\Dsl f = (e_\Ic(f))$, $\Ic=2,3,\ldots,n$, for smooth, time-independent vector
fields $e_\Ic=e_{\Ic}^\Sigma \delb{\Sigma}$, $\; \del{t}  e_\Ic^\Sigma =0$, that are defined on $\overline{\Omega}$ and when restricted
to $\del{}\Omega$ form a local basis
for the tangent space $T\del{}\Omega$.} $\del{}\Omega$ and $\nu_\Sigma$ is an outward pointing conormal to $\del{}\Omega$. Then the derivative
$\del{t}^{2s}(f(u,\Dsl v))$ can be decomposed as
\begin{equation*}
\del{t}^{2s}(f(u,\Dsl v)) = k^{\Sigma}\del{\Sigma}\theta+ h
\end{equation*}
for
\begin{equation*}
(k,\theta,h) \in \bigcap_{\ell=0}^1 W^{\ell,\infty}\bigl([0,T],H^{s+\frac{1}{2}-\frac{\ell}{2}}(\Omega,\Rbb^{n})\bigr)
\times \bigcap_{\ell=0}^1 W^{\ell,\infty}\bigl([0,T],H^{\frac{3}{2}-\frac{\ell}{2}}(\Omega)\bigr)
\times \bigcap_{\ell=0}^1 W^{\ell,\infty}\bigl([0,T],H^{1-\ell}(\Omega)\bigr)
\end{equation*}
that satisfy
$\nu_\Sigma k^\Sigma=0$ on $\del{}\Omega$  and the estimates
\begin{gather*}
\norm{\theta}_{H^{\frac{3}{2}}(\Omega)} + \norm{\del{t}\theta}_{H^{1}(\Omega)} \lesssim 
\zeta,\\
\norm{h}_{H^{1}(\Omega)} + \norm{\del{t}h}_{L^{2}(\Omega)} \leq 
C(\zeta)\zeta,\\
\norm{k}_{H^{s+\frac{1}{2}}(\Omega)}+ \norm{\del{t}k}_{H^{s}(\Omega)} 
\leq 
C(\zeta),\\
\end{gather*} 
where
\begin{equation*}
\zeta = \norm{u}_{E^{s+1,2s}}+ \norm{u}_{E^{s+\frac{1}{2}}}+ \norm{v}_{\Ec^{s+\frac{3}{2}}}.
\end{equation*}
\end{prop}
\begin{proof}
Differentiating $f(u,\Dsl v)$ with respect to $t$ yields
\begin{equation}\label{stpropF1}
\del{t}(f(u,\Dsl v)) = k^\Sigma(u,\Dsl v) \del{\Sigma}\del{t}v + g
\end{equation}
where
\begin{align}
k^\Sigma(u,\Dsl v) &= \frac{\del{}f}{\del{}e_\Ic(u)}(u,\Dsl v)e_\Ic^\Sigma \label{stpropF1a.1}
\intertext{and}
g(u,\del{t}u,\Dsl v) & =  \frac{\del{}f}{\del{}u}(u,\Dsl v)\del{t}u. \label{stpropF1a.2}
\end{align}
Since the $e_\Ic^\Sigma$ are tangent to $\del{}\Omega$, it is clear from \eqref{stpropF1a.1} that  $\nu_\Sigma k^\Sigma=0$ on $\del{}\Omega$. 

From \eqref{stpropF1}, we see, after differentiating with respect to $t$ $2s$-times, that
\begin{equation}\label{stpropF2}
\del{t}^{2s}(f(u,\Dsl v)) = k^\Sigma(u,\Dsl v) \del{\Sigma}\theta + h 
\end{equation}
where
\begin{align}
\theta &= \del{t}^{2s}v \label{stpropF3.1}
\intertext{and}
h &= [\del{t}^{2s-1},k^\Sigma]\del{\Sigma} \del{t}v + \del{t}^{2s-1}g. \label{stpropF3.2}
\end{align}
It is then clear from \eqref{stpropF3.1} that
\begin{equation} \label{stpropF4}
\norm{\theta}_{H^{\frac{3}{2}}(\Omega)} + \norm{\del{t}\theta}_{H^{1}(\Omega)}
\lesssim \zeta, 
\end{equation}
where $\zeta$ is as defined in the statement of the proposition. 

Next, we observe from \eqref{stpropF1a.1} and Proposition \ref{stpropB} that
\begin{equation} \label{stpropF5}
\norm{k}_{H^{s+\frac{1}{2}}(\Omega)} = \norm{k}_{E^{s+\frac{1}{2},0}} \leq C\big(\norm{u}_{E^{s+\frac{1}{2},0}},\norm{D v}_{E^{s+\frac{1}{2},0}}\bigr)
\leq  C\big(\norm{u}_{E^{s+\frac{1}{2},0}},\norm{v}_{E^{s+\frac{3}{2},0}}\bigr) \leq C(\zeta)
\end{equation}
and
\begin{equation} \label{stpropF6}
\norm{\del{t}k}_{H^{s}(\Omega)} \lesssim \norm{k}_{E^{s+\frac{1}{2},1}} \leq C\big(\norm{u}_{E^{s+\frac{1}{2},1}},\norm{D v}_{E^{s+\frac{1}{2},1}}\bigr)
\leq  C\big(\norm{u}_{E^{s+\frac{1}{2},1}},\norm{v}_{E^{s+\frac{3}{2},1}}\bigr) \leq C(\zeta).
\end{equation}
By similar arguments, we also obtain
\begin{gather}
\norm{\del{t}k}_{E^{s,2s-a-2}} \lesssim \norm{k}_{E^{s+\frac{1}{2},2s-a-1}} \leq C(\zeta) \label{stpropF7.1}
\intertext{and}
\norm{\del{t}^2 k}_{E^{s-\frac{1}{2},2s-a-2}}\lesssim \norm{k}_{E^{s+\frac{1}{2},2s-a}} \leq C(\zeta). \label{stpropF7.2}
\end{gather}

We can estimate  the second term on the right hand side of \eqref{stpropF3.2} and its first derivative, again using  Proposition \ref{stpropB},  as follows:
\begin{align}
\norm{ \del{t}^{2s-1}g}_{H^{1}(\Omega)} &\lesssim \norm{g}_{E^{s+\frac{1}{2},2s-1}}\notag \\
& \leq C\bigl(\norm{(u,\del{t}u,D v)}_{E^{s+\frac{1}{2},2s-1}} \bigr)\norm{(u,\del{t}u,D v)}_{E^{s+\frac{1}{2},2s-1}} \notag \\
& \leq C\bigl(\norm{u}_{E^{s+1,2s}}+\norm{v}_{E^{s+\frac{3}{2},2s-1}}\bigr)\bigl(\norm{u}_{E^{s+1,2s}}+\norm{v}_{E^{s+\frac{3}{2},2s-1}}\bigr) \notag \\
& \leq C(\zeta)\zeta \label{stpropF8}
\end{align}
and
\begin{align}
\norm{ \del{t}^{2s}g}_{L^{2}(\Omega)} &\lesssim \norm{g}_{E^{s,2s}}\notag \\
& \leq C\bigl(\norm{(u,\del{t}u,D v)}_{E^{s,2s}} \bigr)\norm{(u,\del{t}u,D v)}_{E^{s,2s}} \notag \\
& \leq C\bigl(\norm{u}_{E^{s+\frac{1}{2},2s+1}}+\norm{v}_{E^{s+1,2s}}\bigr)\bigl(\norm{u}_{E^{s+\frac{1}{2},2s+1}}+\norm{v}_{E^{s+1,2s}}\bigr) \notag \\
& \leq C(\zeta)\zeta. \label{stpropF9}
\end{align}
We further observe that the first term on the right hand side of \eqref{stpropF3.2} can be estimated by
\begin{align}
\norm{ [\del{t}^{2s-1},k^\Sigma]\del{\Sigma} \del{t}v }_{H^{1}(\Omega)} &\lesssim
\norm{\del{t}k}_{E^{s.2s-1}}\norm{D\del{t}v}_{E^{s,2s-2}} \notag \\
&\leq C(\zeta)\norm{v}_{E^{s+\frac{3}{2},2s-1}} \notag \\
&\leq  C(\zeta)\zeta,  \label{stpropF10}
\end{align}
where in deriving this we have employed Proposition \ref{stcomA}, with $s_1=s+\frac{1}{2}$, $s_2=s$, $s_3=s+\frac{1}{2}$,
and $\ell = 2s-1$, and  the inequality \eqref{stpropF7.1}.

Differentiating  $[\del{t}^{2s-1},k^\Sigma]\del{\Sigma} \del{t}v$ with respect to $t$ gives
\begin{equation*}
\del{t}\bigl([\del{t}^{2s-1},k^\Sigma]\del{\Sigma} \del{t}v\bigr)=[\del{t}^{2s-1},\del{t}k^\Sigma]\del{\Sigma} \del{t}v+[\del{t}^{2s-1},k^\Sigma]\del{\Sigma} \del{t}^2 v.
\end{equation*}
Both terms  on the right hand side can be estimated in a similar fashion as \eqref{stpropF10} by using Proposition \ref{stcomA}, with
 with $s_1=s$, $s_2=s-\frac{1}{2}$, $s_3=s-\frac{1}{2}$
and $\ell = 2s-1$ for the first term and $s_1=s+\frac{1}{2}$, $s_2=s-1$, $s_3=s-\frac{1}{2}$ and
$\ell = 2s-1$ 
for the second,
and the inequalities  \eqref{stpropF7.1}-\eqref{stpropF7.2}. Doing so shows that
\begin{align}
\norm{ \del{t}\bigl([\del{t}^{2s-1},k^\Sigma]\del{\Sigma} \del{t}v\bigr) }_{L^{2}(\Omega)} &\lesssim  \norm{\del{t}^2 k}_{E^{s-\frac{1}{2},2s-2}}
\norm{D\del{t}v}_{E^{s-\frac{1}{2},2s-2}} +  \norm{\del{t} k}_{E^{s,2s-2}}
\norm{D\del{t}^2v}_{E^{s-1,2s-2}} \notag \\
&\leq C(\zeta)\norm{v}_{E^{s+1,2s}}\notag \\
&\leq C(\zeta)\zeta.\notag 
\end{align}
From this, the definition \eqref{stpropF3.2}, and the estimates \eqref{stpropF8}-\eqref{stpropF10}, we deduce that
\begin{equation*}
\norm{h_a}_{H^{1}(\Omega)} + \norm{\del{t}h}_{L^{2}(\Omega)} \leq 
C(\zeta)\zeta.
\end{equation*}
This result together with \eqref{stpropF4}-\eqref{stpropF6} establishes the stated estimates and completes the proof.
\end{proof}

\section{\label{elliptic}Elliptic systems}

In this appendix, we recall some well known regularity results for elliptic systems. Throughout this section, $\Omega$ will denote an open and bounded set in $\Rbb^n$, $n\geq 2$, with smooth boundary.

\subsection{\label{neumann} Neumann boundary conditions}
We begin by considering elliptic systems with Neumann boundary conditions of the form
\begin{align}
\del{\Lambda}(b^{\Lambda\Sigma}\del{\Sigma}u+m^\Lambda) &= f && \text{in $\Omega$,}\label{ellipA.1}\\
\nu_\Lambda(b^{\Lambda\Sigma}\del{\Sigma}u+m^\Lambda) &= g && \text{in $\del{}\Omega$,}\label{ellipA.2}
\end{align}
where
\begin{enumerate}[(i)]
\item $\nu_\Lambda$ is the outward pointing unit conormal to $\del{}\Omega$,
\item $u=u(x)$, $m^\Lambda=m^\Lambda(x)$, $f=f(x)$ and $g=g(x)$ are $\Rbb^N$-valued maps,
\item and the $b^{\Lambda\Sigma}=b^{\Lambda\Sigma}(x)$ are $\Mbb{N}$-valued maps that satisfy the symmetry condition
\begin{equation}\label{bellip}
(b^{\Lambda\Sigma})^{\tr} = b^{\Sigma\Lambda}
\end{equation}
in $\Omega$,
and there exists constants $\kappa_1 >0$ and $\mu \geq 0$ such that coercive inequality 
\begin{equation}\label{coercellip}
\ip{\del{\Lambda}v}{b^{\Lambda\Gamma}\del{j}v}_{\Omega} \geq \kappa_1 \norm{v}^2_{H^1(\Omega)} -\mu \norm{v}^2_{L^2(\Omega)}
\end{equation}
holds for all $v\in H^1(\Omega,\Rbb^N)$.
\end{enumerate}
We recall, see, for example, \cite[Theorem B.3]{Koch:1990}, that solutions to these systems enjoy the following version of
elliptic regularity:
\begin{thm} \label{ellipticregN}
Suppose $r,s \in \Rbb$, $s>n/2$, $0\leq r \leq s$,
\begin{equation}\label{r*def}
r^* = \begin{cases} r-1 & \text{if $r> 1$}\\
                   0 & \text{otherwise}
\end{cases},
\end{equation}
$b^{\Lambda\Sigma} \in H^s(\Omega,\Mbb{N})$, $m^\Lambda\in H^{r}(\Omega,\Rbb^N)$,
$f\in H^{r^*}(\Omega,\Rbb^N)$, $g\in H^{r-\frac{1}{2}}(\del{}\Omega,\Rbb^N)$,
the $b^{\Lambda\Sigma}$ satisfy \eqref{bellip} and \eqref{coercellip}, and
$u\in H^1(\Omega,\Rbb^N)$ is a weak solution of  \eqref{ellipA.1}-\eqref{ellipA.2}. Then
$u\in H^{r+1}(\Omega,\Rbb^N)$ and
\begin{equation*}
\norm{u}_{H^{r+1}(\Omega)} \leq C\Bigl(\norm{u}_{H^{1}(\Omega)}+\norm{f}_{H^{r^*}(\Omega)}
+ \norm{m}_{H^{r}(\Omega)} + \norm{g}_{H^{r-1/2}(\del{}\Omega)} \Bigr)
\end{equation*}
where $C = C(\kappa_1,\mu,\norm{b}_{H^s(\Omega)})$.
\end{thm}

\subsection{\label{dirichlet} Dirichlet boundary condtions}
Next, we consider elliptic systems with Dirichlet boundary conditions of the form
\begin{align}
\del{\Lambda}\bigl(b^{\Lambda\Sigma}\del{\Sigma}u + m^\Lambda \bigr) &= f && \text{in $\Omega$,}\label{ellipB.1}\\
u &= 0 && \text{in $\del{}\Omega$,}\label{ellipB.2}
\end{align}
where 
\begin{enumerate}[(i)]
\item $u=u(x)$, $m^\Lambda=m^{\Lambda}(x)$, and $f=f(x)$ are $\Rbb^N$ valued maps,
\item and the $b^{\Lambda\Sigma}=b^{\Lambda\Sigma}(x)$ are $\Mbb{N}$-valued
maps that satisfy the symmetry condition \eqref{bellip} and there exists a constant $\kappa > 0$ such the
strong ellipticity condition 
\begin{equation} \label{strongellip}
b^{\Lambda\Sigma}\xi_\Lambda\xi_\Sigma \geq \kappa |\xi|^2 \id, \quad \forall \; \xi=(\xi_\Sigma)\in \Rbb^n,
\end{equation}
holds in $\Omega$.
\end{enumerate}
We recall the following elliptic regularity result, see, for example, \cite[Theorem 4.14]{GiaquintaMartinazzi:2012}, 
satisfied by solutions of these systems:
\begin{thm} \label{ellipticregD}
Suppose $r,s \in \Rbb$, $s>n/2$, $0\leq r \leq s$, $r^*$ is given by \eqref{r*def},
$b^{\Lambda\Sigma} \in H^s(\Omega,\Mbb{N})$, $m^\Lambda\in H^{r}(\Omega,\Rbb^N)$,
$f\in H^{r^*}(\Omega,\Rbb^N)$,
the $b^{\Lambda\Sigma}$ satisfy \eqref{bellip} and \eqref{strongellip}, and
$u\in H^1(\Omega,\Rbb^N)$ is a weak solution of \eqref{ellipB.1}-\eqref{ellipB.2}. Then
$u\in H^{r+1}(\Omega,\Rbb^N)$ and
\begin{equation*}
\norm{u}_{H^{r+1}(\Omega)} \leq C\Bigl(\norm{u}_{H^{1}(\Omega)}+
\norm{f}_{H^{r*}(\Omega)} + \norm{m}_{H^{r}(\Omega)} \Bigr)
\end{equation*}
where $C = C(\kappa,\norm{b}_{H^s(\Omega)})$.
\end{thm}

\section{Linear wave equations\label{LWE}}

In this appendix, we establish an existence and uniqueness theory for systems of linear wave equations with either Dirichlet or acoustic boundary conditions.
Throughout this section we employ the notion
\begin{equation*}
\Omega_T=[0,T]\times \Omega \AND \Gamma_T=[0,T]\times\del{}\Omega,
\end{equation*}
where, as above, $\Omega$ denotes an open and bounded set in $\Rbb^n$, $n\geq 2$, with smooth boundary.

\subsection{Dirichlet boundary conditions\label{DCWE}}
In this section, we consider the initial boundary value problem (IBVP) for systems of wave equations with Dirichlet boundary conditions that are of the following form:
\begin{align}
\del{\alpha}\bigl(b^{\alpha\beta}\del{\beta}u+m^\alpha\bigr) & = f &&\text{in $\Omega_T$,}
\label{linDCa.1}\\
u & = 0 &&\text{in $\Gamma_T$, } \label{linDCa.2}\\
(u,\del{t}u) &= (u_0,u_1) &&\text{in $\Omega_0$,}\label{linDCa.3}
\end{align}
where
\begin{enumerate}[(i)]
\item $u=u(t,x)$, $m^\alpha= m^\alpha(t,x)$, and $f=f(t,x)$ are $\Rbb^N$-valued maps,
\item  the $\Mbb{N}$-valued maps $b^{\alpha\beta}=b^{\alpha\beta}(t,x)$
satisfy
\begin{equation}
(b^{\alpha\beta})^{\tr} = b^{\beta\alpha} \label{bsym}
\end{equation}
in $\Omega_T$, 
and there exists constants  $\kappa_0$, $\kappa_1>0$ and $\mu\geq 0$ such that\footnote{Given $A,B\in \Mbb{N}$, we define $A\leq B$ $\Longleftrightarrow$ $\ipe{\eta}{A\eta} \leq
\ipe{\eta}{B\eta}$ for all $\eta\in \Rbb^N$.}
\begin{equation}
 b^{00} \leq -\kappa_0\id\label{b00lb}
\end{equation}
in $\Omega_T$ and the coercive inequality
\begin{equation} \label{coercDC}
\ip{\del{\Lambda}u}{b^{\Lambda\Sigma}(t)\del{\Sigma}u}_\Omega \geq \kappa_1 \norm{u}^2_{H^1(\Omega)} - \mu \norm{u}^2_{L^2(\Omega)}
\end{equation}
holds for all $(t,u)\in [0,T]\times   H^1_0(\Omega,\Rbb^N)$.
\end{enumerate}

\begin{rem} \label{strongelliprem}
Provided that the matrices $b^{\Lambda\Sigma}$ are uniformly continuous on $\Omega_T$, it is known, see Theorem 3 from Section 6 of \cite{SimpsonSpector:1987},
that the strong ellipticity condition \eqref{strongellip} holds in $\Omega_T$ if and only if the coercive inequality \eqref{coercDC} holds. 
\end{rem}

\begin{Def} \label{dircompat}
Given $\st=\kt/2$ with $\kt\in \mathbb{Z}_{\geq 0}$, we say the the initial data\footnote{See \eqref{melldef} for a definition of $m(s,\ell)$.} 
\begin{equation*}
(u_0,u_1) \in H^{\st+1}(\Omega,\Rbb^N)\times H^{\st+1-\frac{m(\st+1,1)}{2}}(\Omega,\Rbb^N)
\end{equation*}
for the IBVP \eqref{linDCa.1}-\eqref{linDCa.3} satisfies the \textit{compatibility conditions to order $2\st+1$} if
the higher formal
time derivatives
\begin{equation*}
u_\ell = \del{t}^\ell u \bigl|_{\Omega_0}, \quad \ell=2,3,\ldots,2\st+1,
\end{equation*} 
which are  generated from the initial data by differentiating the wave equation \eqref{linDCa.1} formally with respect to $t$ the required number of times
and setting $t=0$, satisfy 
\begin{equation*}
u_\ell \in H^{\st+1-\frac{m(\st+1,\ell)}{2}}(\Omega,\Rbb^N), \quad 0\leq \ell \leq 2\st+1,
\end{equation*}
and
\begin{equation*}
\del{t}^\ell u = 0\quad \text{in $\Gamma_0$}, \quad \ell=0,1,\ldots,2\st.
\end{equation*}
\end{Def}

\begin{thm} \label{DCthm}
Suppose $T>0$, $s>n/2$, $s=k/2$ and $\st=\kt/2$ for $k,\kt\in \Zbb_{\geq 0}$ where $\kt\leq k$,  $\,b^{\alpha\beta}, \del{t}b^{\alpha\beta}  \in X_{T}^{s,2\st}(\Omega,\Mbb{N})$, 
$\,f,\del{t}f\in X_{T}^{\st-1}(\Omega,\Rbb^N)$, $\del{t}^{2\st}f\in L^\infty([0,T],L^2(\Omega,\Rbb^N))$, $\,m^{\alpha}, \del{t}m^{\alpha}  \in X_{T}^{\st}(\Omega,\Rbb^{N})$,
the initial data 
\begin{equation*}
(u_0,u_1)\in  H^{\st+1}(\Omega,\Rbb^N)\times  H^{\st+1-\frac{m(\st+1,1)}{2}}(\Omega,\Rbb^N)
\end{equation*} 
satisfies the compatibility conditions to
order $2\st+1$, and the matrices $b^{\alpha\beta}$ satisfy \eqref{bsym} and the inequalities \eqref{b00lb}-\eqref{coercDC} for some constants $\kappa_0,\kappa_1>0$ and $\mu\geq 0$.
Then there exists a unique solution\footnote{When $\st=0$, the solution must be interpreted in the weak sense, see \cite[Definiton 2.1]{Koch:1993}. } $u\in C\Xc^{\st+1}_T(\Omega,\Rbb^N)$ to the IBVP \eqref{linDCa.1}-\eqref{linDCa.3} that satisfies the energy estimate
\begin{equation*}
\norm{u(t)}_{\Ec^{\st+1}} \leq C\Bigl( \norm{u(0)}_{\Ec^{\st+1}} + 
\alpha_2(0) +  \int^t_0 \alpha_1(\tau)\norm{u(\tau)}_{\Ec^{\st+1}} + \alpha_2(\tau) \, d\tau \Bigr)
\end{equation*}
where\footnote{See \eqref{r*def} for the definition of $(\st-1)^*$.} $C=C\bigl(\vec{\kappa},\mu,\norm{b(0)}_{L^\infty(\Omega)},\norm{b(t)}_{H^s(\Omega)}\bigr)$, $\vec{\kappa}=(\kappa_0,\kappa_1)$,
\begin{align*}
\alpha_1(t) &= 1+\norm{b(t)}_{H^{s}(\Omega)}+\norm{\del{t}b(t)}_{E^{s,2\st}},
\intertext{and}
\alpha_2(t) &=  \norm{m(t)}_{H^{\st}(\Omega)}+\norm{\del{t}m(t)}_{E^{\st}}+\norm{f(t)}_{H^{(\st-1)^*}(\Omega)} +
 \norm{\del{t}f(t)}_{E^{(\st-1)^*}}+\norm{\del{t}^{2\st}f(t)}_{L^2(\Omega)}.
\end{align*}
\end{thm}
\begin{proof}
Assuming for the moment that $b^{\alpha\beta}$, $m^\alpha$, $f$ $\in$ $C^\infty(\overline{\Omega}_T)$ and the initial
data satisfies the compatibility conditions in the sense of \cite{Koch:1993} to order $2\st+1$, that is,
$\del{t}^\ell u|_{\Omega_0} \in H^{2\st+1-\ell}(\Omega,\Rbb^N)$,  $0\leq \ell \leq 2\st+1$, and
$\del{t}^\ell u|_{\Gamma_0}=0$, $0\leq \ell \leq 2\st$,
then it follows from Theorem 2.5 of \cite{Koch:1993} that there exists a unique solution 
\begin{equation*}
u \in \bigcap_{j=0}^{2\st+1} C^{\ell}\bigl([0,T],H^{2\st+1-\ell}(\Omega,\Rbb^N)\bigr)
\end{equation*}
of the IBVP \eqref{linDCa.1}-\eqref{linDCa.3}. Differentiating \eqref{linDCa.1}-\eqref{linDCa.2} $\ell$ times, $0\leq \ell \leq 2\st$,
with respect to $t$ shows that
$\del{t}^\ell u$ satisfies
\begin{align}
\del{\alpha}\bigl(b^{\alpha\beta}\del{\beta}\del{t}^\ell u+ M^\alpha_\ell \bigr) & = \del{t}^\ell f &&\text{in $\Omega_T$,}
\label{DCthm1.1}\\
\del{t}^\ell u & = 0 &&\text{in $\Gamma_T$, } \label{DCthm1.2}
\end{align}
where
\begin{equation} \label{DCthm2}
M^\alpha_\ell =  [\del{t}^\ell, b^{\alpha\beta}]\del{\beta}u + \del{t}^\ell m^\alpha.
\end{equation}
In particular, $\del{t}^\ell u$ defines a weak solution of the wave equation \eqref{DCthm1.1}, and hence,
by Theorem 2.2 of \cite{Koch:1993}, it satisfies the energy estimate
\begin{align}
E\bigl(\del{t}^\ell u(t)\bigr) \leq C\biggl(E\bigl(\del{t}^\ell u(0)\bigr) + 
 \norm{\vec{M}_\ell(0) }_{L^2(\Omega)} + \int_0^t \bigl(1+&\norm{b(\tau)}_{L^\infty(\Omega)}+\norm{\del{t}b(\tau)}_{L^\infty(\Omega)}\bigr) E\bigl(\del{t}^\ell u(\tau)\bigr)
 \notag \\
 + \norm{\vec{M}_\ell(\tau)}_{L^2(\Omega)}& 
 +\norm{\del{t} M_\ell(\tau)}_{L^2(\Omega)}+\norm{\del{t}^\ell f(\tau)}_{L^2(\Omega)} \, d\tau \biggr), \label{DCthm3}
\end{align} 
\bigskip
where  $C=C\bigl(\vec{\kappa},\mu,\norm{b(0)}_{L^\infty(\Omega)}\bigr)$, $\vec{M}=(M^\Sigma)$, and
\begin{equation*}
E\bigl(\del{t}^\ell u(t)\bigr) =\sqrt{ \norm{\del{t}^\ell u(t)}_{H^1(\Omega)}^2 + \norm{\del{t}^{\ell+1}u(t)}^2_{L^2(\Omega)}}.
\end{equation*}

Viewing \eqref{DCthm1.1}-\eqref{DCthm1.2} as an elliptic system, that is,
\begin{align*}
\del{\Lambda}\bigl(b^{\Lambda\Sigma}\del{\Sigma}\del{t}^\ell u+ M^\Lambda_\ell + b^{\Lambda 0}\del{t}^{\ell+1} u\bigr) & =
 F_\ell &&\text{in $\Omega_T$,} \\
\del{t}^\ell u & = 0 &&\text{in $\Gamma_T$, }
\end{align*}
where
\begin{equation} \label{DCthm4}
F_\ell =-b^{0\Lambda}\del{\Lambda}\del{t}^{\ell+1} u - \del{t}b^{0\Lambda}\del{\Lambda}\del{t}^\ell u -\del{t}b^{00}\del{t}^{\ell+1}u 
-b^{00}\del{t}^{\ell+2}u -\del{t}M^0_\ell+\del{t}^\ell f,
\end{equation}
we can appeal to elliptic regularity, see Theorem \ref{ellipticregD} and Remark \ref{strongelliprem}, to obtain the estimate
\begin{equation}  \label{DCthm5}
\norm{\del{t}^\ell u}_{H^{\st+1-\frac{\ell}{2}}(\Omega)} \leq C
\Bigr(E\bigl(\del{t}^\ell u\bigr) + \norm{\vec{M}_\ell}_{H^{\st-\frac{\ell}{2}}(\Omega)}+ 
\norm{b^{0\Lambda}\del{t}^{\ell+1} u}_{H^{\st-\frac{\ell}{2}}(\Omega)} +
\norm{F_\ell}_{H^{(\st-\frac{\ell}{2})^*}\!(\Omega)}
\Bigl)
\end{equation}
for $0\leq \ell \leq 2\st-1$
where $C=C\bigl(\kappa_1,\mu,\norm{b(t)}_{H^s(\Omega)}\bigr)$ and $(\st-\frac{\ell}{2})^*$ is defined by \eqref{r*def}.

To proceed, we estimate the terms and their time derivatives on the right hand side of \eqref{DCthm5} by the energy norms. We collect the
relevant estimates in the following lemma.
\begin{lem} \label{DClem}
The following estimates hold:\footnote{See \eqref{r*def} for the definition of $(\st-\frac{\ell}{2})^*$.}
\begin{enumerate}[(i)]
\item
\begin{gather*}
\norm{M_\ell}_{H^{\st-\frac{\ell}{2}}(\Omega)} \lesssim  (\norm{b}_{E^{s,2\st}}+\norm{\del{t}b}_{E^{s,2\st}})\norm{u}_{\Ec^{\st+1}}+\norm{m}_{E^{\st}}, \\
\norm{\del{t}M_\ell}_{H^{\st-\frac{\ell}{2}}(\Omega)}+
\norm{\del{t}^2 M_\ell}_{H^{(\st-\frac{\ell}{2})*}\!(\Omega)} \lesssim (\norm{b}_{E^{s,2\st}}+
\norm{\del{t}b}_{E^{s,2\st}})\norm{u}_{\Ec^{\st+1}} +
\norm{\del{t}m}_{E^{\st}}
\intertext{and}
\norm{b^{0\Lambda}\del{t}^{\ell+1} u}_{H^{\st-\frac{\ell}{2}}(\Omega)}  \lesssim (\norm{b}_{E^{s,2\st}}+
\norm{\del{t}b}_{E^{s,2\st}})\norm{u}_{\Ec^{\st+1}}
\end{gather*}
for $0\leq \ell \leq 2\st-1$,
\item 
\begin{gather*}
\norm{F_\ell}_{H^{(\st-\frac{\ell}{2})^*}\!(\Omega)}+  \norm{\del{t}F_\ell}_{H^{(\st-\frac{\ell}{2})^*}\!(\Omega)} \lesssim (\norm{b}_{E^{s,2\st}}+\norm{\del{t}b}_{E^{s,2\st}})\norm{u}_{\Ec^{\st+1}} + \norm{f}_{E^{\st-1}}+\norm{\del{t}f}_{E^{\st-1}} + \norm{\del{t}m}_{E^{\st}}
\intertext{and}
\norm{ \del{t}(b^{0\Lambda}\del{t}^{\ell+1}u)}_{H^{\st-\frac{\ell}{2}}(\Omega)}  \lesssim 
(\norm{b}_{E^{s,2\st}}+\norm{\del{t}b}_{E^{s,2\st}})\norm{u}_{\Ec^{\st+1}}
\end{gather*}
for $0\leq \ell \leq 2\st-2$,
\item
\begin{gather*}
\norm{F_{2\st-1}}_{L^2(\Omega)} \lesssim \norm{b}_{H^s(\Omega)}
E(\del{t}^{2\st}u)+ \norm{\del{t}b^{0\Lambda}\del{\Lambda}\del{t}^{2\st-1} u}_{L^2(\Omega)} +\norm{\del{t}b^{00}\del{t}^{2\st}u}_{L^2(\Omega)} \\
+ \norm{\del{t}^{2\st} M}_{L^2(\Omega)} + \norm{\del{t}^{2\st-1}f}_{L^2(\Omega)}, \\
 \norm{\del{t}b^{0\Lambda}\del{\Lambda}\del{t}^{2\st-1} u}_{L^2(\Omega)} + \norm{\del{t}(\del{t}b^{0\Lambda}\del{\Lambda}\del{t}^{2\st-1} u)}_{L^2(\Omega)} 
 \lesssim (\norm{b}_{E^{s,2\st}}+\norm{\del{t}b}_{E^{s,2\st}})\norm{u}_{\Ec^{\st+1}}
 \intertext{and}
 \norm{\del{t}b^{00}\del{t}^{2\st}u}_{L^2(\Omega)} +\norm{\del{t}(\del{t}b^{00}\del{t}^{2\st}u)}_{L^2(\Omega)} 
 \lesssim (\norm{b}_{E^{s,2\st}}+\norm{\del{t}b}_{E^{s,2\st}})\norm{u}_{\Ec^{\st+1}},
\end{gather*}
\item and
\begin{align*}
\norm{\vec{M}_\ell}_{L^2(\Omega)}+\norm{\del{t}M_\ell}_{L^2(\Omega)} \lesssim  
(\norm{b}_{E^{s,2\st}}+\norm{\del{t}b}_{E^{s,2\st}})\norm{u}_{\Ec^{\st+1}}+\norm{m}_{E^{\st}}+\norm{\del{t}m}_{E^{\st}}
\end{align*}
for $0\leq \ell \leq 2\st$.
\end{enumerate}
\end{lem}
\begin{proof}
The inequalities all follow from a repeated application of the commutator estimate from Proposition \ref{stcomA} and the multiplication estimates from Theorem \ref{calcpropB} and Proposition \ref{stpropA}. We will only provide detailed proofs for the estimates from part (i) that involve $M^\alpha_\ell$. The remainder of the estimates can be established in a similar fashion.

By \eqref{DCthm1.2}, we see, for $0\leq \ell \leq 2\st-1$, that we can estimate $M^\alpha_\ell$ by
\begin{equation*}
\norm{M^\alpha_\ell}_{H^{\st-\frac{\ell}{2}}(\Omega)} \leq \norm{[\del{t}^\ell,b^{\alpha\beta}]\del{\beta}u}_{H^{\st-\frac{\ell}{2}}(\Omega)}+
\norm{\del{t}^\ell m^\alpha}_{H^{\st-\frac{\ell}{2}(\Omega)}}.
\end{equation*} 
Applying the commutator estimate from Proposition \ref{stcomA} (with $s_1=s$, $s_2=\st$ and $s_3=\st$) to the above inequality yields
\begin{equation*}
\norm{M^\alpha_\ell}_{H^{\st-\frac{\ell}{2}}(\Omega)} \lesssim \norm{\del{t}b}_{E^{s-\frac{1}{2},\ell-1}}\norm{\del{}u}_{E^{\st,\ell-1}}+ \norm{m}_{E^{\st,\ell}}
\lesssim \norm{b}_{E^{s,2\st}}\norm{u}_{\Ec^{\st+1}}+\norm{m}_{E^{\st}}
\end{equation*} 
for $0\leq \ell \leq 2\st-1$, which establishes the first estimate from part (i). 

Next, differentiating \eqref{DCthm1.2} gives
\begin{equation} \label{DClem2}
\del{t}M^\alpha_\ell = [\del{t}^\ell,\del{t}b^{\alpha\beta}]\del{\beta}u+[\del{t}^\ell,b^{\alpha\beta}]\del{\beta}\del{t}u+\del{t}^\ell \del{t}m^\alpha,
\end{equation}
which allows us to estimate $\del{t}M^\alpha_\ell$ by
\begin{equation*}
\norm{\del{t}M^\alpha_\ell}_{H^{\st-\frac{\ell}{2}}(\Omega)} \leq \norm{[\del{t}^\ell,\del{t}b^{\alpha\beta}]\del{\beta}u}_{H^{\st-\frac{\ell}{2}}(\Omega)}+
\norm{[\del{t}^\ell,b^{\alpha\beta}]\del{\beta} \del{t}u}_{H^{\st-\frac{\ell}{2}}(\Omega)}+
\norm{\del{t}^\ell \del{t}m^\alpha}_{H^{\st-\frac{\ell}{2}(\Omega)}}
\end{equation*}
for $0\leq \ell \leq 2\st-1$. Using the commutator estimate from Proposition \ref{stcomA} to estimate the first (with $s_1=s$, $s_2=\st$, and $s_3=\st$)
and second (with $s_1=s+\frac{1}{2}$, $s_2 = \st-\frac{1}{2}$, and $s_3=\st$) terms, we get that
\begin{align}
\norm{\del{t}M^\alpha_\ell}_{H^{\st-\frac{\ell}{2}}(\Omega)} &\lesssim  \norm{\del{t}^2 b}_{E^{s-\frac{1}{2},\ell-1}}\norm{\del{}u}_{E^{\st,\ell-1}}
+ \norm{\del{t}b}_{E^{s,\ell-1}}\norm{\del{}\del{t}u}_{E^{\st-\frac{1}{2},\ell-1}}+\norm{\del{t}m}_{E^{\st}} \notag \\
& \lesssim \norm{\del{t}b}_{E^{s,2\st}}\norm{u}_{\Ec^{\st+1}}+\norm{\del{t}m}_{E^{\st}}\label{DClem3}
\end{align}
for $0\leq \ell \leq 2\st-1$.

Differentiating \eqref{DClem2} again gives
\begin{equation}\label{DClem4}
\del{t}^2M^\alpha_\ell = [\del{t}^\ell,\del{t}^2 b^{\alpha\beta}]\del{\beta}u+2[\del{t}^\ell,\del{t}b^{\alpha\beta}]\del{\beta}\del{t}u+
[\del{t}^\ell,b^{\alpha\beta}]\del{\beta}\del{t}^2u+\del{t}^\ell \del{t}^2m^\alpha.
\end{equation}
Using this, we can bound $\del{t}^2M^\alpha_\ell$ by
\begin{align*}
\norm{\del{t}^2M^\alpha_\ell}_{H^{\st-\frac{\ell}{2}-1}(\Omega)} \lesssim & \norm{[\del{t}^\ell,\del{t}^2b^{\alpha\beta}]\del{\beta}u}_{H^{\st-\frac{\ell}{2}-1}(\Omega)}+
\norm{[\del{t}^\ell,\del{t}b^{\alpha\beta}]\del{\beta} \del{t}u}_{H^{\st-\frac{\ell}{2}-1}(\Omega)} \notag \\
& + \norm{[\del{t}^\ell,b^{\alpha\beta}]\del{\beta} \del{t}^2u}_{H^{\st-\frac{\ell}{2}-1}(\Omega)}+ \norm{\del{t}^\ell \del{t}^2 m^\alpha}_{H^{\st-\frac{\ell}{2}-1}(\Omega)}
\end{align*}
for $0\leq \ell \leq 2\st-2$. Noting from \eqref{r*def} that $(s-\frac{\ell}{2})^*=s-\frac{\ell}{2}-1$, $0\leq \ell \leq 2\st-2$, the same arguments that led to \eqref{DClem3} yields the inequality
\begin{equation} \label{DClem5}
\norm{\del{t}^2 M^\alpha_\ell}_{H^{(\st-\frac{\ell}{2})^*}(\Omega)} \lesssim  \norm{\del{t}b}_{E^{s,2\st}}\norm{u}_{\Ec^{\st+1}}+\norm{\del{t}m}_{E^{\st}}
\end{equation}
for $0\leq \ell \leq 2\st-2$.

Next, we use \eqref{DClem4} to estimate $\del{t}M^\alpha_{2\st-1}$ by
\begin{align*}
\norm{\del{t}^2M^\alpha_{2\st-1}}_{L^2(\Omega)} \lesssim & \norm{[\del{t}^{2\st-1},\del{t}^2b^{\alpha\beta}]\del{\beta}u}_{L^2(\Omega)}+
\norm{[\del{t}^{2\st-1},\del{t}b^{\alpha\beta}]\del{\beta} \del{t}u}_{L^2(\Omega)}\notag \\
& + \norm{[\del{t}^{2\st-1},b^{\alpha\beta}]\del{\beta} \del{t}^2u}_{L^2(\Omega)}+ \norm{\del{t}^{2\st-1}\del{t}^2 m^\alpha}_{L^2(\Omega)}.
\end{align*}
Using the commutator estimate from Proposition \ref{stcomA} to estimate the first (with $s_1=s-\frac{1}{2}$, $s_2=\st$ and $s_3=\st-\frac{1}{2}$), 
second (with $s_1=s$, $s_2=\st-\frac{1}{2}$ and $s_3=\st-\frac{1}{2}$) and third 
(with $s_1=s+\frac{1}{2}$, $s_2=\st-1$, $s_3=\st-\frac{1}{2}$) terms, we find that
\begin{align} 
\norm{\del{t}^2M^\alpha_{2\st-1}}_{L^2(\Omega)} &\lesssim 
  \norm{\del{t}^3 b}_{E^{s-1,2\st-2}}\norm{\del{}u}_{E^{\st,2\st-2}}+ \norm{\del{t}^2 b}_{E^{s-\frac{1}{2},2\st-2}}\norm{\del{}\del{t}u}_{E^{\st-\frac{1}{2},2\st-2}} 
 \notag \\ 
  & \hspace{2.0cm} +  \norm{\del{t} b}_{E^{s,2\st-2}}\norm{\del{}\del{t}^2 u}_{E^{\st-1,2\st-2}} + \norm{\del{t}m}_{E^{\st}} \notag \\
  &\lesssim   \norm{\del{t}b}_{E^{s,2\st}}\norm{u}_{\Ec^{\st+1}}+\norm{\del{t}m}_{E^{\st}}\label{DClem6}.
\end{align}
Together, \eqref{DClem3}, \eqref{DClem5} and \eqref{DClem6} establish the validity of the second estimate from part (i).
\end{proof}

Writing $\vec{M}$, $b^{0\Lambda}\del{t}^{\ell} u$ and $F_\ell$ as 
 $\vec{M}(t) = \vec{M}(0)+\int_0^t \vec{M}(\tau)\, d\tau$,  $(b^{0\Lambda}\del{t}^{\ell}u)(t) =(b^{0\Lambda}\del{t}^{\ell} u)(0)+\int_0^t (b^{0\Lambda}\del{t}^{\ell} u)(\tau)\, d\tau$, and $F_\ell(t) = F_\ell(0)+\int_0^t F_\ell(\tau)\, d\tau$, respectively, we see, after applying the appropriate norm to each of these expression and summing the result, that
\begin{align*}
 &\norm{\vec{M}_\ell(t)}_{H^{\st-\frac{\ell}{2}}(\Omega)}+  
\norm{(b^{0\Lambda}\del{t}^{\ell} u)(t)}_{H^{\st-\frac{\ell}{2}}(\Omega)} +
\norm{F_\ell(t)}_{H^{(\st-\frac{\ell}{2})^*}\!(\Omega)} \lesssim  \norm{\vec{M}_\ell(0)}_{H^{\st-\frac{\ell}{2}}(\Omega)}+ 
\norm{(b^{0\Lambda}\del{t}^{\ell} u)(0)}_{H^{\st-\frac{\ell}{2}}(\Omega)} \\
&\hspace{1.1cm}+\norm{F_\ell(0)}_{H^{(\st-\frac{\ell}{2})^*}\!(\Omega)} + \int_0^t 
\norm{\del{t}\vec{M}_\ell(\tau)}_{H^{\st-\frac{\ell}{2}}(\Omega)}+ 
\norm{\del{t}(b^{0\Lambda}\del{t}^{\ell} u)(\tau)}_{H^{\st-\frac{\ell}{2}}(\Omega)} +
\norm{\del{t}F_\ell(\tau)}_{H^{(\st-\frac{\ell}{2})^*}\!(\Omega)}\, d\tau
\end{align*}
for $0\leq \ell \leq 2\st-1$.
Estimating the right hand side using Lemma \ref{DClem}.(i)-(ii), we deduce that
\begin{gather}
\norm{\vec{M}_\ell(t)}_{H^{\st-\frac{\ell}{2}}(\Omega)}+  
\norm{(b^{0\Lambda}\del{t}^{\ell} u)(t)}_{H^{\st-\frac{\ell}{2}}(\Omega)} +
\norm{F_\ell(t)}_{H^{(\st-\frac{\ell}{2})^*}\!(\Omega)} \lesssim (\norm{b(0)}_{E^{s,2\st}}+\norm{\del{t}b(0)}_{E^{s,2\st}})\norm{u(0)}_{\Ec^{\st+1}} \notag \\
+ 
\norm{f(0)}_{E^{\st-1}}+\norm{\del{t}f(0)}_{E^{\st-1}} +  \norm{m(0)}_{E^{\st}}+ \norm{\del{t}m(0)}_{E^{\st}}
+ \int_0^t (\norm{b(\tau)}_{E^{s,2\st}}+\norm{\del{t}b(\tau)}_{E^{s,2\st}})\norm{u(\tau)}_{\Ec^{\st+1}}\notag \\
+ \norm{f(\tau)}_{E^{\st-1}}+
\norm{\del{t}f(\tau)}_{E^{\st-1}} +  \norm{m(\tau)}_{E^{\st}}+\norm{\del{t}m(\tau)}_{E^{\st}}
\, d\tau \label{DCthm6}
\end{gather}
for $0\leq \ell \leq 2 \st-2$, and
\begin{gather}
\norm{\vec{M}_{2\st-1}(t)}_{H^{\st-\frac{{2\st-1}}{2}}(\Omega)}+  
\norm{(b^{0\Lambda}\del{t}^{{2\st-1}} u)(t)}_{H^{\st-\frac{{2\st-1}}{2}}(\Omega)} 
 \lesssim (\norm{b(0)}_{E^{s,2\st}}+\norm{\del{t}b(0)}_{E^{s,2\st}})\norm{u(0)}_{\Ec^{\st+1}} \notag \\
+  \norm{m(0)}_{E^{\st}}+ \norm{\del{t}m(0)}_{E^{\st}}
+ \int_0^t (\norm{b(\tau)}_{E^{s,2\st}}+\norm{\del{t}b(\tau)}_{E^{s,2\st}})\norm{u(\tau)}_{\Ec^{\st+1}} +  \norm{m(\tau)}_{E^{\st}}+\norm{\del{t}m(\tau)}_{E^{\st}}
\, d\tau .\label{DCthm7}
\end{gather}
By similar arguments and Lemma \ref{DClem}.(iii)-(iv), we find also that
\begin{gather}
\norm{F_{2\st-1}(t)}_{L^2(\Omega)} \lesssim \norm{b}_{H^s(\Omega)}
E(\del{t}^{2\st}u(0)) + (\norm{b(0)}_{E^{s,2\st}}+\norm{\del{t}b(0)}_{E^{s,2\st}})\norm{u(0)}_{\Ec^{\st+1}} 
+ 
\norm{f(0)}_{E^{\st-1}}\notag \\
+\norm{\del{t}f(0)}_{E^{\st-1}} +  \norm{m(0)}_{E^{\st}}+ \norm{\del{t}m(0)}_{E^{\st}}
+ \int_0^t (\norm{b(\tau)}_{E^{s,2\st}}+\norm{\del{t}b(\tau)}_{E^{s,2\st}})\norm{u(\tau)}_{\Ec^{\st+1}}\notag \\
+ \norm{f(\tau)}_{E^{\st-1}}+
\norm{\del{t}f(\tau)}_{E^{\st-1}} +  \norm{m(\tau)}_{E^{\st}}+\norm{\del{t}m(\tau)}_{E^{\st}}
\, d\tau. \label{DCthm8}
\end{gather}

Using \eqref{DCthm3} to bound the term $E(\del{t}^\ell u)$ on the right hand side of \eqref{DCthm5}, for $0\leq \ell \leq 2\st-1$,
and the estimates \eqref{DCthm6}-\eqref{DCthm8} and those from Lemma \ref{DClem}.(iv) to bound right hand side
of the resulting inequality, we see, after summing over
$\ell$ from $0$ to $2\st-1$ and adding the result to \eqref{DCthm5} with $\ell=2\st$, that
$u$ satisfies the energy estimate 
 \begin{equation} \label{DCthm9}
\norm{u(t)}_{\Ec^{\st+1}} \leq C(\vec{\kappa},\mu,\norm{b(t)}_{H^s(\Omega)})\Bigl( \norm{u(0)}_{\Ec^{\st+1}} + 
\alpha_2(0) +  \int^t_0 \alpha_1(\tau)\norm{u(\tau)}_{\Ec^{\st+1}} + \alpha_2(\tau) \, d\tau \Bigr) 
\end{equation}
where
\begin{align*}
\alpha_1(t) &= 1+\norm{b(t)}_{E^{s,2\st}}+\norm{\del{t}b(t)}_{E^{s,2\st}}
\intertext{and}
\alpha_2(t) &=  \norm{m(t)}_{E^{\st}}+\norm{\del{t}m(t)}_{E^{\st}}+\norm{f(t)}_{E^{(\st-1)^*}} +
 \norm{\del{t}f(t)}_{E^{(\st-1)^*}}+\norm{\del{t}^{2\st}f(t)}_{L^2(\Omega)}.
\end{align*}

Thus far, we have established the existence of solutions that satisfy the energy estimate under that assumption that the coefficients in the wave equation are smooth an the initial data satisfies the compatibility conditions in the sense of Koch to
order $2\st+1$.
Existence for the general case, where the coefficients satisfy $\,b^{\alpha\beta}, \del{t}b^{\alpha\beta}  \in X_{T}^{s,2\st}(\Omega,\Mbb{N})$,
$\,f,\del{t}f\in X_{T}^{\st-1}(\Omega,\Rbb^N)$, $\del{t}^{2\st}f\in L^\infty([0,T],L^2(\Omega,\Rbb^N))$,
$\,\ell^{\alpha}, \del{t}\ell^{\alpha}  \in X_{T}^{\st}(\Omega,\Rbb^{N})$, and the initial conditions satisfy the compatibility
conditions to order $2\st+1$ in the sense of Definition \eqref{dircompat}, 
follows from an approximation and limiting argument, the details of which we leave to the interested reader. Furthermore,  uniqueness of solutions follows
directly from the energy estimate \eqref{DCthm9} in the usual way.
\end{proof}

\subsection{Acoustic boundary conditions\label{AWEI}}

In this section, we consider the initial boundary value problem (IBVP) for systems of wave equations with acoustic boundary conditions that are of the following form:
\begin{align}
\del{\alpha}\bigl(b^{\alpha\beta}\del{\beta}u+m^\alpha\bigr) &= f && \text{in $\Omega_T$,}
\label{linAa.1}\\
\nu_\alpha\bigl(b^{\alpha\beta}\del{\beta} u + m^\alpha) & = q\del{t}^2 u + p \del{t} u + g
 && \text{in $\Gamma_T$,} \label{linAa.2}\\
(u,\del{t}u) &= (u_0,u_1) && \text{in $\Omega_0$}, \label{linAa.3}
\end{align}
where
\begin{enumerate}[(i)]
\item $\nu_\alpha = \delta^\Sigma_\alpha \nu_\Sigma$ where $\nu_\Sigma$ is a time-independent outward pointing conormal to $\del{}\Omega$,
\item $u=u(t,x)$, $m^\alpha= m^\alpha(t,x)$, $f=f(t,x)$ and $g=g(t,x)$ are $\Rbb^N$-valued maps,
\item  the $\Mbb{N}$-valued maps $b^{\alpha\beta}=b^{\alpha\beta}(t,x)$
satisfy the symmetry condition \eqref{bsym}
in $\Omega_T$ 
and there exists constants  $\kappa_0$, $\kappa_1>0$ and $\mu\geq 0$ such that the inequalities \eqref{b00lb}
and
\begin{equation} \label{coercN}
\ip{\del{\Lambda}u}{b^{\Lambda\Sigma}(t)\del{\Sigma}u}_\Omega \geq \kappa_1 \norm{u}^2_{H^1(\Omega)} - \mu \norm{u}^2_{L^2(\Omega)},
\qquad \forall \, (t,u)\in [0,T]\times   H^1(\Omega,\Rbb^N), 
\end{equation}
hold,
\item  and $p=p(t,x)$ and $q=q(t,x)$ are $\Mbb{N}$-valued maps where
$q$
satisfies
\begin{equation}\label{qconditions}
q^{\tr} = q, \quad q \leq 0, \quad -\frac{1}{\gamma} q \leq q^2 \leq -\gamma q \AND
\text{rank}(q)= N_q \qquad \text{in $\Gamma_T$}
\end{equation}
for some positive constants $\gamma>0$ and $0 \leq N_q \leq N$.
\end{enumerate}

\begin{rem} \label{coercNrem}$\;$
\begin{enumerate}[(a)]
\item 
The coercive condition \eqref{coercN} is known to be equivalent to the matrix $b^{\Sigma\Lambda}$
being strongly elliptic, see \eqref{strongellip}, in $\overline{\Omega}$ and satisfying the strong complementing
condition on the boundary $\del{}\Omega$. For a proof of this equivalence, see
Theorem 3 from Section 6 of \cite{SimpsonSpector:1987}.
\item Letting $\Pbb_q$ and $\Pbb_q^\perp$ denote the
projection onto $\Ran(q)$ and its orthogonal complement, we can decompose $\Rbb^N$ as
\begin{equation*}
\Rbb^N = \Pbb_q\Rbb^N\!\oplus \Pbb_q^\perp\Rbb^N
\end{equation*}
and write $q$ as
\begin{equation*}
q = \Pbb_q q \Pbb_q.
\end{equation*}
\end{enumerate}
\end{rem}

\subsubsection{Weak solutions\label{weaksol}}
To define a weak solution of the system \eqref{linAa.1}-\eqref{linAa.3}, we must supplement the initial conditions \eqref{linAa.3} with an additional one
given by
\begin{equation}\label{linAa.4}
\Pbb_q \del{t}u = U_1 \quad \text{in $\Gamma_0$}.
\end{equation}
This initial condition determines the piece of the time derivative $\del{t}u$ that lies in the range of $q|_{t=0}$ when restricted to the boundary. However, because we can choose initial data for weak solutions
such that the time derivative $\del{t}u|_{t=0}$ lies in $L^2(\Omega)$, the trace of $\del{t}u|_{t=0}$ on the boundary $\Gamma$ will, in general, not be defined. Thus the initial condition \eqref{linAa.4}, as stated, cannot be literally true and must be interpreted in a suitable weak sense; see \eqref{weaksoldef2a} below. With that said, for regular 
enough initial data, e.g. $\del{t}u|_{t=0}\in H^1(\Omega)$, \eqref{linAa.4} can be used to define $U_1$ by taking the trace of $\Pbb_q \del{t}u$ on the boundary $\del{}\Omega$ at $t=0$.

\begin{Def} \label{weaksoldef}
A pair $(u,U)\in H^1(\Omega_T,\Rbb^N)\times L^2(\Gamma_T,\Rbb^N)$ is called a \emph{weak solution} of 
\eqref{linAa.1}-\eqref{linAa.3} and \eqref{linAa.4} if  $(u,U)$ define maps $u \, : \, [0,T]\rightarrow H^1(\Omega,\Rbb^N)$, $\del{t}u \, :
\, [0,T]\rightarrow L^2(\Omega,\Rbb^N)$ and $U \,:\, [0,T] \longrightarrow L^2(\Gamma,\Rbb^N)$ that
satisfy\footnote{We use the standard notation
$\rightharpoonup$ to denote weak convergence.}
\begin{equation*}
(u(t),\del{t}u(t)) \rightharpoonup (u_0,u_1) \quad\text{in $H^1(\Omega,\Rbb^N)\times L^2(\Omega,\Rbb^N)$},  \AND U(t) \rightharpoonup U_1 \quad\text{in $L^2(\Gamma,\Rbb^N)$ }
\end{equation*}
as $t\searrow 0$,
\begin{equation*}
U \in \Ran(q)\hspace{0.4cm} \text{in $\Gamma_T$,}
\end{equation*}
and
\begin{equation}\label{weaksoldef2a}
\ip{qU}{\phi}_{\Gamma_T} = - \ip{\del{t}q u}{\phi}_{\Gamma_T}-\ip{q u}{\del{t}\phi}_{\Gamma_T}
\end{equation}
and
\begin{equation}\label{weaksoldef3}
\ip{b^{\alpha\beta}\del{\beta}u + m^\alpha}{\del{\alpha} \phi}_{\Omega_T} +
\ip{(\del{t}q-p)\del{t}u}{\phi}_{\Gamma_T}-\ip{g}{\phi}_{\Gamma_T}+\ip{qU}{\del{t}\phi}_{\Gamma_T}= -\ip{f}{\phi}_{\Omega_T}
\end{equation}
for all $\phi \in C^1_0\bigl([0,T],C^1(\overline{\Omega},\Rbb^N)\bigr)$.
\end{Def}

\begin{rem} \label{weakrem}$\;$

\begin{itemize}
\item[(a)]
As in \cite{Koch:1993}, the boundary
terms $\ip{g}{\phi}_{\Gamma_T}$ and $\ip{(\del{t}q-p)\del{t}u}{\phi}_{\Gamma_T}$ are defined via the
expressions\footnote{For sufficiently differentiable vector valued and
matrix valued maps $\{u,\phi\}$ and $S$, respectively, such that $\phi|_t=\phi|_{t=T}=0$, the identities
\begin{align*}
\ip{u}{\phi}_{\Gamma_T} &= \int_{\Omega_T}\del{\alpha}\bigl[ \nu^\alpha \ipe{u}{\phi}\bigr]\, d^{n+1} x =
\ip{\nu(u)+\del{\Sigma}\nu^\Sigma u}{\phi}_{\Omega_T} + \ip{u}{\nu(\phi)}_{\Omega_T}, \\
0&=\int_{\Omega_T}\del{\beta}\bigl[\delta^\beta_0\ipe{S\nu(u)}{\phi}\bigr]\, d^{n+1} x =
\ip{\del{t}S\nu(u)+S\del{t}\nu(u)}{\phi}_{\Omega_T}+\ip{S\nu(u)}{\del{t}\phi}_{\Omega_T}
\end{align*}
follow from the divergence theorem. The second identity together with one more application of
the divergence theorem then yields
\begin{equation*}
\ip{S\del{t}u}{\phi}_{\Gamma_T} = \int_{\Omega_T}\del{\beta}\bigl[\nu^\beta\ipe{S\del{t}u}{\phi}\bigr]\, d^{n+1} x = \ip{\nu(S)\del{t}u-\del{t}S\nu(u) +
\del{\alpha}\nu^{\alpha}
S\del{t}u}{\phi}_{\Omega_T}+ \ip{S\del{t}u}{\nu(\phi)}_{\Omega_T}-
\ip{S\nu(u)}{\del{t}\phi}_{\Omega_T}.
\end{equation*}
}
\begin{equation}\label{weakrem2a}
\ip{g}{\phi}_{\Gamma_T} = \ip{\nu(g)+\del{\Sigma}\nu^\Sigma g}{\phi}_{\Omega_T} + \ip{g}{\nu(\phi)}_{\Omega_T},
\end{equation}
and
\begin{align}
&\ip{(\del{t}q-p)\del{t}u}{\phi}_{\Gamma_T} = \ip{\nu(\del{t}q-p)\del{t}u-\del{t}(\del{t}q-p)\nu(u)}{\phi}_{\Omega_T} \notag \\
 &\hspace{0.2cm} +
\ip{\del{\alpha}\nu^{\alpha}
(\del{t}q-p)\del{t}u}{\phi}_{\Omega_T}
+ \ip{(\del{t}q-p)\del{t}u}{\nu(\phi)}_{\Omega_T}-
\ip{(\del{t}q-p)\nu(u)}{\del{t}\phi}_{\Omega_T}, \label{weakrem2.1}
\end{align}
respectively,
where $\nu(\cdot) = \nu^\alpha\del{\alpha}(\cdot)$, $\nu^\alpha = \delta^{\alpha \Sigma}\nu_\Sigma$, and
$\nu_\Sigma$ is any smooth extension to $\Omega$ of the outward pointing unit normal to $\del{}\Omega$.
\item[(ii)] The condition \eqref{weaksoldef2a}
implies that $u$ \textit{weakly} satisfies
\begin{equation*}
q U = q \del{t} u \hspace{0.4cm} \text{in $\Gamma_T$,}
\end{equation*}
where here, we are again defining the
boundary terms on the right hand side of \eqref{weaksoldef2a} using the same type of formula
as \eqref{weakrem2a}.
\end{itemize}
\end{rem}

The existence and uniqueness of weak solutions is a consequence of the following theorem, which follows easily from
a special case of Theorem 7.5 from \cite{Oliynyk:Bull_2017}.

\begin{thm} \label{weakthm}
Suppose $s>n/2$, $0\leq \lambda \leq \epsilon \leq s$, $u_0\in H^1(\Omega,\Rbb^N)$, $u_1\in L^2(\Omega,\Rbb^N)$, $U_1 \in L^2(\del{}\Omega,\Rbb^N)$
and satisfies $U_1\in \Ran(q|_{t=0})$, $m=(m^\alpha) \in W^{1,2}\bigl([0,T],L^2(\Omega,\Rbb^N)\bigr)$,
$f \in L^2(\Omega_T,\Rbb^N)$, $\, p,q,\del{t}q \in W^{1,\infty}(\Omega_T,\Mbb{N})$, 
$b^{\alpha\beta}\in W^{1,\infty}\bigl([0,T],L^\infty(\Omega,\Mbb{N})\bigr)$, $b^{\alpha\beta}$
satisfy the coercive condition \eqref{coercN}, the symmetry condition \eqref{bsym} and inequality \eqref{b00lb},  $q$ satisfies \eqref{qconditions}, $g$ can be written
as
\begin{equation*}
g = k^\alpha\del{\alpha} \theta + G
\end{equation*}
where $k\in  L^\infty([0,T],H^{s+1-\ep}(\Omega))$, $\del{t}k\in  L^\infty([0,T],H^{s-\lambda}(\Omega))$, $\theta\in  L^\infty([0,T],H^{1+\ep}(\Omega))$,
$\del{t}\theta\in  L^\infty([0,T],H^{1+\lambda}(\Omega))$ and $G\in H^1(\Omega_T,\Rbb^N)$,
and $p,q$ satisfy the inequality
\begin{equation*}
p-\Half \del{t}q-\chi q \leq 0 \hspace{0.5cm}\text{in $\Gamma_T$}
\end{equation*}
for some $\chi \in \Rbb$.
Then
there exists a unique weak solution $(u,U)$ to the IBVP consisting of \eqref{linAa.1}-\eqref{linAa.3} and \eqref{linAa.4}, which possesses 
the additional regularity
$(u,U) \in \bigcap_{j=0}^1 C^j\bigl([0,T],H^{1-j}(\Omega,\Rbb^N)\bigr)
\times C^0\bigl([0,T],L^2(\del{}\Omega,\Rbb^N)\bigr)$. Morever,
this solution satisfies the energy estimate
\begin{align*}
E(t) \leq& E(0) + C \int_0^t\bigl(1+\norm{\del{t}b(\tau)}_{L^\infty(\Omega)}+|\chi|
\bigr)
E(\tau) 
+ \norm{\del{t}m(\tau)}_{L^2(\Omega)}  + \norm{f(\tau)}_{L^2(\Omega)}+
\norm{G(\tau)}_{H^1(\Omega)} \notag \\
&+ \norm{\del{t}G(\tau)}_{L^2(\Omega)}
+ \bigl(\norm{k(\tau)}_{H^{s+1-\ep}(\Omega)}+  \norm{\del{t}k(\tau)}_{H^{s-\lambda}(\Omega)}\bigr)\bigl(\norm{\del{t}\theta(\tau)}_{H^{1+\lambda}(\Omega)} +
\norm{\theta(\tau)}_{H^{1+\ep}(\Omega)}\bigr) \, d\tau
\end{align*}
for $0\leq t\leq T$, where $C=C\bigl(\vec{\kappa},\mu,\gamma)$, $\vec{\kappa}=(\kappa_0,\kappa_1)$, 
\begin{align*}
E(t)^2 = \frac{1}{2}\ip{\del{\Lambda}u(t)}{b^{\Lambda\Sigma}(t)\del{\Sigma}u(t)}_{\Omega}&-\frac{1}{2}\ip{\del{t}u(t)}{b^{00}(t)\del{t}u(t)}_{\Omega}
- \frac{1}{2}\ip{U(t)}{q(t)U(t)}_{\del{}\Omega} \notag \\
& +\ip{\del{\Lambda}u(t)}{M^\Lambda(t)}_{\Omega}+ \frac{\mu}{2}\norm{u(t)}_{L^2(\Omega)}^2 +
\frac{2}{\kappa_1}\norm{\vec{M}(t)}^2_{L^2(\Omega)},
\end{align*}
\begin{equation*} 
M^\Lambda=m^\Lambda + \frac{2}{|\nu|^2}\nu^{[\Sigma} k^{\Lambda]}\del{\Sigma}\theta, \qquad \vec{M}=(M^\Lambda),
\end{equation*}
and the energy norm $E(t)$ is bounded below by
\begin{equation*}
\norm{(u,U)}^2_E = \norm{u}_{H^1(\Omega)}^2 + \norm{\del{t}u}^2_{L^2(\Omega)} + \ip{U}{(-qU}_{\del{}\Omega} 
\end{equation*}
according to
\begin{equation*}
E^2 \geq \min\biggl\{\frac{\kappa_0}{2},\frac{\kappa_1}{8},\frac{1}{2}\biggr\}\norm{(u,U)}^2_E,
\end{equation*}
and above by
\begin{equation*}
E^2\leq C(k_1,\mu)\Bigl(\bigl(1+\norm{b}_{L^\infty(\Omega)}\bigr)\norm{(u,U)}^2_E+
\norm{\vec{m}}^2_{L^2(\Omega)}+\norm{k}^2_{H^{s+1-\ep}(\Omega)}\norm{\theta}^2_{H^{1+\ep}(\Omega)}\Bigr).
\end{equation*}

\end{thm}
\begin{proof}
Since, $\nu_\alpha k^\alpha = k^0=0$ and $g = k^\alpha\del{\alpha} \theta + G$, we can write the boundary condition \eqref{linAa.2} as 
\begin{equation} \label{weakthm5}
\nu_\alpha \bigl(b^{\alpha\beta}\del{\beta}u + M^\alpha\bigr) = q\del{t}^2 u + p\del{t} u + G \hspace{0.5cm}\text{in $\Gamma_T$,}
\end{equation}
where
\begin{equation*} 
M^\alpha=m^\alpha + \frac{2}{|\nu|^2}\nu^{[\beta} k^{\alpha]}\del{\beta}\theta.
\end{equation*}
We also observe that the wave equation \eqref{linAa.3} can be expressed as
\begin{equation}\label{weakthm6}
\del{\alpha}\bigl(b^{\alpha\beta}\del{\beta}u+ M^\alpha\bigr) = F,
\end{equation}
where
\begin{equation*}
F= f+ \del{\alpha}\biggl(\frac{2}{|\nu|^2}\nu^{[\beta} k^{\alpha]}\biggr)\del{\beta}\theta.
\end{equation*}
From the fractional multiplication inequality, see Theorem \ref{calcpropB}, we see, for $0\leq \lambda \leq  \ep \leq s$, that
\begin{gather}
\norm{F}_{L^2(\Omega)} \lesssim \norm{f}_{L^2(\Omega)}+ \norm{k}_{H^{s+1-\ep}(\Omega)}
\norm{\theta}_{H^{1+\ep}(\Omega)}, \label{weakthm7}\\
\norm{M}_{L^2(\Omega)}\lesssim 
\norm{m}_{L^2(\Omega)}+\norm{k}_{H^{s+1-\ep}(\Omega)}
\norm{\theta}_{H^{1+\ep}(\Omega)} \label{weakthm8}
\intertext{and}
\norm{\del{t}M}_{L^2(\Omega)} \lesssim 
\norm{\del{t}m}_{L^2(\Omega)} +  \norm{k}_{H^{s+1-\ep}(\Omega)}\norm{\del{t}\theta}_{H^{1+\lambda}(\Omega)} + 
\norm{\del{t}k}_{H^{s-\lambda}(\Omega)}
\norm{\theta}_{H^{1+\ep}(\Omega)}\label{weakthm9}.
\end{gather}

The existence and uniqueness of a weak solution to the system of wave equations \eqref{weakthm6} with acoustic boundary conditions \eqref{weakthm5} is then
an immediate consequence of Theorem 7.5 from \cite{Oliynyk:Bull_2017}\footnote{We are actually employing a variation of Theorem 7.5 from  \cite{Oliynyk:Bull_2017}  where $E(t)$ is replaced by its square root. This replacement follows from
modifying the proof by dividing the differential energy inequality (7.48) from \cite{Oliynyk:Bull_2017} by $\sqrt{E_M}$ to yield a differential energy inequality involving $\sqrt{E_M}$
instead of $E_M$. The remainder of the proof follows essentially unchanged.}. The stated energy estimates and bounds on the energy norm $E(t)$ also follow directly
from this theorem and the estimates \eqref{weakthm7}-\eqref{weakthm9}. 
\end{proof}

\subsubsection{Higher regularity\label{highreg}}

As it stands, Theorem \ref{weakthm} is not useful for developing an existence and uniqueness theory for non-linear
wave equations with acoustic boundary conditions. For application to non-linear problems, Theorem \ref{weakthm} must be improved to establish the existence and uniqueness of
solutions with higher regularity. The key to doing this is to choose initial data that satisfies certain \textit{compatibility conditions},    
which are made precise in the following definition.

\begin{Def} \label{acouscompat}
Given $\st=\kt/2$ with $\kt\in \mathbb{Z}$, we say the the initial data\footnote{See \eqref{melldef} for a definition of $m(s,\ell)$.}
\begin{equation*}
(u_0,u_1,U_{2\st+1}) \in H^{\st+1}(\Omega,\Rbb^N)\times H^{\st+1-\frac{m(\st+1,1)}{2}}(\Omega,\Rbb^N)\times
L^2(\del{}\Omega,\Rbb^N)
\end{equation*}
for the IBVP \eqref{linAa.1}-\eqref{linAa.3} satisfies the \textit{compatibility conditions to order $2\st+1$} if
the higher formal
time derivatives
\begin{equation*}
u_\ell = \del{t}^\ell u |_{\Omega_0}, \quad \ell=2,3,\ldots,2\st+1,
\end{equation*} 
which are  generated from the initial data by differentiating the wave equation \eqref{linAa.1} formally with respect to $t$ the required number of times
and setting $t=0$, satisfy 
\begin{equation*}
u_\ell \in H^{\st+1-\frac{m(\st+1,\ell)}{2}}(\Omega,\Rbb^N), \quad 0\leq \ell \leq 2\st+1,
\end{equation*}
and
\begin{equation*}
\del{t}^\ell \bigl(\nu_\alpha (b^{\alpha\beta}\del{\beta}u +m^\alpha)-q\del{t}^2 u -p\del{t}u-g\bigr)\bigl|_{\Gamma_0} = 0, \quad \ell=0,1,\ldots,2\st-1,
\end{equation*}
where the $(2\st+1)$ formal time derivative of $u$ restricted to the boundary $\del{}\Omega$ at $t=0$ is given by $U_{2\st+1}$, 
that is\footnote{In general, the expression \label{acouscompat1} has to be taken as a definition and not as the restriction of $u_{2\st+1}$ to the boundary
$\del{}\Omega_0$ since the trace of $u_{2\st+1}$ on the boundary is not necessarily well defined due to our assumption that
$u_{2\st+1}\in L^2(\Omega,\Rbb^N)$.}
\begin{equation*}
\del{0}^{2\st+1} u|_{\Gamma_0} = U_{2\st+1}. 
\end{equation*}
\end{Def}

\begin{thm} \label{LAWthm}
Suppose $T>0$, $s>n/2$, $s=k/2$ and $\st=\kt/2$ for $k,\kt\in \Zbb_{\geq 1}$ where $\kt\leq k$, $\,b^{\alpha\beta}, 
\del{t}b^{\alpha\beta}  \in  X^{s,2\st}_T(\Omega,\Mbb{N})$, $\,p,q, \del{t}p,\del{t}q  \in X_{T}^{s,2\st-2}(\Omega,\Mbb{N})$, 
$\del{t}q\in X_T^{s+\frac{1}{2},2\st-1}(\Omega,\Mbb{N})$, 
$\,f,\del{t}f\in X_{T}^{\st-1}(\Omega,\Rbb^N)$, $\del{t}^{2\st}f\in L^\infty([0,T],L^2(\Omega,\Rbb^N))$, $\,m^{\alpha}, \del{t}m^{\alpha}  \in X_{T}^{\st}(\Omega,\Rbb^{N})$, $\,g,\del{t}g \in X_{T}^{\st,2\st-2}(\Omega,\Rbb^{N})$,
the initial data 
\begin{equation*}
(u_0,u_1,U_{2\st+1})\in 
 H^{\st+1}(\Omega,\Rbb^N)\times  H^{\st+1-\frac{m(\st+1,1)}{2}}(\Omega,\Rbb^N) \times  L^2(\del{}\Omega,\Rbb^N)
\end{equation*}
satisfies the compatibility conditions to
order $2\st+1$, the matrices $b^{\alpha\beta}$ satisfy \eqref{bsym} and the inequalities \eqref{b00lb} and \eqref{coercN} for some constants $\kappa_0,\kappa_1>0$ and $\mu\geq 0$, the matrix $q$ satisfies the relations \eqref{qconditions} for some constant $\gamma >0$ and the inequalities
\begin{equation*}
p + \biggl(\ell-\frac{1}{2} \biggr)\del{t}q - \chi_\ell q \leq 0 \hspace{0.5cm}\text{in $\Gamma_T$}
\end{equation*}
for $0\leq \ell \leq 2\st$ for some constants $\chi_\ell \in \Rbb$, and there exists maps
\begin{gather*} 
\theta_1 \in 
\bigcap_{j=0}^1 W^{j,\infty}([0,T],H^{s-\st+\frac{3-j}{2}}(\Omega,\Rbb^N)),\quad  \theta_2 \in\bigcap_{j=0}^1 W^{j,\infty}([0,T],H^{\frac{3-j}{2}}(\Omega,\Rbb^N)), \\
h_1 \in \bigcap_{j=0}^1 W^{j,\infty}([0,T],H^{s-\st+1-j}(\Omega,\Rbb^N)),\quad h_2 \in \bigcap_{j=0}^1 W^{j,\infty}([0,T],H^{1-j}(\Omega,\Rbb^N))
\intertext{and}
k_a^\Sigma\in  
\bigcap_{j=0}^1 W^{j,\infty}(([0,T],H^{s+\frac{1-j}{2}}(\Omega,\Mbb{N})), \quad a=1,2,
\end{gather*}
 which satisfy $\nu_\Sigma k_a^\Sigma =0$, such that
\begin{equation*}
\del{t}^{2\st}p=k^\Sigma_1\del{\Sigma}\theta_1+h_1  \AND
\del{t}^{2\st}g=k^\Sigma_2\del{\Sigma}\theta_2 + h_2.
\end{equation*}
Then there exists maps
\begin{equation*}
(u,U)\in C\Xc^{\st+1}_T(\Omega,\Rbb^N)\times C([0,T],L^2(\Omega,\Rbb^N))
\end{equation*}
that determine the unique solution to the IBVP \eqref{linAa.1}-\eqref{linAa.3} satisfying the
additional property that $U|_{t=0}=U_{2\st+1}$ and the pair
$(\del{t}^{2\st}u,U)$ defines a weak solution solution of the linear wave equation obtained by differentiating \eqref{linAa.1}-\eqref{linAa.2} $2\st$-times with respect to $t$. Moreover, the solution $(u,U)$
satisfies the energy estimate
\begin{align*}
\nnorm{(u(t),U(t))}_{\st+1} \leq
C &\biggl(\nnorm{(u(0),U(0))}_{\st+1} \\
 &+\alpha_2(0)+
 \int_0^t  \alpha_1(\tau)\nnorm{(u(\tau),U(\tau))}_{\st+1}+\alpha_2(\tau)\,d\tau \biggr)
\end{align*}
where\footnote{See \eqref{r*def} for the definition of $(\st-1)^*$ and $(2\st-2)^*$.} $C = C\bigl(\vec{\kappa},\mu,\gamma,\vec{\chi},\alpha_3(t)\bigr)$, $\vec{\chi}=(\chi_\ell)$, $\vec{\kappa}=(\kappa_0,\kappa_1)$,
\begin{equation*}
\nnorm{(u(t),U(t))}_{\st+1}^2 = \norm{u(t)}_{\Ec^{\st+1}}^2 + \ip{U(t)}{(-q)U(t)}_{\del{}\Omega},
\end{equation*}
\begin{align*}
\alpha_1(t)  =& 1+\norm{b(t)}_{H^{s}(\Omega)}+\norm{\del{t}b(t)}_{E^{s,2\st}}
 + \norm{q(t)}_{H^{s}(\Omega)}+
\norm{\del{t}q}_{E^{s+\frac{1}{2},2\st-1}}+\norm{p(t)}_{H^{s}(\Omega)}\\
 &+\norm{\del{t}p}_{E^{s,(2\st-2)*}} 
 + \norm{h_1(t)}_{H^{s-\st+1}(\Omega)}+ \norm{\del{t} h_1(t)}_{H^{s-\st}(\Omega)} \\
&+ \Bigl(\norm{k_1}_{H^{s+\frac{1}{2}}(\Omega)}+ \norm{\del{t}k_1}_{H^{s}(\Omega)}\Bigr)\Bigl(\norm{\del{t}\theta_1(t)}_{H^{s-\st+1}(\Omega)} +
\norm{\theta_1(t)}_{H^{s-\st+\frac{3}{2}}(\Omega)}\Bigr),\\
\alpha_2(t) =& \norm{m(t)}_{H^{\st}(\Omega)}+\norm{\del{t}m(t)}_{E^{\st}}+\norm{f(t)}_{H^{(\st-1)^*}(\Omega)} +
 \norm{\del{t}f(t)}_{E^{(\st-1)^*}}+\norm{\del{t}^{2\st}f(t)}_{L^2(\Omega)} \\
&+\norm{g(t)}_{H^{\st}(\Omega)}+\norm{\del{t}g(t)}_{E^{\st,(2\st-2)^*}}
+ \norm{h_2(t)}_{H^{1}(\Omega)}+ \norm{\del{t} h_2(t)}_{L^{2}(\Omega)}\\
&+ \Bigl(\norm{k_2(t)}_{H^{s+\frac{1}{2}}(\Omega)}+\norm{\del{t}k_2(t)}_{H^{s}(\Omega)}\Bigr)
\Bigl(\norm{\theta_2(t)}_{H^{\frac{3}{2}}(\Omega)}+\norm{\del{t}\theta_2(t)}_{H^{1}(\Omega)}\Bigr)
\intertext{and}
\alpha_3(t) =& 1+\norm{b(0)}_{H^{s}(\Omega)}+ \norm{b(t)}_{H^{s}(\Omega)}+ \norm{p(t)}_{H^{s}(\Omega)}+  \norm{q(t)}_{H^{s}(\Omega)}+ \norm{\del{t}q(t)}_{H^{s}(\Omega)}.
\end{align*}
\end{thm}
\begin{proof}
We begin by assuming that $b^{\alpha\beta}$, $m^\alpha$, $f$, $g$, $q$, $p$  $\in$ $C^\infty(\overline{\Omega}_T)$ and the initial
data satisfies the compatibility conditions to order $2s+3$.
Since $s>n/2$ by assumption, it follows from Theorem 7.16 of \cite{Oliynyk:Bull_2017} that there exists a map
\begin{equation*}
(u,U) \in C\Xc_T^{s+2}(\Omega,\Rbb^N)\times  C([0,T],L^2(\del{}\Omega,\Rbb^N))
\end{equation*}
that determines the unique solution to the IBVP \eqref{linAa.1}-\eqref{linAa.3} satisfying the
additional property that $U|_{t=0}=U_{2s+3}$ and the pair
$(\del{t}^{2s+2}u,U)$ defines a weak solution solution of the linear wave equation obtained by differentiating \eqref{linAa.1}-\eqref{linAa.2} $2s$-times with respect to $t$.

Differentiating \eqref{linAa.1}-\eqref{linAa.2} $\ell$ times with respect to $t$, where $0\leq \ell \leq 2\st$, gives
\begin{align}
\del{\alpha}\bigl(b^{\alpha\beta}\del{\beta}\del{t}^\ell u+M^\alpha_\ell\bigr) &= \del{t}^\ell f && \text{in $\Omega_T$,}
\label{linBa.1}\\
\nu_\alpha\bigl(b^{\alpha\beta}\del{\beta} \del{t}^\ell u + M^\alpha_\ell) & = q\del{t}^2 \del{t}^\ell u + 
P_\ell\del{t} \del{t}^\ell u + G_\ell
 && \text{in $\Gamma_T$,} \label{linBa.2}
\end{align}
where $M^\alpha_\ell$ is as defined previously by \eqref{DCthm2}, 
\begin{align}
P_\ell &= \ell \del{t}q + p \label{LAWthmP20}
\intertext{and}
G_\ell &= \sum_{r=0}^{\ell-2}\binom{\ell}{r}\del{t}^{\ell-r}q\del{t}^{2+r}u + [\del{t}^\ell,p]\del{t}u + \del{t}^\ell g.
\label{LAWthmP21}
\end{align}
Since, by assumption, $P_\ell -\frac{1}{2}\del{t}q-\chi_\ell q = p+\bigl(\ell-\frac{1}{2}\bigr)\del{t}q - \chi_\ell q \leq 0$ for $0\leq \ell \leq 2\st$, 
and $\del{t}^{2\st}g = k_2^\Sigma\del{\Sigma}\theta_2 +h_2$, it follows
from Theorem \ref{weakthm} that $\del{t}^\ell u$ satisfies the energy estimate
\begin{align}
\Et\bigl(\del{t}^\ell &u(t)\bigr) \leq C\biggl( \Et\bigl(\del{t}^\ell u(0)\bigr) +\norm{\vec{M}_\ell(0)}_{L^2(\Omega)} + \int_0^t\bigl(1+
\norm{b(\tau)}_{L^\infty(\Omega)}+\norm{\del{t}b(\tau)}_{L^\infty(\Omega)}\bigr)\Et\bigl(\del{t}^\ell u(t)\bigr) \notag \\
&+\norm{\vec{M}_\ell(\tau)}_{L^2(\Omega)}+ \norm{\del{t}M_\ell(\tau)}_{L^2(\Omega)}  + \norm{\del{t}^\ell f(\tau)}_{L^2(\Omega)}+
\norm{G_\ell(\tau)}_{H^1(\Omega)} + \norm{\del{t}G_\ell(\tau)}_{L^2(\Omega)} \, d\tau\biggr) \label{LAWthmP20a}
\end{align}
for $0\leq \ell \leq 2\st-1$, 
and
\begin{align}
\Et\bigl(&\del{t}^{2\st} u(t)\bigr) \leq C\biggl( \Et\bigl(\del{t}^{2\st} u(0)\bigr) +\norm{\vec{M}_{2\st}(0)}_{L^2(\Omega)} 
+\norm{k_2(0)}_{H^{s+1-\ep}(\Omega)}\norm{\theta_2(0)}_{H^{1+\ep}(\Omega)} \notag \\
&+\int_0^t\bigl(1+
\norm{b(\tau)}_{L^\infty(\Omega)}+\norm{\del{t}b(\tau)}_{L^\infty(\Omega)}\bigr)\Et\bigl(\del{t}^{2\st} u(t)\bigr) +\norm{\vec{M}_{2\st}(\tau)}_{L^2(\Omega)}
+ \norm{\del{t}M_{2\st}(\tau)}_{L^2(\Omega)} \notag \\
& + \norm{\del{t}^{2\st} f(\tau)}_{L^2(\Omega)}+
\norm{(G_{2\st}-\del{t}^{2\st}p \del{t}h-\del{t}^{2\st}g)(\tau)}_{H^1(\Omega)} + 
\norm{\del{t}(G_{2\st}-\del{t}^{2\st}p \del{t}h-\del{t}^{2\st}g)(\tau)}_{L^2(\Omega)} \notag \\
&+\Bigl((\norm{(k_1\del{t}u)(\tau)}_{H^{\st+\frac{1}{2}}(\Omega)}+\norm{\del{t}(k_2\del{t}h )(\tau)}_{H^{\st}(\Omega)}\Bigr)\Bigl(\norm{\del{t}\theta_1(\tau)}_{H^{s-\st+1}(\Omega)} +
\norm{\theta_1(\tau)}_{H^{s-\st+\frac{3}{2}}(\Omega)}\Bigr)  \notag \\
&+\Bigl(\norm{k_2(\tau)}_{H^{s+\frac{1}{2}}(\Omega)}+\norm{\del{t}k_2(\tau)}_{H^{s}(\Omega)}\Bigr)\Bigl(\norm{\del{t}\theta_2(\tau)}_{H^{1}(\Omega)} +
\norm{\theta_2(\tau)}_{H^{\frac{3}{2}}(\Omega)}\Bigr)  \notag \\
&+\norm{h_2(\tau)}_{H^{1}(\Omega)}+ \norm{\del{t} h_2(\tau)}_{L^{2}(\Omega)} \, d\tau\biggr), \label{LAWthmP20b}
\end{align}
where $C=C\bigl(\kappa_0,\kappa_1,\mu,\gamma,\vec{\chi},\norm{b(0)}_{L^\infty(\Omega)}\bigr)$ and
\begin{equation*} 
\Et\bigl(\del{t}^\ell u(t)\bigr)^2 =  \norm{\del{t}^{\ell}u(t)}_{H^{1}(\Omega)}^2 + \norm{\del{t}^{\ell+1}u(t)}_{L^2(\Omega)}^2 - 
\ip{\del{t}^{\ell+1}u(t)}{q(t)\del{t}^{\ell+1}u(t)}_{\del{}\Omega}.
\end{equation*}

To proceed, we view \eqref{linBa.1}-\eqref{linBa.2} as an elliptic system by expressing it as
\begin{align}
\del{\Lambda}\bigl(b^{\Lambda\Sigma}\del{\Sigma}\del{t}^\ell u+M^\Lambda_\ell+b^{\Lambda 0}\del{t}^{\ell+1}u\bigr) 
&= F_\ell && \text{in $\Omega_T$,}
\label{linBb.1}\\
\nu_\Lambda\bigl(b^{\Lambda\Sigma}\del{\Sigma} \del{t}^\ell u + M^\Lambda_\ell+b^{\Lambda 0}\del{t}^{\ell+1}u) & = 
q\del{t}^{\ell+2} u + 
P_\ell \del{t+1}^\ell u + G_\ell
 && \text{in $\Gamma_T$,} \label{linBb.2}
 \end{align}
where $F_\ell$ is as previously defined by \eqref{DCthm4} and we have used the fact that $\nu_0 = 0$. Elliptic regularity, see Theorem \ref{ellipticregN}, then gives
\begin{align} 
\norm{\del{t}^\ell u}_{H^{\st+1-\frac{\ell}{2}}(\Omega)} \leq C&
\Bigr(\Et\bigl(\del{t}^\ell u\bigr) + \norm{\vec{M}_\ell}_{H^{\st-\frac{\ell}{2}}(\Omega)}+ 
\norm{b^{0\Lambda}\del{t}^{\ell+1} u}_{H^{\st-\frac{\ell}{2}}(\Omega)} +
\norm{F_\ell}_{H^{(\st-\frac{\ell}{2})^*}\!(\Omega)} \notag\\
&
+\norm{q\del{t}^{\ell+2}u}_{H^{\st-\frac{\ell}{2}-\frac{1}{2}}(\del{}\Omega)}+
\norm{P_\ell \del{t}^{\ell+1}u}_{H^{\st-\frac{\ell}{2}-\frac{1}{2}}(\del{}\Omega)}+
\norm{G_\ell}_{H^{\st-\frac{\ell}{2}-\frac{1}{2}}(\del{}\Omega)}
\Bigl),  \label{LAWthmP22}
\end{align}
for $0\leq \ell \leq 2\st-1$,
where $C=C\bigl(\kappa_1,\mu,\norm{b(t)}_{H^s(\Omega)}\bigr)$ and $(\st-\frac{\ell}{2})^*$ is defined by \eqref{r*def}.

\begin{lem} \label{AWhestA}
The following estimates hold:
\begin{itemize}
\item[(i)]

\begin{equation*}
\norm{q\del{t}^{\ell+2} u}_{H^{\st-\frac{\ell}{2}-\frac{1}{2}}(\del{}\Omega)} \lesssim \norm{q}_{H^s(\Omega)}
\norm{\del{t}^{\ell+2}u}_{H^{\st+1-\frac{\ell+2}{2}}(\Omega)}, \quad 0\leq \ell \leq 2\st -2,
\end{equation*}
and
\begin{equation*}
\norm{q\del{t}^{(2\st-1)+2} u}_{H^{\st-\frac{2\st-1}{2}-\frac{1}{2}}(\del{}\Omega)} \leq 
C(\norm{q}_{H^s(\Omega)})\Et\bigl(\del{t}^{2\st}u\bigr)^{\frac{1}{2}},
\end{equation*}
\item [(ii)] and for any $\ep >0$, there exists a constant $c(\ep)>0$ such that
\begin{equation*}
\norm{P_\ell \del{t}^{\ell+1} u}_{H^{\st-\frac{\ell}{2}-\frac{1}{2}}(\del{}\Omega)} \lesssim 
\bigl(\norm{\del{t}q}_{H^s(\Omega)}+\norm{p}_{H^s(\Omega)} \bigr)
\bigl(\ep \norm{u}_{\Ec^{\st+1}}+ c(\ep)\Et\bigl(\del{t}^{\ell+1} u\bigr)^{\frac{1}{2}} \bigr) , \quad 0\leq \ell \leq 2\st -1.
\end{equation*}
\end{itemize}
\end{lem}
\begin{proof}
We begin by noting that the first estimate from part (i) follows directly from the Trace Theorem, see Theorem \ref{trace}, and the multiplication estimates from Theorem \ref{calcpropB}.  The second estimate from part (i) follows is a similar fashion once we observe that
\begin{align*}
\norm{q\del{t}^{(2\st-1)+2} u}_{H^{\st-\frac{2\st-1}{2}-\frac{1}{2}}(\del{}\Omega)} &\lesssim \norm{(-q)^{\frac{1}{2}}}_{H^{s-\frac{1}{2}}(\del{}\Omega)}
\norm{(-q)^{\frac{1}{2}}\del{t}^{(2\st-1)+2} u}_{L^2(\del{}\Omega)}   
&& \text{(since $(-q)=(-q)^{\frac{1}{2}}(-q)^{\frac{1}{2}})$}\\
 &\lesssim \norm{(-q)^{\frac{1}{2}}}_{H^{s-\frac{1}{2}}(\del{}\Omega)}\Et\bigl(\del{t}^{2\st}u\bigr)^{\frac{1}{2}}
&& \text{(since $((-q)^{\frac{1}{2}})^{\tr}=(-q)^{\frac{1}{2}})$} \\
& \leq C(\norm{q}_{H^s(\Omega)})\Et\bigl(\del{t}^{2\st}u\bigr)^{\frac{1}{2}},
\end{align*}
where the last inequality is a consequence of the Trace Theorem (Theorem \ref{trace}), the analyticity of the map $(-q)\mapsto (-q)^{\frac{1}{2}}$, and the Moser estimate 
from Theorem \ref{calcpropC}.

Turning to the estimate from part (ii), we observe from the definition \eqref{LAWthmP20}, the Trace Theorem, and the multiplication estimates that 
\begin{equation} \label{AWhestA1}
\norm{P_\ell \del{t}^{\ell+1} u}_{H^{\st-\frac{\ell}{2}-\frac{1}{2}}(\del{}\Omega)} \lesssim \bigl(\norm{\del{t}q}_{H^s(\Omega)}+\norm{p}_{H^s(\Omega)} \bigr)\norm{\del{t}^{\ell+1} u}_{H^{\st-\frac{\ell}{2}}(\Omega)}, \quad 0\leq \ell \leq 2\st-2,
\end{equation}
and
\begin{align}
\norm{P_{2\st-1} \del{t}^{2\st} u}_{H^{\st-\frac{2\st-1}{2}-\frac{1}{2}}(\del{}\Omega)} &\lesssim \bigl(\norm{\del{t}q}_{H^{s-\frac{1}{2}}(\del{}\Omega)}+
\norm{p}_{H^{s-\frac{1}{2}}(\del{}\Omega)} \bigr)\norm{\del{t}^{2\st} u}_{L^2(\del{}\Omega)} \notag \\
& \lesssim \bigl(\norm{\del{t}q}_{H^s(\Omega)}+\norm{p}_{H^s(\Omega)} \bigr)\norm{\del{t}^{2\st} u}_{H^1(\Omega)} \notag \\
& \lesssim \bigl(\norm{\del{t}q}_{H^s(\Omega)}+\norm{p}_{H^s(\Omega)} \bigr)\Et\bigl(\del{t}^{2\st} u\bigr)^{\frac{1}{2}} \label{AWhestA2}
\end{align}
With the help of Ehrling's lemma (Lemma \ref{Ehrling}), we deduce from the inequality \eqref{AWhestA1} that
\begin{align}
\norm{P_\ell \del{t}^{\ell+1} u}_{H^{\st-\frac{\ell}{2}-\frac{1}{2}}(\del{}\Omega)} &\lesssim \bigl(\norm{\del{t}q}_{H^s(\Omega)}+\norm{p}_{H^s(\Omega)} \bigr)
\bigl(\ep \norm{\del{t}^{\ell+1} u}_{H^{\st+1-\frac{\ell+1}{2}}(\Omega)}+ C(\ep)\norm{\del{t}^{\ell+1}u}_{H^1(\Omega)}\bigr), \notag \\
 &\lesssim 
\bigl(\norm{\del{t}q}_{H^s(\Omega)}+\norm{p}_{H^s(\Omega)} \bigr)
\bigl(\ep \norm{u}_{\Ec^{\st+1}}+ c(\ep)\Et\bigl(\del{t}^{\ell+1} u\bigr)^{\frac{1}{2}} \bigr)
\label{AWhestA3}
\end{align}
for $0\leq \ell \leq 2 \st-2$ any $\ep >0$. Taken together, \eqref{AWhestA2} and \eqref{AWhestA3} show that the inequality from part (ii) holds, and the proof is complete. 
\end{proof}

From the above lemma and the elliptic estimates \eqref{LAWthmP22}, we see that
\begin{align}
\norm{\del{t}^\ell u}_{H^{\st+1-\frac{\ell}{2}}(\Omega)} &\leq C\Bigl(\norm{\del{t}^{\ell+2} u}_{H^{\st+1-\frac{\ell+2}{2}}(\Omega)} + \ep \norm{u}_{\Ec^{\st+1}}
+ r_\ell\Bigr), \quad 0\leq \ell \leq 2\st -2, \label{LAWthmP23}
\intertext{and}
\norm{\del{t}^{2\st-1} u}_{H^{\frac{3}{2}}(\Omega)} &\leq C \bigl(\Et\bigl(\del{t}^{2\st}u\bigr)+r_{2\st-1}\bigr), \label{LAWthmP24}
\end{align}
where 
\begin{equation} \label{LAWthmP25}
C= C\bigl(\kappa_1,\mu,\norm{b(t)}_{H^s(\Omega)},\norm{q(t)}_{H^s(\Omega)},\norm{\del{t}q(t)}_{H^s(\Omega)},\norm{p(t)}_{H^s(\Omega)}\bigr)
\end{equation}
and
\begin{equation*}
r_\ell = \Et\bigl(\del{t}^\ell u\bigr) + C(\ep)\Et\bigl(\del{t}^{\ell+1} u\bigr) +\norm{\vec{M}_\ell}_{H^{\st-\frac{\ell}{2}}(\Omega)}+ 
\norm{b^{0\Lambda}\del{t}^{\ell+1} u}_{H^{\st-\frac{\ell}{2}}(\Omega)} +
\norm{F_\ell}_{H^{(\st-\frac{\ell}{2})^*}\!(\Omega)} +
\norm{G_\ell}_{H^{\st-\frac{\ell}{2}-\frac{1}{2}}(\del{}\Omega)}.
\end{equation*}
We then deduce from the estimates \eqref{LAWthmP23}-\eqref{LAWthmP24} via a simple induction argument that
\begin{align*}
\norm{\del{t}^\ell u}_{H^{\st+1-\frac{\ell}{2}}(\Omega)} \leq &C\biggl( \ep \norm{u}_{\Ec^{\st+1}} + c(\ep)\sum_{m=0}^{2\st} \Et\bigl(\del{t}^m u\bigr) + \sum_{m=0}^{2\st-1}
\Bigl(\norm{\vec{M}_m}_{H^{\st-\frac{m}{2}}(\Omega)} \\
&+ 
\norm{b^{0\Lambda}\del{t}^{m+1} u}_{H^{\st-\frac{m}{2}}(\Omega)} +
\norm{F_m}_{H^{(\st-\frac{m}{2})^*}\!(\Omega)}+
\norm{G_\ell}_{H^{\st-\frac{\ell}{2}-\frac{1}{2}}(\del{}\Omega)}\Bigr)\biggr)
\end{align*}
for  $0\leq \ell \leq 2\st -1$, where the constant $C$ is of the form \eqref{LAWthmP25}. Summing the above inequality over $\ell$ and choosing $\epsilon$ small enough shows that
 \begin{equation} \label{LAWthmP26}
\norm{u}_{E^{\st+1,2\st-1}} \leq C\biggl(\sum_{m=0}^{2\st} \Et\bigl(\del{t}^m u\bigr)+ \Fc \biggr),
\end{equation}
where $C$ is again of the form \eqref{LAWthmP25} and
\begin{equation*}
\Fc = \sum_{m=0}^{2\st-1}
\Bigl(\norm{\vec{M}_m}_{H^{\st-\frac{m}{2}}(\Omega)}+
\norm{b^{0\Lambda}\del{t}^{m+1} u}_{H^{\st-\frac{m}{2}}(\Omega)} +
\norm{F_m}_{H^{(\st-\frac{m}{2})^*}\!(\Omega)}+
\norm{G_\ell}_{H^{\st-\frac{\ell}{2}-\frac{1}{2}}(\del{}\Omega)}\Bigr).
\end{equation*}

\begin{lem} \label{AWhestB} 
The following estimates hold:
\begin{itemize}
\item[(i)]
\begin{align*}
\norm{G_\ell}_{H^{\st-\frac{\ell}{2}-\frac{1}{2}}(\del{}\Omega)}+
\norm{\del{t}G_\ell}_{H^{\st-\frac{\ell}{2}-\frac{1}{2}}(\del{}\Omega)}
\lesssim \sigma \norm{u}_{\Ec^{\st+1}} + \norm{g}_{E^{\st,2\st-2}}+\norm{\del{t}g}_{E^{\st,2\st-2}}
\end{align*}
for  $0\leq \ell \leq 2\st -2$,
where 
\begin{equation*}
\sigma =  
\norm{p}_{E^{s,2\st-2}}+
\norm{\del{t}p}_{E^{s,2\st-2}}+
\norm{q}_{E^{s,2\st-2}}+
\norm{\del{t}q}_{E^{s,2\st-2}},
\end{equation*}
\item[(ii)]
\begin{align*}
\norm{\del{t}^{2\st}p\del{t}u}_{L^2(\del{}\Omega)} & \lesssim \Bigl(\norm{k_1 }_{H^{s+\frac{1}{2}}(\Omega)}\norm{\theta_1}_{H^{s-\st+\frac{3}{2}}(\Omega)}+\norm{h_1}_{H^{s-\st+1}(\Omega)}\Bigr)
\norm{u}_{\Ec^{\st+1}}
\intertext{and}
\norm{\del{t}^{2\st}g}_{L^2(\del{}\Omega)} &\lesssim \norm{k_2}_{H^{s+\frac{1}{2}}(\Omega)}\norm{\theta_2}_{H^{\frac{3}{2}}(\Omega)}+ \norm{h_2}_{H^1(\Omega)},
\end{align*}
\item[(iii)]
\begin{gather*}
\norm{G_\ell-\del{t}^\ell g}_{H^1(\Omega)}+\norm{\del{t}(G_{\ell}-\del{t}^\ell g)}_{L^2(\Omega)}\lesssim \beta \norm{u}_{\Ec^{\st+1}}, 
\quad 0\leq \ell \leq 2\st-1,
\intertext{and}
\norm{G_{2\st}-\del{t}^{2\st}p\del{t}u-\del{t}^{2\st}g}_{H^1(\Omega)}+\norm{\del{t}(G_{2\st}-\del{t}^{2\st}p\del{t}u-\del{t}^{2\st}g)}_{L^2(\Omega)} 
\lesssim \beta \norm{u}_{\Ec^{\st+1}},
\end{gather*}
where
\begin{align*}
\beta = &\norm{\del{t}p}_{E^{s,2\st-2}}+
\norm{\del{t}q}_{E^{s+\frac{1}{2},2\st-1}}+\norm{\del{t}q}_{E^{s,2\st}} \notag \\
&+ \norm{k_1}_{H^{s+\frac{1}{2}}(\Omega)} \norm{\theta_1}_{H^{s-\st+\frac{3}{2}}(\Omega)}
+ \norm{h_1}_{H^{s-\st+1}(\Omega)}+ \norm{\del{t} h_1}_{H^{s-\st}(\Omega)},
\end{align*}
\item[(iv)] and
\begin{align*}
\norm{k_1 \del{t}u}_{H^{\st+\frac{1}{2}}(\Omega)} &\lesssim \norm{k_1}_{H^{s+\frac{1}{2}}(\Omega)}\norm{u}_{\Ec^{\st+1}}
\intertext{and}
\norm{k_1 \del{t}u}_{H^{\st+\frac{1}{2}}(\Omega)} &\lesssim \bigl( \norm{k_1}_{H^{s+\frac{1}{2}}(\Omega)}+ \norm{\del{t}k_1}_{H^{s}(\Omega)}\bigr)\norm{u}_{\Ec^{\st+1}}.
\end{align*}
\end{itemize}
\end{lem}
\begin{proof}
We will only prove statements (ii) and (iii). Statements (i) and (iv) can be verified using similar arguments.

\bigskip

\noindent\textbf{(ii)}:
Letting $\perp^\Sigma_\Lambda = \delta^\Sigma_\Lambda - \nu^\Sigma \nu_\Lambda$, $\nu^\Sigma = \delta^{\Sigma\Omega}\nu_\Omega$, denote the orthogonal projection onto the the subspace orthogonal to the normal vector $\nu_\Lambda$,
we set $\Dsl_\Sigma = \perp_\Sigma^\Lambda\del{\Lambda}$,
which defines a complete collection of derivatives tangent to $\del{}\Omega$. This allows us to write
$k^\Sigma_a\del{\Sigma} = k^\Sigma_a \Dsl_{\Sigma}$, $a=1,2$,
since
$k^\Sigma_a = \perp^\Sigma_\Lambda k^\Lambda_a$ by assumption. Then, from the decomposition $\del{t}^{2\st}p =  k_1^\Sigma\del{\Sigma}\theta_1+h_1$,
the Sobolev and multiplication inequalities (Theorems \ref{FSobolev} and \ref{calcpropB}), and the Trace Theorem (Theorem \ref{trace}), we see that
\begin{align*}
\norm{\del{t}^{2\st}p \del{t}u}_{L^2(\del{}\Omega)} &\leq \norm{k_1^\Sigma\Dsl_{\Sigma}\theta_1 \del{t}u}_{L^2(\del{}\Omega)}+\norm{h_1\del{t}u}_{L^2(\del{}\Omega)} \\
&\lesssim  \bigl(\norm{k_1 \Dsl\theta_1}_{H^{s-\st}(\del{}\Omega)}+\norm{h_1}_{H^{s-\st}(\del{}\Omega)}\bigr)\norm{\del{t}u}_{H^{\st}(\del{}\Omega)} \\
&\lesssim \bigl(\norm{k_1 }_{H^{s}(\del{}\Omega)}\norm{\theta_1}_{H^{s-\st+1}(\del{}\Omega)}+\norm{h_1}_{H^{s-\st}(\del{}\Omega)}\bigr)\norm{\del{t}u}_{H^{\st}(\del{}\Omega)} \\
&\lesssim \Bigl(\norm{k_1 }_{H^{s+\frac{1}{2}}(\Omega)}\norm{\theta_1}_{H^{s-\st+\frac{3}{2}}(\Omega)}+\norm{h_1}_{H^{s-\st+1}(\Omega)}\Bigr)
\norm{\del{t}u}_{H^{\st+\frac{1}{2}}(\Omega)},
\end{align*}
where in deriving this we have used that $\st \geq \frac{1}{2}$ by assumption. Using similar arguments,
it is not difficult to verify  from the
decomposition $\del{t}^{2\st}g =  k_2^\Sigma\del{\Sigma}\theta_2+h_2$ that
\begin{equation*}
\norm{\del{t}^{2\st}g}_{L^2(\del{}\Omega)} \lesssim \norm{k_2}_{H^{s+\frac{1}{2}}(\Omega_2)}\norm{\theta}_{H^{\frac{3}{2}}(\Omega)}+ \norm{h_2}_{H^1(\Omega)}.
\end{equation*}

\end{proof}

From Lemma \ref{AWhestB}, the Trace Theorem (Theorem \ref{trace}), and the integral representation $G_\ell(t) = G_\ell(0)+\int_0^t \del{t}G_\ell(\tau)\, d\tau$, we observe that $G_\ell$ can be estimated by
\begin{equation} \label{LAWthmP27}
\norm{G_\ell(t)}_{H^{\st-\frac{\ell}{2}-\frac{1}{2}}(\Omega)} \lesssim \alpha_2(0) + \int_{0}^t \alpha_1(\tau)\norm{u(\tau)}_{\Ec^{\st+1}} +  \alpha_2(\tau) \, d\tau,
\quad 0 \leq \ell \leq 2\st-1,
\end{equation}
where
\begin{align*}
\alpha_1(t)  =& 1+\norm{b(t)}_{H^{s}(\Omega)}+\norm{\del{t}b(t)}_{E^{s,2\st}}
 + \norm{q(t)}_{H^{s}(\Omega)}+
\norm{\del{t}q}_{E^{s+\frac{1}{2},2\st-1}}+\norm{p(t)}_{H^{s}(\Omega)}
 +\norm{\del{t}p}_{E^{s,(2\st-2)^*}} \\
&+ \Bigl(\norm{k_1}_{H^{s+\frac{1}{2}}(\Omega)}+ \norm{\del{t}k_1}_{H^{s}(\Omega)}\Bigr)\Bigl(\norm{\del{t}\theta_1(t)}_{H^{s-\st+1}(\Omega)} +
\norm{\theta_1(t)}_{H^{s-\st+\frac{3}{2}}(\Omega)}\Bigr)\\
&+ \norm{h_1}_{H^{s-\st+1}(\Omega)}+ \norm{\del{t} h_1}_{H^{s-\st}(\Omega)}
\intertext{and}
\alpha_2(t) =& \norm{m(t)}_{H^{\st}(\Omega)}+\norm{\del{t}m(t)}_{E^{\st}}+\norm{f(t)}_{H^{(\st-1)^*}(\Omega)} +
 \norm{\del{t}f(t)}_{E^{(\st-1)*}}+\norm{\del{t}^{2\st}f(t)}_{L^2(\Omega)} 
+\norm{g(t)}_{H^{\st}(\Omega)}\\
&+\norm{\del{t}g(t)}_{E^{\st,(2\st-2)^*}}+ \Bigl(\norm{k_2(t)}_{H^{s+\frac{1}{2}}(\Omega)}+\norm{\del{t}k_2(t)}_{H^{s}(\Omega)}\Bigr)
\Bigl(\norm{\del{t}\theta_2(t)}_{H^{1}(\Omega)} +
\norm{\theta_2(t)}_{H^{\frac{3}{2}}(\Omega)}\Bigr)\\
&+ \norm{h_2(t)}_{H^{1}(\Omega)}+ \norm{\del{t} h_2(t)}_{L^{2}(\Omega)}.
\end{align*}
Using \eqref{LAWthmP27} in conjuction with  \eqref{DCthm6}, \eqref{DCthm7} and \eqref{DCthm8}, it then clear that $\Fc$ satisfies the estimate
\begin{equation} \label{LAWthmP28}
\Fc(t)  \lesssim \alpha_2(0) + \int_{0}^t \alpha_1(\tau)\norm{u(\tau)}_{\Ec^{\st+1}} +  \alpha_2(\tau) \, d\tau,
\end{equation}
while we note, for $0\leq \ell \leq 2\st$, that
\begin{equation}\label{LAWthmP29}
\Et\bigl(\del{t}^{\ell} u(t)\bigr) \leq C\biggl( \nnorm{(u(0),\del{t}^{2\st+1}u(0)|_{\del{}\Omega})}_{\st+1} + \alpha_2(0) + \int_{0}^t \alpha_1(\tau)\norm{u(\tau)}_{\Ec^{\st+1}} +  \alpha_2(\tau) \, d\tau\biggr),
\end{equation}
where $C=C\bigl(\kappa_0,\kappa_1,\mu,\gamma,\vec{\chi},\norm{b(0)}_{L^\infty(\Omega)},\norm{q(0)}_{L^\infty(\Omega)}\bigr)$, is a direct consequence of \eqref{LAWthmP20a}, \eqref{LAWthmP20b},  Lemma \ref{DCthm} and Lemma \ref{AWhestB}.  
The energy estimate
\begin{align*}
\nnorm{(u(t),\del{t}^{2\st+1}u(t)|_{\del{}\Omega})}_{\st+1} \leq
C &\biggl( \nnorm{(u(0),\del{t}^{2\st+1}u(0)|_{\del{}\Omega})}_{\st+1} \\
 &+\alpha_2(0)+
 \int_0^t  \alpha_1(\tau)\nnorm{(u(\tau),U(\tau))}_{\st+1}+\alpha_2(\tau)\,d\tau \biggr),
\end{align*}
where $C = C\bigl(\kappa_0,\kappa_1,\mu,\gamma,\vec{\chi},\alpha_3(t)\bigr)$ and
\begin{equation*}
\alpha_3(t) = 1+\norm{b(0)}_{H^{s}(\Omega)}+ \norm{b(t)}_{H^{s}(\Omega)}+ \norm{p(t)}_{H^{s}(\Omega)}+  \norm{q(t)}_{H^{s}(\Omega)}+ \norm{\del{t}q(t)}_{H^{s}(\Omega)},
\end{equation*}
then follows directly from the inequalities \eqref{LAWthmP26}, \eqref{LAWthmP28} and \eqref{LAWthmP29}.
 
We have, so far, established the existence of solutions that satisfy the required energy estimate under the assumption that the coefficients are smooth and that initial data satisfies the compatibility conditions to order $2s+3$.
Existence for the general case, where the coefficients satisfy the hypotheses of the theorem and the initial conditions satisfy the compatibility
conditions to order $2\st+1$, 
follows from an approximation and limiting argument, the details of which we leave to the interested reader. 
Furthermore,  the uniqueness of solutions follows
directly from the energy estimate in the usual way.
\end{proof}

\section{Index of notation} \label{index}

\bigskip

\begin{longtable}{ll}
$M$ & ambient spacetime manifold; \S \ref{introIBVP}\\
$\Omega_T$ & spacetime fluid domain; \S \ref{introIBVP}\\
$g_{\mu\nu}$ & Lorentzian spacetime metric; \S \ref{introIBVP} \\
$p$ & fluid pressure; \S \ref{intro} \& \S \ref{waveEul}; eqns.~\eqref{prel.1}-\eqref{zpfix} \\
$\rho$ & fluid proper energy density; \S \ref{intro} \& \S \ref{waveEul}, eqn.~\eqref{theta2rho}\\
$s^2$ & square of the fluid sound speed; \S \ref{intro}, eqn.~\eqref{s2}\\
$v^\mu$ & fluid 4-velocity; \S \ref{intro} \& \S \ref{waveEul}, eqns. \eqref{xihdef} \& \eqref{theta2v}\\
$h_{\mu\nu}$ & induced Riemannian metric on subspace orthogonal to $v^\mu$; \S \ref{intro}, eqn.~\eqref{pidef}\\
$\del{\mu}$ & partial derivatives, \S \ref{pderivatives} \\
$\del{}$, $\delsl{}$, $D$ & gradients, \S \ref{pderivatives} \\
$\del{}^{|k|}$, $\delsl{}^{|k|}$, $D^{|k|}$ & derivatives of order $\leq k$, \S \ref{pderivatives}\\
$\Rf$, $\Qf$, \ldots & generic constraint terms; \S \ref{constnot}\\
$\ipe{\cdot}{\cdot}$, $|\cdot|$ & Euclidean inner-product and norm; \S \ref{norms}\\
$|\cdot|_g$, $|\cdot|_m$ & tensor norms; \S \ref{norms}\\
$\ip{\cdot}{\cdot}$ & $L^2$ inner-product; \S \ref{spatialFS}\\
$W^s(\Omega,V)$, $H^s(\Omega,V)$ & Sobolev spaces; \S \ref{spatialFS}\\
$X^{s,r}_T(\Omega,V)$ & spacetime function space; \S \ref{spacetimeFS}, eqn.~\eqref{XTdefA}\\
$X^s_T(\Omega,V)$ & spacetime function space; \S \ref{spacetimeFS}, eqn.~\eqref{XTdefB}\\
$\Xc^s_T(\Omega,V)$ & spacetime function space; \S \ref{spacetimeFS}, eqns.~\eqref{XcTdefB}\\
$\mathring{X}^{s,r}_T(\Omega,V)$ & spacetime function space; \S \ref{spacetimeFS}, eqns.~\eqref{Xrdef}\\
$\Xcr^s_T(\Omega,V)$ & spacetime function space; \S \ref{spacetimeFS}, eqns.~\eqref{Xcrdef}\\
$m(s,\ell)$ &indexing function; \S \ref{spacetimeFS}, eqn.~\eqref{melldef}\\
$\norm{\cdot}_{E^{s,r}}$ & energy norm; \S \ref{spacetimeFS} \\
$\norm{\cdot}_{E^s}$ & energy norm; \S \ref{spacetimeFS} \\
$\norm{\cdot}_{\Ec^s}$ & energy norm; \S \ref{spacetimeFS} \\
$\norm{\cdot}_{X_T^{s,r}}$ & spacetime energy norm; \S \ref{spacetimeFS} \\
$\norm{\cdot}_{X_T^s}$ & spacetime energy norm; \S \ref{spacetimeFS} \\
$\norm{\cdot}_{\Xc^s}$ & spacetime energy norm; \S \ref{spacetimeFS} \\
$\zeta$, $\thetah^0_\mu$ & primary fields; \S \ref{primaux}\\
$\theta^0_\mu$ & timelike co-frame field; \S \ref{primaux}, eqn.~\eqref{theta0def}\\
$\theta^I_\mu$ & spatial co-frame (auxiliary fields); \S \ref{primaux} \\
$e_j^\mu$ & frame dual to $\theta^i_\nu$; \S \ref{primaux}, eqn.~\eqref{framedef}\\
$\gamma_{ij}$ & frame metric; \S  \ref{primaux} (see also Appendix \ref{DG}, eqn.~\eqref{DGframemet})\\
$\omega_{i}{}^k{}_j$ & connection coefficients; \S  \ref{primaux} (see also Appendix \ref{DG}, eqn.~\eqref{DGconndefA})\\
$\sigma_i{}^k{}_j$ & auxiliary fields; \S \ref{primaux} (see also \S \ref{consEul}, eqns.~\eqref{ccon}-\eqref{dcon})\\
$\xi^\mu$ & \S \ref{primaux}, eqn.~\eqref{xidef}\\
$\ac$ & bulk constraint, \S \ref{consEul}, eqn.~\eqref{acon}\\
$\bc^J$ & bulk constraint, \S \ref{consEul}, eqn.~\eqref{bcon}\\
$\cc^k_{ij}$ & bulk constraint, \S \ref{consEul}, eqn.~\eqref{ccon}\\
$\dc^k_j$ & bulk constraint, \S \ref{consEul}, eqn.~\eqref{dcon}\\
$\ec^K$ & bulk constraint, \S \ref{consEul}, eqn.~\eqref{econ}\\
$\Fc$ & bulk constraint, \S \ref{consEul}, eqn.~\eqref{fcon}\\
$\gc$ & bulk constraint, \S \ref{consEul}, eqn.~\eqref{gcon}\\
$\hc$ & bulk constraint, \S \ref{consEul}, eqn.~\eqref{hcon}\\
$\jc$ & bulk constraint, \S \ref{consEul}, eqn.~\eqref{jcon}\\
$\chi$ & collection of bulk constraints, \S \ref{consEul}, eqn.~\eqref{chidef}\\
$f(\lambda)$ & \S \ref{consEul}, eqn.~\eqref{fdef}\\
$\gammah^{00}$ & \S \ref{consEul}, eqn.~\eqref{gammah00}\\
$\kc$ & boundary constraint, \S \ref{consEul}, eqn.~\eqref{kcon}\\
$\thetah^3$ &\S \ref{consEul}, eqn.~\eqref{thetah3def}\\
$\Wc^{\alpha\mu}$ & \S \ref{wavesysEul}, eqn.~\eqref{Wcdef}\\
$m^{\alpha\beta}$ & Riemannian spacetime metric; \S \ref{wavesysEul}, eqn.~\eqref{mdef}\\
$a^{\alpha\beta}$ & acoustic metric; \S \ref{wavesysEul}, eqn.~\eqref{adef}\\
$\Ec^\mu$ & equations of motion for $\thetah^0_\mu$; \S \ref{wavesysEul}, eqn.~\eqref{Ecaldef}\\
$\Hc^\mu$ &  \S \ref{wavesysEul}, eqn.~\eqref{Hbrdef}\\
$\breve{\Hc}_\nu$ &  \S \ref{wavesysEul}, eqn.~\eqref{Hdef}\\
$C^\lambda_{\alpha\beta}$ &  \S \ref{wavesysEul}, eqn.~\eqref{Cdef}\\
$\Fc_{\alpha\beta}$ &  Electromagnetic field tensor associated to $\theta^0_\mu$; \S \ref{wavesysEul}, eqn.~\eqref{fcomp} (see also \eqref{fcon})\\
$\Kc$ &  \S \ref{wavesysEul}, eqn.~\eqref{Kcdef}\\
$\Bc^\mu$  & boundary conditions for $\thetah^0_\mu$; \S \ref{wavebcs}, eqn.~\eqref{Bcdef}\\
$\ell^\mu$ &  \S \ref{wavebcs}, eqn.~\eqref{elldef}\\
$\Lc$ &  \S \ref{wavebcs}, eqn.~\eqref{Lcdef}\\
$|\gamma|$ & determinant of the frame metric; Appendix \ref{DG}, eqn.~\eqref{detgammadef}\\
$\Upsilon^\mu_\omega$ &  \S \ref{wavebcs}, eqn.~\eqref{Upsilondef}\\
$\psi_\nu$ &  \S \ref{wavebcs}, eqn.~\eqref{psidef}\\
$\psih_\nu$ &  \S \ref{wavebcs}, eqn.~\eqref{psihdef}\\
$p^\mu_\nu$ &  projection onto the subspace $g$-orthogonal to $\psih^\mu$; \S \ref{wavebcs}, eqn.~\eqref{pdef} \\
$\Pi^\mu_\nu$ & projection onto the subspace $g$-orthogonal to $\psih^\mu$ and $v^\mu$; \S \ref{wavebcs}, eqn.~\eqref{Pidef}\\
$\nu_{\nu\alpha\beta\gamma}$ &  volume form associate to $g_{\mu\nu}$; \S \ref{wavebcs}, eqn.~\eqref{nudef}\\
$N_\nu$ &  \S \ref{wavebcs}, eqn.~\eqref{Ndef}\\
$s_{\mu\nu}{}^\gamma$ &  \S \ref{wavebcs}, eqn.~\eqref{Sdef}\\
$\Pibr^\mu_\nu$ & complementary projection to $\Pi^\mu_\nu$; \S \ref{wavebcs}, eqn.~\eqref{Pibrdef}\\
$x^\mu=\phi^\mu(\xb^\lambda)$ & Lagrangian coordinates; \S \ref{constprop}, eqn.~\eqref{lemBa.1a} (see also \S \ref{lag})\\
$\hbr$ & \S \ref{constprop}, eqn.~\eqref{lemC.8b}\\
$Y_{\Kc}$  & \S \ref{constprop}, eqn.~\eqref{Ydef}\\
$|\cdot|_{\Lambda}$  & \S \ref{constprop}, eqn.~\eqref{YLnorm}\\
$\Lambda^{\Kc\Lc}$  & \S \ref{constprop}, eqn.~\eqref{Lambdadef}\\
$\Ecr^\mu$  & the velocity differentiated equations of motion for $\thetah^0_\mu$; \S \ref{choice}, eqn.~\eqref{Ecrdef}\\
$\Bcr^\mu$  & the velocity differentiated boundary conditions for $\thetah^0_\mu$; \S \ref{choice}, eqn.~\eqref{Bcrdef}\\
$\ep$  & \S \ref{choice}, eqns.~\eqref{epfix} \& \eqref{epformula}\\
$\ellt^\mu$  & \S \ref{choice}, eqn.~\eqref{elltdef}\\
$\Upsilonch^\mu_\nu$  & inverse of $\Upsilon^\mu_\nu$; \S \ref{choice}, eqn.~\eqref{Upsilonchdef}\\
$J^\mu_\nu$ & Jacobian matrix of the Lagrangian map $\phi$; \S \ref{lag}, eqn.~\eqref{Jdef}\\
$\Jch^\mu_\nu$ & inverse Jacobian matrix of the Lagrangian map $\phi$; \S \ref{lag}, eqn.~\eqref{Jdef}\\
$\underline{f}$ & the pull-back of a scalar function $f$ by $\phi$; \S \ref{lag}, eqn.~\eqref{ful}\\
$\overline{Q}^{\mu_1\ldots \mu_r}_{\nu_1\ldots\nu_s}$ &  the pull-back of a tensor $Q^{\mu_1\ldots \mu_r}_{\nu_1\ldots\nu_s}$ by $\phi$; \S \ref{lag}\\
$\Hct^\mu$ & \S \ref{wformlag}, eqn.~\eqref{Hctdef}\\
$r$ & \S \ref{wformlag}, eqn.~\eqref{rfix}\\
$\vartheta_\nu$ &  $\nabla_v\thetah^0_\nu$ evaluated in Lagrangian coordinates;  \S \ref{wformlag}, eqn.~\eqref{varthetadef}\\
$\Ic^\mu$ & \S \ref{wformlag}, eqn.~\eqref{Icdef}\\
$\Asc^{\alpha\beta\mu\nu}$ & \S \ref{wformlag}, eqn.~\eqref{Ascdef}\\
$\Xsc^{\alpha\mu}$ & \S \ref{wformlag}, eqn.~\eqref{Xscdef}\\
$\Hsc^\mu$ & \S \ref{wformlag}, eqn.~\eqref{Hscdef}\\
$\Psi_\nu$ & \S \ref{wformlag}, eqn.~\eqref{Psidef}\\
$\Ssc^{\mu\nu\gamma}$ & \S \ref{wformlag}, eqn.~\eqref{Sscdef}\\
$\Psc^{\mu\nu}$ & \S \ref{wformlag}, eqn.~\eqref{Pscdef}\\
$\Gsc^\mu$  & \S \ref{wformlag}, eqn.~\eqref{Gscdef}\\
$\Msc^{\alpha\beta}$  & \S \ref{wformlag}, eqn.~\eqref{Mscdef}\\
$\Ksc$  & \S \ref{wformlag}, eqn.~\eqref{Kscdef}\\
$\Qsc^{\mu\nu}$  & \S \ref{wformlag}, eqn.~\eqref{Qscdef}\\
$\Rsc^{\mu}_{\nu}$  & \S \ref{wformlag}, eqn.~\eqref{Rscdef}\\
$\alpha^{\mu\nu}$  & \S \ref{wformlag}, eqn.~\eqref{alphadef}\\
$\betat^{\mu}_\nu$  & \S \ref{wformlag}, eqn.~\eqref{betatdef}\\
$\lambda$  & \S \ref{wformlag}, eqn.~\eqref{lambdadef}\\
$\beta^{\mu}_\nu$  & \S \ref{wformlag}, eqn.~\eqref{betadef}\\
$\Bsc^{\alpha\beta\mu\nu}$  & \S \ref{wformlag}, eqn.~\eqref{Bscdef}\\
$\Fsc^{\mu}$  & \S \ref{wformlag}, eqn.~\eqref{Fscdef}\\
$E^\mu_i$ & orthonormal frame; \S \ref{tdiffIBVP}, (see, in particular,  eqn.~\eqref{E0E3def})\\
$\Theta_\mu^i$ & dual frame to $E^\mu_i$; \S \ref{tdiffIBVP} \\
$\Pbb_0$, $\Pbb_3$, $\Pbb$, $\Pbbbr$ & constant projection matrices; \S \ref{tdiffIBVP}, eqns.~\eqref{Pbb3def}-\eqref{Pbb0def}\\
$q$ & \S \ref{tdiffIBVP}, eqn.~\eqref{qdef} \\
$\Qtt$ & \S \ref{tdiffIBVP}, eqn.~\eqref{Qttdef} \\
$\Btt^{\alpha\beta}$ & \S \ref{tdiffIBVP}, eqn.~\eqref{Bttdef} \\
$\Xtt^\alpha$ & \S \ref{tdiffIBVP}, eqn.~\eqref{Xttdef} \\
$\Ftt$ & \S \ref{tdiffIBVP}, eqn.~\eqref{Fttdef} \\
$\Uc$ & \S \ref{tdiffIBVP} (see also eqns.~\eqref{gubnd}-\eqref{zubnd}) \\
$\Ett_0$, $\Ett_1$ & \S \ref{tdiffIBVP}, eqn.~\eqref{dttE} \\
$\Ptt$ & \S \ref{tdiffIBVP}, eqn.~\eqref{Pttdef} \\
$\Gtt$ & \S \ref{tdiffIBVP}, eqn.~\eqref{Gttdef} \\
$\Ktt$ & \S \ref{tdiffIBVP}, eqn.~\eqref{Kttdef} \\
$\Ytt^\alpha$ & \S \ref{tdiffIBVP}, eqn.~\eqref{Yttdef} \\
$\Ztt$ & \S \ref{tdiffIBVP}, eqn.~\eqref{Zttdef} \\
$\Nc$ & \S \ref{coefsmooth} (see Lemma \ref{coeflemA})\\
$\Yc^{s}_T$  & \S \ref{Lwfsec}, eqn.~\eqref{YcsTdef}\\
\end{longtable}

\bigskip

\bibliographystyle{amsplain}
\bibliography{cpropB}

\end{document}